\documentclass{amsbook}

\newtheorem{theorem}{Theorem}[chapter]
\newtheorem{lemma}[theorem]{Lemma}
\newtheorem{proposition}[theorem]{Proposition}

\theoremstyle{definition}
\newtheorem{definition}[theorem]{Definition}
\newtheorem{example}[theorem]{Example}

\newtheorem{conventions}[theorem]{Conventions}

\theoremstyle{remark}
\newtheorem{remark}[theorem]{Remark}

\usepackage{hyperref}
\usepackage{color}
\usepackage{graphicx}
\usepackage{pb-diagram,pb-xy}
\usepackage{amsmath,amsfonts,amssymb}
\usepackage{latexsym,mathrsfs,amsthm, amscd, url}
\usepackage{psfrag}
\usepackage[all,graph]{xy}
\xyoption{import}

\def\R{\mathbb{R}}
\def\Z{\mathbb{Z}}

\def\Q{{\mathbb{Q}}}
\def\C{{\mathcal{C}}}

\def\CP1{\mathbb{CP}^1}

\def\wt{\widetilde}
\def\ol{\overline}
\def\a{\alpha}

\def\g{\gamma}
\def\G{\Gamma}
\def\zp{\mathbb{Z}[\pi]}
\def\zpdag{\Z[\pi^{\dag}]}
\def\zpd{\Z[\pi']}
\def\hp{h_{\partial}}
\def\NG{\mathcal{N}\Gamma}
\def\K{\mathcal{K}}
\def\Ups{\Upsilon}
\def\zh{\Z[\Z \ltimes H]}
\def\zhd{\Z[\Z \ltimes H']}
\def\zhdag{\Z[\Z \ltimes H^{\dag}]}
\def\zhddag{\Z[\Z \ltimes H^{\ddag}]}
\def\zhpc{\Z[\Z \ltimes H^\%]}
\def\COT{Cochran-Orr-Teichner }
\def\COTN{Cochran-Orr-Teichner}
\def\CHL{Cochran-Harvey-Leidy }
\def\CG{Casson-Gordon }

\def\UG{\mathcal{U}\Gamma}

\def\zpx{\Z[\pi_1(X)]}
\def\px{\pi_1(X)}
\def\la{\lambda}
\def\P{\mathcal{P}}
\def\Y{\mathcal{Y}}
\def\ac2{\mathcal{AC}_2}

\def\ba{\begin{array}}
\def\ea{\end{array}}
\def\bn{\begin{enumerate}}
\def\en{\end{enumerate}}

\def\toiso{\xrightarrow{\simeq}}

\def\mc{\mathcal}

\newcommand{\eps}{\varepsilon}

\DeclareMathOperator{\Ext}{Ext}
\DeclareMathOperator{\Hom}{Hom}

\DeclareMathOperator{\im}{im}

\DeclareMathOperator{\cl}{cl}

\DeclareMathOperator{\Wr}{Wr}

\DeclareMathOperator{\coker}{coker}
\DeclareMathOperator{\tr}{tr}
\DeclareMathOperator{\spec}{spec}

\DeclareMathOperator{\Id}{Id}

\DeclareMathOperator{\Bl}{Bl}

\begin{document}
\frontmatter
\title{A Second Order Algebraic Knot Concordance Group}

\author{Mark Powell}
\address{Department of Mathematics, Rawles Hall, 831 East Third Street, Bloomington, IN 47405, USA}
\email{macp@indiana.edu}



\begin{abstract}
Let $Knots$ be the abelian monoid of isotopy classes of knots $S^1 \subset S^3$ under connected sum, and let $\mathcal{C}$ be the topological knot concordance group of knots under connected sum modulo slice knots.  Cochran, Orr and Teichner defined a filtration of $\C$:
\[\C \supset \mathcal{F}_{(0)} \supset \mathcal{F}_{(0.5)} \supset \mathcal{F}_{(1)} \supset \mathcal{F}_{(1.5)} \supset \mathcal{F}_{(2)} \supset \dots \]
The quotient $\mathcal{C}/\mathcal{F}_{(0.5)}$ is isomorphic to Levine's algebraic concordance group, which we denote $\mathcal{AC}_1$; $\mathcal{F}_{(0.5)}$ is the algebraically slice knots.  The quotient $\mathcal{C}/\mathcal{F}_{(1.5)}$ contains all metabelian concordance obstructions.  The Cochran-Orr-Teichner $(1.5)$-level two stage obstructions map the concordance class of a knot to a pointed set $(\mathcal{COT}_{(\C/1.5)},U)$.

We define an abelian monoid of chain complexes $\mathcal{P}$, with a monoid homomorphism $Knots \to \mathcal{P}$.  We then define an algebraic concordance equivalence relation on $\mathcal{P}$ and therefore a group $\mathcal{AC}_2 := \mathcal{P}/\sim$, our \emph{second order algebraic knot concordance group}.  Our results can be summarised in the following diagram:
\[\xymatrix @R+1cm @C+1cm{
Knots \ar[r] \ar @{->>}[d] & \mathcal{P} \ar @{->>}[d]  \\
\mathcal{C} \ar[r] \ar @{->>} [d] & \mathcal{AC}_2 \ar@{-->}[d]  \\
\mathcal{C}/\mathcal{F}_{(1.5)} \ar@{-->}[r]  \ar[ur] & \mathcal{COT}_{(\C/1.5)}.
}\]
That is, we define a group homomorphism $\mathcal{C} \to \mathcal{AC}_2$ which factors through $\mathcal{C}/\mathcal{F}_{(1.5)}$.  We can extract the two stage Cochran-Orr-Teichner obstruction theory from $\mathcal{AC}_2$: the dotted arrows are morphisms of pointed sets.  There is a surjective homomorphism $\mathcal{AC}_2 \to \mathcal{AC}_1$, and we show that the kernel of this homomorphism is non--trivial.  Our second order algebraic knot concordance group $\mathcal{AC}_2$ is a \emph{single stage obstruction group.}
\end{abstract}

\maketitle

\setcounter{page}{4}
\tableofcontents

\mainmatter


\chapter{Introduction}



\section{Introduction to knot concordance}

\begin{definition}
An oriented knot $K$ is an oriented, locally flat embedding $K \colon S^1 \subset S^{3}$.  An oriented knot $K$ is topologically \emph{slice} if there is an oriented locally flat embedding of a disk $D^2 \subseteq D^4$ whose boundary $\partial D^2 \subset \partial D^4 = S^3$ is the knot $K$.  Here locally flat means locally homeomorphic to a standardly embedded $\R^k \subseteq \R^{k+2}$.

Two knots $K_1, K_2 \colon S^1 \subset S^3$ are \emph{concordant} if there is an oriented locally flat embedding of an annulus $S^1 \times I \subset S^3 \times I$ such that $\partial (S^1 \times I)$ is $K_1 \times \{0\} \subseteq S^3 \times \{0\}$ and $-K_2 \times \{1\} \subset S^3 \times \{1\}$.   Given a knot $K$, the knot $-K$ arises by reversing the orientation of the knot and of the ambient space $S^3$: on diagrams reversing the orientation of $S^3$ corresponds to switching under crossings to over crossings and vice versa.  The set of concordance classes of knots form a group $\C$ under the operation of connected sum with the identity element given by the class of slice knots, or knots concordant to the unknot.
\qed \end{definition}

Fox and Milnor first defined the knot concordance group $\C$ in \cite{fm}; they were interested in removing singularities of piecewise--linear surfaces in a 4-manifold: a singularity is removable if a piecewise--linear sphere centred on the singularity intersects the surface in a slice knot.  They gave a condition which a slice knot satisfies, namely that, up to a unit, its Alexander polynomial factorises in the form $f(t)f(t^{-1})$ for some $f$.

One can also consider smoothly slice knots and require that embeddings are smooth rather than just locally flat, but we will primarily consider topological manifolds and locally flat embeddings in this work.

The aim of this work is to unify some previously known obstructions to the concordance of knots using chain complexes with a Poincar\'{e} duality structure, and to present the beginning of a framework with which to apply the algebraic theory of surgery of A. Ranicki \cite{Ranicki3} to classification problems involving 3- and 4-dimensional manifolds.

The first major progress in the study of the concordance group was in 1968 when Levine defined an \emph{algebraic concordance group} $\mathcal{AC}_1$, namely the Witt group of integral Seifert forms.  The \emph{Seifert form} is the linking form on the first homology $H_1(F;\Z)$ of a Seifert surface $F$, defined by pushing one of a pair of curves off the surface slightly along a normal vector.  A form is said to be \emph{algebraically null-concordant} if it is represented by a matrix congruent to one of the form:
\[\left(
    \begin{array}{cc}
      0 & A \\
      B & C \\
    \end{array}
  \right),
\]
for block matrices $A,B,C$ such that $C=C^T$ and $A-B^T$ is invertible.  To obtain a group, we add two forms together via direct sum and $-V$ is the inverse of $V$.

The idea is that if there is a half-basis of curves on $F$ with self linking zero, it might be possible to cut the Seifert surface along these curves and glue in discs, embedded in $D^4$, so as to construct a slice disc.  This is called \emph{ambient surgery}.  For knots with Alexander polynomial one, this is possible \cite{FQ}, \cite{GT03}; we can embed the discs topologically.  However, in general this is problematic as we shall see below.  Certainly a slice knot has an algebraically null-concordant Seifert form, so we have an algebraic obstruction.  Levine \cite{Levine} and Stoltzfus \cite{Stoltzfus} calculated the Witt group of integral Seifert forms $\mathcal{AC}_1$ to be isomorphic to the countably infinitely generated abelian group:
\[\bigoplus_{\infty}\;\Z \oplus \bigoplus_{\infty}\;\Z_2 \oplus \bigoplus_{\infty}\;\Z_4.\]
The infinite cyclic summands are detected by the Levine-Tristram $\omega$-signatures: for a Seifert form $V$ and $\omega \in S^1\setminus \{1\} \subset \mathbb{C}$, the $\omega$-signature is the signature of:
\[(1-\omega)V + (1-\ol{\omega})V^T.\]

\begin{definition}\label{Defn:highdimknots}
An oriented $m$-dimensional knot $K$ is an oriented, locally flat embedding of $S^m \subset S^{m+2}$.  An $m$-knot is topologically slice if there is an oriented, locally flat embedding of a disk $D^{m+1} \subseteq D^{m+3}$ whose boundary $\partial D^{m+1} \subset \partial D^{m+3} = S^{m+2}$ is the knot $K$.  The group of concordance classes of $m$-knots is denoted $\C_m$.
\qed \end{definition}

Every $m$-knot has a Seifert $(m+1)$-manifold $F$ in $S^{m+2}$, with boundary the knot, and there is a linking form on the middle dimensional homology of $F$ defined as above which gives us the Seifert form.  We push the interior of $F$ into $D^{m+3}$, and try to perform ambient surgery in $D^{m+3}$ on the Seifert manifold to make it highly connected and therefore, by the $h$-cobordism theorem, a disk $D^{m+1}$.
In the case of even-dimensional knots there is no obstruction to this, and we can always guarantee by general position that we can glue in embedded rather than immersed discs when we try to do ambient surgery.  Kervaire \cite{Kervaire2} showed that:
\[\C_{2n} \cong 0.\]
For odd dimensional knots $K \colon S^{2n-1} \subset S^{2n+1}$, the algebraic concordance class of the Seifert form obstructs the possibility of embedding all of the surgery disks. Levine \cite{Levine} showed for odd high dimensional knots, with $n \geq 2$, that this is the only obstruction, so that\footnote{Actually, $\mathcal{AC}_1$ takes a slightly different form when $n$ is even: we require the Seifert form $V$ to satisfy that $V + (-1)^n V^T$ is invertible over $\Z$.  Also $\C_3$ maps to an index 2 subgroup of $\mathcal{AC}_1$.  See \cite{Levine} for details.}:
\[\C_{2n-1} \xrightarrow{\simeq} \mathcal{AC}_1.\]
For high-dimensional knots we can always assume by surgery that the fundamental group of the complement of a Seifert $2n$-manifold pushed into $D^{2n+2}$ is $\Z$, and using the Whitney trick we can always guarantee that we can glue in embedded discs, as long as the algebraic obstruction vanishes, when we try to do ambient surgery.  An odd-dimensional knot in high dimensions, so when $n > 1$, is slice if and only if it is algebraically null-concordant.  However when $n=1$, our case of interest, the Whitney trick fails, this program does not work and Levine's map is only a surjection.  Whenever we try to do surgery to kill an element of the fundamental group of the knot complement, we simultaneously create another element of the fundamental group, so we cannot assume, even up to concordance, that the knot group is $\Z$; indeed, by the Loop theorem of Papakyriakopoulos \cite{Papa57}, \cite{Hempel}, the only knot with cyclic fundamental group is the unknot.  As a result, the fundamental group of a slice disc complement will not typically be $\Z$, unless the Alexander polynomial of the knot is one, but will also be more complicated.  In dimension four there is no guarantee that disks can be embedded, only immersed, even if the linking form obstruction vanishes, and attempts to remove intersection points create further problems with the fundamental group.  These problems do not disappear in general unless, as was done by Casson and Freedman (\cite{Casson}, \cite{FQ}), we can push them away to infinity.  The fundamental groups of knot concordance exteriors are in general not ``good'' in the sense of Freedman, so this will not be possible.  Obstructing concordance of knots in dimension three starts with the high-dimensional obstruction, but in contrast to the high-dimensional case, this is only the first stage.

There is a more intrinsic version of the algebraic concordance obstruction.  We will primarily make use of this version.  We now return, for the exposition, to considering 1-dimensional knots, as they are our primary case of interest\footnote{The following construction has an analogous version for all odd dimensional knots.}. If we cut the knot exterior
\[X := \cl(S^3 \setminus (K(S^1) \times D^2))\]
open along a Seifert surface, and then glue infinitely many copies of $X$ together along the Seifert surface, we obtain a space $X_{\infty}$, the \emph{infinite cyclic} or \emph{universal abelian cover} of the knot exterior, which is independent of the choice of Seifert surface.  The $\Z[\Z]$-module $H_1(X_{\infty};\Z) \cong H_1(X;\Z[\Z])$, called the \emph{Alexander module}, is therefore an invariant of the knot.  It is a torsion module (see \cite{Levine2}), and we can define the Blanchfield homology linking pairing
\[\Bl \colon H_1(X;\Z[\Z]) \times H_1(X;\Z[\Z]) \to \Q(t)/\Z[t,t^{-1}]\]
as follows \cite{Blanchfield}.  For $x,y \in C_1(X_{\infty};\Z),$  find $z \in C_2(X_{\infty};\Z)$ such that $\partial z = p(t) x$ for some Laurent polynomial $p(t) \in \Z[\Z]= \Z[t,t^{-1}]$ where $t$ generates the deck transformation group of $X_{\infty}$ (taking $p(t) = \Delta_K(t)$, the Alexander polynomial, will always work, for instance).  Then define:
\[\Bl(x,y) = \frac{\sum_{i=-\infty}^{\infty}(z,yt^{-i})t^i}{p(t)} \in \Q(t)/\Z[t,t^{-1}]\]
where $(\text{ },\text{ })$ is the $\Z$-valued intersection pairing of chains in $C_2$ and $C_1$.

This is equivalent to defining the Blanchfield pairing via the isomorphisms:
\begin{eqnarray*}H_1(X;\Z[t,t^{-1}]) \xrightarrow{\simeq} H^2(X;\Z[t,t^{-1}]) \xrightarrow{\simeq} H^1(X; \Q(t)/\Z[t,t^{-1}])\\ \xrightarrow{\simeq} \Hom_{\Z[t,t^{-1}]}(H_1(X;\Z[t,t^{-1}]),\Q(t)/\Z[t,t^{-1}])\end{eqnarray*}
where the isomorphisms come from Poincar\'{e} duality, a connecting Bockstein homomorphism, and a Universal Coefficient Spectral Sequence.  The Blanchfield form arises from a Seifert matrix $V$ as follows (see \cite{Kearton}):
\[\Bl(a,b) = a^T(1-t)(tV-V^T)^{-1}b \mod \Z[\Z].\]
Note that in order to invert the matrix it is necessary to pass to the field of fractions $\Q(t)$ of $\Z[t,t^{-1}]$.  The appearance of the factor $(1-t)$ corresponds to the duality; it measures the intersection of 2-chains and 1-chains in a certain handle decomposition which begins with the Seifert surface: \cite[page~158]{Kearton2}.  For a slice knot, the Blanchfield form is \emph{metabolic}; that is, there is a submodule $P \subset H_1(X;\Z[\Z])$, a metaboliser, such that $P = P^{\bot}$, where
\[P^{\bot}:= \{v \in H_1(X;\Z[\Z])\,|\, \Bl(v,w) = 0 \text{ for all } w \in P\}.\]

The next significant development in the study of the classical knot concordance group $\C_1$ was the seminal work of Casson and Gordon \cite{CassonGordon}, who found the first algebraically null-concordant knots which are not slice; they used the metaboliser of a linking form on a $k$-fold branched covering of $S^3$ over a knot, for prime power $k$, to define representations of the fundamental group of a 4-manifold whose boundary is $M_K$, the result of performing zero-framed surgery on $K$.  They used these representations to calculate the signature of the twisted intersection form of the 4-manifold.  They made use of the key observation that the vanishing of first-order linking information in a 3-manifold controls the representations of the fundamental group which extend over a 4-manifold which has the 3-manifold as its boundary.  This enables the construction of a second order intersection form on the 4-manifold.  For a slice disc exterior the signatures of the intersection form which Casson and Gordon defined vanish, yielding an obstruction theory.

In 1999, Cochran-Orr-Teichner \cite{COT} defined an infinite filtration of the concordance group.  They understood that the Casson-Gordon invariants obstructed sliceness on a second level.  Recall the heuristic above that if the Seifert form is algebraically null--concordant we can attempt to surger along the curves with zero self-linking and try to create a slice disk.  Instead of being able to glue in disks, we can certainly glue in surfaces.  We can then ask whether these surfaces have sufficiently many curves with zero self linking: the Casson-Gordon invariants obstruct, roughly speaking, the existence of these curves.  The \COT filtration essentially iterates this idea.  It is defined by looking at successive quotients of the derived series (Definition \ref{Defn:derived_subgroups}) of the fundamental group, and constructing so-called higher order Blanchfield forms to control which representations extend over their 4-manifolds.  By using the Blanchfield form on the infinite cyclic cover instead of the $\Q/\Z$-valued linking forms on the finite cyclic covers as in the Casson-Gordon type representations, \COT keep greater control on the fundamental group, which significantly improves the power of their obstruction theory.  Their representations map into fixed groups which they call universally solvable groups, and the values of the representations depend for their definitions on choices of the way in which the lower level obstructions vanish.  See Chapter \ref{chapter:COTsurvey} for a survey of the \COT theory.

Finally, with this extra control on the fundamental group, extra technology is required to extract invariants of the Witt classes of intersection forms.  \COT use the theory of $L^{(2)}$-signatures, in particular the \emph{Cheeger--Gromov--Von--Neumann $\rho$-invariant}, to obtain signatures which capture their obstruction theory and are able to show that their filtration is highly non-trivial.

The goal of this work is to present a unified obstruction theory for the first two stages of the \COT filtration, which does not depend on any choices.

\begin{definition}
We recall the definition of the zero-framed surgery along $K$ in $S^3$, which we denote by $M_K$: attach a solid torus to the boundary of the knot exterior $$X=\cl(S^3 \setminus (K(S^1) \times D^2))$$ in such a way that the longitude of the knot bounds in the solid torus.
\[M_K= X \cup_{S^1 \times S^1} D^2 \times S^1.\]
\qed \end{definition}

The \COT filtration is based on the following characterisation of topologically slice knots: notice that the exterior of a slice disc for a knot $K$ is a 4-manifold whose boundary is $M_K$, since the extra $D^2 \times S^1$ which is glued onto the knot exterior $X$ is the boundary of a regular neighbourhood of a slice disc.

\begin{proposition}\label{basicfact_intro}
A knot $K$ is topologically slice if and only if $M_K$ bounds a topological 4-manifold $W$ such that
\begin{description}
\item[(i)]$i_* \colon H_1(M_K;\Z) \xrightarrow{\simeq} H_1(W;\Z)$ where $i \colon M_K \hookrightarrow W$ is the inclusion map;
\item[(ii)] $H_2(W;\Z) \cong 0$; and
\item[(iii)] $\pi_1(W)$ is normally generated by the meridian of the knot.
\end{description}
\end{proposition}
\begin{proof}
The exterior of a slice disc $D$, $W:= \cl(D^4 \setminus (D \times D^2))$, satisfies all the conditions of the proposition, as can be verified using Mayer-Vietoris and Seifert-Van Kampen arguments on the decomposition of $D^4$ into $W$ and $D \times D^2$.  Conversely, suppose we have a manifold $W$ which satisfies all the conditions of the proposition.  Glue in $D^2 \times D^2$ to the $D^2  \times S^1$ part of $M_K$.  This gives us a 4-manifold $W'$ with $H_*(W';\Z) \cong H_*(D^4;\Z)$, $\pi_1(W') \cong 0$ and $\partial W' = S^3$, so $K$ is slice in $W'$.  We can then apply Freedman's topological $h$-cobordism theorem \cite{FQ} to show that $W' \approx D^4$ and so $K$ is in fact slice in $D^4$.
\end{proof}

We give the definition of the \COT filtration of the knot concordance group.  An $(n)$-solution $W$ is an approximation to a slice disc complement; if $K$ is slice then it is $(n)$-solvable for all $n$, so if we can obstruct a knot from being $(n)$- or $(n.5)$-solvable then in particular we show that it is not slice.

\begin{definition}[\cite{COT} Definition 1.2]\label{Defn:COTnsolvable_intro}
A \emph{Lagrangian} of a symmetric form $\lambda \colon P \times P \to R$ on a free $R$-module $P$ is a submodule $L \subseteq P$ of half-rank on which $\lambda$ vanishes.  For $n \in \mathbb{N}_0 := \mathbb{N} \cup \{0\}$, let $\lambda_n$ be the intersection form, and $\mu_n$ the self-intersection form, on the middle dimensional homology $H_2(W^{(n)};\Z) \cong H_2(W;\Z[\pi_1(W)/\pi_1(W)^{(n)}])$ of the $n$th derived cover of a 4-manifold $W$, that is the regular covering space $W^{(n)}$ corresponding to the subgroup $\pi_1(W)^{(n)} \leq \pi_1(W)$:
\[\lambda_n \colon H_2(W^{(n)};\Z) \times H_2(W^{(n)};\Z) \to \Z[\pi_1(W)/\pi_1(W)^{(n)}].\]
An $(n)$-\emph{Lagrangian} is a submodule of $H_2(W^{(n)};\Z)$, on which $\lambda_n$ and $\mu_n$ vanish, which maps via the covering map onto a Lagrangian of $\lambda_0$.

We say that a knot $K$ is \emph{$(n)$-solvable} if $M_K$ bounds a topological spin 4-manifold $W$ such that the inclusion induces an isomorphism on first homology and such that $W$ admits two dual \emph{$(n)$-Lagrangians}.  In this setting, dual means that $\lambda_n$ pairs the two Lagrangians together non-singularly and their images freely generate $H_2(W;\Z)$.

We say that $K$ is \emph{$(n.5)$-solvable} if in addition one of the $(n)$-Lagrangians is the image of an $(n+1)$-Lagrangian.
\qed \end{definition}

\section{Uniting abelian and metabelian concordance obstructions}

In this section we give a summary of the main results of this monograph.  We will focus on the $(0.5),(1)$ and $(1.5)$ levels of the filtration, corresponding ot the abelian and metabelian quotients of the the fundamental group.  The \COT obstructions to a knot being $(1.5)$-solvable depend for their definitions on the vanishing of the first order obstructions; that is, for each metaboliser of the Blanchfield form, we have a different obstruction.  Our goal is to have an algebraically defined \emph{second order algebraic concordance group}, which obstructs $(0.5)$-, $(1)$- and $(1.5)$-solvability in a single stage definition.  Rather than filter the condition that zero-surgery bounds a 4-manifold whose intersection form is hyperbolic with respect to coefficients of increasing complexity, we filter the condition that the \emph{chain complex} of the zero-surgery bounds an \emph{algebraic} 4-manifold which is a $\Z$-homology circle, with respect to coefficients of increasing complexity.

Something similar, but with respect to the \CG invariants, was attempted by Gilmer\footnote{Unfortunately \cite[page~43]{Friedl}, there is a gap in Gilmer's proofs.} in \cite{Gilmer}: his work was an inspiration for this work.  Gilmer uses homology pairings, however, and his group is altogether different in character from ours.

The knot exterior $X$ is a manifold with boundary $S^1 \times S^1$.  We can split $S^1 \times S^1$ into $S^1 \times D^1 \cup_{S^1 \times S^0} S^1 \times D^1$, cutting the longitude of the knot in two.  We think of this as two trivial cobordisms of the circle.  We use the \emph{symmetric chain complex of the universal cover of the knot exterior}, considered as a chain complex cobordism from the chain complex of $S^1 \times D^1$ to itself, as our fundamental object.  A manifold triad is a manifold with boundary $(X,\partial X)$ such that the boundary splits along a submanifold into two manifolds with boundary:
\[\partial X = \partial X_{0} \cup_{\partial X_{01}} \partial X_1.\]
In our case we have the manifold triad:
\[\xymatrix{
S^1 \times S^0 \ar[r] \ar[d] & S^1 \times D^1 \ar[d]\\ S^1 \times D^1 \ar[r] & X.
}\]
We think of the fundamental object as a $\Z$-homology chain complex cobordism from the chain complex of  $S^1\times D^1$ to itself, which is a product along the boundary; the knot exterior has the homology of a circle and the inclusion of each of the boundary components induces an isomorphism on $\Z$-homology.

We now give an outline of the contents of each chapter.  Broadly, Chapters \ref{Chapter:handledecomp} -- \ref{Chapter:duality_symm_structures} describe an algorithm to produce the symmetric Poincar\'{e} triad associated to the knot exterior, starting with a diagram of a knot.  Chapters \ref{Chapter:2ndderived} -- \ref{Chapter:extractingCOTobstructions} then fit these objects into our group $\ac2$, and relate $\ac2$ to the \COT theory.

Our geometric constructions are described in Chapter \ref{Chapter:handledecomp}.  We explain how to decompose a knot exterior into handles, algorithmically, based on a diagram of the knot. We have:
\begin{theorem}[Theorem \ref{Thm:includingboundary}]
Given a reduced diagram (Definition \ref{Defn:reduceddiagram}) for a knot $K: S^1 \hookrightarrow S^3$, with $c \geq 3$ crossings, there is a handle decomposition of the knot exterior $X$ which includes a regular neighbourhood of the boundary $\partial X \times I \approx S^1 \times S^1 \times I$ as a sub-complex:
\[X = \hp^0 \cup \bigcup_{i=1}^{c+2}\, h_i^1 \cup \bigcup_{j=1}^{c+3}\,h_j^2 \cup \bigcup_{k=1}^{2}\,h_k^3. \]
\end{theorem}

There are relatively few handles, so that the chain complexes which arise from the handle decompositions can be explicitly exhibited.  In Chapter \ref{chapter:chaincomplex}, we do this, and in doing so pass from geometry to algebra.  In order to include the unknot we can either make a reduced diagram for the unknot with 3 crossings, or can work out handle decompositions and chain complexes separately for this case, since it is relatively simple.    The main theorem of Chapter \ref{chapter:chaincomplex} is:
\begin{theorem}[Theorem \ref{Thm:mainchaincomplex}]
Suppose that we are given a knot $K$ with exterior $X$, and a reduced knot diagram for $K$ with $c \geq 3$ crossings.  Denote by $F(g_1, \dots ,g_c)$ the free group on the letters $g_1,\dots,g_c$, and let $l \in F(g_1,\dots,g_c)$ be the word corresponding to a zero--framed longitude of $K$.  Then there is a presentation
\[\pi_1(X) = \langle\,g_1,\dots,g_c, \mu, \lambda\,|\,r_1,\dots,r_c, r_{\mu},r_{\lambda},r_{\partial}\,\rangle\]
with the Wirtinger relations $r_1,\dots,r_c \in F(g_1,\dots,g_c)$ read off from the knot diagram, and
\[r_{\mu} = g_1\mu^{-1};\;\;r_{\lambda} = l\lambda^{-1};\; r_{\partial} = \lambda\mu\lambda^{-1}\mu^{-1}.\]
The generators $\mu$ and $\lambda$ correspond to the generators, and $r_{\partial}$ to the relation, for the fundamental group of the boundary torus $\pi_1(S^1 \times S^1) \cong \Z \oplus \Z$. The generator $\mu$ is a meridian and $\lambda$ is a longitude.  The relations $r_{\mu}$ and $r_{\lambda}$ are part of Tietze moves: they show the new generators to be consequences of the original generators.

The handle chain complex of the $\pi$-cover $\widetilde{X}$ (the cover with deck group $\pi := \pi_1(X)/S$ for some normal subgroup $S \unlhd \pi_1(X)$), with chain groups being based free left $\Z[\pi]$-modules, and with the chain complex $C(\wt{\partial X})$ of the $\pi_1(X)$-cover of $\partial X$ as a sub-complex, is given, with the convention that matrices act on row vectors on the right, by:

\[\xymatrix @C+2cm{
\bigoplus_2\,\Z[\pi] \cong \langle h^3_o,h^3_{\partial} \rangle \ar[d]^{\partial_3}\\
 \bigoplus_{c+3} \,\Z[\pi] \cong \langle h^2_1,\dots,h^2_c,h^2_{\partial \mu},h^2_{\partial \lambda},h^2_{\partial} \rangle \ar[d]^{\partial_2}\\
  \bigoplus_{c+2} \,\Z[\pi] \cong \langle h^1_1,\dots,h^1_c,h^1_{\mu},h^1_{\lambda} \rangle \ar[d]^{\partial_1}\\
  \Z[\pi] \cong \langle h^0_{\partial} \rangle
}\]
where:
\[\ba{rcl}\partial_3 &=& \left(
               \begin{array}{cccccccc}
                 w_1 &  & \hdots &  & w_c & 0 & 0 & 0 \\
                 -u_1 &  & \hdots &  & -u_c & 1-\lambda & \mu-1 & -1 \\
               \end{array}
             \right);\\ & & \\
\partial_2 &= &\left(
                 \begin{array}{ccccccc}
                   \left(\partial r_1/\partial g_1\right) &  & \hdots &  & \left(\partial r_1/\partial g_c\right) & 0 & 0\\
                    &  &  &  &  &  & \\
                   \vdots &  & \ddots &  & \vdots & \vdots & \vdots\\
                    &  &  &  &  &  &  \\
                   \left(\partial r_c/\partial g_1\right) &  & \hdots &  & \left(\partial r_c/\partial g_c\right) & 0 & 0\\
                   1 & 0 & \hdots & 0 & 0 & -1 & 0\\
                   \left(\partial l/\partial g_1\right) &  & \hdots &  & \left(\partial l/\partial g_c\right) & 0 & -1\\
                   0 &  & \hdots &  & 0 & \lambda-1 & 1-\mu\\
                 \end{array}
               \right); \text{ and}\\
                & & \\
\partial_1 &=& \left(
               \begin{array}{ccccccc}
                 g_1 - 1 &  & \hdots &  & g_c - 1 & \mu-1 & \lambda-1\\
               \end{array}
             \right)^T \ea \]
See Theorem \ref{Thm:mainchaincomplex} for the full explanation of the $\partial_i$.
There is then a pair of chain complexes,
\[f \colon C_*(\partial X;\Z[\pi]) \to C_*(X;\Z[\pi]),\] with the map $f$ given by inclusion, expressing the manifold pair $(X,\partial X)$.
\end{theorem}
We therefore obtain, using this, from a knot diagram, a triad of chain complexes:
\[\xymatrix{
C_*(S^1 \times S^0;\Z[\pi_1(X)]) \ar[r]^{i_-} \ar[d]_{i_+} & C_*(S^1 \times D^1_-;\Z[\pi_1(X)]) \ar[d]^{f_-}\\ C_*(S^1 \times D^1_+;\Z[\pi_1(X)]) \ar[r]^{f_+} & C_*(X;\Z[\pi_1(X)]).
}\]
In Chapter \ref{Chapter:duality_symm_structures}, we define and explain the extra structure, namely the \emph{symmetric structure}, with which we endow our chain complexes in order to be able to use Ranicki's theory of algebraic surgery.  The symmetric structure is the chain level version of Poincar\'{e} duality.  We explain the vital relationship between the symmetric structure on the boundary and that on the interior of a manifold, and the further complications which arise when the boundary splits into two along a submanifold.  Making use of formulae of Trotter \cite{Trotter}, we obtain this structure for a knot exterior, and extract what we call the \emph{fundamental symmetric Poincar\'{e} triad of a knot}:
\[\xymatrix{
(C_*(S^1 \times S^0;\Z[\pi_1(X)]),\varphi\oplus -\varphi) \ar[r]^{i_-} \ar[d]_{i_+} & (C_*(S^1 \times D^1_-;\Z[\pi_1(X)]),0) \ar[d]^{f_-}\\ (C_*(S^1 \times D^1_+;\Z[\pi_1(X)]),0) \ar[r]^{f_+} & (C_*(X;\Z[\pi_1(X)]),\Phi).
}\]
In Chapter \ref{Chapter:2ndderived}, we explain how to add knots together.  The desire for the ability to perform addition is the reason for splitting the boundary into two.  The connected sum of knots corresponds to gluing the two knot exteriors together along one $S^1 \times D^1$ half of each of their boundaries.  This operation translates very well into the algebraic gluing of chain complexes, so that we can define a monoid of chain complexes - see Chapter \ref{Chapter:semigroup} for the use of the gluing construction to add together symmetric Poincar\'{e} triads.  We first need to know how addition of knots translates onto the fundamental groups.
 \begin{proposition}[Proposition \ref{prop:addingknotgroups}]\label{Prop:addingknots_intro}
Denote by $X^{\ddag} := \cl(S^3 \setminus N(K \, \sharp \, K^{\dag}))$ the knot exterior for the connected sum $K^{\ddag} := K \, \sharp \, K^{\dag}$ of two oriented knots.  Let $g_1,g_1^{\dag}$ be chosen generators in the fundamental groups $\pi_1(X;x_0)$ and $\pi_1(X^{\dag};x_0^{\dag})$ respectively, generating preferred subgroups $\langle g_1 \rangle \xrightarrow{\simeq} \Z \leq \pi_1(X;x_0)$ and $\langle g_1^{\dag} \rangle \xrightarrow{\simeq} \Z \leq \pi_1(X^{\dag};x_0^{\dag})$.

The knot group for a connected sum $K\,\sharp\,K^{\dag}$ is given by the amalgamated free product of the knot groups of $K$ and $K^{\dag}$, with our chosen meridians identified: \[\pi_1(X^{\ddag}) = \pi_1(X) \ast_{\Z} \pi_1(X^{\dag}),\]
so that $g_1 = g_1^{\dag}$.
\end{proposition}
For the rest of Chapter \ref{Chapter:2ndderived}, we study the quotient of knot groups $\pi_1(X)/\pi_1(X)^{(2)}$, which it turns out has the structure of a semi-direct product $\Z \ltimes H_1(X;\Z[\Z])$.  This is the coefficient group over which we work in order to obtain $(1.5)$-level, or metabelian, obstructions. We prove:
\begin{proposition}[Proposition \ref{prop:2ndderivedsubgroup}]
Let $\phi$ be the quotient map
\[\phi \colon \frac{\pi_1(X)}{\pi_1(X)^{(2)}} \to \frac{\pi_1(X)}{\pi_1(X)^{(1)}} \xrightarrow{\simeq} \Z.\]
Then for each choice of homomorphism
\[\psi \colon \Z \to \frac{\pi_1(X)}{\pi_1(X)^{(2)}}\]
such that $\phi \circ \psi = \Id$, there is an isomorphism:
\[ \theta \colon \frac{\pi_1(X)}{\pi_1(X)^{(2)}} \xrightarrow{\simeq} \Z \ltimes H,\]
where $H := H_1(X;\Z[\Z])$ is the Alexander module.  In the notation of Proposition \ref{Prop:addingknots_intro}, and denoting $H^{\dag} := H_1(X^{\dag};\Z[\Z])$ and $H^{\ddag} := H_1(X^{\ddag};\Z[\Z])$, the behaviour of the second derived quotients under connected sum is given by:
\[\frac{\pi_1(X^{\ddag})}{\pi_1(X^{\ddag})^{(2)}} \cong \Z \ltimes H^{\ddag} \cong  \Z \ltimes (H \oplus H^{\dag}).\]
That is, we can take the direct sum of the Alexander modules.
\end{proposition}

In Chapter \ref{Chapter:semigroup} we define, in purely algebraic terms, a monoid $\P$ of chain complexes (Definition \ref{Defn:algebraicsetofchaincomplexes}).  Our monoid comprises triples $(H,\Y,\xi)$, where $H$ is $\Z[\Z]$-module which satisfies certain conditions which we call the conditions to be an \emph{Alexander Module} (Theorem \ref{Thm:Levinemodule}, \cite{Levine2}), $\Y$ is a $3$-dimensional symmetric Poincar\'{e} triad over the group ring $\zh$, of the form:
\[\xymatrix{
(C_*(S^1 \times S^0;\zh),\varphi\oplus -\varphi) \ar[r]^-{i_-} \ar[d]_{i_+} & (C_*(S^1 \times D^1_-;\zh),0) \ar[d]^{f_-}\\ (C_*(S^1 \times D^1_+;\zh),0) \ar[r]^-{f_+} & (Y,\Phi),
}\]
and $\xi \colon H \toiso H_1(\Z[\Z] \otimes_{\zh} Y)$ is an isomorphism.
The $\zh$-module chain complex $Y$ represents, pedagogically, the chain complex of the knot exterior; however it need not be the chain complex of any manifold. We require that:
\[f_{\pm} \colon H_*(S^1 \times D^1_{\pm};\Z) \toiso H_*(Y;\Z).\]
We call the existence of $\xi$ the \emph{consistency condition}.  We consider two such triples $(H,\Y,\xi)$ and $(H^\%,\Y^\%,\xi^\%)$ to be equivalent, corresponding to isotopy of knots, if there exists an isomorphism $\omega \colon H \toiso H^\%$ and a chain equivalence of triads $j \colon \Z[\Z \ltimes H^\%] \otimes_{\zh} \Y \xrightarrow{\sim} \Y^\%$ such that the following induced diagram commutes:
\[\xymatrix @R+1cm @C+1cm {H \ar[r]_-{\xi}^-{\cong} \ar[d]^-{\omega}_-{\cong} & H_1(\Z[\Z] \otimes_{\zh} Y) \ar[d]^-{j_*}_-{\cong}\\
H^\%  \ar[r]_-{\xi^{\%}}^-{\cong} & H_1(\Z[\Z] \otimes_{\Z[\Z \ltimes H^\%]} Y^\%).}\]

To add two elements $(H,\Y,\xi),(H^\dag\Y^{\dag},\xi^\dag)$ of $\P$ we first tensor all chain complexes up over $\Z[\Z \ltimes (H \oplus H^{\dag})]$, and then use the following diagram.
\[\xymatrix @R+1cm{
 C(S^1 \times D^1) \ar[d]^{f_-} & C(S^1 \times S^0) \ar[l]_-{i_-} \ar[d]_{i_+ = i_-^{\dag}} \ar[r]^-{i_+^{\dag}} & C(S^1 \times D^1) \ar[d]^{f_+^{\dag}} \\
Y & C(S^1 \times D^1) \ar[l]_-{f_+} \ar[r]^-{f_-^{\dag}} & Y^{\dag}
}\]
By taking the mapping cone $Y^{\ddag} := \mathscr{C}((-f_+,f_-^{\dag})^T)$, and using the algebraic gluing construction of Definition \ref{Defn:unionconstruction}, we construct a new element of $\P$.  We have:
\begin{proposition}[Proposition \ref{Prop:abelianmonoid}]
The set $\P$ with the addition $\sharp$ yields an abelian monoid $(\P,\sharp)$.  That is, the sum operation $\sharp$ on $\P$ is abelian, associative and has an identity, namely the fundamental symmetric Poincar\'{e} triad of the unknot.  Let ``$Knots$'' denote the abelian monoid of isotopy classes of locally flat knots in $S^3$ under the operation of connected sum.  Then we have a homomorphism $Knots \to \P$.
\end{proposition}
Chapter \ref{chapter:COTsurvey} contains a survey of the work of \COTN, which motivates the definition in Chapter \ref{chapter:algconcordance} of an algebraic notion of concordance of chain complexes.  We have the following generalisation of Proposition \ref{basicfact_intro}:
\begin{proposition}[Proposition \ref{lemma:basicfactconcordance}]
Two knots $K$ and $K^{\dag}$ are topologically concordant if and only if the 3-manifold
\[Z := X \cup_{\partial X = S^1 \times S^1} S^1 \times S^1 \times I \cup_{S^1 \times S^1 = \partial X^{\dag}} -X^{\dag}\]
is the boundary of a topological 4-manifold $W$ such that
\begin{description}
\item[(i)] the inclusion $i \colon Z \hookrightarrow W$ restricts to $\Z$-homology equivalences $$H_*(X;\Z) \xrightarrow{\simeq} H_*(W;\Z) \xleftarrow{\simeq} H_*(X^\dag;\Z); \text{ and}$$
\item[(ii)] the fundamental group $\pi_1(W)$ is normally generated by a meridian of (either of) the knots.
\end{description}
\end{proposition}
The algebraic definition is similar: see Figure \ref{Fig:existenceofinverse7}; we say that two triples $(H,\Y,\xi)$ and $(H^\dag,\Y^\dag,\xi^\dag)$ are \emph{second order algebraically concordant} if there exists a $\Z$-homology chain complex cobordism $(V,\Theta)$ between two symmetric pairs $$C_*(S^1 \times S^1;\zh) \to Y$$ and $$C_*(S^1 \times S^1;\zhdag) \to Y^\dag$$ which is a product cobordism on the boundary.  Since the Alexander module changes in a concordance of knots, we require the existence of a $\Z[\Z]$-module $H'$ with a homomorphism $\omega \colon H \oplus H^\dag \to H'$. The algebraic cobordism $V$ must be over the ring $\Z[\Z \ltimes H']$.  We tensor the two symmetric pairs with $\zhd$ so that this makes sense.  There is a similar consistency condition: we require that there is an isomorphism
\[\xi' \colon H' \cong H_1(\Z[\Z] \otimes_{\zhd} V)\]
such that the following diagram commutes:
\[\xymatrix @R+1cm @C+1cm {H \oplus H^\dag \ar[r]_-{\xi \oplus \xi^\dag}^-{\cong} \ar[d]^-{\omega} & H_1(\Z[\Z] \otimes_{\zh} Y) \oplus H_1(\Z[\Z] \otimes_{\zhdag} Y^\dag) \ar[d] \\
H'  \ar[r]_-{\xi'}^-{\cong} & H_1(\Z[\Z] \otimes_{\Z[\Z \ltimes H']} V).}\]
This guarantees that the correspondence between the group of the complex and the 1-chains of the complex remains strong: we do not use the Blanchfield pairing, but we still need a mechanism to make sure we have the control that it exercises over the fundamental group, and representations of it, in the work of \COT (see Chapter \ref{chapter:COTsurvey}).
We say that two knots are \emph{second order algebraically concordant} if their triples are, and that a knot which is second order algebraically concordant to the unknot is \emph{second order algebraically slice} or \emph{algebraically $(1.5)$-solvable}.

\begin{figure}
    \begin{center}
 {\psfrag{A}{$Y$}
 \psfrag{B}{$C(S^1 \times D^1)$}
 \psfrag{C}{$C(S^1 \times D^1)$}
 \psfrag{D}{$V$}
 \psfrag{E}{}
 \psfrag{F}{}
 \psfrag{G}{$C(S^1 \times D^1)$}
 \psfrag{H}{$Y^{\dag}$}
 \psfrag{J}{$C(S^1 \times D^1)$}
 \psfrag{K}{$C(S^1 \times S^0)$}
 \psfrag{L}{$C(S^1 \times S^0)$}
\includegraphics[width=10cm]{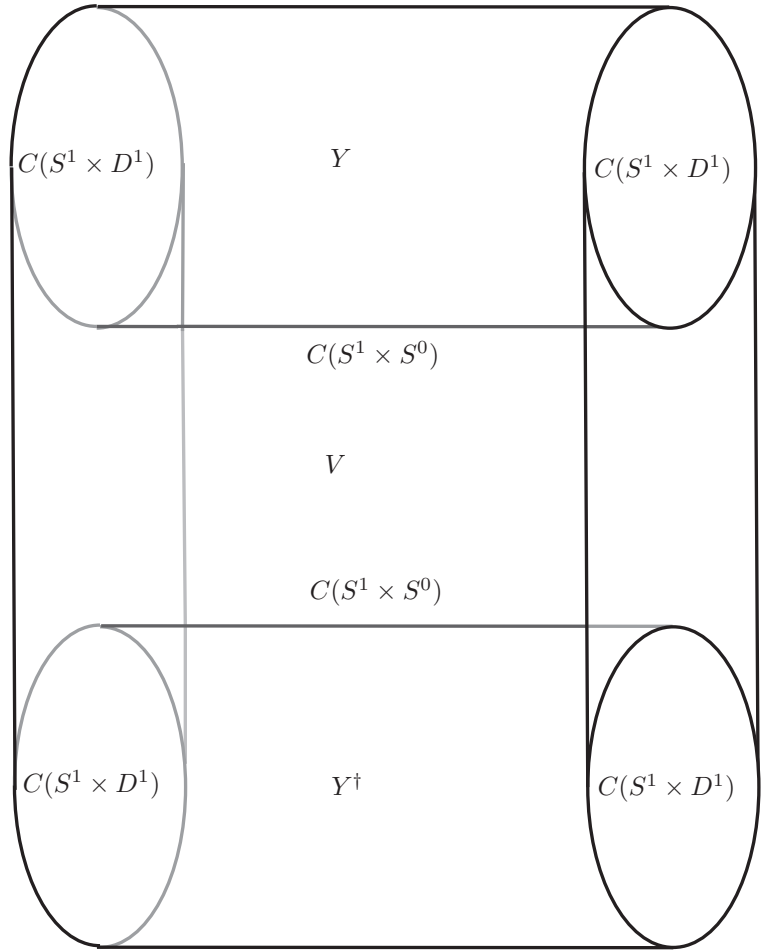}
 }
 \caption{The cobordism which shows that $\mathcal{Y} \sim \mathcal{Y}^{\dag}$.}
 \label{Fig:existenceofinverse7}
 \end{center}
\end{figure}

Taking the quotient of $\P$ by our algebraic concordance equivalence relation we finally arrive at the definition of our group $\ac2$.  The result of Chapter \ref{chapter:algconcordance} is that there is a diagram:
\[\xymatrix @R+1cm @C+1cm{
Knots \ar[r] \ar @{->>}[d] & \P \ar @{->>}[d]  \\
\C \ar[r] & \mathcal{AC}_2,
}\]
where the top row consists of monoids and the bottom row consists of groups.  The maps are therefore monoid homomorphisms, except for the map in the bottom row which is a group homomorphism.

We proceed to show how our group $\ac2$ relates to the \COT filtration and the \COT obstruction theory.  Roughly speaking, our group lies in between the two. First, Chapter \ref{chapter:one_point_five_solvable_knots} contains the proof of the following theorem:
\begin{theorem}[Theorem \ref{Thm:1.5solvable=>2nd_alg_slice}]\label{Thm:1.5solvableintro}
A $(1.5)$-solvable knot is second order algebraically slice.
\end{theorem}
The idea is that the assumptions satisfied by the symmetric chain complex associated to a $(1.5)$-solution, which is a 4-manifold $W$ whose intersection form is hyperbolic, yield precisely the data needed to do \emph{algebraic surgery} on the chain complex of $W$ to make it into a $\Z$-homology circle without affecting $H_1(W;\Z[\Z])$.  The resulting chain complex may not be the chain complex of any manifold, since the corresponding geometric surgeries will not in general be possible: if they were, the knot would be slice rather than just $(1.5)$-solvable.

Our construction, the group $\ac2$, is in some sense very clean.  We avoid references to homology pairings and we avoid the use of the Ore localisation.  We also avoid the problems of universal coefficients which often require the ad-hoc introduction of principal ideal domains, and we obtain a \emph{group} with a \emph{non-trivial homomorphism} $\C \to \ac2$: the chain complexes behave well under connected sum.  Traditionally, cobordism groups use disjoint union to define their addition operation.  Our operation of addition is superior because it mirrors much more closely the geometric operation of addition of knots.  Most importantly, by defining our obstruction in terms of chain complexes, we have a \emph{single stage} obstruction which captures the first two main stages of the \COT obstruction theory.  In Chapter \ref{Chapter:extractingCOTobstructions}, we have the following results.  First:
\begin{proposition}[Proposition \ref{Thm:zeroAC2_goesto_zeroAC1}]
There is a surjective homomorphism
\[\ac2 \to \mathcal{AC}_1,\]
where $\mathcal{AC}_1$ is the algebraic concordance group of Seifert forms.
\end{proposition}
The key point is that a symmetric chain complex contains all the information necessary to algebraically define the Blanchfield form.  This gives us the algebraic concordance group.  When the image in $\mathcal{AC}_1$ vanishes, we can then use the Blanchfield form to define the second order \COT obstructions.

Depending on a choice of metaboliser for the Blanchfield form, the \COT obstructions map a knot to an element of the $L$-group $L^4(\Q\G,\Q\G-\{0\})$ (see Definitions \ref{Defn:Lgroups} and \ref{defn:localisationexactsequence}), where $\G := \Z \ltimes \Q(t)/\Q[t,t^{-1}]$ is the \COT universally $(1)$-solvable group.  For an element $(H,\Y,\xi) \in \ac2$, we define the chain complex $N$ to be:
\[N := Y \cup_{C(S^1 \times S^1)} C(S^1 \times D^2),\]
the algebraic equivalent of the zero surgery $M_K := X \cup_{S^1 \times S^1} S^1 \times D^2$; we perform the gluing so that the longitude bounds.  For each $p \in H$, there is a representations $\rho \colon \Z \ltimes H \to \G$, which defines the tensor product $(\Q\G \otimes_{\zh} N)_p$.  We use the subscript to remind us that this depends on a choice of $p$.  We prove:
\begin{theorem}[Theorem \ref{Thm:betterstatementzeroAC2_to_zeroL4QG}]\label{Thm:ac2_to_zeroL4QG_intro}
Let $(H,\Y,\xi) \in \ac2$ be in the equivalence class of the fundamental symmetric Poincar\'{e} triad $(0,\Y^U,\Id_{\{0\}})$ of the unknot.  Then there exists a metaboliser $P = P^{\bot}$ of the rational Blanchfield form, such that for any $p \in P$, using a representation $\rho \colon \Z \ltimes H \to \G$ which depends on $\xi$, $\Bl$ and $p$, $\Y$ produces:
\[(\Q\G \otimes_{\zh} N, \theta)_p,\]
which represents:
\[0 \in L^4(\Q\G,\Q\G - \{0\}).\]
\end{theorem}
We denote by $(\mathcal{COT}_{(\C/1.5)},U)$ the \emph{\COT pointed set}, which captures the \COT obstruction theory: see Definition \ref{defn:COTobstructionset_2}.  We then use Theorems \ref{Thm:1.5solvableintro} and \ref{Thm:ac2_to_zeroL4QG_intro} to extend our diagram to:
\[\xymatrix @R+1cm @C+1cm{
Knots \ar[r] \ar @{->>}[d] & \P \ar @{->>}[d]  \\
\C \ar[r] \ar @{->>} [d] & \mathcal{AC}_2 \ar@{-->}[d]  \\
\C/\mathcal{F}_{(1.5)} \ar@{-->}[r]  \ar[ur] & \mathcal{COT}_{(\C/1.5)}.
}\]
The homomorphism $\C \to \ac2$ factors through $\mathcal{F}_{(1.5)}$ by Theorem \ref{Thm:1.5solvable=>2nd_alg_slice}, and we can extract the \COT obstructions from our algebraic group of chain complexes.  As above, the maps starting in the top row are monoid homomorphisms.  The dotted arrows are maps of pointed sets.  The maps emanating from the middle row are group homomorphisms.  We would prefer that the homomorphism $\C/\mathcal{F}_{(1.5)} \to \ac2$ were injective but this seems hard with present technology.

Furthermore, we are able to use $L^{(2)}$-signatures to obstruct elements in $\ac2$ from being second order algebraically slice.  We have a purely algebraic definition of a Von Neumann $\rho$--invariant.

\begin{definition}[Theorem \ref{Defn:alg(1)solvable}]
We say that an element $(H,\Y,\xi) \in \ac2$ with image $0 \in \mathcal{AC}_1$ is \emph{algebraically $(1)$-solvable} if the following holds.  There exists a metaboliser $$P = P^{\bot} \subseteq H_1(\Q[\Z] \otimes_{\zh} N)$$ for the rational Blanchfield form such that for any $p \in H$ satisfying $\xi(p) \in P$, we obtain an element:
\[\Q\G \otimes_{\zh} N \in \ker(L^4(\Q\G,\Q\G -\{0\}) \to L^3(\Q\G)),\]
via a symmetric Poincar\'{e} pair over $\Q\G$: \[(j \colon (\Q\G \otimes_{\zh} N)_p \to V_p, (\Theta_p,\theta_p)),\] with
\[P = \ker(j_* \colon H_1(\Q[\Z] \otimes_{\zh} N) \to H_1(\Q[\Z] \otimes_{\Q\G} V_p)),\]
and such that:
\[j_*\colon H_1(\Q \otimes_{\zh} N) \xrightarrow{\simeq} H_1(\Q \otimes_{\Q\G} V_p)\]
is an isomorphism.  We call each such $(V_p,\Theta_p)$ an \emph{algebraic $(1)$-solution}.
\qed\end{definition}
\begin{definition}
Let $\K$ be the skew field which comes from the Ore localisation of $\Q\G$ with respect to $\Q\G - \{0\}$ (Definition \ref{Defn:OreLocalisation}). See \cite[Section~5]{COT}, or our section \ref{chapter:COTsurvey}.\ref{Chapter:L2signature}, for the definition of the \emph{$L^{(2)}$-signature homomorphism}:
\[\sigma^{(2)} \colon L^0(\K) \to \R,\]
which we use to detect non-trivial elements of the Witt group $L^0(\K)$ of non-singular Hermitian forms over $\K$.
\qed\end{definition}
\begin{theorem}[Theorem \ref{Thm:extractingL2signatures}]
Suppose that $(H,\Y,\xi) \in \ac2$ is algebraically $(1)$-solvable with algebraic $(1)$-solution $(V_p,\Theta_p)$.  Then since:
\[\ker(L^4(\Q\G,\Q\G -\{0\}) \to L^3(\Q\G)) \cong \frac{L^4(\K)}{L^4(\Q\G)} \cong \frac{L^0(\K)}{L^0(\Q\G)},\]
we can apply the $L^{(2)}$-signature homomorphism:
\[\sigma^{(2)} \colon L^0(\K) \to \R,\]
to the intersection form:
\[\lambda_{\K} \colon H_2(\K \otimes_{\Q\G} V_p) \times H_2(\K \otimes_{\Q\G} V_p) \to \K.\]
We can also calculate the signature $\sigma(\lambda_{\Q})$ of the ordinary intersection form:
\[\lambda_{\Q} \colon H_2(\Q \otimes_{\Q\G} V_p) \times H_2(\Q \otimes_{\Q\G} V_p) \to \Q,\]
and so calculate the reduced $L^{(2)}$-signature \[\wt{\sigma}^{(2)}(V_p) = \sigma^{(2)}(\lambda_{\K}) - \sigma(\lambda_{\Q}).\]  This is independent, for fixed $p$, of changes in the choice of chain complex $V_p$.    Provided we check that the reduced $L^{(2)}$-signature does not vanish, for each metaboliser $P$ of the rational Blanchfield form with respect to which $(H,\Y,\xi)$ is algebraically $(1)$-solvable, and for each $P$, for at least one $p \in P \setminus \{0\}$, then we have a \emph{chain--complex--Von--Neumann $\rho$--invariant} obstruction.  This obstructs the image of the element $(H,\Y,\xi)$ in $\mathcal{COT}_{(\C/1.5)}$ from being $U$, and therefore obstructs $(H,\Y,\xi)$ from being second order algebraically slice.
\end{theorem}

This shows that our group $\ac2$ is highly non-trivial.   Previous definitions of $\rho$--invariants invoke a 4-manifold in some way for the definition.


Philosophically, when we talk about obstructions to topological knot concordance, we are really talking about algebraic obstructions to $\Z$-homology chain complex cobordism.  The more sophisticated the obstructions that we are dealing with, the more complicated must the coefficient ring be to which we are able to lift our $\Z$-homology chain complex cobordism.  As such, we outline, in Appendix \ref{Appendix:nth_order_group}, a definition of what we conjecture to be an $n$th order algebraic concordance group $\mathcal{AC}_n$, which, as we hope to show in future work, should extend the results of this present work to capture the whole of the \COT filtration.

\section*{Acknowledgement}

This work is essentially my PhD thesis.  Most of all, I would like to thank my supervisor Andrew Ranicki, for all the help he has given me over the last three and a half years, in particular for suggesting this project, and for the ideas and advice which were instrumental in solving so many of the problems encountered.

Stefan Friedl and Kent Orr were involved in this project from an early stage.  Their suggestions and criticism have proved vital in the evolution of the ideas presented in this work.  I would like to take this opportunity to thank both Kent and Stefan for all their advice and encouragement.

During Andrew's sabbaticals in the autumns of 2008 and 2010 in M\"{u}nster and MPIM Bonn respectively, Wolfgang L\"{u}ck and Peter Teichner arranged that I could accompany him to Germany, where I enjoyed productive and stimulating research environments.  I would also like to thank Peter Teichner for helpful discussions, particularly concerning the equivalence relation which defines the group of chain complexes.

I would also like to thank Diarmuid Crowley, Tibor Macko, Daniel Moskovich and Dirk Sch\"{u}tz for their useful advice and helpful comments.

Throughout my time at Edinburgh, I have had countless, invaluable mathematical conversations with my friends and fellow graduate students Spiros Adams-Florou, Julia Collins and Paul Reynolds.  I can only hope they found my contributions to these dialogues half as useful as I found theirs.




\chapter{A Handle Decomposition of a Knot Exterior}\label{Chapter:handledecomp}

In this chapter we shall explain how to obtain, in an algorithmic fashion, a concrete handle decomposition of the exterior $X_K = X$, of a knot $K$, given a diagram of the knot.  This will be the starting point in geometry, from which we pass to algebra via the handle chain complex associated to our decomposition.

The results of this chapter are not particularly novel.  However, the details presented here are crucial for the rest of the construction, in chapters \ref{chapter:chaincomplex} and \ref{Chapter:duality_symm_structures}, of an element of our group $\ac2$.

The boundary $\partial X$ of a knot exterior $X$ is the boundary of a neighbourhood $S^1 \times D^2$ of the knot.  Since in our applications we shall be using the chain complex level version of the Poincar\'{e}-Lefschetz duality of $X$ relative to $\partial X$, it is important that we have a handle decomposition which includes, as a sub-complex, a handle decomposition of the boundary.  Moreover, in order to define a group, we have to be able to add the chain complexes of the knot exteriors together.  We therefore split the boundary $S^1 \times S^1$ into two copies of $S^1 \times D^1$, splitting the longitude in two, and consider the knot exterior as a cobordism from $S^1 \times D^1$ to itself, relative to a product cobordism on $S^1 \times S^0$.  We call this cobordism the \emph{fundamental cobordism} of a knot, since the chain complex of this cobordism, with the duality maps, will be the main algebraic object which we associate to a knot in order to obtain an element of our algebraic concordance group.

The cylinder $S^1 \times D^1$ is a trivial cobordism from the circle $S^1$ to itself.  The knot exterior is a $\Z$-homology cobordism from $S^1 \times D^1$ to itself; it is an easy Mayer-Vietoris argument to show that both of the inclusion induced maps $H_*(S^1 \times D^1;\Z)$ $\xrightarrow{\simeq}$ $H_*(X;\Z)$ are isomorphisms; this is the definition of a $\Z$-homology cobordism.  Therefore $X$ looks like a product cobordism to $\Z$-homology.  We will consider covering spaces of $X$ whose homology modules have more discerning coefficients.  For these coefficients $X$ does not typically have the homology of a product.  Our obstruction theory measures the obstruction to changing (by a homology cobordism) the chain complex of the fundamental cobordism of the knot to being the chain complex of a product cobordism.

We begin by making explicit some standard notation.
\begin{definition}\label{conventions0}
We denote $D^n := \{x \in \R^{n} \,|\, |x|\leq 1\}.$  In particular $D^1 = [-1,1].$  We denote $I := [0,1]$, the unit interval.  Although they are topologically the same, there are semantic differences.  We shall principally use $I$ when thickening a submanifold in the neighbourhood of a boundary.  The notation $\mathring{D}^n$ and $\mathring{I}$ shall be used to mean the respective interiors $D^n \setminus \partial D^n \approx \R^n$ and $(0,1)$.
\qed \end{definition}

\begin{definition}
We adopt the following conventions for certain common equivalence relations.  For algebraic objects such as groups, rings and modules, we use the symbol $\cong$ for abstract isomorphism; we use $\xrightarrow{\simeq}$ when there is a choice of map which induces the isomorphism; we use $=$ when there is equality but differing notation for the same object. We use $\simeq$ to denote homotopy equivalence of topological spaces or chain equivalence of chain complexes, and $\approx$ denotes homeomorphism, while $\xrightarrow{\approx}$ denotes a particular choice of homeomorphism.
\qed \end{definition}

\begin{definition}
A \emph{knot} is an isotopy class of oriented locally flat embeddings $K:S^1 \hookrightarrow S^3$.  Such an embedding is \emph{locally flat} if it is locally homeomorphic to a ball-arc pair: that is, for all $x \in S^1$, there is a neighbourhood $U$ of $K(x)$, such that $(U,U \cap K(S^1)) \approx (\mathring{D}^3,\mathring{D}^1)$.  We often abuse notation and also refer to the image $K(S^1)$ as $K$.
\qed \end{definition}
\begin{remark}
In particular a tame knot in the smooth or piecewise-linear categories is locally flat.  We do not make restrictions to the smooth or piecewise-linear categories in this work; our obstructions work at the level of the topological category.  We therefore work with topological manifolds and locally flat embeddings when there is no comment otherwise.  Our obstructions are not intended to detect the gap between smooth and topological concordance.
\end{remark}
\begin{definition}
A \emph{knot diagram} for an oriented knot $K$ is the image of the composite of a representative embedding in the isotopy class $K \colon S^1 \to S^3$, with a  projection $p \colon S^3 \backslash \{\infty\} \approx \R^3 \to S^2$.  If necessary, we must either isotope the knot slightly, or perturb $p$, so that $p \circ K \colon S^1 \to S^2$ is an immersion with at most transverse double points.  We make pictures in the plane to define knots\footnote{The locally flat condition ensures that space filling curves do not occur, so the images $K(S^1)$ and $p \circ K(S^1)$ can be isotoped away from the $\infty$ points of the respective spheres in which they lie, and nothing is lost by considering the diagram as the image of $p \colon \R^3 \to \R^2$.}, and mark the diagram with the crossing points of the knot, the image of the double point set of $p \circ K$, using a small break in the line to indicate an under-crossing.  We also mark the diagram with an arrow to indicate the orientation of the knot.
\qed \end{definition}

\begin{definition}
We can associate a \emph{graph} to a knot diagram in $S^2$ by putting a vertex at each crossing point.  All the arcs of the diagram, between the crossings, become the edges of the graph, and the areas bounded by the edges are the \emph{regions}.  Note that the area outside the knot also counts as a region, sometimes called the unbounded region, since we are considering the diagram as part of $S^2$. (See Figure \ref{Fig:regions})
\qed \end{definition}
\begin{figure}[h!]
 \begin{center}
 \includegraphics [width=7cm] {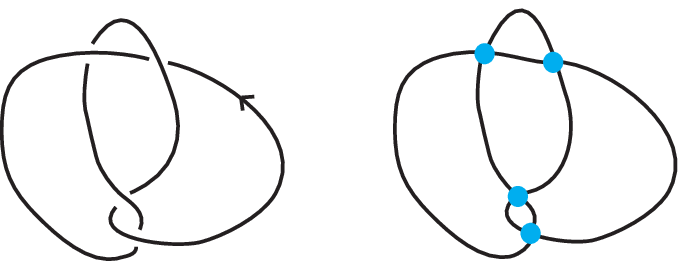}
 \caption {A knot diagram for the figure eight knot, and its associated graph.  The graph has 4 vertices, 8 edges, and 6 regions.}
 \label{Fig:regions}
 \end{center}
\end{figure}

\begin{definition}\label{Defn:reduceddiagram}
A knot diagram is \emph{reduced} if there does not exist a region in the associated graph which abuts itself at a vertex.  At each vertex, four regions meet;  in a reduced diagram, they must be four distinct regions.  If a diagram is un--reduced, there is a move similar to a Reidemeister type I move with some possibly non-trivial part of the knot diagram on either side, which can be made.  Take a closed curve inside the region which abuts itself, which starts and ends at the crossing.  By the Jordan Curve Theorem this divides $S^2$ into two discs, each containing a part of the knot diagram.  To remove the offending crossing, lift one of the parts up to $S^3$ and rotate it by $\pi$ radians, before projecting down again.  This constitutes an isotopy of the knot.  Thus any knot has a reduced diagram.\\
Figure \ref{Fig:reduceddiagram} shows some un--reduced diagrams.  Note that the region which abuts itself can be an inside region, or it can be the region ``outside'' the knot diagram.
\qed \end{definition}
\begin{figure}[h!]
 \begin{center}
 {\psfrag{X}{$X$}
 \psfrag{Y}{$Y$}
 \psfrag{A}{$X'$}
 \psfrag{B}{$Y'$}
 \includegraphics [width=9cm] {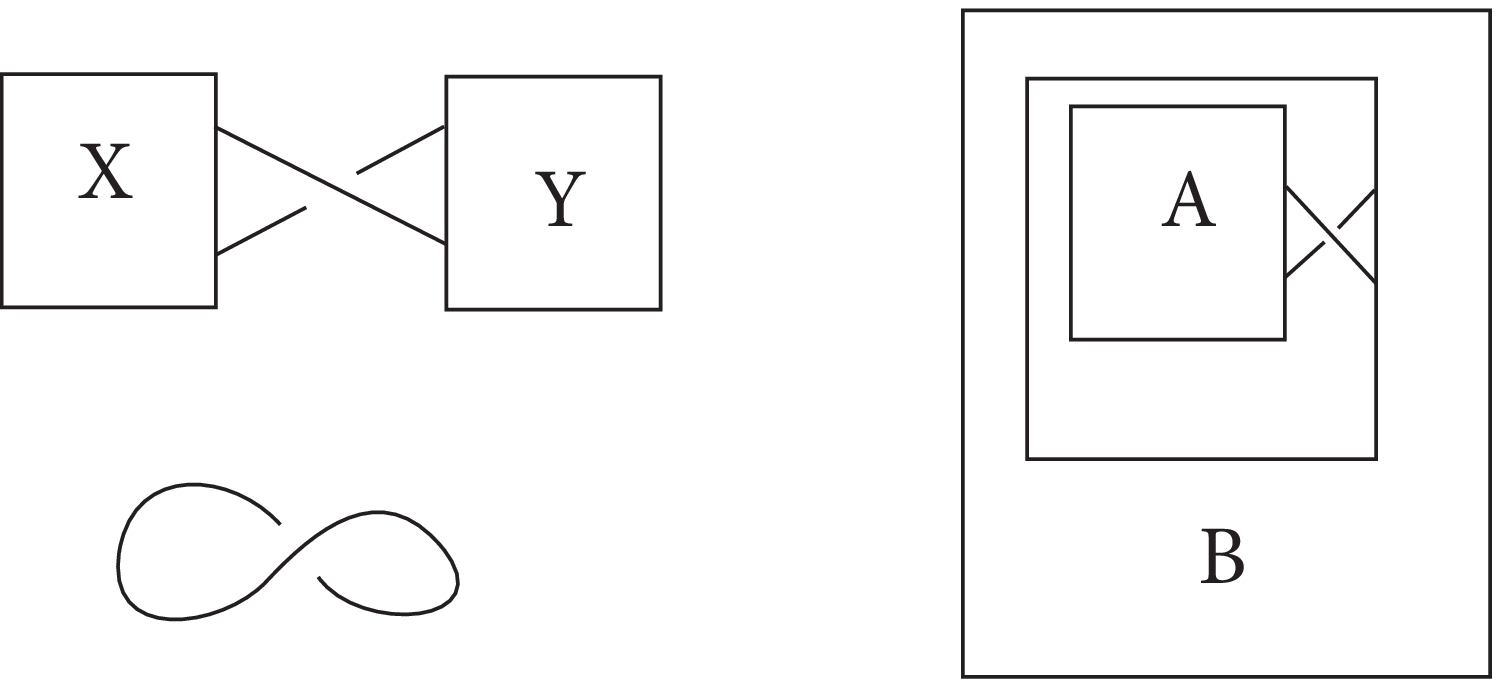}}
 \caption{Some examples of un--reduced diagrams.  The labels $X$, $X'$, $Y$ and $Y'$ denote some other part of the knot diagrams.}
  \label{Fig:reduceddiagram}
 \end{center}
\end{figure}

\begin{definition}\label{Defn:quaddecomp}
A reduced knot diagram with a non-zero number of crossings determines a \emph{quadrilateral decomposition} of $S^2$ (\cite{Sanderson}).  It is the graph which is \emph{dual} to the associated graph.  This means putting a vertex in each region, and then joining a pair of vertices with an edge if the original regions were separated by an edge in the original graph.
\begin{figure}[h!]
 \begin{center}
 \includegraphics [width=7cm] {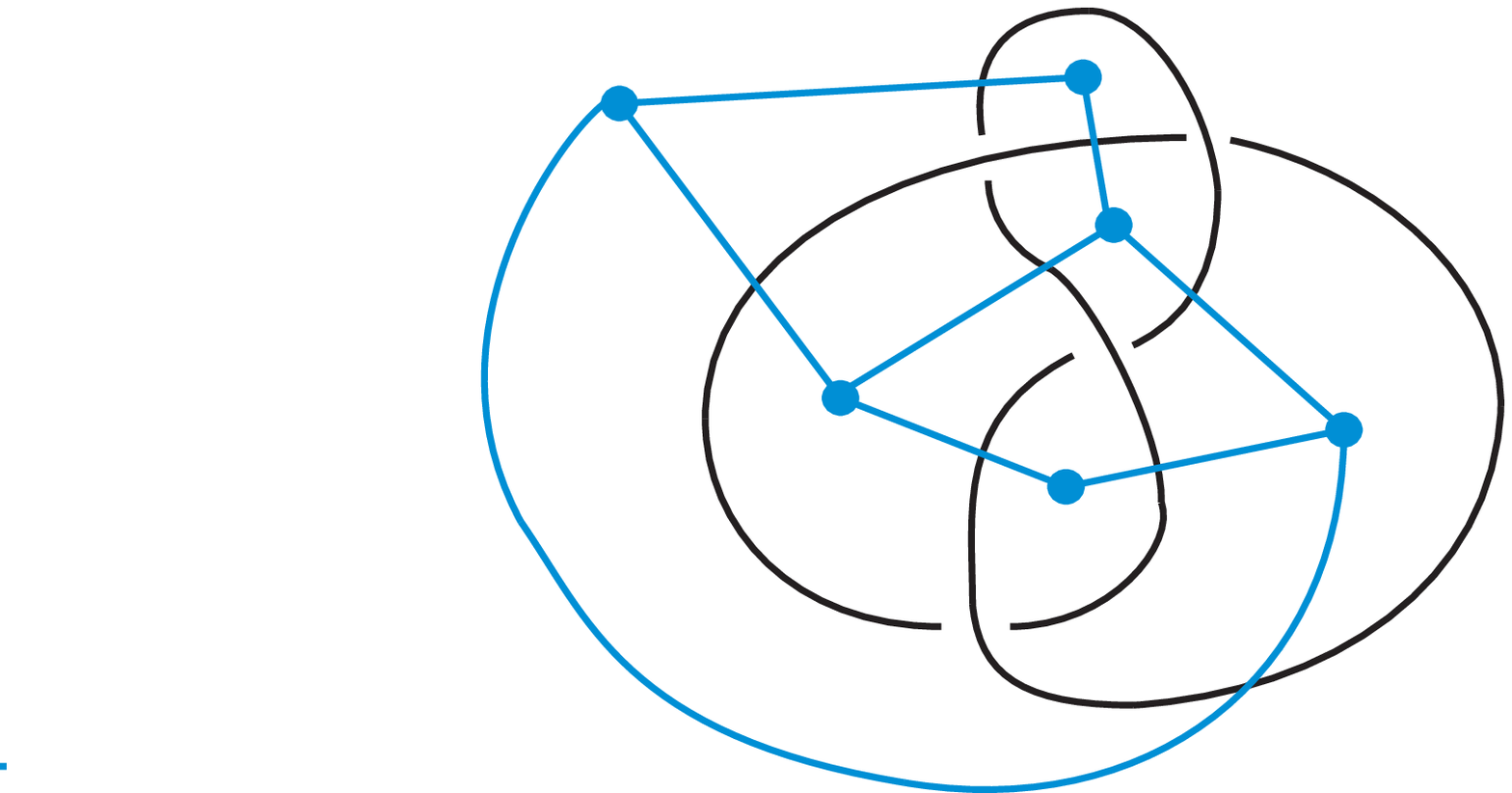}
 \caption {Quadrilateral decomposition for our diagram of the figure eight.}
 \label{Fig:quaddecomp}
 \end{center}
\end{figure}
  Each region in the dual graph then has a single crossing in its interior, and as the original graph is four-valent, each region is a quadrilateral.\\
\qed \end{definition}

\begin{remark}
If we want to include the unknot then we can make a reduced diagram of the unknot with at least 3 crossings e.g. simply change one of the crossings in a 3 crossing diagram for the trefoil.  Alternatively, note that the exterior of the unknot is a solid torus $S^1 \times D^2$ and one can then fathom simple handle decompositions for this manifold; this will be made explicit soon.

Just as the quadrilateral decomposition fails for unreduced diagrams, it also fails for those which are not connected; it does not generalise to split link diagrams.
\end{remark}

\begin{definition}
For a knot $K \colon S^1 \hookrightarrow S^3$, the \emph{knot exterior} is $X_K := \overline{S^3 \backslash N}$ where $N \approx S^1 \times D^2$ is a regular neighbourhood of the knot.  Where there is only one (usually generic) knot being considered, this will be written just as $X$.
\qed \end{definition}

\begin{definition}
We attach an $r$-handle $h^r = D^r \times D^{3-r}$ to a $3$-manifold with boundary $(M,\partial M)$ by gluing $S^{r-1} \times D^{3-r} = \partial D^r \times D^{3-r} \subset D^r \times D^{3-r}$ to an embedding \\$S^{r-1} \times D^{3-r} \hookrightarrow \partial M$.

The subset $S^{r-1} \times \{0\} \subseteq D^r \times D^{3-r}$ is known as the \emph{attaching sphere}, $D^r \times \{0\}$ is the \emph{core}, while $\{0\} \times D^{3-r}$ is called the \emph{cocore}, and $\{0\} \times S^{3-r-1}$ is the \emph{belt sphere}.  By convention, $S^{-1} := \emptyset$.

By Morse theory, any manifold $X$ can be constructed by iteratively attaching handles, yielding a \emph{handle decomposition} of $X$.

The $k$-\emph{skeleton} $X^{(k)}$ of a handle decomposition of $X$ is the union of all the $i$-handles of the decomposition for all $i \leq k$.  We shall principally be concerned with 3-manifold decompositions, so the definition was given in this case, but handles which decompose $n$-manifolds are defined by replacing $3$ with $n$ in the definition above.

We employ the standard techniques of rounding or smoothing the corners without further comment.
\qed \end{definition}

\begin{remark}
In order to specify the attaching of an $r$-handle, it suffices to give the embedding of the attaching $(r-1)$-sphere in the boundary of $X^{(r-1)}$, and to give a framing for this embedding.  In the case of 3-manifolds, the embeddings are of 0, 1, or 2 manifolds into surfaces, and so there is essentially only one choice of framing.  0- and 3- handles are copies of $D^3$; a 1-handle is attached at two points, called its feet, and a 2-handle is attached along a circle.  Figure \ref{Fig:handlesall} shows a copy of each of these handles.
\begin{figure}[h!]
 \begin{center}
 \includegraphics [width=9cm] {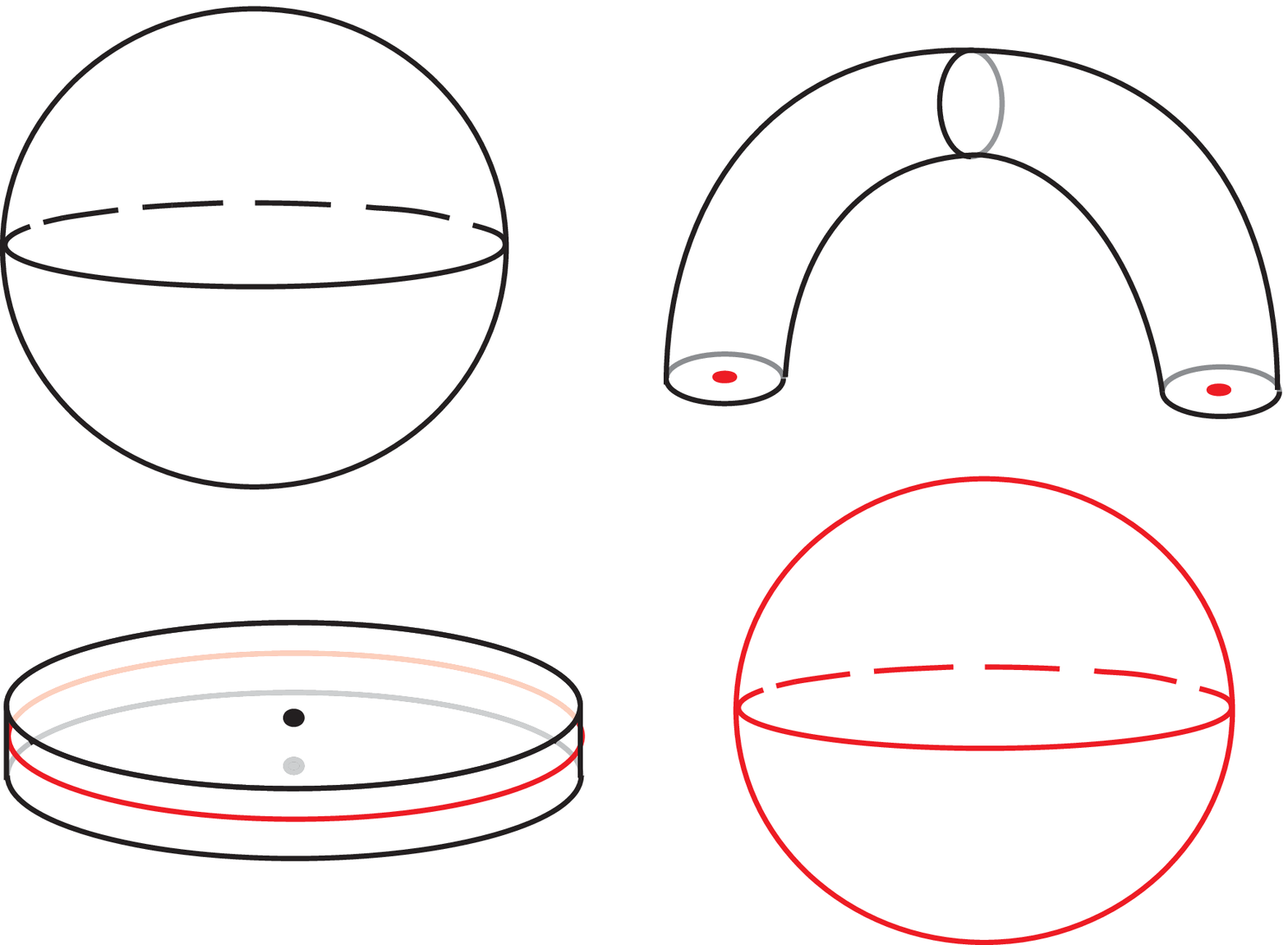}
 \caption {A standard $r$-handle for $r=0,1,2,3$.}
 \label{Fig:handlesall}
 \end{center}
\end{figure}

\end{remark}

\begin{theorem}\label{handledecomp}
Given a reduced diagram for a knot $K: S^1 \hookrightarrow S^3$, with $c \geq 3$ crossings, there is a handle decomposition of the knot exterior $X = \overline{S^3 \setminus N}$:
\[X = h^0 \cup \bigcup_{i=1}^{c}\, h_i^1 \cup \bigcup_{j=1}^{c}\,h_j^2 \cup h^3. \]
\end{theorem}
\begin{proof}
This is based on notes of Sanderson \cite{Sanderson}.  Divide $S^3$ into an upper and lower hemisphere: $S^3 \approx D_-^3 \cup_{S^2} D_+^3$.  Let the knot diagram be in $S^2$, and arrange the knot itself to be close to its image in the diagram in $S^2$ but all contained in $D_+^3$.  Let $D_-^3$ be $h^0$.  Attach 1-handles which start and end at the 0-handle and go over the knot, one for each edge of the quadrilateral decomposition of $S^2$.  The feet of each one handle should be at or near the vertices of the quadrilateral decomposition, which is a graph in $S^2$, which are at either end of the corresponding edge.  Figure \ref{Fig:bigcrossing2} shows a single crossing of the knot inside one region of the quadrilateral decomposition, and the part of the boundary of $D^3_- = h^0$ which lies inside this region.  According to some numbering of the crossings of the knot diagram (see Conventions \ref{conventions}) and therefore of the regions of the quadrilateral decomposition we call this crossing $i$.

\begin{figure}[h]
  \begin{center}
    {\psfrag{A}{$= K(S^1)$ near crossing $i$.}
    \psfrag{B}{$=$ the boundary of the $0$-handle.}
    \includegraphics[width=9cm]{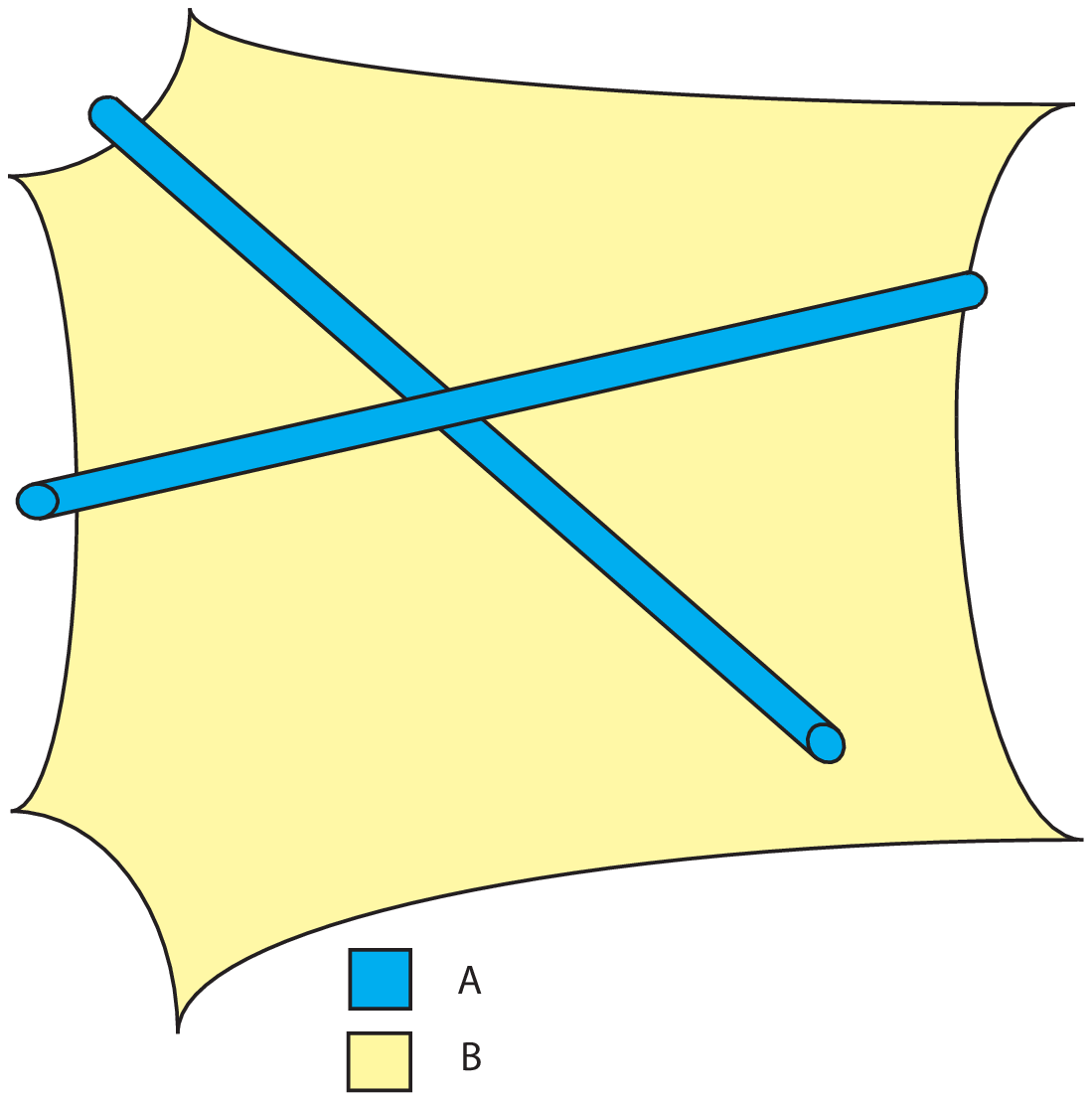}}
  \caption{A close up of one region of the quadrilateral decomposition.}
  \label{Fig:bigcrossing2}
  \end{center}
\end{figure}

Figure \ref{Fig:bigcrossing25} shows the attaching of the 1-handles adjacent to this crossing.  It also shows the orientation of the over-strand and some labels which we associate to each of the 1-handles.

\begin{figure}[h]
  \begin{center}
    {\psfrag{A}{$=$ 1-handles for under-strand.}
    \psfrag{B}{$=$ 1-handle for over-strand.}
    \psfrag{C}{$h^1_{i_1}$}
    \psfrag{D}{$h^1_{i_2}$}
    \psfrag{E}{$h^1_{i_3}$}
    \psfrag{F}{$h^1_{i_4}$}
    \includegraphics[width=9cm]{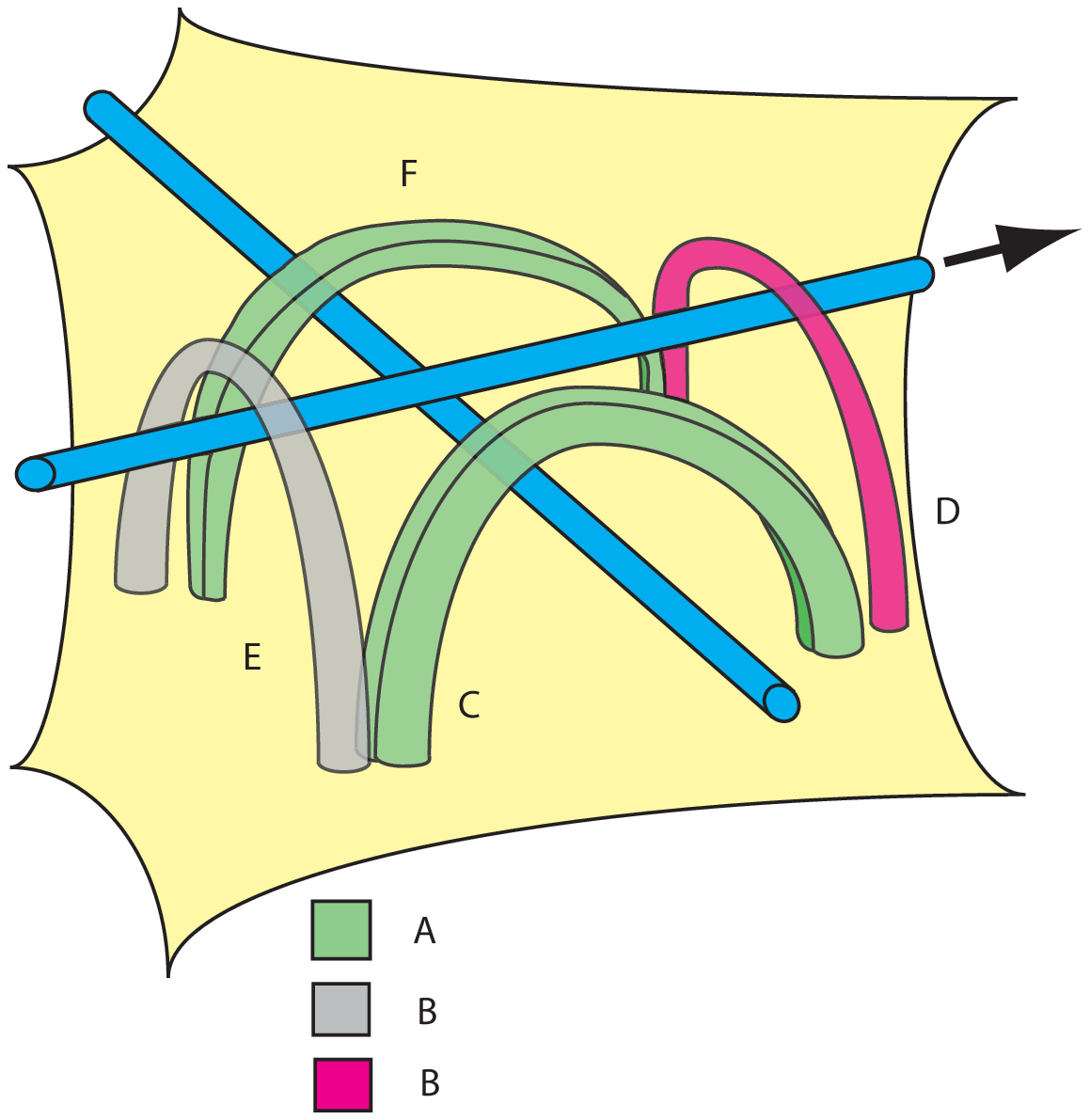}}
  \caption{Attaching the 1-handles.}
  \label{Fig:bigcrossing25}
  \end{center}
\end{figure}

There are $c$ regions and therefore $2c$ edges and currently $2c$ 1-handles.  Now, for each crossing, attach a 2-handle which goes between the strands of the knot, so that the 1-handles which are labelled $h^1_{i_1}$ and $h^1_{i_4}$ from Figure \ref{Fig:bigcrossing25}, and this 2-handle, by handle cancellation, can be amalgamated into a single 1-handle.  Look now at Figure \ref{Fig:bigcrossing29} for an illustration.

\begin{figure}[h]
  \begin{center}
    {\psfrag{A}{$=$ original 1-handles.}
    \psfrag{B}{$=$ cancelling 2-handle.}
    \psfrag{C}{$h^1_{i_1}$}
    \psfrag{D}{$h^1_{i_4}$}
    \includegraphics[width=9cm]{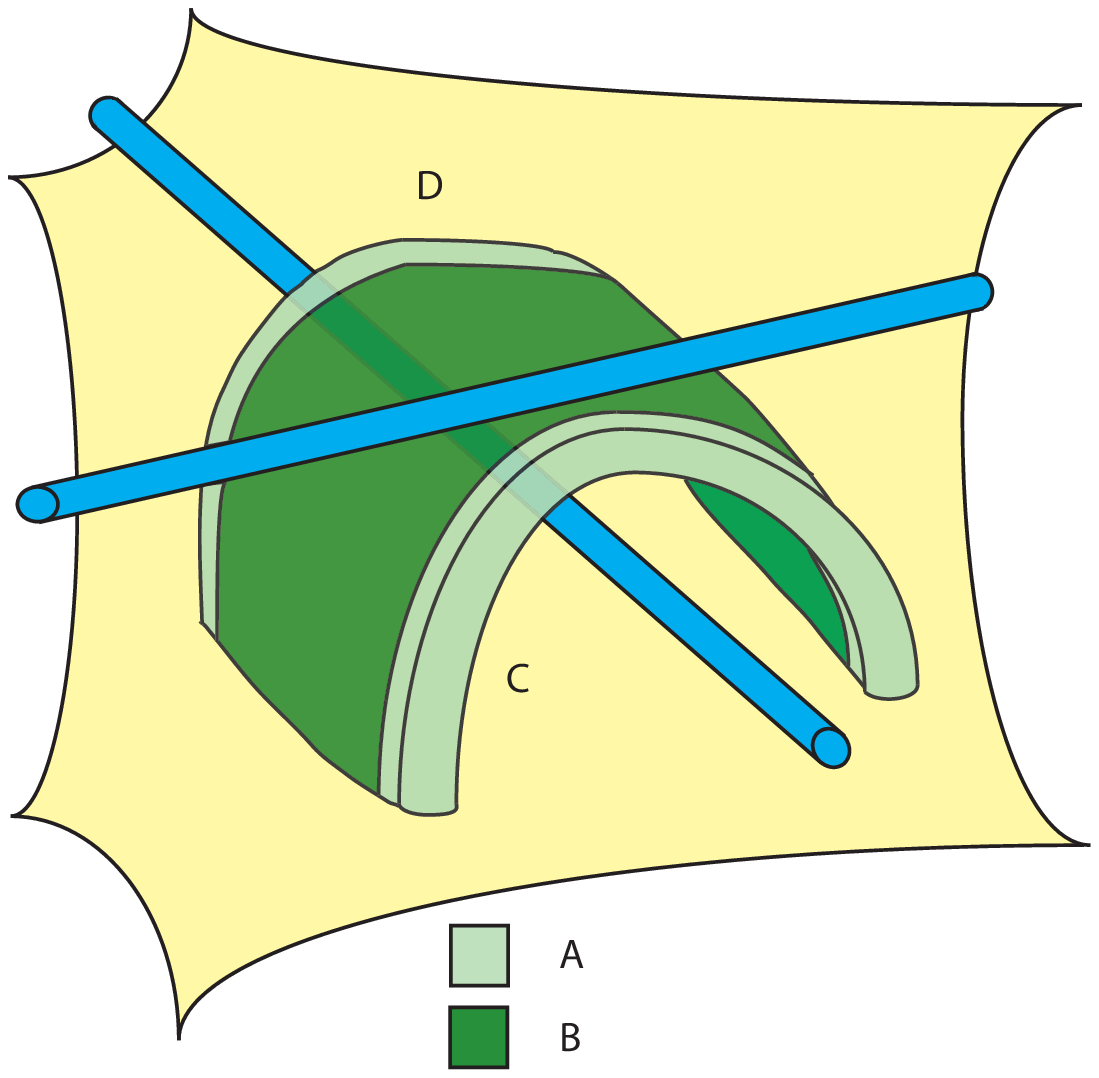}}
  \caption{Combining two of the 1-handles.}
  \label{Fig:bigcrossing29}
  \end{center}
\end{figure}

The outcome of this is that we half the number of 1-handles, so there are now $c$ in total.  Figure \ref{Fig:bigcrossing3} shows the final configuration on 1-handles at each crossing.

\begin{figure}[h]
  \begin{center}
    {\psfrag{A}{$=$ 1-handles.}
    \psfrag{C}{$h^1_{i_1}$}
    \psfrag{D}{$h^1_{i_2}$}
    \psfrag{B}{$h^1_{i_3}$}
    \includegraphics[width=9cm]{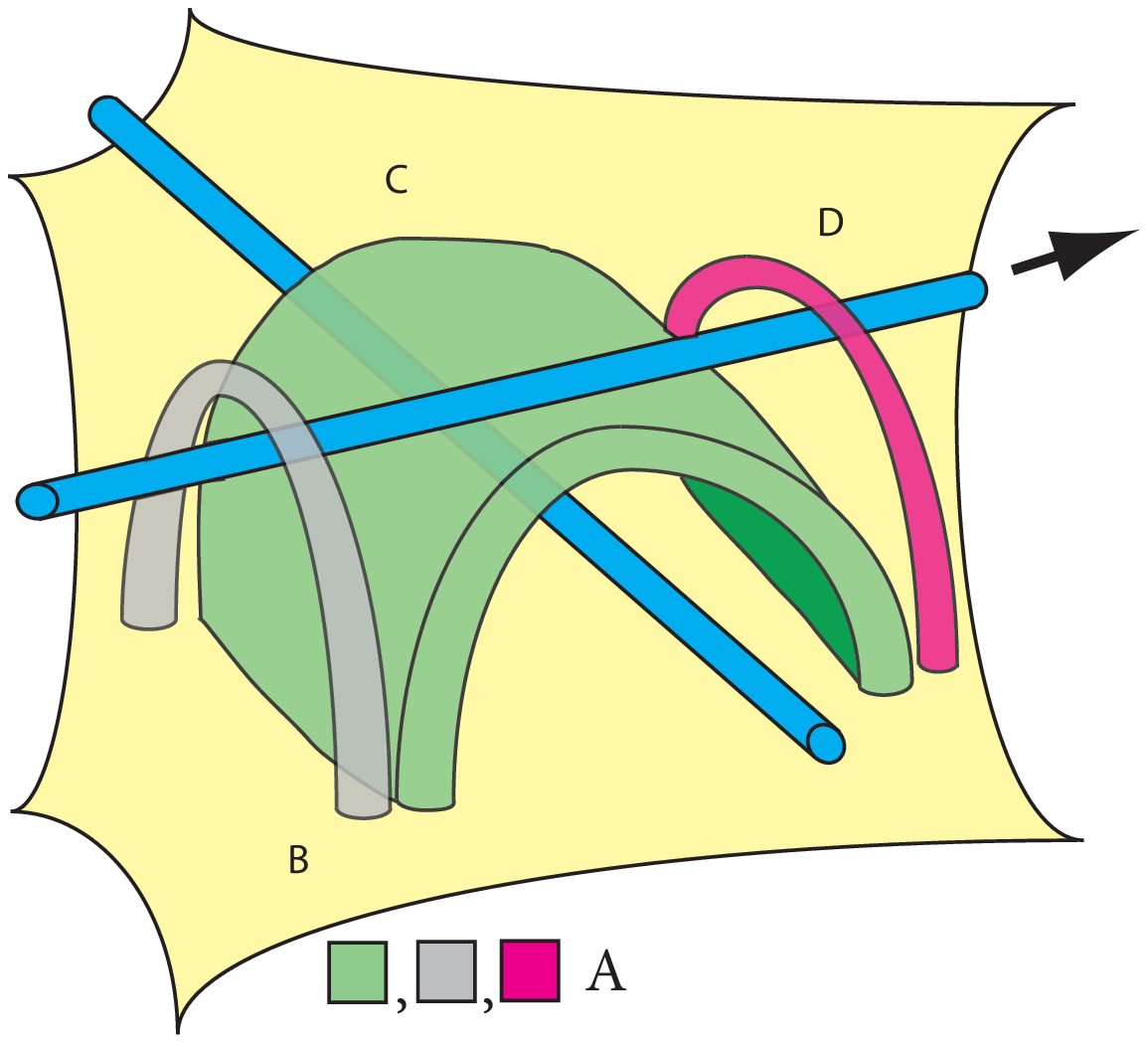}}
  \caption{The 1-handles.}
  \label{Fig:bigcrossing3}
  \end{center}
\end{figure}

The next step is to attach the 2-handles.  For each crossing, and therefore region of the quadrilateral decomposition of $S^2$, we glue a 2-handle on top of the knot, with boundary circle which goes around the 1-handles according to the boundary of the region of $S^2$.  Figure \ref{Fig:bigcrossingfinish} shows the status of the handles thus far; however it only shows the core of the 2-handle for clarity.

\begin{figure}[h]
  \begin{center}
    {\psfrag{A}{$= K(S^1)$ near crossing $i$.}
    \psfrag{B}{$=$ boundary of $h^0$.}
    \psfrag{C}{$=$ attaching sphere $S^1 \times \{0\}$ of $h^2_i$.}
    \psfrag{D}{$=$ core $D^2 \times \{0\}$ of $h^2_i$.}
    \psfrag{E}{$= h^1_{i_1}.$}
    \psfrag{F}{$= h^1_{i_2}$ and $h^1_{i_3}$.}
    \includegraphics[width=10cm]{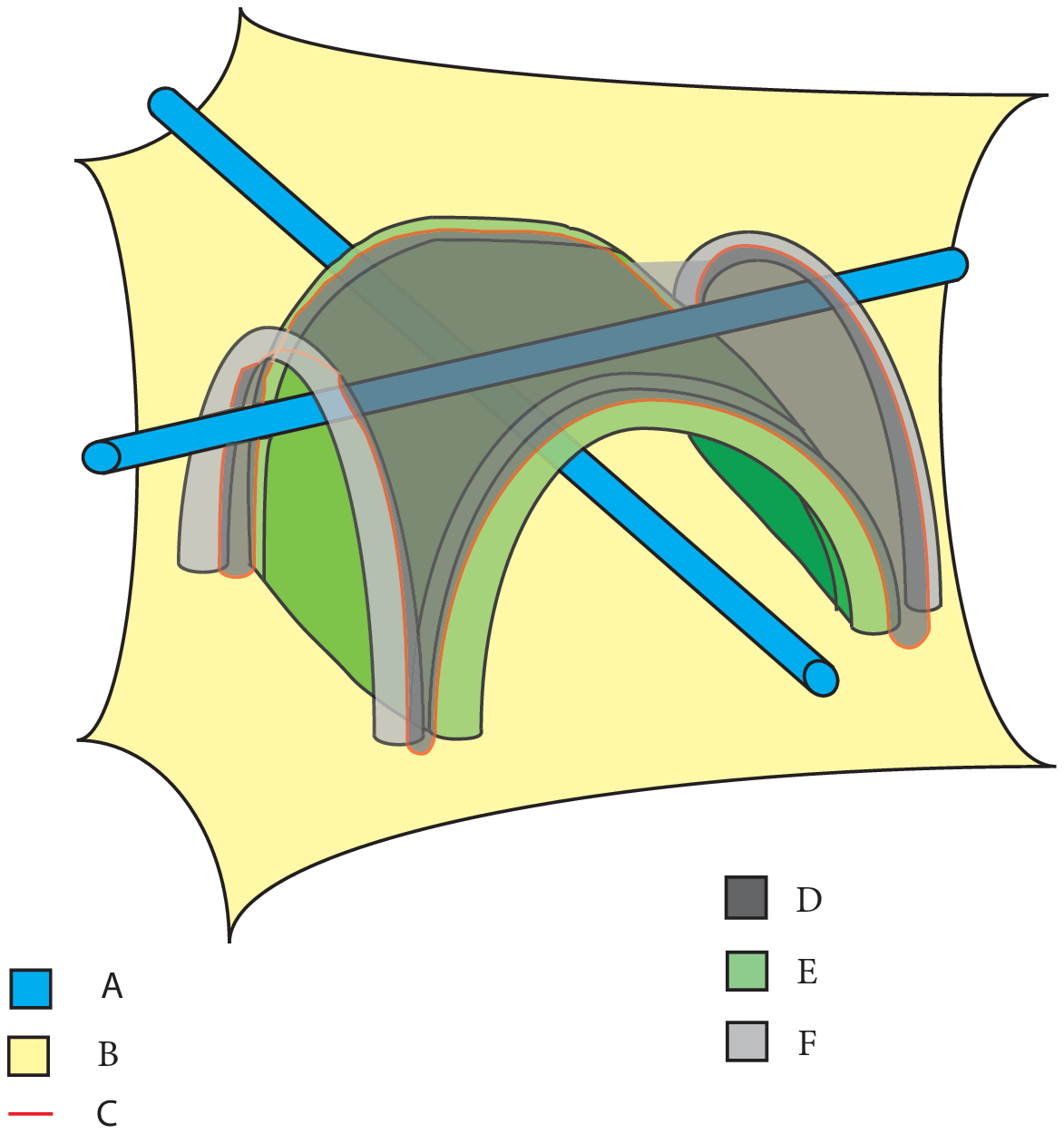}}
  \caption{The 2-handle attachment.}
  \label{Fig:bigcrossingfinish}
  \end{center}
\end{figure}

Finally, after a 2-handle is attached over each crossing of the knot, we have $c$ 2-handles, and the upper boundary of the 2-skeleton is again homeomorphic to $S^2$.  This means that we can attach a 3-handle to fill in the rest of $S^3$ and so complete the handle decomposition of $X$ as claimed.
\end{proof}

\begin{proposition}\label{wirtingerpresentation}
Let $x_0 \in X$ be $D^0 \times \{0\}$, the centre of the 0-handle.  A presentation for the fundamental group of $X$ is given by:
\[\pi_1(X,x_0)\; = \; \langle \; g_1,..,g_{c}\,|\, r_1,..,r_c\,\rangle\]
where:
\[r_i =
\begin{cases}
g_{i_2}^{-1}g_{i_1}^{-1}g_{i_3}g_{i_1} & \text{if crossing }i \text{ is of sign }+1;\\
g_{i_2}^{-1}g_{i_1}g_{i_3}g_{i_1}^{-1} & \text{if crossing }i \text{ is of sign }-1.
\end{cases}\]
A crossing of sign $+1$ is shown in Figure \ref{Fig:crossinglabels}.  A crossing of sign $-1$ has the orientation on the under-strand reversed.  There is one more relation in this presentation than necessary.
\end{proposition}

\begin{proof} If we crush each of the thickening disks (the $D^{3-r}$) of the handles in our handle decomposition for $X$, each $r$-handle $h^r_j$ becomes an $r$-cell $e^r_j$, and so we are left with a CW complex $X'$ which is homotopy equivalent to $X$.  The fundamental group of $X$ only depends on the homotopy type, and it can be simply derived from $X'$.
Each 1-cell $e^1_j$ corresponds to a generator $g_j$ for the knot group $\pi_1(X)$, and each 2-cell to a relation, in the Wirtinger presentation (see \cite{BZ}, \cite{Fox1}).  We list the four 1-cells for which the composition of the attaching map of the 2-cell with the collapse to a single 1-cell,
$\theta \colon \partial D^2 = S^1 \to \bigvee_{i=1}^c \, S^1 \to S^1$, has non-zero degree \emph{i.e.} the four 1-handles over which the attaching circle of the 2-handle runs, in order, with an exponent $\pm 1$ according to the degree of $\theta$.

The 3-handle corresponds to a redundancy between the relations, which arises as all the 2-cells together make up a sphere, which is the boundary of the 3-cell (see Remark \ref{Rmk:identitypresentation} and \cite{Trotter}).
\end{proof}

\begin{remark}\label{Remark:unknothandles}
Although by making a diagram for the unknot with $3$ or more crossings we can include it into the discussion thus far, it can be dealt with much more simply than this.  We elucidate how for completeness and since this will be of use later.  We note that the exterior of the unknot $U$ is homeomorphic to $D^2 \times S^1$. The decomposition $S^3 \approx S^1 \times D^2 \cup_{S^1 \times S^1} D^2 \times S^1$ is sometimes called the Clifford decomposition of $S^3$.  As we are about to do for any knot, we give a handle decomposition which contains the boundary torus $S^1 \times S^1 \times I$ as a sub-complex; it includes a collar neighbourhood since all our handles must be 3-dimensional.  There is one 0-handle, $h^0_{\partial}$.  Attached to this are two 1-handles, $h^1_{\mu}$ and $h^1_{\lambda}$, the meridian and longitude 1-handles respectively, with corresponding generators of the fundamental group $\mu$ and $\lambda$.  We then glue a 2-handle $\hp^2$ to this using the word $\lambda \mu \lambda^{-1} \mu^{-1}$.

We then fill in the exterior of our torus to make a solid torus $D^2 \times S^1$.  To begin, we attach another 2-handle $h^2_s$ whose boundary is the longitude.  We then fill in the remaining exterior, which is now a 3-ball, with a 3-handle $h^3_s$.  The 3-handle attaches to either side of $h^2_s$, and to the inside of $\hp^2$.  This completes our handle decomposition of $D^2 \times S^1 = X_U$.
The fundamental group of a solid torus is of course \[\pi_1(X_U) \cong \pi_1(S^1 \times D^2) \xrightarrow{\simeq} \langle \mu \rangle \cong \Z.\]
\end{remark}

\begin{conventions}\label{conventions}
It is worthwhile at this point to establish our numbering and orientation conventions.  First, we describe how to number the crossings of the knot, and hence both the regions of the quadrilateral decompositions, and the 2-handles, in a coherent way.  Call the crossing in the infinite region number 1.  Then, starting at this crossing, we go along the over crossing strand and follow around the knot, according to its orientation.  \emph{We enumerate each crossing as we come to it along an over strand.}

At a crossing $i$ of positive sign we label the 1-handles which go over its strands, and the corresponding generators of $\pi_1(X)$, as shown in Figure \ref{Fig:crossinglabels}.  That is, the over-strand has handles $h^1_{i_2}$ and $h^1_{i_3}$, the former being the furthest along the knot with respect to the orientation.  The handle over the under strand is labelled $h^1_{i_1}$; recall that we used handle cancellation in order to have only one 1-handle here.  The same convention for naming the 1-handles is used if the orientation on the under strand of the knot is reversed, and the crossing is a negative one.  Let\footnote{Unless two or more of the 1-handles $h^1_{1_1}, h^1_{1_2}$, or $h^1_{1_3}$ turn out to be the same due to the handle cancellation of the 2-handles which go through the crossings.} $h^1_1 = h^1_{1_1}$, $h^1_2 = h^1_{1_2}$ and $h^1_3 = h^1_{1_3}$.  After that, all other naming of 1-handles is arbitrary.

We use the notation $g_i$ when we are referring to the group element in $\pi_1(X)$ or in the free group $F(g_1,\dots,g_c)$ for which the core of the handle (and some paths from the basepoint to its feet) is a representative.  We reserve the notation $h^1_i$ for when we are referring to the ``physical'' handle.

We orient each of the 1-handles according to the right-hand thumb rule: one puts the thumb of one's right hand along the knot in line with its orientation and one's fingers then indicate the direction of the orientation to be put on the 1-handle which goes over the knot there.  This ensures that each 1-handle has linking number $+1$ with the knot\footnote{Technically in order to define linking number we have to talk about the 1-cell which is created when the handles are contracted to cells as in the derivation of $\pi_1(X)$.}.  The 2-handles are oriented in an anti-clockwise manner, looking down on them from within the 3-handle.  Note that in the knot diagram quadrilateral decomposition the 2-handles are oriented in an anti-clockwise manner as we look at them, except that the 2-handle over the infinite region appears to be oriented in a clockwise manner, since it is in fact the rest of $S^2$.  When listing relations in the knot group, as the boundary of a 2-handle, we always list $g_{i_2}$ first.

\begin{figure}[h]
  \begin{center}
 {\psfrag{A}{$h^1_{i_3}, g_{i_3}$}
 \psfrag{B}{$h^1_{i_2}, g_{i_2}$}
 \psfrag{D}{$h^1_{i_1}, g_{i_1}$}
 \includegraphics[width=7cm]{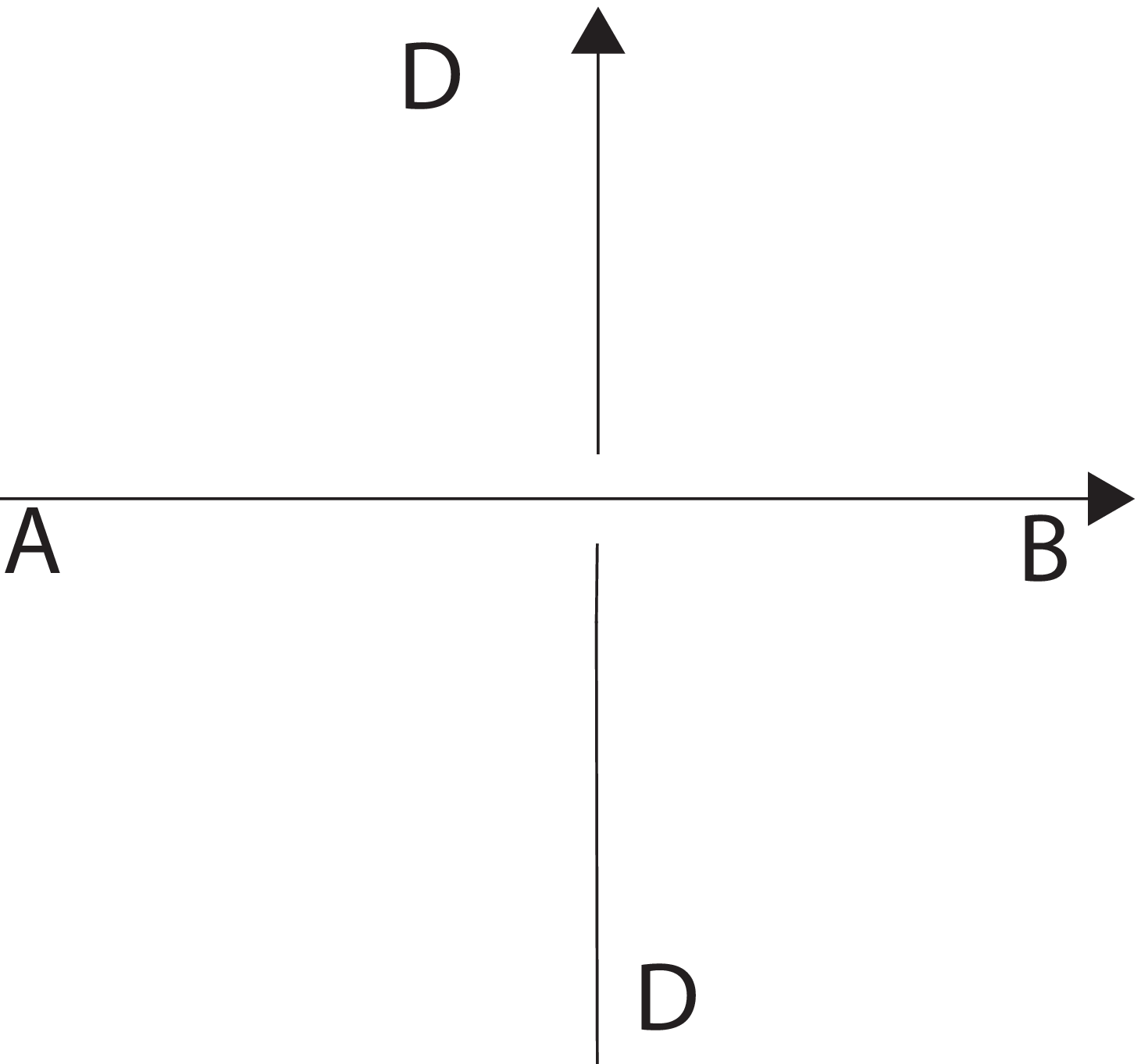}}
 \caption{A crossing of positive sign and the labelling of the fundamental group generators associated to it.}
\label{Fig:crossinglabels}
\end{center}
\end{figure}

\qed \end{conventions}

We now extend our decomposition so as to include the boundary of $X$.

\begin{remark}\label{Rmk:needperipheralstructure}
The handle decomposition of Theorem \ref{handledecomp} depends only on the fundamental group of $X$ (and a presentation of the group).  Indeed, since a knot exterior is an Eilenberg-Maclane space $K(\pi,1)$ (see Remark \ref{Rmk:identitypresentation} for an explanation of this), the homotopy type of $X$ only depends on the knot group.  As explained in the introduction we need more information than this, which we now illustrate.  Recall (see \cite{Fox1}) that the reef and granny knots have isomorphic groups.  However the reef knot is the trefoil connect summed with the reverse of its obverse, and so is slice.  The granny knot is the trefoil summed with itself, so has signature 4 and is not slice; the two knots are not concordant.  We must have boundary information in order to define concordance invariants; it is the inclusion of the boundary which differentiates between the reef and granny knots. In order to have this information algebraically in the chain complex of the fundamental cobordism of a knot, we must first include the boundary in our handle decomposition.
\end{remark}

\begin{theorem}\label{Thm:includingboundary}
Given a reduced diagram for a knot $K: S^1 \hookrightarrow S^3$, with $c \geq 3$ crossings, there is a handle decomposition of the knot exterior $X = \overline{S^3 \setminus N}$ which includes a regular neighbourhood of the boundary $\partial X \times I \approx S^1 \times S^1 \times I$ as a sub-complex:
\[X = \hp^0 \cup \bigcup_{i=1}^{c+2}\, h_i^1 \cup \bigcup_{j=1}^{c+3}\,h_j^2 \cup \bigcup_{k=1}^{2}\,h_k^3. \]
\end{theorem}

\begin{proof}
We begin by renaming the 0- and 3-handles which are already in our decomposition of $X$ as $h^0_o = h^0_1 := h^0$ and $h^3_o = h^3_1 :=h^3$, where the o stands for original.

The idea of the construction is to begin by including a decomposition of the boundary torus with collar neighbourhood $S^1 \times S^1 \times I \approx \partial X \times I$: we use the decomposition of the torus described above in Remark \ref{Remark:unknothandles}.

We define $\hp^0:= h^0_2$, $h^1_{\mu}:=h^1_{c+1}$, $h^1_{\lambda}:=h^1_{c+2}$ and $\hp^2:=h^2_{c+1}$.

We then must connect the boundary to the rest of $X$.  We do this by adding, for each $n$-handle of $\partial X$, an $(n+1)$-handle of $X$.  In this way we realise the inclusion of the boundary as a sequence of elementary handle additions: a simple homotopy equivalence.  Equivalently, we are taking the mapping cylinder of the inclusion map $f \colon S^1 \times S^1 \hookrightarrow S^1 \times S^1 \times I \approx \partial X \times I \to X$.

To begin, we add a connecting 1-handle for the 0-handle $\hp^0$, which connects it to $h^0_o$, oriented so as to point from $h^0_o$ to $\hp^0$.  We call this $h^1_{\partial}$.  We then need to connect each of the 1-handles in the boundary to the 1-handles already there.  To do this, first pick a 1-handle: we choose $h^1_1$ for no special reason.  This is a meridian so we add a 2-handle, which we call $h^2_{\partial \mu}=h^2_{c+2}$, with an attaching map which starts at $\hp^0$, goes around $h^1_{\mu}$ against its orientation, along $h^1_{\partial}$, around $h^1_1$ with its orientation, and then back along $h^1_{ \partial}$.

Next, we need to see how the longitude lives in our handle decomposition.  Look again at Figure \ref{Fig:bigcrossingfinish}, and imagine the longitude as a curve following the knot, just underneath it.
\begin{definition}
To crossing $j$ of an oriented knot diagram we associate a sign $\eps_j \in \{-1,+1\}$, which is $+1$ if the crossing is as shown in Figure \ref{Fig:crossinglabels}, and $-1$ if the orientation on the under-strand is reversed.  The writhe of the diagram, $\Wr$, is \[\Wr := \sum_{j=1}^c \, \eps_j.\]
\qed \end{definition}
Since the writhe of the diagram is potentially non-zero, in order to have the zero-framed longitude, we take it to wind $(-\Wr)$ times around the knot, underneath the tunnel created by the 1-handle $h^1_1$.  Then, deforming the longitude directly downwards, that is towards the 0-handle $h^0_o$, everywhere apart from underneath $h^1_1$, we can see that at the over-strand of crossing $j$, the longitude follows the 1-handle $h^1_{i_1}$ respecting the orientation if $\eps_j = 1$ and opposite to the orientation if $\eps_j=-1$.  As it follows under-strands we deform it to the 0-handle, so these have no contribution to the longitude.  However, within the tunnel underneath $h^1_1$, we instead deform the longitude outwards to see that it follows $h^1_1$, $(-\Wr)$ times.  A word for the longitude, as an element of $\pi_1(X)$, in terms of the Wirtinger generators, is:
\[l = g_1^{-\Wr}g_{k_1}^{\eps_k}g_{(k+1)_1}^{\eps_{k+1}}\dots g_{(k+c-1)_1}^{\eps_{k+c-1}};\]
where $k$ is the number of the crossing reached first as an over crossing, when starting on the under crossing strand of the knot which lies in region 1; the indices $k, k+1,\dots,k+c-1$ are to be taken mod $c$, with the exception that we prefer the notation $c$ for the equivalence class of $0 \in \Z_c$.  This happily coincides with \cite[Remark~3.13]{BZ}.


This now enables us to attach the longitude 1-handle of the boundary to the rest of the 1-handles.  We use a 2-handle $h^2_{\partial \lambda} = h^2_{c+3}$ which has an attaching map which starts at $\hp^0$, traverses $h^1_{\partial}$ and then follows the 1-handles according to the letters in the word $l$.  It then traverses $\hp^1$, before finally following along $h^1_{\lambda}$ against its orientation.

Finally we include a 3-handle into the gap remaining, in order to relate $h^2_{\partial}$ to the 2-handles already in $X$, which we call  $h^3_{\partial} = h^3_2$.  It attaches to $\hp^2$, and to both $h^2_{\partial \mu}$ and $h^2_{\partial \lambda}$ on the top and the bottom.  It also attaches to the underneath of the 2-handles $h^2_k,..,h^2_{k+c-1}$, again with the same convention on the indices, so as to connect these handles to the boundary 2-handle.

To finish we remark that we can, to reduce the number of handles, cancel the 1-handle $\hp^1$ with the 0-handle $h^0_o$, amalgamating $\hp^0, \hp^1$ and $h^0_o$ into a single 0-handle, which we still call $\hp^0$ to emphasise that it is also the 0-handle of the boundary sub-complex.

This completes our description of the additional handles required to include the boundary as a sub-complex.
\end{proof}

\begin{definition}
We denote by $M_K$ the zero-surgery on the knot; by which we mean the manifold obtained from $S^3$ by 0-framed surgery along the knot $K \colon S^1 \hookrightarrow S^3$.
\[M_K := X \cup_{S^1 \times S^1} D^2 \times S^1,\]
where the gluing is done so that the \emph{longitude} bounds.
\qed \end{definition}

\begin{remark}\label{rmk:zerosurgery1}
If a knot $K$ is a slice knot then if $W$ is the exterior of a slice disk $\Delta^2$, \[W := \ol{D^4 \setminus (\Delta^2 \times D^2)},\] then the boundary of $W$ is the zero-surgery $M_K$.  The zero-surgery has the useful property that it is a closed manifold, and so has Poincar\'{e} duality without having to factor out the boundary.  In many applications this makes it simpler to work with than the knot exterior.  We can construct a handle decomposition for the zero-surgery by adding just two handles to our decomposition for $X$, to make a handle decomposition of the solid torus $D^2 \times S^1$.  We may as well go back to the original decomposition of Theorem \ref{handledecomp}, since the boundary and connected handles of Theorem \ref{Thm:includingboundary} will now be superfluous.  The first is a 2-handle, $h^2_s$, where the s stands for surgery, which has as its attaching map the longitude of the knot, much like the attachment of the 2-handle $h^2_{\partial \lambda}$ in Theorem \ref{Thm:includingboundary}.  The rest of the solid torus is a 3-ball, so we attach a 3-handle $h^3_s$ to fill it in, to either side of the 2-handle $h^2_s$, and to the underside of each of the 2-handles of $X$, much like the attachment of the 3-handle $h^3_{\partial}$ in Theorem \ref{Thm:includingboundary}.
\end{remark}

Finally in this chapter, we describe the handle decomposition of a torus $S^1 \times S^1$ split into two, as the union of two copies of $S^1 \times D^1$ along their common boundary $S^1 \times S^0$.

\begin{figure}[h]
    \begin{center}
 {\psfrag{A}{$h^0_+$}
 \psfrag{B}{$h^0_-$}
 \psfrag{C}{$h^1_-$}
 \psfrag{D}{$h^1_+$}
 \psfrag{E}{$h^1_a$}
  \psfrag{F}{$h^1_b$}
 \psfrag{G}{$h^2_b$}
 \psfrag{H}{$h^2_a$}
 \includegraphics[width=7.5cm]{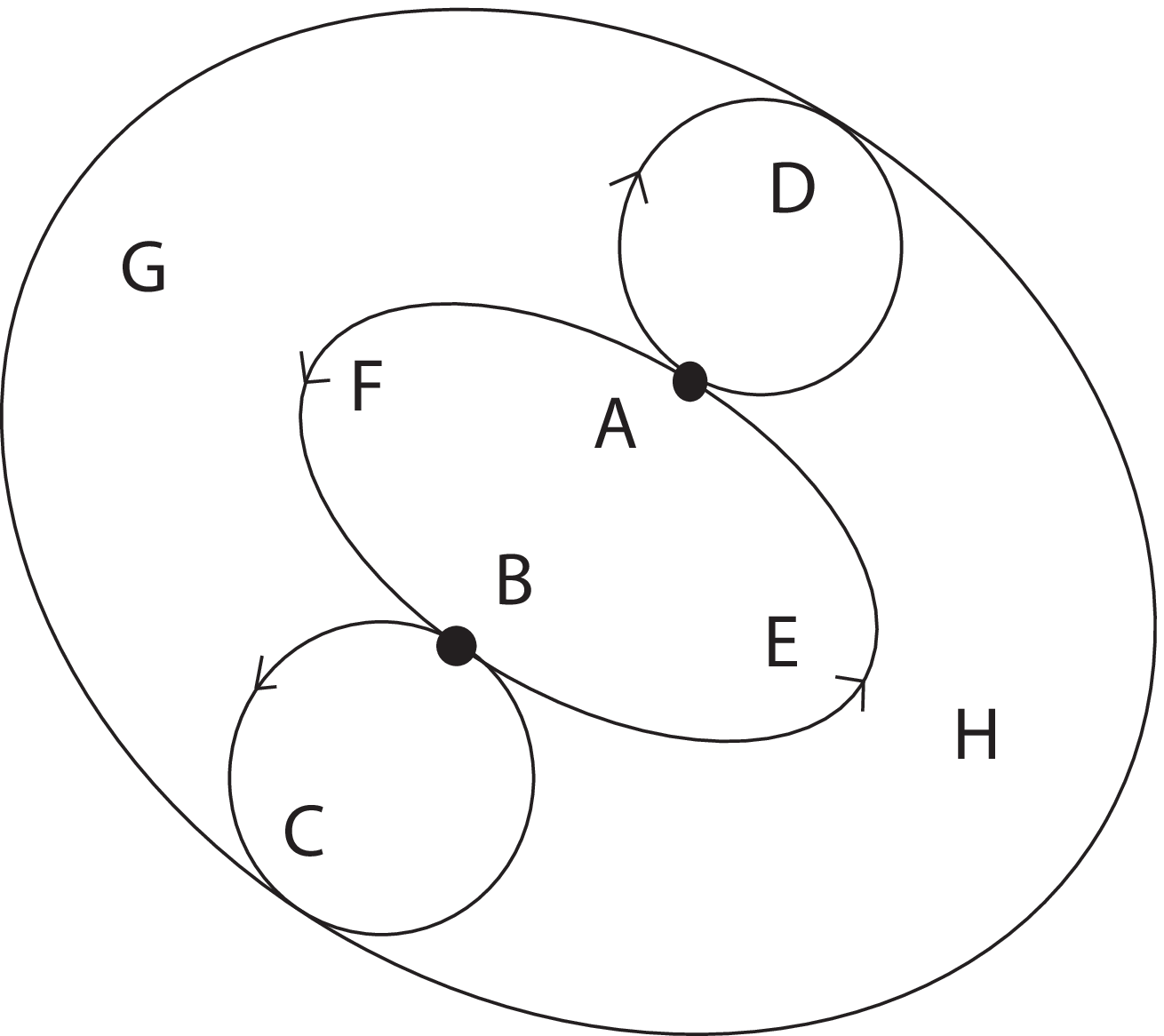}
 }
 \caption{A handle decomposition of the torus with a splitting into two cylinders.}
 \label{Fig:torusintwocylinders}
 \end{center}
\end{figure}

There are two 0-handles, $h^0_+$ and $h^0_-$.  Attached to each of these is a 1-handle, $h^1_+$ and $h^1_-$, which form the boundaries of the two cylinders, a copy of $S^1 \times S^0$.  We then join the two circles with the longitudinal 1-handles, $h^1_{a}$ and $h^1_{b}$, before filling in the holes to complete the torus with 2-handles $h^2_a$ and $h^2_b$.

We shall make use of this as we use these handle decompositions to yield chain complexes.  Considering the boundary split in this manner allows us to consider the knot exterior as the fundamental cobordism of the knot, as described in the introduction to this chapter. 

\chapter[The chain complex of a knot exterior]{A Chain Complex of the Universal Cover of a Knot Exterior}\label{chapter:chaincomplex}

In this chapter we make the transfer from geometry to algebra: using our handle decomposition of a knot exterior, we derive a chain complex of its universal cover, and thence of any covering space.

The constructions are given in some detail.  The reader interested in the symmetric structure, whose application to low dimensional knot theory is perhaps less common, may wish to skip to chapter \ref{Chapter:duality_symm_structures}.

Let $S \leq \pi_1(X)$ be a normal subgroup of the fundamental group of $X$, and let $\wt{X}$ be the regular covering space associated to it, so that $\pi_1(\wt{X}) \cong S$, and the deck transformations of the covering space are $\pi := \pi_1(X)/S$.  We refer to $\wt{X}$ as the $\pi$-\emph{covering} of $X$.  We will be interested primarily in the cases where $S$ is an element of the derived series, which is given by taking iterated commutator subgroups, of $\pi_1(X)$.  If we work initially with the universal covering space, then when situations demand we can go to simpler covering spaces at will; by taking account of the whole fundamental group we do not lose any information at this crucial initial stage of converting the geometry into algebra.  One of the broad goals of this work is to support the philosophy that the chain complex of the universal cover of the knot exterior is a universal invariant for topological concordance.
\begin{definition}\label{handlechain}
The \emph{handle chain complex} of a $\pi$-covering space $\wt{X}$ of an $n$-manifold $X$ with a handle decomposition has as chain groups $C_i(\wt{X})$ the free left $\Z[\pi]$-modules generated on the $i$-handles, since there is one lift of each handle for each element of $\pi$.  The boundary maps $\partial_{i+1} \colon C_{i+1}(\wt{X}) \to C_i(\wt{X})$ are given by the \emph{twisted $\Z[\pi]$-incidence numbers} $\langle h^{i+1}_j \,|\, h^i_k \rangle_{\Z[\pi]}$ (defined below):
\[\partial_{i+1}(\widetilde{h}^{i+1}_j) = \sum_k  \, \langle h^{i+1}_j \,|\, h^i_k \rangle_{\Z[\pi]}\,\widetilde{h}^i_k,\]
where $\wt{h}^i_j$ is a chosen lift of $h^i_j$.  The boundary maps need only be defined on such elements, and the $\Z[\pi]$-module structure determines the rest.

The $\Z$-incidence numbers $\langle h^{i+1}_j\,|\,h^i_k \rangle$ are the algebraic intersection numbers of the attaching sphere $S^i \times \{0\}$ of $h^{i+1}_j$ with the belt sphere $\{0\} \times S^{n-i-1}$ of $h^i_k$.  These spheres both live with complementary dimensions in the $(n-1)$-manifold which is the boundary of $X^{(i)}$, and so their intersection can be made transverse, and hence a finite number of points with signs, which can then be counted algebraically.

\begin{conventions}\label{orientationconventions}
In order to be able to attribute signs to intersection points, we need to fix our orientation conventions.  To put an orientation on the boundary of $X^{(i)}$ as a basis for comparison, we pick an outward pointing normal vector to $\partial X^{(i)}$ and list this \emph{first}, followed by a basis of $n-1$ tangent vectors to $\partial X^{(i)}$.  We choose the orientation on $\partial X^{(i)}$ in such a way that the orientation of these $n$ tangent vectors to $X$ agrees with our fixed standard orientation of $X$.  Our handles are oriented as subsets of $X$.  We described in Conventions \ref{conventions} how to orient the cores $D^i \times \{0\}$ of each of our handles.  Given the fixed orientation on $X$, this fixes an orientation on the cocores $\{0\} \times D^{n-i}$.  The attaching and belt spheres inherit orientations from the cores and cocores respectively, again using the outward-pointing-normal-first convention.  For $i=n-1$ the belt spheres are copies of $S^0$, while for $i=1$ the attaching spheres are copies of $S^0$.  In these cases, rather than giving an ordering of a basis for the tangent space, which is hard to accomplish for a 0-dimensional vector space, we give each point of $S^0 = \{-1,+1\}$ a sign in the obvious way: it inherits this from the orientation of $D^1 = [-1,1]$ of which it is a boundary.

For simplicity we can assume we are working with smooth manifolds, in order to define tangent spaces and use transversality here.  However we know that these constructions can be carried out with considerably more effort for topological manifolds too: see \cite{MilnorMicrobundles}.  We therefore do not actually have to restrict categories here.

\end{conventions}

In the case where $\pi$ is non-trivial, each intersection point of the attaching and belt spheres has not only a sign, but also an element of $\pi_1(X)$ associated to it, and hence the incidence number lives in $\Z[\pi]$.  We have to record whether, when lifted, the intersection is between $\widetilde{h}^{i+1}_j$ and $\widetilde{h}^i_k$ or if in fact the intersection is with some other lift, or $\pi$-translate, of $h^i_k$.

A \emph{threaded} handle (\cite{Scorpan}, 1.7) is a handle $h^i_k$ of an $n$-manifold $X$ together with a path $c^i_k \colon [0,1] \to X$ from the base point $x_0$ of $X$ to the centre $\{0\} \times \{0\} \in D^i \times D^{n-i}$.  A handle $h^i_k$ is contractible, so for an intersection point $p$ there is a unique homotopy class of paths $[\gamma^i_{kp}]$, from the centre of $h^i_k$ to $p$.  We can then form a loop associated to an intersection point of two threaded handles $h^{i+1}_j$ and $h^i_k$:
\[c^{i+1}_j \ \ast_p \overline{c^i_k} := c^{i+1}_j \cdot \gamma^{i+1}_{jp} \cdot \overline{\gamma^i_{kp}} \cdot \overline{c^i_k},\]
where the bar means that we take the reverse of the path.
This represents a homotopy class in $\pi_1(X,x_0)$, giving us an element $[c^{i+1}_j \ \ast_p \overline{c^i_k}] \in \pi$.  We define
\[\langle h^{i+1}_j \,|\, h^i_k \rangle_{\Z[\pi]} := \sum_p \pm [c^{i+1}_j \ \ast_p \overline{c^i_k},]\]
taking the sum over all intersections of the attaching and belt spheres of the two handles in $X$; the sign is from matching orientations, just as with the $\Z$-incidence numbers.
\qed \end{definition}

An element of $\pi_1(X)$ is represented by a word $w$ in $F$, the free group on $g_1,\dots,g_c$.  This in turn determines a path in $\partial \wt{X}^{(1)}$, which in the case $w=r_i$ is a lift of the attaching sphere of $h^2_i$.  The free differential calculus (Definition \ref{freederivative}), due to Fox (\cite{Fox1}), is a formalism that tells us which chain this path is in $C_1(\wt{X})$.  In particular, it will be used to derive the boundary map $\partial_2$ in our chain complex:
\[\frac{\partial r_i}{\partial g_j} = \langle h^2_i \,|\, h^1_j \rangle_{\Z[\pi]}.\]

\begin{definition}\label{freederivative}
The \emph{free derivative} of a word $w$ in a free group $F$ with respect to a generator $g_i$ is a map $\frac{\partial}{\partial g_i} \colon F \to \Z[F]$ defined inductively, using the following rules:
\[\frac{\partial (1)}{\partial g_j} = 0; \;\;\; \frac{\partial g_i}{\partial g_j} = \delta_{ij}; \;\;\; \frac{\partial(uv)}{\partial g_j} = \frac{\partial u}{\partial g_j} + u \frac{\partial v}{\partial g_j}.\]
Extending this using linearity makes the free derivative into an endomorphism of the group ring $\Z[F]$.
\qed \end{definition}

\begin{conventions}\label{conventions2}
We consider the handle chain groups as based free left $\Z[\pi]$-modules in order to define our conventions: the chain groups inherit a particular choice of basis from our geometric constructions.  For some basis element, $\wt{h}^i$, and for $g, g_1, g_2 \in \pi$, we define $g\wt{h}^i$ to be the lift of the handle $h^i$ arrived at by translating $\wt{h}^i$ along $g$.  In particular, note that this means that we define $g_1g_2\wt{h}^i$ to be the lift of the handle $h^i$ arrived at by translating $\wt{h}^i$ first along $g_1$, and then along $g_2$.
We define module homomorphisms only on the basis elements of a free module, and use the left $\zp$ module structure to define the map on the whole module.  This has the effect, in the non-commutative setting, that when we want to formally represent elements of our based free modules as vectors with entries in $\Z[\pi]$ detailing the coefficients, then the vectors are written as row vectors, and the matrices representing a map must be multiplied on the right.  This is because the order of multiplication of two matrices should be preserved when multiplying elements to calculate the coefficients.

In later chapters, when considering matrices of homomorphisms which compose on the left in the usual way, acting on direct sums of modules which themselves may or may not be free modules, we retain the usual convention of column vectors and matrices acting on the left.
\qed\end{conventions}

\begin{theorem}\label{Thm:unicoverchaincomplex}
Suppose that we are given a knot $K$ with exterior $X$, and a reduced knot diagram for $K$ with $c \geq 3$ crossings.  Then there is a presentation
\[\pi_1(X) = \langle\,g_1,\dots,g_c\,|\,r_1,\dots,r_c\,\rangle\]
with the Wirtinger relations $r_1,\dots,r_c \in F(g_1,\dots,g_c)$ read off from the knot diagram.  The handle chain complex of the $\pi$-cover $\widetilde{X}$ (the cover with deck group $\pi := \pi_1(X)/S$ for some normal subgroup $S \unlhd \pi_1(X)$), with chain groups being free left $\Z[\pi]$-modules is given by:
\[\Z[\pi] \xrightarrow{\partial_3} \bigoplus_{c} \,\Z[\pi] \xrightarrow{\partial_2} \bigoplus_{c} \,\Z[\pi] \xrightarrow{(\Phi(g_1-1),\dots,\Phi(g_c-1))^T} \Z[\pi]\]
\[(\partial_2)_{ij} = \Phi \left( \frac{\partial r_i}{\partial g_j}\right)\]
\[\partial_3 = (\Phi(w_1),\dots,\Phi(w_c)).\]
The ring homomorphism $\Phi \colon \Z[F] \to \Z[\pi]$ is defined by linearly extending the homomorphism $\phi \colon F \to \pi$.
To determine the words $w_i$ which arise in $\partial_3$, consider the quadrilateral decomposition of the knot diagram (Definition \ref{Defn:quaddecomp}).  At each crossing $i$, we have a distinguished edge which we always list first in the relation, $g_{i_2}$.  Choose the vertex, call it $v_i$, which is at the end of $g_{i_2}$.  For crossing $i$, choose a path in the 1-skeleton of the quadrilateral decomposition from $v_1$ to $v_i$.  This yields a word $w_i$ in $g_1,\dots,g_c$.  Then the component of $\partial(h^3)$ along $h^2_i$ is $\Phi (w_i)$.
\end{theorem}

\begin{proof}
Follow Definitions \ref{handlechain} and \ref{freederivative}.  We therefore have to thread the handles.  Let $x_0$ be the centre of the 0-handle.  We define the paths $c^i$ using the Conventions \ref{conventions}.  For 1-handles, go from $x_0$ to the foot at which it starts, with respect to the orientation of the 1-handle, then along its core to its centre.  For a 2-handle $h^2_i$, we have a vertex $v_i$ on the quadrilateral decomposition at which the word $r_i$ which represents the boundary starts.  Take for $c^2_i$ a path from $x_0$ to this vertex on $\partial h^0$, then follow from there to the centre of the 2-handle.  For $c^3$, go from $x_0$ to $v_1$, and from there up to the centre of the 3-handle.
\end{proof}

\begin{remark}
The fundamental formula of Fox for his derivative (see \cite{Fox2}) says that for a word $w \in F(g_1,\dots,g_c)$ we have
\[w-1 = \sum_{i=1}^c\, \frac{\partial w}{\partial g_i} (g_i - 1).\]
Suppose that a word $w=\wt{w}g_j$, for some $\wt{w}$ and some $j$.  Then
\[\wt{w}g_j = \wt{w}(g_j-1) + \wt{w}.\]
Similarly if $w=\wt{w}g_j^{-1}$ then
\[\wt{w}g_j^{-1} = -\wt{w}g_j^{-1}(g_j-1) + \wt{w}.\]
Working inductively on the length of $w$, at each letter $g_j$ or $g_j^{-1}$ of $w$ one factors out $g_j-1$ as above.  The formula then follows using the inductive definition of the Fox derivative.  In the case that $w \in \ker \phi$, the fundamental formula verifies that $\partial_2 \circ \partial_1 = 0$.
\end{remark}

\begin{remark}\label{Rmk:identitypresentation}
For the universal cover, the only map of the chain complex which has non-zero kernel is $\partial_2$.  This kernel is generated by the sum of all the 2-handles over the top of the knot, \emph{i.e.} $\sum_i \phi(w_i)\wt{h}^2_i$.  However this is of course precisely the boundary of $\wt{h}^3$, and so, after augmenting the chain complex with $C_0(\wt{X}) \to \Z$, we have an acyclic complex.  This is just the fact that $X$ is an Eilenberg-Maclane space; the only opportunity for $\pi_2(X)$ to be non-trivial is through 2-spheres which encompass the knot: by the Sphere theorem of Papakyriakopoulos \cite{Papa57}, a non-zero class in $\pi_2(X)$ can be represented by an embedded $S^2$; the Sch\"{o}nflies theorem means that this sphere must encompass the knot on one side, and a 3-handle on the other side, and therefore is in fact null-homotopic.  Once $\pi_2(X)$ is seen to be zero, the rest of the homotopy groups $\pi_j(X)$ for $j \geq 3$ also vanish by the Hurewicz theorem applied to the universal cover $\wt{X}$, since $H_j(\wt{X};\Z) \cong 0$ for $j \geq 3$, and since $\pi_j(X) \cong \pi_j(\wt{X})$ for $j \geq 2$.

The fact that the knot exterior is an Eilenberg-MacLane space is borne out in the way that the chain map $\partial_3$ can be expressed purely from knowledge of the fundamental group.  Given a presentation for the knot group of deficiency zero, one of the relations will be a consequence of all the others.  This is of course due to the fact that the sum of certain lifts of the 2-handles form a cycle in $C_2(\wt{X})$.
\end{remark}

\begin{definition}\label{identitypresentation}
Let $\langle \,g_1,..,g_a\,|\,r_1,\dots,r_c\,\rangle$ be a presentation of a group and let $F = F(g_1,\dots,g_a)$ be the free group on the generators $g_i$. Following Trotter \cite{Trotter}, let $P$ be the free group on letters $\rho_1,..,\rho_c$, and let $\psi:P \ast F \to F$ be the homomorphism such that $\psi(\rho_i)=r_i$ and $\psi(g_j)=g_j$. An \emph{identity of the presentation} is a word in $\ker(\psi) \leq P \ast F$ of the form:
\begin{equation}\label{identity}
s = \prod_{k=1}^c \, w_{j_k} \rho_{j_k}^{\eps_{j_k}} w_{j_k}^{-1}
\end{equation}
where $\eps_{j_k} = \pm 1$.
\qed \end{definition}

The $w_i$ here in our case coincide with the $w_i$ from Theorem \ref{Thm:unicoverchaincomplex}.  Here, however, the word chosen matters, or in other words the path in the 1-skeleton of the quadrilateral decomposition matters, rather than just the element of $\pi_1(X)$ represented, \emph{i.e.} the end point of the path.

\begin{example}

Figure \ref{Fig:labelledtrefoil} shows a reduced diagram of a trefoil, with quadrilateral decomposition.  The knot is oriented, as shown by the arrows, and the edges of the quadrilateral decomposition have been correspondingly oriented, so that they have linking number 1 with the knot.  The edges and the regions have been labelled according to the conventions laid out in \ref{conventions}.  Each region has a specific vertex, labelled $v_1, v_2$ and $v_3$, as described in the statement of Theorem \ref{Thm:unicoverchaincomplex}.

\begin{figure}
  \begin{center}
    {\psfrag{g1}{$g_1$}
    \psfrag{g2}{$g_2$}
    \psfrag{g3}{$g_3$}
    \psfrag{A}{$1$}
    \psfrag{B}{$2$}
    \psfrag{C}{$3$}
    \psfrag{va}{$v_1$}
    \psfrag{v2}{$v_2$}
    \psfrag{vc}{$v_3$}
    \includegraphics[width=6cm]{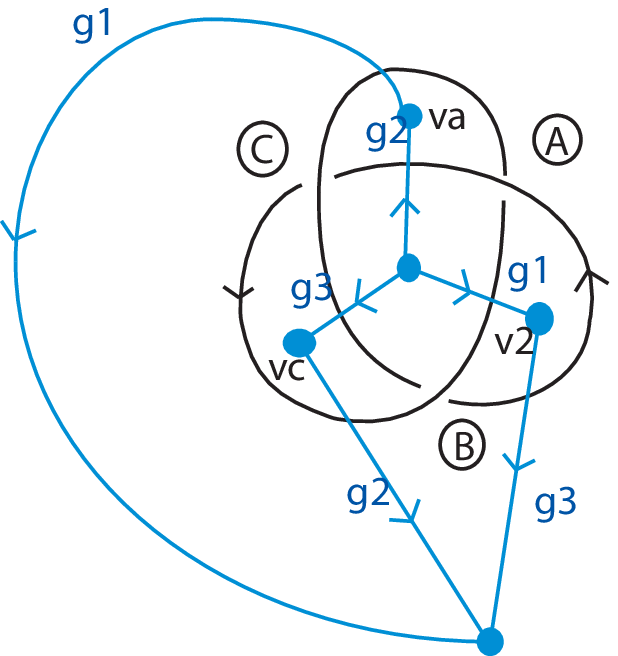}}
  \caption{An oriented diagram of a trefoil with labelled quadrilateral decomposition.}
  \label{Fig:labelledtrefoil}
  \end{center}
\end{figure}

We apply the construction of the handle decomposition associated to this diagram as in Theorem \ref{handledecomp}.  A presentation for the fundamental group of the exterior, $X$, of the knot, can be read off from the diagram, by looking at the boundaries of the regions:

\[ \pi_1(X,x_0)\; = \; \langle \; g_1,g_2,g_3\,|\, r_1 = g_2^{-1}g_1g_3g_1^{-1}, r_2 = g_1^{-1}g_3g_2g_3^{-1}, r_3 = g_3^{-1}g_2g_1g_2^{-1}\,\rangle\]
where the basepoint $x_0$ is the centre $D^0 \times \{0\}$ of the 0-handle.

By finding paths in the quadrilateral decomposition we can find words which give the boundary of the 3-handle.  Note that it is allowed here to take a different vertex of the diagram for our base vertex; this corresponds to choosing a different lift of $h^3$ as our chosen lift $\wt{h}^3$.  In this example we take the central vertex, just because it gives a more symmetrical answer, and then read off the boundary map:
\[\partial_3(\wt{h}^3) = (w_1, w_2, w_3) = (g_2,g_1,g_3).\]
The reader can check that $s_1 = g_1r_2g_1^{-1}g_2r_1g_2^{-1}g_3r_3g_3^{-1}$ is an identity of the presentation; that is, when $r_1, r_2$, and $r_3$ are substituted, $s_1 = 1 \in F(g_1,g_2,g_3)$.
We can therefore take the free derivatives of the relations in order to find $\partial_2$, and so give the handle chain complex of the universal cover of the trefoil exterior, with $\pi := \pi_1(X)$, as:
\[\Z[\pi] \xrightarrow{(g_2, g_1, g_3)} \bigoplus_{3} \,\Z[\pi] \xrightarrow{\partial_2} \bigoplus_{3} \,\Z[\pi] \xrightarrow{(g_1-1,g_2-1,g_3-1)^T} \Z[\pi]\]
where
\[\partial_2 = \left(
          \begin{array}{ccc}
            g_2^{-1} - 1 & -g_2^{-1} & g_2^{-1}g_1\\
            -g_1^{-1} & g_1^{-1}g_3 & g_1^{-1} - 1\\
             g_3^{-1}g_2  &  g_3^{-1} - 1 & -g_3^{-1}\\
          \end{array}
        \right).\]
The reader may check that the composite maps are zero here.
\end{example}

We now describe the chain complex of the boundary torus, its attachment to $X$, and the expression of this as the chain complex of a cobordism.  Our strategy is as follows.  We begin by describing the chain complex of the torus and then attach it to the rest of $C(X;\Z[\pi_1(X)])$ using the attaching handles as in Theorem \ref{Thm:includingboundary}.
Next, recall that we split the torus into two halves, each a copy of $S^1 \times D^1$, glued together along $S^1 \times S^0$ using two inclusions $i_{\pm} \colon S^1 \times S^0 \hookrightarrow S^1 \times D^1_{\pm}$.  We consider the pull-back $\pi_1(X)$-covers of these spaces (Definition \ref{pullbackcovers}), and describe their chain complexes of finitely-generated free $\Z[\pi_1(X)]$-modules.  These can then be mapped via the induced maps of the inclusions $f_{\pm} \colon S^1 \times D^1_{\pm} \hookrightarrow \partial X \hookrightarrow X$ into the chain complex of universal cover of the knot exterior, to form the triad:
\[\xymatrix{C_*(S^1 \times S^0;\Z[\pi_1(X)]) \ar[r]^{i_-} \ar[d]_{i_+} & C_*(S^1 \times D^1_-;\Z[\pi_1(X)]) \ar[d]^{f_-} \\ C_*(S^1 \times D^1_+;\Z[\pi_1(X)]) \ar[r]^{f_+} & C_*(X;\Z[\pi_1(X)]),
}\]
the algebraic version of a cobordism relative to the identity cobordism of the boundary $S^1 \times D^1$ to itself.  The compositions $f_{-}\circ i_{-}$ and $f_+ \circ i_+$ will typically not coincide on the chain level so we will also have a chain homotopy $g \colon f_{-}\circ i_{-} \simeq f_+ \circ i_+$ as part of the algebraic data which enables us to write the chain map \[\eta \colon \mathscr{C}((i_-,i_+)^T ) \simeq C_*(S^1 \times S^1;\Z[\pi_1(X)]) \to C_*(X;\Z[\pi_1(X)])\]
which is the algebraic version of the pair $(X,\partial X)$.  See Definition \ref{Defn:algmappingcone} for the algebraic mapping cone construction $\mathscr{C}$.
\begin{definition}\label{pullbackcovers}
Let $p \colon \wt{X} \to X$ be a covering space with deck transformation group $G$; we call this the $G$-cover of $X$.  Given a space $Y$ and a map $g \colon Y \to X$ we define the \emph{pull-back $G$-cover of $Y$}, $\wt{Y}_G$, to be induced from the diagram:
\[\xymatrix{\wt{Y}_G \ar[r] \ar[d] & \wt{X} \ar[d]^p \\ Y \ar[r]_g & X;}\]
\[\wt{Y}_G := \{(y,x) \in Y \times \wt{X} \,|\, g(y) = p(x)\}.\]
This construction can yield both irregular and disconnected covering spaces, since there is no requirement that $G$ be a normal subgroup of $\pi_1(Y)$.
\qed \end{definition}

\begin{proposition}\label{prop:toruschaincopmlex}
The chain complex for the $\pi_1(X)$-cover of the torus $S^1 \times S^1 \approx \partial X$ is given below, where the image in $\pi_1(X)$ of the generators of $\pi_1(\partial X)$ are $\mu$ and $\lambda$, a meridian and longitude of the knot respectively:
\[\Z[\pi_1(X)] \xrightarrow{\partial_2 = \left(\begin{array}{cc} \lambda-1 & 1-\mu \end{array} \right)} \bigoplus_2 \, \Z[\pi_1(X)] \xrightarrow{\partial_1 = \left(\begin{array}{c} \mu-1 \\ \lambda - 1 \end{array} \right)} \Z[\pi_1(X)].\]
\end{proposition}
\begin{proof}
This corresponds to the standard handle decomposition for the torus described in Remark \ref{Remark:unknothandles}.  We have to thread the handles.  Each threading begins at the basepoint $x_0 \in X$ at the centre of $\hp^0$; recall that we amalgamated the 0-handles of $X$ into one 0-handle.  The 1-handles $h^1_{\mu}$ and $h^1_{\lambda}$ are threaded by following their cores, agreeing with their orientations, until reaching their centres.  The 2-handle $\hp^2$ is threaded by following along the threading for $h^1_{\lambda}$ before leaving in the direction of the orientation of $h^1_{\mu}$ and heading straight to the centre of $\hp^2$.  The boundary maps claimed then follow by considering the concatenation of paths described in Definition \ref{handlechain}.  Alternatively we can see, where $r_{\partial} = \lambda\mu\lambda^{-1}\mu^{-1}$, that $\partial_2 = (\partial r_{\partial}/\partial \mu,\partial r_{\partial}/\partial \lambda)$ (see Definition \ref{freederivative}).
\end{proof}

As promised, we now include the chain complex of the $\pi_1(X)$-cover of the boundary into the chain complex of the universal cover of $X$.  We therefore need to describe the chain complex $C(X;\Z[\pi_1(X)])$ with the additional summands generated by the additional handles which make up the boundary and the attaching handles for the boundary.
The following theorem is an extension of Theorem \ref{Thm:unicoverchaincomplex} but rather than just stating the new assertions we state it in full so that the full result is given in one location.  As before we work at the level of the universal cover, presenting our results in this generality so that we can work at the level of any smaller covering space we require.

\begin{theorem}\label{Thm:mainchaincomplex}
Suppose that we are given a knot $K$ with exterior $X$, and a reduced knot diagram for $K$ with $c \geq 3$ crossings.  Denote by $F(g_1, \dots ,g_c)$ the free group on the letters $g_1,\dots,g_c$, and let $l \in F(g_1, \dots,g_c)$ be the word for the longitude defined in the proof of Theorem \ref{Thm:includingboundary}.  Then there is a presentation
\[\pi_1(X) = \langle\,g_1,\dots,g_c, \mu, \lambda\,|\,r_1,\dots,r_c, r_{\mu},r_{\lambda},r_{\partial}\,\rangle\]
with the Wirtinger relations $r_1,\dots,r_c \in F(g_1,\dots,g_c)$ read off from the knot diagram, and
\[r_{\mu} = g_1\mu^{-1};\;\;r_{\lambda} = l\lambda^{-1};\; r_{\partial} = \lambda\mu\lambda^{-1}\mu^{-1}.\]
The generators $\mu$ and $\lambda$ correspond to the generators, and $r_{\partial}$ to the relation, for the fundamental group of the boundary torus $\pi_1(S^1 \times S^1) \cong \Z \oplus \Z$. The generator $\mu$ is a meridian and $\lambda$ is a longitude.  The relations $r_{\mu}$ and $r_{\lambda}$ are part of Tietze moves: they show the new generators to be consequences of the original generators.

The handle chain complex of the $\pi$-cover $\widetilde{X}$, that is the cover with deck group $\pi := \pi_1(X)/S$ for some normal subgroup $S \unlhd \pi_1(X)$, with chain groups being based free left $\Z[\pi]$-modules, and with the chain complex $C(\wt{\partial X})$ of the $\pi_1(X)$-cover of $\partial X$ as a sub-complex, is given, recalling the convention of \ref{conventions2} that matrices act on row vectors on the right, by:

\[\xymatrix @C+2cm{
\bigoplus_2\,\Z[\pi] \cong \langle h^3_o,h^3_{\partial} \rangle \ar[d]^{\partial_3}\\
 \bigoplus_{c+3} \,\Z[\pi] \cong \langle h^2_1,\dots,h^2_c,h^2_{\partial \mu},h^2_{\partial \lambda},h^2_{\partial} \rangle \ar[d]^{\partial_2}\\
  \bigoplus_{c+2} \,\Z[\pi] \cong \langle h^1_1,\dots,h^1_c,h^1_{\mu},h^1_{\lambda} \rangle \ar[d]^{\partial_1}\\
  \Z[\pi] \cong \langle h^0_{\partial} \rangle
}\]
where:
\[\ba{rcl} \partial_3 &=& \left(
               \begin{array}{cccccccc}
                 w_1 &  & \hdots &  & w_c & 0 & 0 & 0 \\
                 -u_1 &  & \hdots &  & -u_c & 1-\lambda & \mu-1 & -1 \\
               \end{array}
             \right);\\ & & \\
\partial_2 &=& \left(
                 \begin{array}{ccccccc}
                   \left(\partial r_1/\partial g_1\right) &  & \hdots &  & \left(\partial r_1/\partial g_c\right) & 0 & 0\\
                    &  &  &  &  &  & \\
                   \vdots &  & \ddots &  & \vdots & \vdots & \vdots\\
                    &  &  &  &  &  &  \\
                   \left(\partial r_c/\partial g_1\right) &  & \hdots &  & \left(\partial r_c/\partial g_c\right) & 0 & 0\\
                   1 & 0 & \hdots & 0 & 0 & -1 & 0\\
                   \left(\partial l/\partial g_1\right) &  & \hdots &  & \left(\partial l/\partial g_c\right) & 0 & -1\\
                   0 &  & \hdots &  & 0 & \lambda-1 & 1-\mu\\
                 \end{array}
               \right);\text{ and} \\ & & \\
\partial_1 &=& \left(
               \begin{array}{ccccccc}
                 g_1 - 1 &  & \hdots &  & g_c - 1 & \mu-1 & \lambda-1\\
               \end{array}
             \right)^T. \ea\]
The word $l$ for the longitude was defined in the proof of Theorem \ref{Thm:includingboundary}.  We have:
\[l = g_1^{-\Wr}g_{k_1}^{\eps_k}g_{(k+1)_1}^{\eps_{k+1}}\dots g_{(k+c-1)_1}^{\eps_{k+c-1}},\]
where $k$ is the number of the crossing reached first as an over crossing, when starting on the under crossing strand of the knot which lies in region 1; the indices $k, k+1,\dots,k+c-1$ are to be taken mod $c$, with the exception that we prefer the notation $c$ for the equivalence class of $0 \in \Z_c$.  The sign of crossing $j$ is $\eps_j$ and $\Wr$ is the writhe of the diagram, which is the sum of the $\eps_j$.
The $u_{k+i}$ are given by $g_1^{1-\Wr}$ followed by the next $i+1$ letters in the word for the longitude:
\[u_{k+i} = g_1^{1-\Wr}g_{k_1}^{\eps_k}g_{(k+1)_{1}}^{\eps_{k+1}}\dots g_{(k+i)_1}^{\eps_{k+i}}.\]

To determine the words $w_i$ which arise in $\partial_3$, consider the quadrilateral decomposition of the knot diagram (Definition \ref{Defn:quaddecomp}).  At each crossing $i$, we have a distinguished edge which we always list first in the relation, $g_{i_2}$.  Choose the vertex, call it $v_i$, which is at the end of $g_{i_2}$.  For crossing $i$, choose a path in the 1-skeleton of the quadrilateral decomposition from $v_1$ to $v_i$.  This yields a word $w_i$ in $g_1,\dots,g_c$.  Then the component of $\partial(h^3)$ along $h^2_i$ is $w_i$.

As written, each component of the boundary matrices is an element of the group ring on the free group $\Z[F(g_1,\dots,g_c,\mu,\lambda)]$.  We therefore act on each element by the homomorphism \[\Phi \colon \Z[F(g_1,\dots,g_c,\mu,\lambda)] \to \Z[\pi]\] defined by linearly extending the group homomorphism $\phi \colon F(g_1,\dots,g_c,\mu,\lambda) \to \pi$.

There is then a pair of chain complexes,
\[f \colon C_*(\partial X;\Z[\pi]) \to C_*(X;\Z[\pi]),\] with the map $f$ given by inclusion, expressing the manifold pair $(X,\partial X)$.
\end{theorem}

\begin{proof}
The new presentation for the fundamental group of $X$ reflects the new handles which have been added.  There are two new generators and two new relations which express the new generators as consequences of the old generators.  The extra relation $r_{\partial}$ is already a consequence of the Wirtinger relations,
since the meridian and the longitude already commute in a knot exterior, entirely independently of a presentation chosen for $\pi_1(X)$.  The reason for its inclusion is that we shall require, in the next chapter, that a presentation for the fundamental group of the boundary sits inside our presentation for the group of the whole manifold.

In order to see that the boundary maps are as claimed we need to describe the threadings.  These have already been described for the interior handles in the proof of Theorem \ref{Thm:unicoverchaincomplex} and for the boundary handles in the proof of Proposition \ref{prop:toruschaincopmlex}.  We therefore only need to describe the threadings for handles which are the attaching handles for our boundary handles, namely for $h^2_{\partial\mu},h^2_{\partial\lambda}$, and $h^3_{\partial}$.  Recall that we are taking the basepoint $x_0 \in X$ to be the centre of the 0-handle $h^0_{\partial}$.

To thread the 2-handles we have a choice as to where to enter the 2-handle, which corresponds to choosing a preferred lift of the 2-handle in a covering space. To thread $h^2_{\partial\mu}$, we follow from $h^0_{\partial}$, along the core of $h^1_{\mu}$ in agreement with its orientation and then enter the 2-handle from the centre of $h^1_{\mu}$ passing directly to the centre of $h^2_{\partial \mu}$.  The handle $h^2_{\partial \lambda}$ is also threaded in this manner; starting from $h^0_{\partial}$ we follow $h^1_{\lambda}$ along its core in agreement with its orientation until we reach its centre, and from there we pass directly to the centre of $h^2_{\partial \lambda}$.

Finally, to thread the 3-handle $h^3_{\partial}$ we follow the threading for $h^2_{\partial}$ to its centre before passing directly to the centre of $h^3_{\partial}$.

All of the boundary maps claimed then follow by considering the concatenation of paths which define the twisted intersection numbers as in Definition \ref{handlechain} and expressing the loops which result in terms of the $g_i$.  The pair \[f \colon C_*(\partial X;\Z[\pi]) \to C_*(X;\Z[\pi])\] is as claimed, with a $f$ a split injection of free modules.

One should also note that the entries in the matrix for $\partial_2$ are given by taking the free derivatives of the relation words, and that the following are identities of this presentation, and yield the boundary map $\partial_3$, by taking the word which conjugates each relation, and the sign $\pm 1$ according to the exponent of each relation:
\[s_o = \prod_{k=1}^c \, w_{j_k}r_{j_k}w_{j_k}^{-1} = 1 \in F(g_1,\dots,g_c,\mu,\lambda);\]
\begin{multline*} s_{\partial} = (r_{\partial}^{-1})(\lambda r_{\mu}^{-1} \lambda^{-1})(r_{\lambda}^{-1}) \left(\prod_{j=0}^{c-1}\,u_{k+j}r^{-1}_{k+j}u_{k+j}^{-1}\right) (r_{\mu})(\mu r_{\lambda}\mu^{-1})\\
 = 1 \in F(g_1,\dots,g_c,\mu,\lambda).\end{multline*}
\end{proof}

\begin{remark}\label{Rmk:relativefundclass}
Passing to $\Z$ coefficients, the 3-dimensional chain
\[[X,\partial X] := h^3_o + h^3_{\partial} \in C_3(X;\Z) = \Z \otimes_{\Z[\pi_1(X)]} C_3(X;\Z[\pi_1(X)])\]
represents a cycle in $C_3(X,\partial X;\Z)$, since $\partial_3([X,\partial X]) = -h^2_{\partial} = (-1)^3f([\partial X]) \in C_2(X;\Z)$.  This is the \emph{relative fundamental class} for the knot exterior, which we shall use in Chapter \ref{Chapter:duality_symm_structures} to derive the symmetric structure on the chain complex.
\end{remark}

We now describe how to construct a chain complex of the boundary torus $C_*(\partial X;\Z[\pi_1(X)])$ with a splitting.  We define maps:
\[(i_-,i_+)^T \colon C_*(S^1 \times S^0;\Z[\pi_1(X)]) \to C_*(S^1 \times D^1_-;\Z[\pi_1(X)]) \oplus C_*(S^1 \times D^1_+;\Z[\pi_1(X)]),\]
with chain maps \[f_{\pm} \colon C_*(S^1 \times D^1_{\pm};\Z[\pi_1(X)]) \to C_*(X;\Z[\pi_1(X)])\] corresponding to the geometric inclusion maps, and a chain homotopy \[g \colon f_-\circ i_- \simeq f_+ \circ i_+ \colon C_*(S^1 \times S^0;\Z[\pi_1(X)])_* \to C_{*+1}(X;\Z[\pi_1(X)])\] which measures the failure of these two compositions to coincide on the chain level.  Combining $f_{\pm}$ and $g$ will yield a chain equivalence $\eta$ from the mapping cone (Definition \ref{Defn:algmappingcone}) to the chain complex of the boundary of $X$ from Proposition \ref{prop:toruschaincopmlex}: \[\eta \colon \mathscr{C}((i_-,i_+)^T) \xrightarrow{\sim} C_*(\partial X;\Z[\pi_1(X)]).\]  Since this latter complex is a sub-complex of $C_*(X;\Z[\pi_1(X)])$, we can then use the inclusion to define the map $f \circ \eta$ which gives us the pair of complexes: \[f \circ \eta \colon \mathscr{C}((i_-,i_+)^T) \to C_*(X;\Z[\pi_1(X)]),\] so that we have a triad of chain complexes:
\[\xymatrix{C_*(S^1 \times S^0;\Z[\pi_1(X)]) \ar @{} [dr] |{\stackrel{g}{\sim}} \ar[r]^{i_-} \ar[d]_{i_+} & C_*(S^1 \times D^1_-;\Z[\pi_1(X)]) \ar[d]^{f_-} \\ C_*(S^1 \times D^1_+;\Z[\pi_1(X)]) \ar[r]^-{f_+} & C_*(X;\Z[\pi_1(X)]).
}\]
as desired.

In order to better understand the following, recall our splitting of the torus into two halves as shown in Figure \ref{Fig:torusintwocylinders2}.
\begin{figure}[h]
    \begin{center}
 {\psfrag{A}{$h^0_+$}
 \psfrag{B}{$h^0_-$}
 \psfrag{C}{$h^1_-$}
 \psfrag{D}{$h^1_+$}
 \psfrag{E}{$h^1_a$}
 \psfrag{F}{$h^1_b$}
 \psfrag{G}{$h^2_b$}
 \psfrag{H}{$h^2_a$}
 \includegraphics[width=7.5cm]{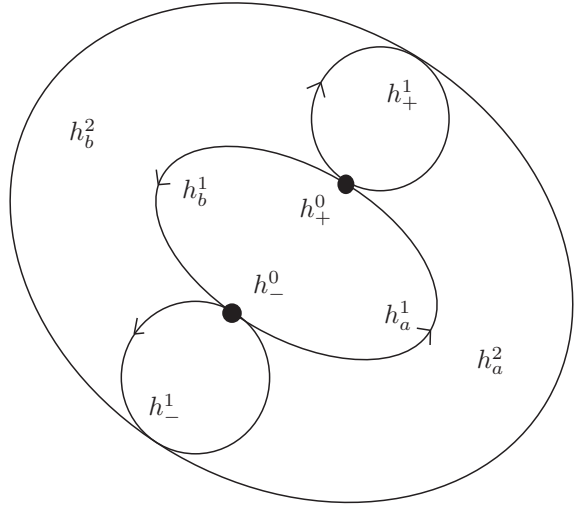}
 }
 \caption{A handle decomposition of the torus $\partial X \approx S^1 \times S^1$ with a splitting into two cylinders.}
 \label{Fig:torusintwocylinders2}
 \end{center}
\end{figure}

\begin{proposition}\label{Prop:circlechaincomplex}
An equivariant chain complex of the universal cover of the circle $C_*(\wt{S^1};\Z) \cong C_*(S^1;\Z[\Z])$ is given by the following $\Z[\Z] = \Z[t,t^{-1}]$-module chain complex:
\[\Z[\Z] \xrightarrow{\partial_1 = (t-1)} \Z[\Z].\]
A chain complex of $S^1 \times S^0$ is therefore given by the $\Z[\Z \oplus\Z] = \Z[t,t^{-1},s,s^{-1}]$-module chain complex $C_*(S^1;\Z[t,t^{-1}]) \oplus C_*(S^1;\Z[s,s^{-1}])$:
\[\bigoplus_2 \, \Z[\Z \oplus \Z]  \xrightarrow{\partial_1 = \left(\begin{array}{cc} t-1 & 0 \\ 0 & s-1 \end{array}\right)}  \bigoplus_2 \, \Z[\Z \oplus \Z].\]
Let $g_1,g_q$ also denote the images of $t,s$ in $\pi_1(X)$ under the induced map of the composition of geometric maps $f_+ \circ i_+$ (which agrees on $S^1 \times S^0$ with $f_- \circ i_-$), for some $q$ such that $2 \leq q \leq c$, depending on where we split the torus.  Then a chain complex of the $\pi_1(X)$-cover of $S^1 \times S^0$ is given by:
\[\bigoplus_2 \, \Z[\pi_1(X)]  \xrightarrow{\partial_1 = \left(\begin{array}{cc} g_1-1 & 0 \\ 0 & g_q-1 \end{array}\right)}  \bigoplus_2 \, \Z[\pi_1(X)].\]

\end{proposition}
\begin{proof}
The circle has an obvious handle decomposition involving one 0-handle and one 1-handle.  Thread the 1-handle by starting at its beginning and following the orientation until reaching its centre.  The boundary map $\partial_1$ is therefore as claimed.\\
For the chain complex of $S^1 \times S^0$ over $\Z[\Z \oplus \Z]$, one merely needs to observe that disjoint union of topological spaces corresponds to the direct sum of their chain complexes.  The chains are free modules over the group ring of $\pi_1(S^1;x_-) \oplus \pi_1(S^1;x_+)$, where $x_-,x_+$ are base-points which belong to each of the connected components of $S^1 \times S^0$, the centres of the 0-handles $h^0_-$ and $h^0_+$ from Figure \ref{Fig:torusintwocylinders2}.
For the chain complex of the $\pi_1(X)$-cover, one tensors the chain complex over $\Z[\Z \oplus \Z]$ on the left with $\Z[\pi_1(X)]$;
 \[\Z[\pi_1(X)] \otimes_{\Z[\Z \oplus \Z]} C_*(S^1 \times S^0;\Z[\Z\oplus \Z])\]
using the homomorphism $\Z \oplus \Z \to \pi_1(X)$ given by $t \mapsto g_1$ and $s \mapsto g_q$.  This map is derived by conjugating the inclusions of loops in $\pi_1(S^1 \times S^0; x_{\pm})$ with paths from the basepoint of $x_0 \in X$ to the images of the basepoints $f_+ \circ i_+(x_{\pm})$.  We have that $f_+ \circ i_+(x_{-}) = x_0$, so no path is required here, whereas $f_+ \circ i_+(x_{+})$ is half way around the boundary so requires a path in $X$ from $x_0$ to the image $f_+\circ i_+(x_+)$.  This choice of path determines the element $g_q$; that is, it determines where we split the torus.
\end{proof}

We now include $S^1 \times S^0$ into the cylinder $S^1 \times D^1$ in two ways, so that we have two null-cobordisms of $S^1 \times S^0$:
\[i_{\pm} \colon S^1 \times S^0 \hookrightarrow S^1 \times D^1.\]
Let $\{1\} \times \{-1\} \in S^1 \times S^0 \subset S^1 \times D^1$ be the basepoint of both spaces, considering each $S^i$ with its standard embedding in $\R^{i+1}$ in order to describe coordinates.  Recall (Definition \ref{conventions0}) that $D^1 \cong [-1,1]$ and $S^0 = \{-1,1\}$.  Let $j \colon S^0 \to D^1$ be the inclusion.  We define the maps $i_{\pm}$ to be:
\[i_{\pm} =(\Id, \pm \Id \circ j) \colon S^1 \times S^0 \to S^1 \times D^1.\]

\begin{proposition}\label{prop:chainmapsiplusiminus}
The cylinder $S^1 \times D^1$ is homotopy equivalent to $S^1$, so we can use the same chain complex for the two spaces.  The chain maps $i_-$ and $i_+$ which are induced by the inclusions $i_{\pm} \colon S^1 \times S^0 \to S^1 \times D^1_{\pm}$ on the chain complexes of the $\pi_1(X)$-covers are given by:

\[\xymatrix @R+1cm @C-0.4cm{C_*(S^1 \times D^1_-;\Z[\pi_1(X)])\colon & \Z[\pi_1(X)] \ar[rrrrr]^{\left(\begin{array}{c}g_1-1\end{array}\right)} &&&&& \Z[\pi_1(X)] \\
C_*(S^1 \times S^0;\Z[\pi_1(X)])\colon \ar[u]_{i_-} \ar[d]^{i_+} & \bigoplus_2\,\Z[\pi_1(X)] \ar[rrrrr]_{\left(\begin{array}{cc}g_1-1 & 0 \\ 0 & g_q-1 \end{array}\right)} \ar[u]^{\left(\begin{array}{c} 1 \\ l_a^{-1} \end{array}\right)} \ar[d]_{\left(\begin{array}{c} l_b^{-1} \\ 1 \end{array}\right)} &&&&& \bigoplus_2\,\Z[\pi_1(X)] \ar[u]_{\left(\begin{array}{c} 1 \\ l_a^{-1} \end{array}\right)} \ar[d]^{\left(\begin{array}{c} l_b^{-1} \\ 1 \end{array}\right)} \\
C_*(S^1 \times D^1_+;\Z[\pi_1(X)])\colon & \Z[\pi_1(X)] \ar[rrrrr]_{\left(\begin{array}{c}g_q-1\end{array}\right)} &&&&& \Z[\pi_1(X)],
}\]
where the words $l_a,l_b$ are given by splitting the word $l$ for the longitude in two as follows.  We take the letters of the word for $l$ before and after a certain point, which corresponds to the point that the knot passes under $h^1_q$.  That is, let $p$, between $0$ and $c-2$ be such that $g_{(k+p)_2} = g_q$.  Then, with the indices taken mod $c$ as above:
\[l_a := g_1^{-\Wr}g_{k_1}^{\eps_k}g_{(k+1)_1}^{\eps_{k+1}}\dots g_{(k+p)_1}^{\eps_{k+p}}; \text{ and}\]
\[l_b := g_{(k+p+1)_1}^{\eps_{k+1}}\dots g_{(k+c-1)_1}^{\eps_{k+p}}.\]
\end{proposition}
\begin{proof}
The basepoint of the cylinder $S^1 \times D^1_{\pm}$ is the point $x_{\pm}$ at the centre of $h^0_{\pm}$ as in the proof of Proposition \ref{Prop:circlechaincomplex} and the preamble to this proposition.  This means that the chain complexes of the $\pi_1(X)$-covers are the same as the chain complexes of $S^1 \times \{-1\}_{\pm} \subset S^1 \times D^1_{\pm}$: we can retract onto this circle without moving the basepoint.  The chain maps $i_{\pm}$ are derived by considering the loops created by concatenating the paths which begin at the basepoint $x_0 = x_- \in X$, follow the threading of the handle in $S^1 \times S^0$, then pass using the geometric map $i_{\pm}$ to the relevant handle of $S^1 \times D^1_{\pm}$, before returning to the basepoint using the threading of this latter handle.  The threadings of the handles of $S^1 \times S^0$ pass through the southern hemisphere ($D^3_-$ from the proof of Theorem \ref{handledecomp}) so avoid the knot entirely.  Since we consider $S^1 \times D^1_{\pm}$ as being retracted onto the end which contains its basepoint the threadings here are identical.  The coefficients $l_a^{-1}$ and $l_b^{-1}$ in the chain maps arise since the effect of this retract is to make the geometric map pass around half the boundary torus for the pairs in which the basepoints do not coincide.  The fundamental group elements $l_a$ and $l_b$ can be visualised in Figure \ref{Fig:torusintwocylinders2} as following the cores of the handles labelled $h^1_a$ and $h^1_b$ respectively.  We need to check that the maps of complexes given are indeed chain maps; for this one needs the following relations, which can be checked algebraically using the Wirtinger relations, and which should geometrically hold:
\[g_q = l_a^{-1}g_1l_a;\;\text{ and}\]
\[g_q = l_bg_1l_b^{-1}.\]
The homotopy for the first relation deforms the loop across the core of $h^2_a$ from Figure \ref{Fig:torusintwocylinders2} while the homotopy for the second relation deforms across $h^2_b$.
\end{proof}

In order to glue the two cylinders together along their common boundary we use the algebraic mapping cone construction.

\begin{definition}\label{Defn:algmappingcone}
The \emph{algebraic mapping cone} $\mathscr{C}(g)$ of a chain map $g \colon C \to D$ is the chain complex given by:
\[d_{\mathscr{C}(g)} = \left(\begin{array}{cc} d_D & (-1)^{r-1}g \\ 0 & d_C \end{array} \right) \colon \mathscr{C}(g)_r = D_r \oplus C_{r-1} \to \mathscr{C}(g)_{r-1} = D_{r-1} \oplus C_{r-2}.\]
This of course mirrors the geometric mapping cone construction algebraically.
\qed \end{definition}


\begin{proposition}\label{Prop:torussplitchainequivalence}
For purposes of brevity we make the following definitions:
\[C:= C_*(S^1 \times S^0;\Z[\pi_1(X)]); \text{ and}\]
\[D_{\pm} := C_*(S^1 \times D^1_{\pm};\Z[\pi_1(X)]).\]
There is a chain complex of the $\pi_1(X)$-cover of the torus $S^1 \times S^1$ given by the mapping cone
\[E := \mathscr{C}((i_-,i_+)^T \colon C \to D_- \oplus D_+)\]
with
\[E_r = (D_-)_r \oplus C_{r-1} \oplus (D_+)_r,\]
so that the chain complex is given by:
\[\xymatrix{ E_2 \cong \bigoplus_2 \, \zpx \cong \langle h^2_a,h^2_b \rangle \ar[d]_{\partial_2} \\
E_1 \cong \bigoplus_4 \, \zpx \cong \langle h^1_-,h^1_a,h^1_b,h^1_+ \rangle \ar[d]_{\partial_1} \\
E_0 \cong \bigoplus_2 \, \zpx \cong \langle h^0_-,h^0_+ \rangle
}\]
where:
\[\ba{rcl}\partial_2 &=& \left(\begin{array}{cccc} -1 & g_1-1 & 0 & -l_b^{-1} \\ -l_a^{-1} & 0 & g_q-1 & -1 \end{array} \right); \text{ and}\\
& & \\
\partial_1 &=& \left(\begin{array}{cc} g_1-1 & 0  \\ 1 & l_b^{-1} \\ l_a^{-1} & 1 \\ 0 & g_q-1 \end{array} \right).\ea\]
We included the geometric interpretation relating each of the $\zpx$-summands to handles in Figure \ref{Fig:torusintwocylinders2}.
This chain complex is chain equivalent to the chain complex for the torus \[E':= C_*(S^1 \times S^1;\zpx) \subset C_*(X;\zpx)\] from Proposition \ref{prop:toruschaincopmlex}, given again here:
\[\Z[\pi_1(X)] \xrightarrow{\partial_2 = \left(\begin{array}{cc} l_al_b-1 & 1-g_1 \end{array} \right)} \bigoplus_2 \, \Z[\pi_1(X)] \xrightarrow{\partial_1 = \left(\begin{array}{c} g_1-1 \\ l_al_b - 1 \end{array} \right)} \Z[\pi_1(X)].\]
\end{proposition}

\begin{proof}
The first statement is just an application of the algebraic mapping cone construction.  In defining the chain equivalence, since the chain complex of the $\pi_1(X)$-cover of the torus, $C_*(S^1 \times S^1;\zpx)$, is a sub-complex of the chain complex $C_*(X;\zpx)$, we shall simultaneously define the maps \[f_{\pm} \colon D_{\pm}=C_*(S^1 \times D^1_{\pm};\zpx) \to C_*(X;\zpx),\] and the chain homotopy  \[g \colon f_-\circ i_- \simeq f_+ \circ i_+ \colon C_*(S^1 \times S^0;\Z[\pi_1(X)])_* \to C_{*+1}(X;\Z[\pi_1(X)]):\]
The geometric maps $f_-\circ i_-$ and $f_+ \circ i_+$ coincide, whereas their algebraic counterparts do not; $g$ is the algebraic data which reflects this.
We use these maps to construct the chain equivalence:
\[\eta \colon E \xrightarrow{\sim} E'\]
shown here (where since we are dealing with matrices of maps and not group ring elements, chain groups are column vectors and matrices act on the left):
\[ \xymatrix @C+0.8cm @R+1cm{  E_2 = C_1 \ar[r]^-{\partial_E} \ar[d]^{-g = \eta} & E_1 = (D_-)_1 \oplus C_0 \oplus (D_+)_1 \ar[r]^-{\partial_E} \ar[d]_{(f_-,g,-f_+)=\eta} & E_0 = (D_-)_0 \oplus (D_+)_0 \ar[d]_{(f_-,-f_+) = \eta} \\ E'_2 \ar[r]^{\partial_{E'}} & E'_1 \ar[r]^{\partial_{E'}} & E'_0,
}\]
where
\[\partial_E = (-i_-,\partial_C,-i_+)^T \colon E_2 = C_1 \to E_1 = (D_-)_1 \oplus C_0 \oplus (D_+)_1; \text{ and}\]
\[\partial_E = \left(\begin{array}{ccc} \partial_{D_-} & i_- & 0 \\ 0 & i_+ & \partial_{D_+}\end{array} \right) \colon E_1 = (D_-)_1 \oplus C_0 \oplus (D_+)_1 \to E_0 = (D_-)_0 \oplus (D_+)_0.\]

The conditions $f_{\pm}\partial_{D_{\pm}} = \partial_{E'}f_{\pm}$ that $f_{\pm}$ are chain maps and $g\partial_{C} + \partial_{E'}g = f_-\circ i_- - f_+\circ i_+$, that $g$ is chain homotopy, are equivalent to the condition that $\eta$ is a chain map.

Returning to the convention of row vectors and matrices acting on the right, we define the chain map $\eta \colon E \to E'$ explicitly to be:
\[\xymatrix @R+2cm @C+1.9cm{\bigoplus_2 \,\zpx \ar[r]^{\partial_E} \ar[d]^{\left(\begin{array}{c} -l_b^{-1}l_a^{-1} \\ 0 \end{array} \right)} & \bigoplus_4 \,\zpx \ar[r]^{\partial_E} \ar[d]^{\left(\begin{array}{cc} 1 & 0 \\ 0 & l_b^{-1}l_a^{-1} \\ 0 & 0 \\ -l_a^{-1} & 0 \end{array} \right)} & \bigoplus_2 \,\zpx \ar[d]_{\left(\begin{array}{c} 1 \\ -l_a^{-1} \end{array} \right)} \\
\zpx \ar[r]^{\partial_{E'}} &  \bigoplus_2 \,\zpx \ar[r]^{\partial_{E'}} & \zpx.}\]
The reader can check that this is indeed a chain map.  To see that it is a chain equivalence, we exhibit here a chain homotopy inverse $\xi$.
\[\xi \colon E' \xrightarrow{\sim} E;\]
\[\xymatrix @R+2cm @C+1.9cm{\zpx \ar[r]^{\partial_{E'}} \ar[d]^{\left(\begin{array}{cc} -l_al_b & l_a \end{array} \right)} &  \bigoplus_2 \,\zpx \ar[r]^{\partial_{E'}} \ar[d]^{\left(\begin{array}{cccc} 1 & 0 & 0 & 0 \\ 0 & l_al_b & -l_a & 0 \end{array} \right)} & \zpx \ar[d]_{\left(\begin{array}{cc} 1 & 0 \end{array} \right)} \\
\bigoplus_2 \,\zpx \ar[r]^{\partial_E}  & \bigoplus_4 \,\zpx \ar[r]^{\partial_E}  & \bigoplus_2 \,\zpx.
}\]
The reader can check that:
\[\eta \circ \xi - \Id = 0 \colon E' \to E',\]
and that the chain map:
\[\xi \circ \eta - \Id \colon E \to E\]
is given by:
\[\xymatrix @R+2cm @C+1.9cm{
\bigoplus_2 \,\zpx \ar[r]^{\partial_E} \ar[d]^<<<<<<<<<{\left(\begin{array}{cc} 0 & -l_b^{-1} \\ 0 & -1 \end{array} \right)}  & \bigoplus_4 \,\zpx \ar[r]^{\partial_E} \ar[dl]^{k} \ar[d]^{\xi \circ \eta - \Id}  & \bigoplus_2 \,\zpx \ar[d]_>>>>>>>>>{\left(\begin{array}{cc} 0 & 0 \\ -l_a^{-1} & -1 \end{array} \right)} \ar[dl]_{k} \\
\bigoplus_2 \,\zpx \ar[r]^{\partial_E}  & \bigoplus_4 \,\zpx \ar[r]^{\partial_E}  & \bigoplus_2 \,\zpx,
}\]
with
\[\xi \circ \eta - \Id = \left(\begin{array}{cccc} 0 & 0 & 0 & 0 \\ 0 & 0 & -l_b^{-1} & 0 \\ 0 & 0 & -1 & 0 \\ -l_a^{-1} & 0 & 0 & -1 \end{array} \right),\]
and where $k \colon E_i \to E_{i+1}$ is a chain homotopy such that $k\partial_E + \partial_E k = \xi \circ \eta - \Id$, showing that $\eta$ and $\xi$ are indeed inverse chain equivalences, given by:
\[\ba{rcl} k &=& \left(\begin{array}{cccc} 0 & 0 & 0 & 0 \\ 0 & 0 & -1 & 0 \end{array} \right) \colon E_0 \to E_1; \text{ and} \\ & & \\
k &=& \left(\begin{array}{cc} 0 & 0 \\ 0 & 0 \\ 0 & 0 \\ 0 & 1 \end{array} \right) \colon E_1 \to E_2.\ea\]
\end{proof}

In conclusion, since $\eta$ splits up as described into $f_-, f_+$ and $g$, and since $E'$ includes into $C(X;\zpx)$ as a sub-complex, we have now exhibited, as claimed, a triad of chain complexes.  Each of the chain complexes and chain maps are algorithmically extractable from a knot diagram.  There is a homotopy $g$ which measures the chain level failure of the diagram to commute:
\[\xymatrix{\ar @{} [dr] |{\stackrel{g}{\sim}}
C_*(S^1 \times S^0;\Z[\pi_1(X)]) \ar[r]^{i_-} \ar[d]_{i_+} & C_*(S^1 \times D^1_-;\Z[\pi_1(X)]) \ar[d]^{f_-}\\ C_*(S^1 \times D^1_+;\Z[\pi_1(X)]) \ar[r]^{f_+} & C_*(X;\Z[\pi_1(X)]).
}\]


\chapter{Poincar\'{e} Duality and Symmetric Structures}\label{Chapter:duality_symm_structures}

In order to use the chain complex of the knot exterior given in Chapter \ref{chapter:chaincomplex} to generate concordance invariants, we will also need a chain equivalence between the chain complex and its dual, which yields the duality isomorphisms upon passing to homology.  Since the knot exterior is a manifold with boundary, we will in fact require the universal coefficient chain level version of Poincar\'{e}-Lefschetz duality.

This duality comes from taking cap product with a fundamental class $[X,\partial X] \in C_3(X,\partial X;\Z)$.  It turns out that not only the homological duality information in the cap product, but also the cup product and the Steenrod squares, are all encoded on the chain level in the diagonal approximation maps.  This algebraic information encodes the geometric linking and intersection information relating to the manifold.

Note that any concordance invariant requires that duality information is taken into account, whether explicitly or otherwise.  For example, the proof that the Alexander polynomial of a slice knot factorises as $f(t)f(t^{-1})$ uses duality; this was the first slice obstruction described in the original paper of Fox and Milnor on knot concordance \cite{fm}.  As in Remark \ref{Rmk:needperipheralstructure}, for our chain complex we need the peripheral structure (Definition \ref{Defn:peripheral}) for concordance invariants, since without taking account of how the boundary relates to the knot exterior, we cannot obtain a fundamental class and therefore cannot obtain Poincar\'{e}-Lefschetz duality.

\begin{definition}\label{Defn:peripheral}
The \emph{peripheral structure} of a knot is a homomorphism $\Z \oplus \Z \to \pi_1(X)$ which records the image of the longitude and the meridian of the boundary torus $S^1 \times S^1 \xrightarrow{\approx} \partial X$ in the fundamental group of $X$.  The longitude and meridian are canonically defined as the curves in $\pi_1(\partial X) \cong H_1(\partial X;\Z) \leq \pi_1(X)$ which represent the unique non-trivial primitive homology classes in
\[\ker(H_1(\partial X;\Z) \to H_1(X;\Z))\]
and
\[\ker(H_1(\partial X;\Z) \to H_1(N(K);\Z))\]
respectively.
In our construction this information allows us to construct a chain map $$C_*(\partial X;\Z[\pi_1(X)]) \to C_*(X;\Z[\pi_1(X)])$$  which records the inclusion of $\partial X$ algebraically on the chain complex level.  When we refer to the peripheral structure we shall also mean this chain map as well as the map on the level of the fundamental groups.
\qed \end{definition}

For the benefit of the reader we include an introduction to diagonal approximation chain maps and their involvement in the symmetric construction in this chapter.  We shall describe how to produce a symmetric structure on a chain complex, in particular on the chain complex of the universal cover of a 3-dimensional model for an Eilenberg MacLane space $K(\pi,1)$, such as a knot exterior.  This produces what shall be the basic algebraic object of our consideration; that is, a collection of $\Z[\pi_1(X)]$-modules and maps of the form:
\[\xymatrix @R+1cm @C+1cm{
C^0 \ar[r]^{\delta_1} \ar[d]^{\varphi_0}& C^1 \ar[r]^{\delta_2} \ar[d]^{\varphi_0} & C^2 \ar[d]^{\varphi_0} \ar[r]^{\delta_3} & C^3 \ar[d]^{\varphi_0}\\
C_3 \ar[r]_{\partial_3} & C_2 \ar[r]_{\partial_2} & C_1 \ar[r]_{\partial_1} & C_0}\]
There are also higher chain homotopies $\varphi_s \colon C^r \to C_{3-r+s}$ which measure the failure of $\varphi_{s-1}$ to be symmetric on the chain level; we shall describe these in detail later in this chapter.
The main references for this material are \cite{Ranicki2} and \cite[part~I]{Ranicki3}.  Experts may wish to skip to Section \ref{Chapter:duality_symm_structures}.\ref{section:formulaediagonal}.

\section{The symmetric construction}\label{Chapter:symmconstruction}

To begin, for simplicity, we take $M$ to be an $n$-dimensional \emph{closed} manifold with $\pi_1(M) = \pi$ and universal cover $\wt{M}$.  Using the trivial homomorphism $\pi \to \{1\}$, we can form the tensor product \[\Z \otimes_{\Z[\pi]} C_*(\wt{M}) = C_*(M;\Z),\] and so calculate $H_*(M;\Z)$.  With $\Z$ coefficients there is a fundamental class $[M] \in H_n(M;\Z)$, which we require in order to furnish the chain complex with Poincar\'{e} duality.  The universal Poincar\'{e} duality isomorphisms:
\[[M] \cap \bullet\; \colon H^{r}(M;\Z[\pi]) \to H_{n-r}(M;\Z[\pi]),\]
as given by the cap product with the fundamental class, are given explicitly on the chain level using the symmetric construction.  Take an equivariant diagonal chain approximation map:
\[\Delta_0 \colon C(M;\Z[\pi])_* \to (C(M;\Z[\pi]) \otimes_{\Z} C(M;\Z[\pi]))_*;\]
there are many choices of such maps; for singular chains an acyclic models argument can be used to show that they exist and that any two choices are chain homotopic.  For the handle chain complex of an Eilenberg-Maclane space, a theorem of Davis (see Theorem \ref{Thm:davisdiag}) is required.

The diagonal maps are chain maps induced by the diagonal map of a topological space.
\begin{equation}\label{topoldiagonal}
\Delta \colon \wt{M} \to \wt{M} \times \wt{M}; \; y \mapsto (y,y).
\end{equation}
This map is $\pi$-equivariant, so we can take the quotient by the action of $\pi$.  This yields
\begin{equation}\label{topoldiagonal2}
\Delta \colon M \to \wt{M} \times_{\pi} \wt{M},
\end{equation}
where:
\[\wt{M} \times_{\pi} \wt{M} := \frac{\wt{M} \times \wt{M}}{\{(x,y) \sim (gx,gy) \,|\,g \in \pi\}} \]
\begin{theorem}[Eilenberg-Zilber] Let $X$ and $Y$ be topological spaces, and let $C(X),C(Y)$ and $C(X \times Y)$ be the corresponding singular or simplicial chain complexes.  There is a natural chain homotopy equivalence:
\[EZ \colon C(X \times Y) \simeq C(X) \otimes C(Y).\]
\end{theorem}
\begin{proof}
See \cite[pages~315--8]{Bredon}.
\end{proof}
Therefore, algebraically, we want a map:
\[\Delta_0 \colon C(\wt{M}) \to C(\wt{M}) \otimes_{\Z} C(\wt{M}).\]
We then take tensor product over $\Z[\pi]$ with $\Z$, on the left, of both the domain and codomain, to get a chain map:
\begin{equation}\label{tensorwithZ}
\Delta_0 \colon \Z \otimes_{\Z[\pi]} C(\wt{M}) \to \Z \otimes_{\Z[\pi]} (C(\wt{M}) \otimes_{\Z} C(\wt{M})).
\end{equation}
Since $\pi$ acts trivially on $\Z$, and diagonally on $C(\wt{M}) \otimes_{\Z} C(\wt{M})$, we are left with a chain map:
\[\Delta_0 \colon C(M) \to C(\wt{M})^t \otimes_{\Z[\pi]} C(\wt{M})\]
which algebraically encodes the topological map \[\Delta \colon M \to \wt{M} \times_{\pi} \wt{M}\] from equation (\ref{topoldiagonal2}) above.
The superscript $t$ denotes the involution $\overline{g} = g^{-1}$ on $\zp$ being used to make $C(\wt{M})$ into a right module in order to form the tensor product.  This is precisely the effect of tensoring on the left with $\Z$ as in (\ref{tensorwithZ}).

Note that in the case that $\wt{M}$ is contractible, such as when $M$ is a $K(\pi,1)$, the map \[\Delta \colon M \xrightarrow{\sim} \wt{M} \times_{\pi} \wt{M}\]
is a homotopy equivalence.  This means that the composition
\[EZ \circ \Delta_* \colon C_*(M) \to C_*(\wt{M} \times_{\pi} \wt{M}) \to C_*(\wt{M}) \otimes_{\Z[\pi]} C_*(\wt{M})\]
is a chain equivalence, so induces an isomorphism on homology, which as we shall see gives us Poincar\'{e} duality isomorphisms.  The problem is to realise these algebraic maps explicitly on small chain complexes.
\begin{theorem} (cf. \cite{Davis})
Let $S(Y)$ denote the singular chain complex of a topological space $Y$.  Then there exists a chain map
\[\Delta_0 \colon S(Y) \to S(Y) \otimes_{\Z} S(Y)\]
such that $\Delta_0(c) = c \otimes c$ for all $c \in S_0(Y)$.
\end{theorem}
\begin{proof}
The proof is by the method of acyclic models: see e.g. \cite[page~317]{Bredon} for an exposition of the method.
\end{proof}

While the method of acyclic models guarantees the existence of such a map on the singular chain groups, the handle chain groups are considerably smaller.  While this is a virtue in that all the information is contained in something which can be explicitly written down (e.g. Theorem \ref{Thm:unicoverchaincomplex}), it means that the topological diagonal map cannot be approximated nearly as closely, and since the ``models'' are the handles themselves, the complex itself must be acyclic.

The fact is that the product of two handles will not in general be a handle in the diagonal of the product space.  For example, consider the circle decomposed into a 0-handle and a 1-handle.  The universal cover of the circle is $\wt{S^1} = \R$, and the chain groups $C_i(\R)$ are $\Z[\Z]$-modules, generated by a point for $C_0(\R)$, and by the interval $[0,1]$ for $C_1(\R)$.  We seek to approximate $\Delta \colon \R \to \R \times \R; \; x \mapsto (x,x)$, by a map:
\[C(\R) \to C(\R) \otimes C(\R) \simeq C(\R \times \R).\]

\begin{figure}[h!]
 \begin{center}
 {\psfrag{R}{$\R$}
 \psfrag{D}{Diagonal}
 \includegraphics [width=7cm] {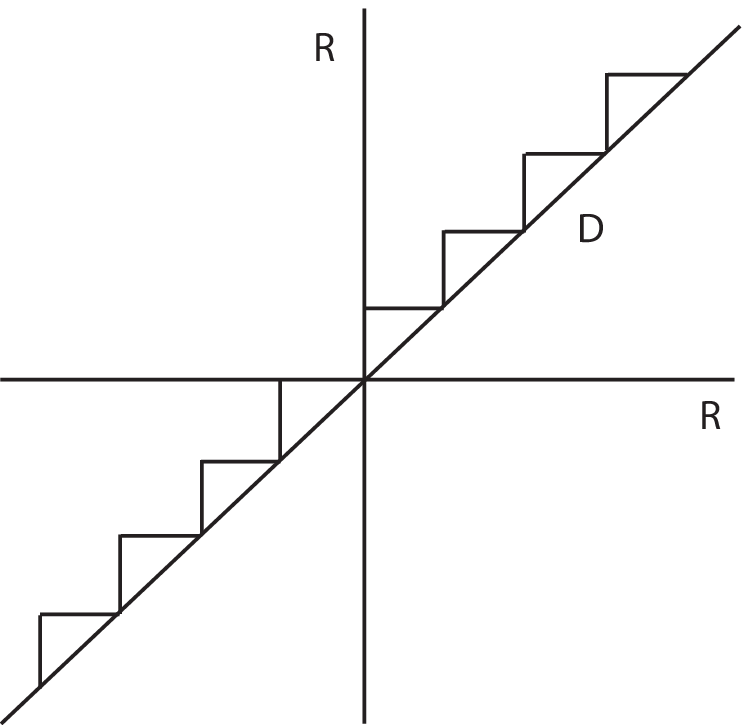}}
 \caption{The diagonal in $\R \times \R = \wt{S^1} \times \wt{S^1}$ with a choice of handle chain approximation to it.}
 \label{circlediagonalapprox}
 \end{center}
\end{figure}

The diagonal map does not map a 1-handle to a 1-handle; the best we can do is to make a staircase with integral increments.  One such is shown in Figure \ref{circlediagonalapprox}.  The question then arises as to whether this staircase should be above or below the diagonal line; Figure \ref{circlediagonalapprox} shows the above version, but putting it below would seem equally as valid.  One is the transpose of the other, and the two can be seen to be geometrically homotopic by homotoping across the boxes which are the product of the 1-handles from each copy of $\R$.  This hopefully motivates the following two definitions.

\begin{definition}
Let $C_*$ be a chain complex of finitely generated (f.g.) projective $A$-modules for a ring with involution $A$, and let $\eps = \pm 1$.  We define the $\eps$-transposition map
\[T_{\eps} \colon C_p^t \otimes C_q \to C_q^t \otimes C_p\]
by
\[x^t \otimes y \mapsto (-1)^{pq} y^t \otimes \eps x.\]
$T$ generates an action of $\Z_2$ on $C \otimes_{A} C$.
We also denote by $T_{\eps}$ the corresponding map on homomorphisms:
\[T_{\eps} \colon \Hom_A(C^p,C_q) \to \Hom_A(C^q,C_p)\]
given by
\[\theta \mapsto (-1)^{pq}\eps \theta^*.\]
\qed \end{definition}

\begin{definition}\label{Defn:higherdiagonalmaps}
A \emph{chain diagonal approximation} is a chain map $\Delta_0 \colon C_* \to C_* \otimes C_*$, with a choice of a collection, for $i \geq 1$, of chain homotopies $\Delta_i \colon C_* \to C_* \otimes C_*$ between $\Delta_{i-1}$ and $T_{\eps}\Delta_{i-1}$.  That is, the $\Delta_i$ satisfy the relations:
\[\partial \Delta_i - (-1)^i\Delta_i\partial = \Delta_{i-1} + (-1)^iT_{\eps}\Delta_{i-1}. \]
Note that $\Delta_i \colon C_k \to (C_* \otimes C_*)_{k+i}$ is a map of degree $i$.
\qed \end{definition}

The following theorem of Davis \cite{Davis} ensures the existence algebraically of the diagonal approximation for an abstract acyclic chain complex.  In particular, it is indeed possible to choose the maps $\Delta_i$ as in Definition \ref{Defn:higherdiagonalmaps} for the handle chain complex of the universal cover of a $K(\pi,1)$ such as the knot exterior, and any choices only affect the answer up to a chain homotopy, as long as the $\Delta_i$ satisfy certain geometrically motivated conditions.

\begin{theorem}\label{Thm:davisdiag}
Let $C = (C_i,\partial)_{0 \leq i \leq n}$ be a chain complex of free $\Z[\pi]$-modules in non-negative dimensions, with augmentation\footnote{We define $\a\left(\left(\sum_{g \in \pi}a_g g\right)\wt{h}^0\right) := \sum_{g \in \pi}a_g$, ($a_g \in \Z$, only finitely many $a_g \neq 0$) for some generator $\wt{h}^0$ of $C_0$, and $\a(x)=0$ for any other generators $x$ of $C_0$.  Also, for convenience, define $\a$ to be zero on $C_i$ for $i>0$.} $\a \colon C_0 \to \Z$, such that the augmented chain complex is acyclic.  Then there exists a $\Z[\pi]$-module chain diagonal approximation $\Delta_i, i=0,1,\dots,n$ ($\Delta_i = 0$ for $i >n$), as in Definition \ref{Defn:higherdiagonalmaps}, satisfying:
\begin{description}
  \item[\emph{(i)}] For all $j$, $\Delta_j(C_i) \subset \bigoplus_{m\leq i,n\leq i}\,C_m\otimes C_n$.
  \item[\emph{(ii)}]  $(\a \otimes 1) \circ \Delta_0 = 1$.
  \item[\emph{(iii)}]$(1 \otimes \a) \circ \Delta_0 = 1$.
  \item[\emph{(iv)}] For all $i$, for any $c \in C_i$, there is an $a \in C_i \otimes C_i$ such that:\\
  \hspace{20pt}$\Delta_i(c) - c \otimes c = a + (-1)^iT_{\eps}a$.
\end{description}
Furthermore, any two choices of such maps are chain homotopic.
\end{theorem}
\begin{proof}
See \cite[Theorem~2.1]{Davis}.  He uses a chain contraction for the augmented complex, which induces a chain contraction on the augmented product complex $C_*^t \otimes_{\Z[\pi]} C_*$, to inductively define the $\Delta_i$.
\end{proof}

We will make use of the diagonal chain approximation maps as follows.  Applying the slant isomorphism (defined below) to the image of the fundamental class \[\Delta_0([M]) \in (C_*(M;\Z[\pi])^t \otimes_{\Z[\pi]} C_*(M;\Z[\pi]))_n\]
yields a set of $\Z[\pi]$-module homomorphisms
\[\varphi_0 \colon C(M;\Z[\pi])^{n-r} \to C(M;\Z[\pi])_{r},\]
which give the cap product explicitly upon descent to homology.  See e.g. \cite[chapter~6]{Bredon} or \cite[chapter~4]{Ranicki} for the standard construction of the cap product using the Alexander-Whitney simplicial diagonal approximation.  We now make the necessary definitions and fix our sign conventions.
\begin{definition}\label{Defn:signsoontensor}
Given chain complexes $(C,d_C)$ and $(D,d_D)$ of f.g. projective left $A$-modules, with $C_r,D_r = 0$ for $r<0$, where $A$ is a ring with involution, we can form the tensor product chain complex $C^t \otimes_{A} D$ defined as:
\[(C^t \otimes_A D)_n := \bigoplus_{p+q=n} \, C^t_p \otimes_A D_q,\]
where the $t$ superscript means that the involution on $A$ is used to make $C_p$ into a right module, with boundary map:
\[d_{\otimes} \colon (C^t \otimes_A D)_n \to (C^t \otimes_A D)_{n-1}\]
given, for $x \otimes y \in C^t_p \otimes_A D_q \subseteq (C^t \otimes_A D)_n$, by
\[d_{\otimes}(x \otimes y) = x \otimes d_D(y) + (-1)^q d_C(x) \otimes y.\]
We define the complex $\Hom_A(C,D)$ by
\[\Hom_A(C,D)_n := \bigoplus_{q-p=n}\,\Hom_A(C_p,D_q)\]
with boundary map
\[d_{\Hom} \colon \Hom_A(C,D)_n \to \Hom_A(C,D)_{n-1}\]
given, for $g \colon C_p \to D_q$, by
\[d_{\Hom}(g) = d_D g + (-1)^q g d_C.\]
The dual complex $C^*$ is defined as a special case of this with $D_0 = A$ as the only non--zero chain group.  Explicitly we define $C^r := \Hom_A(C_r,A)$, with boundary map
\[\delta=d^*_C \colon C^{r-1} \to C^{r}\]
defined as
\[\delta(g) = g \circ d_C.\]
Note that the dual complex $(C^*,\delta)$ consists of chain groups which naturally are right modules, so we use the involution to make them into left modules.
Define, for $g \in C^*$:
\[(a \cdot g)(x) := g(x) \ol{a}. \]
There is an isomorphism:
\[C_* \xrightarrow{\simeq} C^{**} ;\; x \mapsto (f \mapsto \ol{f(x)}).\]
The slant map is:
\[\ba{rcl} \backslash \colon C^t \otimes_A C & \to & \Hom_A(C^{-*},C_*)\\
x \otimes y & \mapsto & \left(g \mapsto \overline{g(x)}y\right) \ea\]
where the chain complex $C^{-*}$ is defined to be
\[(C^{-*})_r = C^{-r};\;\; d_{C^{-*}} = (d_C)^*=\delta.\]
\qed \end{definition}
\begin{proposition}\label{prop:chainmap}
The slant map is an isomorphism between each chain group and commutes with the differentials and is therefore an isomorphism of chain complexes.
\end{proposition}
\begin{proof}
See \cite[Chapter~4]{Ranicki}.
\end{proof}
Let $x \in C_m(M;\Z)$ be a chain.  Since  $\Delta_0$ is a chain map, we have
\[d_{\otimes} \Delta_0(x) = \Delta_0 d_C(x).\]
Therefore \[d_{\Hom}\backslash \Delta_0(x) = \backslash d_{\otimes} \Delta_0(x) = \backslash \Delta_0 (d_C x)\]
since the slant map is a chain isomorphism.  Suppose that $[x] \in H_m(M)$ is homology class.  Then it is a cycle, so
\[d_{\Hom}\backslash \Delta_0([x]) = \backslash \Delta_0 (d_C x) = \backslash \Delta_0 (0) = 0.\]
Therefore $$\backslash \Delta_0([x]) \in \Hom_{\Z[\pi]}(C^{-*}(\wt{M}),C_*(\wt{M}))_m$$ yields a collection of homomorphisms $$g = \{g_r \colon C^{-(m-r)} \to C_{r}\}_{r=0}^m,$$ which satisfy:
\begin{equation}\label{signconventionduality} d_C g_{r+1} + (-1)^r g_{r} \delta = 0.\end{equation}
Note that we can rearrange this to give:
\[d_C g_{r+1} = (-1)^{r+1}g_{r} \delta.\]
We want to use the language of homological algebra to claim that a homology class in $H_m(M;\Z)$ induces a chain map between the chain complex and its dual.  In order to do this we need to take care of the signs.  We therefore, for an $A$-module chain complex $C$, define the complex $C^{m-*}$ by:
\[(C^{m-*})_r = \Hom_A(C_{m-r},A)\]
with boundary maps
\[\partial^* \colon (C^{m-*})_{r+1} \to (C^{m-*})_{r}\]
given by
\[\partial^* = (-1)^{r+1}\delta.\]
With this new chain complex, we have:
\begin{proposition}
A homology class $[x] \in H_m(Y)$ induces a chain homotopy class of chain maps $g = \{g_{m-r} \colon C^{m-r} \to C_r\}_{r=0}^m = \backslash \Delta_0(x) \in \Hom_{\Z[\pi]}(C^{m-*},C_*)$ which descend to give the cap product with $[x]$ on homology.  If $m=n=\dim M$ and $[x] =[M]$, a fundamental class of the manifold, then $g = \varphi_0$ is a chain equivalence which gives rise to the Poincar\'{e} duality isomorphisms between cohomology and homology.
\end{proposition}
\begin{proof}
With the change in sign in the coboundary maps, it is straight-forward that $d_C g_{r+1} = g_{r} \partial^*$.  If we change $x$ to $x+d_C y$ for some $y \in C_{m+1}(Y)$, then this changes the resulting chain map by a boundary in $\Hom_{\Z[\pi]}(C^{n-*},C_*)$, that is to a chain homotopic map.  The corresponding map on homology, which is the cap product with $[x]$, is therefore well defined.
See e.g. \cite{Hatcher}, \cite{Bredon} for the proof that cap product with the fundamental class induces Poincar\'{e} duality isomorphisms.
\end{proof}

While we use this chain complex $C^{n-*}$ to express the cap product maps as chain maps, we prefer to maintain our original notation to discuss the duality maps \emph{i.e.} maps $g \colon C^{n-r} \to C_{r}$ which satisfy equation (\ref{signconventionduality}).  In the case of a closed manifold $M$, we take $x = [M]$ to be the fundamental class and we call these maps \[\varphi_0 := \backslash  \Delta_0 ([M]).\]
Note the subscript on $\Delta_0$ and $\varphi_0$; this is because there are also higher chain homotopies as in Definition \ref{Defn:higherdiagonalmaps} and Theorem \ref{Thm:davisdiag} which take care of the failure of $\Delta_0$ to be symmetric.  They are related to the Steenrod squares which encode higher level information about the intersection properties of the manifold: just as the cup product of $f \in H^i(C)$ and $g \in H^j(C)$ is
\[f\cup g = \Delta_0^*(f^t \otimes g) \in H^{i+j}(C),\]
where $f^t$ is the induced map on $C_i^t$, for a cohomology class $f \in H^r(C)$ we define
\[Sq^i(f) = \Delta_{r-i}^*(f^t \otimes f).\]
Using the higher $\Delta_i$ we can define the entire symmetric structure on a chain complex, which we now proceed to do.

Let $C_*$ be a chain complex of finitely generated projective $A$-modules.  Our principal example is $C_* = C_*(\wt{M})$ with $A=\Z[\pi]$, however once the symmetric structure is obtained a symmetric chain complex is a purely algebraic object.

Recall that the diagonal chain approximation map
\[\Delta_0 \colon C(\wt{M}) \to C(\wt{M})^t \otimes C(\wt{M})\]
was far from unique.  In particular, the transpose $T_{\eps} \circ \Delta_0$ carries essentially the same information: the $\Z_2$ action yields different maps $\varphi_0$ which have the same effect on the homology level; this is the fact that the cup product is signed-commutative on cohomology.  Therefore, as in Definition \ref{Defn:higherdiagonalmaps}, there is a chain homotopy which we call $\Delta_1 \colon C_n \to (C \otimes C)_{n+1}$ between $\Delta_0$ and $T_{\eps}\Delta_0$:
\[d\Delta_1 + \Delta_1 d = \Delta_0 - T_{\eps} \Delta_0.\]
This induces maps $\varphi_1 := \backslash \Delta_1([M]) \colon C^{n-r+1} \to C_r$ which are a chain homotopy from $\varphi_0$ to its transpose; \emph{i.e.} such that:
\[d_C \varphi_1 + (-1)^r \varphi_1 \delta_C + (-1)^n(\varphi_0 - T_{\eps}\varphi_0) = 0 \colon C^{n-r} \to C_r.\]
This process now iterates.  The homotopy $\Delta_1$ and therefore $\varphi_1$ itself fails to be symmetric in general, and so we need a chain homotopy $\Delta_2$ between $\Delta_1$ and its transpose.  Again, $\Delta_2$ fails to be symmetric, and so on, until we reach $\Delta_n \colon C_n \to (C \otimes C)_n$.  The map $\Delta_n \colon C_n \to C_n \otimes_A C_n$ corresponds to the zeroth Steenrod square $Sq^0$ and so must be non-trivial.  All this information can be encoded in a single algebraic object as follows.
\begin{definition}\label{Defn:Qgroups}
Let $W$ be the standard free $\Z[\Z_2]$-resolution of $\Z$, but without the $\Z$ at the end, shown below.  Geometrically it arises as the augmented chain complex of the universal cover $S^{\infty}$ of the model for $K(\Z_2,1)$, namely $\mathbb{RP}^{\infty}$, constructed as a CW complex with a cell decomposition which has one cell in each dimension $0,1,2,\dots$ and so on.  We have:
\[W:\;\;\; \dots \to \Z[\Z_2] \xrightarrow{1+T} \Z[\Z_2] \xrightarrow{1-T}\Z[\Z_2] \xrightarrow{1+T}\Z[\Z_2] \xrightarrow{1-T} \Z[\Z_2].\]
Given a f.g. projective chain complex $C_*$ over $A$ and $\eps \in \{-1,1\}$, define the $\eps$-symmetric $Q$-groups to be:
\[Q^n(C,\eps) := H_n(\Hom_{\Z[\Z_2]}(W,C^t \otimes_A C)) \cong H_n(\Hom_{\Z[\Z_2]}(W,\Hom_{A}(C^{-*},C_*)))\]
An element $\varphi \in Q^n(C,\eps)$ can be represented by a collection of $A$-module homomorphisms
\[\{\varphi_s \in \Hom_A(C^{n-r+s},C_r)\,|\,r \in \Z, s \geq 0\}\]
such that:
\[d_C\varphi_s + (-1)^r \varphi_s\delta_C + (-1)^{n+s-1}(\varphi_{s-1}+(-1)^sT_{\eps}\varphi_{s-1}) = 0 \colon C^{n-r+s-1} \to C_r\]
where $\varphi_{-1} = 0$.  The signs which appear here arise from our choice of convention on the boundary maps in Definition \ref{Defn:signsoontensor}.  If we omit $\eps$ from the notation we take $\eps=1$, so that $Q^n(C) := Q^n(C,1)$.

A pair $(C_*,\varphi)$, with $\varphi \in Q^n(C)$, is called an $n$-dimensional symmetric $A$-module chain complex.  It is called an $n$-dimensional symmetric \emph{Poincar\'{e}} complex if the maps $\varphi_0 \colon C^{n-r} \to C_r$ form a chain equivalence.  In particular this implies that they induce isomorphisms (the cap products) on homology:
\[\varphi_0 \colon H^{n-r}(C) \xrightarrow{\simeq} H_r(C).\]
The symmetric structure is covariantly functorial with respect to chain maps.  A chain map $f \colon C \to C'$ induces a map\footnote{The upper indices here \emph{do not} indicate contravariance; they are used to distinguish from the quadratic structure, which is dual to the symmetric structure in a different way.} $f^\% \colon Q^n(C) \to Q^n(C')$ given by
\[f^\%(\varphi)_s = (f^t \otimes_A f)(\varphi_s) \in C'^t \otimes_A C';\;\text{or}\]
\[\varphi_s \mapsto f\varphi_sf^*.\]
A homotopy equivalence of $n$-dimensional symmetric complexes $f \colon (C,\varphi) \to (C',\varphi')$ is a chain equivalence $f \colon C \to C'$ such that $f^\%(\varphi) = \varphi'$.
\qed \end{definition}

We remark that although we used the geometry to construct the symmetric structure on the chain complex of a manifold, once we have the information we have a purely algebraic object, albeit a fairly complex and unwieldy one, but nevertheless purely algebraic.  This completes our description of the symmetric construction for closed manifolds; we now move on to the important case of manifolds with boundary.

\section{Symmetric structures on manifolds with boundary}

One of the great strengths of the theory of algebraic surgery is that it copes extremely well with manifolds with boundary, particularly when the boundary is split into more than one piece.

So, suppose that instead of a closed manifold $M$ that we have $(X,\partial X)$, an $(n+1)$-manifold with $n$-dimensional boundary.  Then we can take a relative fundamental class $[X,\partial X] \in C_{n+1}(X;\Z)$, which maps in the homology long exact sequence of a pair to a generator of $H_{n+1}(X,\partial X;\Z)$.  On the chain level, $d_C([X,\partial X]) = (-1)^{n+1} f([\partial X]) \in C_{n}(X)$, where $f$ is the chain level inclusion of the boundary into $X$, and $[\partial X]$ is the fundamental class of the boundary $\partial X$.  It is unfortunately necessary to introduce a sign into the identification of the boundary of the fundamental class with the fundamental class of the boundary, in order to fit in with the general scheme of signs in (\cite{Ranicki3} part I) and in Definitions \ref{Defn:signsoontensor} and \ref{Defn:Qgroups}: this sign comes from the use of an algebraic mapping cone of $f \otimes f$ to define the matching conditions of $Q^{n+1}(f)$ in Definition \ref{poincarepaireqns} below.  In the case of a manifold with boundary we have:
\[d_{\otimes} \Delta_0([X,\partial X]) = \Delta_0 d_C ([X,\partial X]) = \Delta_0((-1)^{n+1} f([\partial X]))\]
We adopt the following notation: for a manifold with boundary we call by $\delta \varphi$ the collection of maps given by $\backslash \Delta([X,\partial X])$, and for the duality maps on the boundary $\backslash \Delta([\partial X])$ we use $\varphi$, since $\partial X$ is a closed manifold.  To define a symmetric pair we first recall the algebraic mapping cone construction.
\begin{definition}
The \emph{algebraic mapping cone} $\mathscr{C}(g)$ of a chain map $g \colon C \to D$ is the chain complex given by:
\[d_{\mathscr{C}(g)} = \left(\begin{array}{cc} d_D & (-1)^{r-1}g \\ 0 & d_C \end{array} \right) \colon \mathscr{C}(g)_r = D_r \oplus C_{r-1} \to \mathscr{C}(g)_{r-1} = D_{r-1} \oplus C_{r-2}.\]
\qed \end{definition}

\begin{definition}\label{poincarepaireqns}
The \emph{relative $Q$-groups} of an $A$-module chain map $f \colon C \to D$ are defined to be:
\[Q^{n+1}(f) := H_{n+1}(\Hom_{\Z[\Z_2]}(W,\mathscr{C}(f^t \otimes_{A} f))).\]
An element $(\delta\varphi,\varphi) \in Q^{n+1}(f)$ can be represented by a collection:
\[\{(\delta\varphi_s,\varphi_s) \in (D^t \otimes_A D)_{n+s+1} \oplus (C^t \otimes_A C)_{n+s}\, |\, s \geq 0\}\]
such that:
\begin{eqnarray*}(d_{\otimes}(\delta\varphi_s) + (-1)^{n+s}(\delta\varphi_{s-1}+(-1)^sT_{\eps}\delta\varphi_{s-1}) + (-1)^nf\varphi_sf^*,\\  d_{\otimes}(\varphi_s) + (-1)^{n+s-1}(\varphi_{s-1}+(-1)^sT_{\eps}\varphi_{s-1})) = 0 \\ \in (D^t \otimes_{A} D)_{n+s} \oplus (C^t \otimes C)_{n+s-1}
\end{eqnarray*}
where as before $\delta\varphi_{-1} = 0 = \varphi_{-1}$.
A chain map $f\colon C \to D$ together with an element $(\delta\varphi,\varphi) \in Q^{n+1}(f)$ is called an \emph{$(n+1)$-dimensional symmetric pair}.  A chain map $f$ together with an element of $Q^{n+1}(f)$ is called an \emph{$(n+1)$-dimensional symmetric Poincar\'{e} pair} if the relative homology class\footnote{The (co)homology groups of a chain map are defined to be the homology groups of the (dual of) the algebraic mapping cone.} $(\delta\varphi_0,\varphi_0) \in H_{n+1}(f^t \otimes_A f)$ induces isomorphisms
\[H^{n+1-r}(D,C) := H^{n+1-r}(f) \xrightarrow{\simeq} H_r(D) \;(0 \leq r \leq n+1).\]
For a symmetric Poincar\'{e} pair corresponding to an $(n+1)$-dimensional manifold with boundary, these are the isomorphisms of Poincar\'{e}-Lefschetz duality.
\qed \end{definition}

In the sequel we shall be particularly concerned with the maps for $s=0$, and we shall give explicit formulae for these for knot exteriors, whereas for the higher $\varphi_s$ maps we shall have to content ourselves with the knowledge that these maps exist.  For each null-cobordism in our triad from Chapter \ref{chapter:chaincomplex}, we will have explicit algebraic data which consists of a map of chain complexes, and the duality maps:
\[(f \colon C \rightarrow D,(\delta \varphi_0,\varphi_0)),\]
such that
\begin{equation}\label{compatabilityboundary}
d_{\Hom}(\delta\varphi_0) = \partial(\delta\varphi_0)_{r+1} + (-1)^r(\delta\varphi_0)_{r}\delta = (-1)^{n+1}f (\varphi_0)_r f^*
\end{equation}
where $(\delta\varphi_0)_r \colon D^{n+1-r} \to D_{r}$, and
\[\partial (\varphi_0)_{r+1} + (-1)^r (\varphi_0)_r \delta = 0\]
where $(\varphi_0)_r \colon C^{n-r} \to C_r$.

This algebraic situation mirrors the situation that $D = C(X;\zp)$ and $C=C(\partial X;\zp)$; the condition checks that the duality on the boundary is consistent with that on the interior manifold.  It follows from the fact that the slant map, the diagonal chain approximation and $f$ are chain maps.  We have:
\begin{multline*} d_{\Hom}(\backslash \Delta_0([X,\partial X])) = \backslash d \Delta_0([X,\partial X])= (-1)^{n+1}\backslash\Delta_0(f([\partial X])) \\ = (-1)^{n+1} \backslash (f^t \otimes f)(\Delta_0([\partial X])).
\end{multline*}
Heuristically this says that the boundary represents precisely the failure of Poincar\'{e} duality on $X$; instead there will be Poincar\'{e}-Lefschetz duality.  The chain level version crucially provides more information.  This is because $\delta\varphi_0$ is a chain null-homotopy of $f\varphi_0f^*$; that is a particular reason why cycles of $\partial X$ do not have duals upon inclusion in $X$.  Firstly, the dimension shift means that intersections of a cycle and its dual in $\partial X$ are no longer transverse, since one of them can be pushed into the interior.  The algebraic null-cobordism of the boundary - the chain complex of the interior of the manifold - records which cycles bound in the interior, and how the relative cycles thence created intersect.  In the case where $\partial X = S^1 \times S^1$ we record key algebraic information about the particular knot exterior $X$.

There is another way to express manifolds with boundary algebraically, which we include for completeness, since the distinctions discussed here are common sources of confusion for beginners: the chain complex $(C(X),\delta\varphi)$ is not even a symmetric complex since the maps $\delta\varphi_0$ are not chain maps: the terms $f\varphi_0 f^*$ in Equation (\ref{compatabilityboundary}) prevent this.  However, the relative chain complex $(C_*(X,\partial X), \delta \varphi /\varphi)$ is a $n$-dimensional symmetric chain complex which is \emph{not Poincar\'{e}}.  The chain complex of the boundary of a symmetric chain complex measures in a precise way the failure of the complex to be Poincar\'{e}; the boundary is given by the mapping cone on the duality maps: $\mathscr{C}(\varphi_0 \colon C^{n-r} \to C_r)_{*+1}$.  This algebraic mapping cone is contractible if and only if $\varphi_0$ is a chain equivalence, which is precisely the condition for $(C,\varphi)$ to be a Poincar\'{e} complex.  We can therefore encode a symmetric Poincar\'{e} pair $(f \colon C \to D,(\delta\varphi,\varphi))$ in a single symmetric chain complex $(\mathscr{C}(f),\delta\varphi/\varphi)$.  We call the two ways of expressing a manifold with boundary \emph{the fundamental confusion of algebraic surgery}.

\begin{definition}\label{Defn:algThomcomplexandthickening}
An $n$-dimensional symmetric complex $(C,\varphi \in Q^n(C,\eps))$ is \emph{connected} if
\[H_0(\varphi_0 \colon C^{n-*} \to C_*) = 0.\]
The \emph{algebraic Thom complex} of an $n$-dimensional $\eps$-symmetric Poincar\'{e} pair over $A$
\[(f \colon C \to D, (\delta \varphi, \varphi) \in Q^n(f,\eps))\]
is the connected $n$-dimensional $\eps$-symmetric complex over $A$
\[(\mathscr{C}(f),\delta \varphi/\varphi \in Q^n(\mathscr{C}(f),\eps))\]
where
\begin{eqnarray*}(\delta\varphi/\varphi)_s &:=& \left( \begin{array}{cc} \delta \varphi_s & 0 \\ (-1)^{n-r-1}\varphi_s f^* & (-1)^{n-r+s}T_{\eps} \varphi_{s-1} \end{array} \right) \colon \mathscr{C}(f)^{n-r+s} \\ &=& D^{n-r+s} \oplus C^{n-r+s-1} \to \mathscr{C}(f)_r = D_r \oplus C_{r-1} \;\; (s \geq 0).\end{eqnarray*}

The \emph{boundary} of a connected $n$-dimensional $\eps$-symmetric complex $(C,\varphi \in Q^n(C,\eps))$ over $A$, for $n \geq 1$, is the $(n-1)$-dimensional $\eps$-symmetric \emph{Poincar\'{e}} complex over $A$
\[(\partial C,\partial \varphi \in Q^{n-1}(\partial C,\eps))\]
given by:
\[d_{\partial C} = \left(\begin{array}{cc} d_C & (-1)^r \varphi_0 \\ 0 & \partial^*=d_{C^{n-*}} \end{array} \right) \colon \partial C^{r} = C_{r+1} \oplus C^{n-r} \to \partial C_r = C_{r} \oplus C^{n-r+1};\]
\begin{multline*} \partial \varphi_0 = \left(\begin{array}{cc} (-1)^{n-r-1}T_{\eps}\varphi_1 & (-1)^{r(n-r-1)}\eps \\ 1 & 0 \end{array} \right) \colon \partial C^{n-r-1} = C^{n-r} \oplus C_{r+1} \\ \to \partial C_r = C_{r+1} \oplus C^{n-r}; \end{multline*}
and, for $s \geq 1$,
\begin{multline*} \partial \varphi_s = \left(\begin{array}{cc} (-1)^{n-r+s-1}T_{\eps}\varphi_{s+1} & 0 \\ 0 & 0 \end{array} \right) \colon \partial C^{n-r+s-1} = C^{n-r+s} \oplus C_{r-s+1} \\ \to \partial C_r = C_{r+1} \oplus C^{n-r}. \end{multline*}
See \cite[Part~I,~Proposition~3.4~and~pages~141--2]{Ranicki3} for the full details on the boundary construction.

The \emph{algebraic Poincar\'{e} thickening} of a connected $\eps$-symmetric complex over $A$
\[(C,\varphi \in Q^{n}(C,\eps)),\]
is the $\eps$-symmetric Poincar\'{e} pair over $A$:
\[(i_C \colon \partial C \to C^{n-*}, (0,\partial \varphi ) \in Q^n(i_C,\eps))\]
where
\[i_C = (0 , 1) \colon \partial C = C_{r+1} \oplus C^{n-r} \to C^{n-r}.\]

The algebraic Thom complex and algebraic Poincar\'{e} thickening are inverse operations \cite[part~I,~Proposition~3.4]{Ranicki3}.
\qed \end{definition}

Finally, we give the definition of a symmetric Poincar\'{e} triad.  This is the algebraic version of a manifold with boundary where the boundary is split into two along a submanifold; in other words a cobordism of cobordisms which restricts to a product cobordism on the boundary.  Note that our notion is not quite as general as the notion in \cite[Sections~1.3~and~2.1]{Ranicki2}, since we limit ourselves to the case that the cobordism restricted to the boundary is a product.  We also circumvent the difficult definitions of \cite{Ranicki2}, and define the triads by means of \cite[Proposition~2.1.1]{Ranicki2}, with a sign change in the requirement of $i_-$ to be a symmetric Poincar\'{e} pair.

\begin{definition}\label{Defn:symmPoincaretriad}
A \emph{$(n+2)$-dimensional (Poincar\'{e}) symmetric triad} is a triad of f.g. projective $A$-module chain complexes:
\[\xymatrix @R+1cm @C+1cm{\ar @{} [dr] |{\stackrel{g}{\sim}}
C \ar[r]^{i_-} \ar[d]_{i_+} & D_- \ar[d]^{f_-} \\ D_+ \ar[r]_{f_+} & Y
}\]
with chain maps $i_{\pm},f_{\pm}$, a chain homotopy $g \colon f_- \circ i_- \simeq f_+ \circ i_+$ and structure maps $(\varphi,\delta\varphi_-,\delta\varphi_+,\Phi)$ such that:
$(C,\varphi)$ is an $n$-dimensional symmetric (Poincar\'{e}) complex,
\[(i_+ \colon C \to D_+,(\delta\varphi_+,\varphi))\]
and
\[(i_- \colon C \to D_-,(\delta\varphi_-,-\varphi))\]
are $(n+1)$-dimensional symmetric (Poincar\'{e}) pairs, and
\[(e \colon D_- \cup_{C} D_+ \to Y,(\Phi,\delta\varphi_- \cup_{\varphi} \delta\varphi_+)) \]
is a $(n+2)$-dimensional symmetric (Poincar\'{e}) pair, where:
\[e = \left(\begin{array}{ccc} f_- \, , & (-1)^{r-1}g \, , & -f_+ \end{array} \right) \colon (D_-)_r \oplus C_{r-1} \oplus (D_+)_r \to Y_r.\]
See Definition \ref{Defn:unionconstruction} for the union construction, used to define $(D_- \cup_C D_+,\delta\varphi_- \cup_{\varphi} \delta\varphi_+)$, which glues together two chain complexes along a common part of their boundaries with opposite orientations.

A chain homotopy equivalence of symmetric triads is a set of chain equivalences:
\[\ba{lllcl} \nu_C &\colon& C &\to& C'\;\;\,;\\
\nu_{D_-} &\colon& D_- &\to& D'_-\;;\\
\nu_{D_+} &\colon& D_+ &\to& D'_+ \;;\text{ and}\\
\nu_E &\colon& Y &\to& Y' \ea\]
which commute with the chain maps of the triads up to chain homotopy, and such that the induced maps on $Q$-groups map the structure maps $(\varphi,\delta\varphi_-,\delta\varphi_+,\Phi)$ to the equivalence class of the structure maps $(\varphi',\delta\varphi'_-,\delta\varphi'_+,\Phi')$.  See \cite[Part I,~page~140]{Ranicki3} for the definition of the maps induced on relative $Q$-groups by an equivalence of symmetric pairs.
\qed \end{definition}

\begin{definition}\label{Defn:unionconstruction}
(\cite[Part I,~pages~134--6]{Ranicki3}) An \emph{$\eps$-symmetric cobordism} between symmetric complexes $(C,\varphi)$ and $(C',\varphi')$ is a $(n+1)$-dimensional $\eps$-symmetric Poincar\'{e} pair with boundary $(C \oplus C',\varphi \oplus -\varphi')$:
\[((f_C,f_{C'}) \colon C \oplus C' \to D,(\delta\varphi,\varphi \oplus -\varphi')\in Q^{n+1}((f_C,f_{C'}),\eps)).\]

We define the \emph{union} of two $\eps$-symmetric cobordisms:
\[c  = ((f_C,f_{C'}) \colon C \oplus C' \to D,(\delta\varphi,\varphi \oplus -\varphi')); \text{ and}\]
\[c' = ((f'_{C'},f'_{C''}) \colon C' \oplus C'' \to D',(\delta\varphi',\varphi' \oplus -\varphi'')),\]
to be the $\eps$-symmetric cobordism given by:
\[c \cup c' = ((f''_C,f''_{C''}) \colon C \oplus C'' \to D'',(\delta\varphi'',\varphi \oplus -\varphi'')),\]
where:
\[D''_r:= D_r \oplus C'_{r-1} \oplus D'_r;\]
\[d_{D''} = \left(\begin{array}{ccc} d_D & (-1)^{r-1}f_{C'} & 0 \\ 0 & d_{C'} & 0 \\ 0 & (-1)^{r-1}f'_{C'} & d_{D'}
\end{array}\right)\colon D''_r \to D''_{r-1};\]
\[f''_{C} = \left(
              \begin{array}{c}
                f_C \\
                0 \\
                0
              \end{array}
            \right) \colon C_r \to D''_r;
\]
\[f''_{C'} = \left(
              \begin{array}{c}
                0 \\
                0 \\
                f'_{C''}
              \end{array}
            \right) \colon C_r \to D''_r; \text{ and}
\]
\[\delta\varphi''_s = \left(
                        \begin{array}{ccc}
                          \delta\varphi_s & 0 & 0 \\
                          (-1)^{n-r}\varphi'_sf_{C'}^* & (-1)^{n-r+s+1}T_{\eps}\varphi'_{s-1} & 0 \\
                          0 & (-1)^sf'_{C'}\varphi'_s & \delta\varphi'_s \\
                        \end{array}
                      \right)\colon\]
\[(D'')^{n-r+s+1} = D^{n-r+s+1} \oplus C'^{n-r+s} \oplus D'^{n-r+s+1} \to D''_r = D_r \oplus C'_{r-1} \oplus D'_r \;\;(s \geq 0).\]
We write:
\[(D'' = D \cup_{C'}D',\delta\varphi'' = \delta\varphi\cup_{\varphi'}\delta\varphi').\]
\qed \end{definition}

\section{Formulae for the diagonal chain approximation}\label{section:formulaediagonal}

Trotter \cite{Trotter} gave explicit formulae, which we shall now exhibit, for a choice of diagonal chain approximation map on the 3-skeleton of a $K(\pi,1)$, given a presentation of $\pi$ with a full set of identities.  First, we recall from Definition \ref{identitypresentation}, the concept of an identity of a presentation of a group.

\begin{definition}
Let $\langle \,g_1,..,g_a\,|\,r_1,\dots,r_c\,\rangle$ be a presentation of a group and let $F = F(g_1,\dots,g_a)$ be the free group on the generators $g_i$.  Let $P$ be the free group on letters $\rho_1,\dots,\rho_c$, and let $\psi \colon P \ast F \to F$ be the homomorphism such that $\psi(\rho_i)=r_i$ and $\psi(g_j)=g_j$. An \emph{identity of the presentation} is a word in $\ker(\psi) \leq P \ast F$ of the form:
\begin{equation}\label{identity}
s = \prod_{k=1}^c \, w_{j_k} \rho_{j_k}^{\eps_{j_k}} w_{j_k}^{-1}
\end{equation}
where $\eps_{j_k} = \pm 1$.
\qed \end{definition}

Each identity corresponds to the inclusion of a 3-handle which says that there is a relation amongst the relations.  Recall that the word chosen matters rather than just the element of $\pi_1(X)$ represented, to ensure that we get the trivial element in the free group as the image of $\psi$.  This means more care must be taken, in particular when finding the words $w_i$ from $\partial_3$ of Theorem \ref{Thm:unicoverchaincomplex}: it is not enough to simply choose any path in the quadrilateral decomposition, but rather a path must be chosen such that the relevant cancellation occurs.

\begin{theorem}[\cite{Trotter}]\label{trotterdelta0}
Let $\pi$ be a group with a presentation $$\langle \,g_1,..,g_a\,|\,r_1,\dots,r_b\,\rangle$$ with a full set of identities $s_m = \prod_{k=1}^c \, w_{j_k} r_{j_k}^{\eps_{j_k}} w_{j_k}^{-1}$ for the presentation.  Let $Y$ be a $K(\pi,1)$, and $\wt{Y}$ its universal cover, with a handle structure which corresponds to the presentation and identities.  The diagonal map $\Delta_0 \colon C(\wt{Y}) \to C(\wt{Y}) \otimes_{\Z} C(\wt{Y})$ can be defined on the 3-skeleton $\wt{Y}^{(3)}$ as follows.  Let $h^i$ be the basis elements of the $\Z[\pi]$-modules $C_i(\wt{Y}) \;\; (0 \leq i \leq 3)$, with 1-handles corresponding to generators of $\pi$, 2-handles corresponding to relations, and 3-handles corresponding to identities.  Let $\a \colon F(g_1,\dots,g_a) \to C_1(\wt{Y})$ be given by $\alpha(v) = \sum_i \, \frac{\partial v}{\partial g_i} h^1_i$, using the Fox derivative (defined in Definition \ref{freederivative}). Let $\g \colon F \to C_1(\wt{Y}) \otimes C_1(\wt{Y})$ be the unique homomorphism given by $\g(1)=\g(g_i)=0$, and:
\begin{equation}\label{gammadef}
\g(uv)=\g(u) +u\g(v) + \a(u) \otimes u\a(v).
\end{equation}
Such so-called crossed homomorphisms are well defined and can be arbitrarily prescribed on the generators \cite[page~472]{Trotter}.  Then we can define:
\begin{eqnarray*}
\Delta_0(h^0) & = & h^0 \otimes h^0 \\
\Delta_0(h^1_i) & = & h^0 \otimes h^1_i + h^1_i \otimes g_i h^0 \\
\Delta_0(h^2_j) & = & h^0 \otimes h^2_j + h^2_j \otimes h^0 - \g(r_j) \\
\Delta_0(h_m^3) & = & h^0 \otimes h^3_m + h^3_m \otimes h^0 + \sum_{k=1}^c \, \eps_k \left(\a(w_k) \otimes w_k h^2_{k} + w_k h^2_{k} \otimes \a(w_k)\right)\\
& & + \sum_{k=1}^c \delta_k w_k (h^2_{k} \otimes \alpha(r_{k})) - \sum_{1\, \leq\, l \, <\, k\, \leq \, c} \eps_l w_l h^2_{l} \otimes \eps_k w_k \a(r_{k}).
\end{eqnarray*}
where $\delta_k = \frac{1}{2}(\eps_k - 1)$.
\end{theorem}

\begin{proof}
See \cite[pages 475--6]{Trotter}, where Trotter shows that these are indeed chain maps \emph{i.e.} that\[\Delta_0 \circ \partial = d \circ \Delta_0.\]
Trotter does not state his sign conventions explicitly; however, careful inspection of his calculations shows that his convention for the boundary map of $C^t \otimes C$ disagrees with ours.  We therefore undertook to rework his proof using our sign convention.  It turned out that the only change required in the formulae was a minus sign in front of $\g(r_j)$, which alteration we have already made for the statement of the theorem.
\end{proof}

Note in particular that with $u=g_i, v=g_i^{-1}$, equation (\ref{gammadef}) implies that $\g(g_i^{-1})= g_i^{-1} h^1_i \otimes g_i^{-1} h^1_i$.  When interpreting this formula and those in Theorem \ref{trotterdelta0} we let $\pi$ act on $C(\wt{Y}) \otimes_{\Z} C(\wt{Y})$ by the diagonal action.

\begin{example}
We give the result of the calculation of $\g$ for a typical word which arises in the Wirtinger presentation of the knot group:
\begin{eqnarray*}
\g(g_i^{-1}g_kg_jg_k^{-1}) & = & (g_i^{-1}h^1_i \otimes g_i^{-1}h^1_i) - (g_i^{-1}h^1_i \otimes g_i^{-1}h^1_k) + (h^1_k \otimes h^1_k) -\\ & & (g_i^{-1} g_k h^1_j \otimes h^1_k) + (g_i^{-1} h^1_k - g_i^{-1}h^1_i) \otimes (g_i^{-1}g_k h^1_j - h^1_k).
\end{eqnarray*}
\end{example}

The following fact is now pertinent:
\begin{theorem}\label{Thm:EMspaces}
The knot exterior $X$ and the zero framed surgery $M_K$ (the latter for $K$ not the unknot) are both Eilenberg-MacLane spaces: $X \simeq K(\pi_1(X),1)$ and $M_K \simeq K(\pi_1(M_K),1)$.
\end{theorem}
\begin{proof}
As in Remark \ref{identitypresentation} this follows from the Sphere theorem of Papakyriakopoulos and the Sch\"{o}nflies theorem for $X$.  In addition, for $M_K$ a result of Gabai (\cite[Corollary~5]{Gabai}) using taut foliations says that attaching the solid torus to $X$ to make $M_K$ does not create any new elements of $\pi_2$.
\end{proof}

 Suppose that $X$ is a 3-dimensional manifold with boundary such that both $X$ and $\partial X$ are $K(\pi,1)$s.  Suppose furthermore that we have a presentation of $\pi_1(X)$ which contains a presentation of $\pi_1(\partial X)$ as a sub-presentation, and we have a handle decomposition of $X$ which contains $\partial X$ as a subcomplex, corresponding to the presentation of $\pi_1(X)$.  We can tensor the domain and codomain of Trotter's map with $\Z$ to get a map:
\[\Delta_0 \colon C_*(X;\Z) \to (C(X;\Z[\pi_1(X)])^t \otimes_{\Z[\pi_1(X)]} C(X;\Z[\pi_1(X)]))_*\]
so that
\[\backslash\Delta_0([X,\partial X]) =: \delta\varphi_0;\;\; \backslash\Delta_0([\partial X]) =: \varphi_0,\]
and
\[d_{\otimes} \Delta_0([X,\partial X]) = \Delta_0 d_C ([X,\partial X]) = \Delta_0((-1)^{2+1} f([\partial X]))\]
so that the equations of Definition \ref{poincarepaireqns} are satisfied.  We do not have explicit formulae for the higher diagonal maps $\Delta_i$, for $i=1,2,3$, but at least Theorem \ref{Thm:davisdiag} ensures that they always exist.  We therefore have:
\[(f \colon C(\partial X;\Z[\pi_1(X)]) \to C(X;\Z[\pi_1(X)]),(\delta\varphi,\varphi)),\]
a 3-dimensional symmetric Poincar\'{e} pair, using the pull-back $\pi_1(X)$-cover of $\partial X$.

\section{The fundamental symmetric Poincar\'{e} triad of a knot}\label{Chapter:fundsymmtriad}

We now describe how to use Theorem \ref{trotterdelta0} in order to produce the symmetric Poincar\'{e} triad which will be the main algebraic object which we extract from geometry via our handle decomposition, namely the fundamental symmetric Poincar\'{e} triad of a knot.  Our algebraic concordance group comprises such objects, with some extra data, as the elements of its underlying set.

We proceed as follows.  We first describe the symmetric structure $\varphi$ on our chain complex of the $\px$-cover of a circle $S^1$, thence producing a symmetric Poincar\'{e} complex \[(C_*(S^1 \times S^0;\zpx),\varphi\oplus -\varphi) = (C,\varphi \oplus -\varphi).\]
We have two null-cobordisms of $S^1 \times S^0$ and two algebraic null-cobordisms of its chain complex
\[i_{\pm} \colon C = C_*(S^1 \times S^0;\zpx) \to D_{\pm} = C_*(S^1 \times D^1_{\pm};\zpx).\]
We show that we can consider these as $2$-dimensional symmetric Poincar\'{e} pairs:
\[(i_{\pm} \colon C \to D_{\pm},(\delta\varphi_{\pm}=0,\pm(\varphi \oplus -\varphi))).\]
The orientation induced on the circle at either end of the cylinder is opposite, which is reflected by the symmetric structure on $C$ being $\varphi \oplus -\varphi$.  The next step is to glue the two null-cobordisms together along their common boundary to form a symmetric Poincar\'{e} chain complex of a torus $S^1 \times S^1$, with chain complex $E := \mathscr{C}((i_-,i_+)^T \colon C \to D_- \oplus D_+)$ as in Proposition \ref{Prop:torussplitchainequivalence} with the symmetric structure defined using the union construction (Definition \ref{Defn:unionconstruction}): \[\phi := 0 \cup_{\varphi \oplus -\varphi} 0.\]
We then use the chain equivalence $\eta$ from Proposition \ref{Prop:torussplitchainequivalence} to construct the push-forward symmetric structure on the standard chain complex of the torus, $E'$ from Proposition \ref{Prop:torussplitchainequivalence}:   \[(E',\phi':= \eta\phi\eta^*).\]
We also calculate the symmetric structure $\phi^{Tr}$ which arises on $E'$ from the formula of Trotter in Theorem \ref{trotterdelta0}, and compare the two.  We note that $\phi'_0 - \phi^{Tr}_0 = \backslash d \chi$ for the chain
\[\chi = h^1_{\lambda} \otimes h^2_{\partial} + h^2_{\partial} \otimes h^1_{\mu} \in (E' \otimes E')_3.\]
This enables us to use the formulae of Trotter and our relative fundamental class $[X,\partial X]$ from Remark \ref{Rmk:relativefundclass} to define the symmetric structure on $X$:
\[(C_*(X;\zpx), \Phi)\]
with
\[\Phi_0 := \backslash (\Delta_0([X,\partial X]) - \chi )\]
so that our triad yields a symmetric Poincar\'{e} pair:
\[(f \circ \eta \colon \mathscr{C}((i_-,i_+)^T) = E \to C_*(X;\zpx), (\Phi, \phi')),\]
and we indeed define a symmetric Poincar\'{e} triad.

First, as promised, here is the symmetric structure on a chain complex \\$C_*(S^1;\Z[\Z])$.

\begin{proposition}\label{prop:symmstructurecircle}
The symmetric structure on the chain complex of a circle $C_*(S^1;\Z[\Z])$, where $\Z[\Z] = \Z[t,t^{-1}]$,
\[\xymatrix @C+1cm @R+1cm{ C^0(S^1;\Z[\Z]) \ar[r]^{\delta_1} \ar[d]_{\varphi_0} & C^1(S^1;\Z[\Z]) \ar[d]_{\varphi_0} \ar[dl]^{\varphi_1} \\
C_1(S^1;\Z[\Z]) \ar[r]^{\partial_1} & C_0(S^1;\Z[\Z])}\]
is given by:
\[\xymatrix @C+1cm @R+1cm{ \Z[\Z] \ar[r]^{(t^{-1}-1)} \ar[d]_{(1)} & \Z[\Z] \ar[d]_{(t)} \ar[dl]^{(1)} \\
\Z[\Z] \ar[r]^{(t-1)} & \Z[\Z].}\]
Using the homomorphisms $\Z \to \pi_1(X)$:
\[t \mapsto g_1; \text{ and}\]
\[t \mapsto g_q,\]
we can form two chain complexes:
\[C_*(S^1;\zpx)_j = \zpx \otimes_{\Z[\Z]} C_*(S^1;\Z[\Z]),\]
for $j=1,q$.  The symmetric Poincar\'{e} chain complex of the $\pi_1(X)$-cover of $S^1 \times S^0$ is then:
\begin{eqnarray*}(C,\varphi\oplus-\varphi) &=& (C_*(S^1 \times S^0;\zpx),\varphi^1\oplus-\varphi^q) \\ &=& (C_*(S^1;\zpx)_1,\varphi^1) \oplus (C_*(S^1;\zpx)_q,-\varphi^q).\end{eqnarray*}
Explicitly:
\[\xymatrix @C+2cm @R+1cm{C^0 \ar[r]^{\delta_1} \ar[d]_{(\varphi_0)^1 \oplus -(\varphi_0)^q} & C^1 \ar[d]^{(\varphi_0)^1 \oplus -(\varphi_0)^q} \ar[dl]^{(\varphi_1)^1 \oplus -(\varphi_1)^q}\\
C_1 \ar[r]^{\partial_1} & C_0}\]
is given by:
\[\xymatrix @C+3cm @R+3cm{\bigoplus_2\,\zpx \ar[r]^{\left(\begin{array}{cc} g_1^{-1}-1 & 0 \\ 0 & g_q^{-1}-1 \\ \end{array} \right)} \ar[d]_{\left(\begin{array}{cc} 1 & 0 \\ 0 & -1 \\ \end{array} \right)} & \bigoplus_2\,\zpx \ar[d]^{\left(\begin{array}{cc} g_1 & 0 \\ 0 & -g_q \\ \end{array} \right)} \ar[dl]_{\left(\begin{array}{cc} 1 & 0 \\ 0 & -1 \\ \end{array} \right)}\\
\bigoplus_2\,\zpx \ar[r]^{\left(\begin{array}{cc} g_1-1 & 0 \\ 0 & g_q-1 \\ \end{array} \right)} & \bigoplus_2\,\zpx.}\]
\end{proposition}
\begin{proof}
Let $h^0$ be the $0$-handle, and $h^1$ be the 1-handle of the circle; we use the same notation for the corresponding generators of $C_*(S^1;\Z[\Z])$.  The diagonal map applied to the fundamental class $[S^1] = 1 \otimes_{\Z[\Z]} h^1 \in C_1(S^1;\Z) = \Z \otimes_{\Z[\Z]} C_1(S^1;\Z[\Z])$ of the circle yields (Theorem \ref{trotterdelta0}):
\[\Delta_0([S^1]) = h^0 \otimes_{\Z[\Z]} h^1 +h^1 \otimes_{\Z[\Z]} t h^0.\]
Application of the slant map to this gives us the $\varphi_0$ maps as claimed.  Note that:
\[\partial \varphi_0 + \varphi_0 \delta = 0.\]
The map $\varphi_1=1$ arises as the solution to the equations of a symmetric complex:
\[\partial\varphi_s + (-1)^r \varphi_s\delta_C + (-1)^{n+s-1}(\varphi_{s-1}+(-1)^sT_{\eps}\varphi_{s-1}) = 0 \colon C^{n-r+s-1} \to C_r,\]
which in this case give us the equations, for $r=0,1$:
\[\partial \varphi_1 + (-1)^r\varphi_1 \delta = \varphi_0 - T\varphi_0 \colon C^{1-r} \to C_r.\]
We check these for $r=0,1$:
\[\varphi_0 - T\varphi_0 = t-1 = \partial \varphi_1 + 0 = \partial\varphi_1 + \varphi_1 \delta \colon C^1 \to C_0,\]
and
\[\varphi_0 - T\varphi_0 = 1-t^{-1} = 0 - (t^{-1}-1) = \partial\varphi_1 - \varphi_1 \delta \colon C^0 \to C_1.\]
The $\varphi_0$ maps induce isomorphisms on the chain groups and therefore induce isomorphisms on homology, so the complex is Poincar\'{e}.

We use the two homomorphisms to $\pi_1(X)$ to reflect the two copies of $S^1$ as representing two meridians of the knot, $g_1$ and $g_q$, which encircle the knot at different places.  The chain complex of a disjoint union of spaces is just the direct sum of the chain complexes.  We take opposite orientations on the components so that they are jointly the boundary of $S^1 \times D^1$, as we shall see presently.
\end{proof}

We now check that the complexes $D_{\pm} = C_*(S^1 \times D^1_{\pm};\zpx)$ give two algebraic null-cobordisms of $(C, \varphi \oplus -\varphi)$.

\begin{lemma}\label{lemma:productcobordism}
Given a homotopy equivalence
\[f \colon (C,\varphi) \to (C',\varphi')\]
of $n$-dimensional symmetric Poincar\'{e} chain complexes such that $f^{\%}(\varphi) = \varphi'$, there is a symmetric cobordism, corresponding to a product cobordism in geometry:
\[((f,1) \colon C \oplus C' \to C',(0,\varphi \oplus - \varphi')).\]
This symmetric pair is also Poincar\'{e}.
\end{lemma}
\begin{proof}
We need to check that the symmetric structure maps $(0,\varphi\oplus -\varphi') \in Q^{n+1}((f,1))$ induce isomorphisms:
 \[H^r((f,1)) \xrightarrow{\simeq} H_{n+1-r}(C').\]
We use the long exact sequence in cohomology of a pair, associated to the short exact sequence
\[0 \to C' \xrightarrow{j} \mathscr{C}((f,1)) \to S(C \oplus C') \to 0\]
 to calculate the homology $H^r((f,1))$.  The sequence is:
 \[H^r(C \oplus C') \xleftarrow{(f^*,1^*)^T}H^r(C') \xleftarrow{j^*} H^r((f,1)) \xleftarrow{\partial} H^{r-1}(C \oplus C') \xleftarrow{(f^*,1^*)^T} H^{r-1}(C').\]
We have that $$\ker((f^*,1^*)^T \colon H^r(C') \to H^r(C \oplus C')) \cong 0,$$ so $j^*$ is the zero map, and therefore $\partial$ is surjective.  The image \[\im((f^*,1^*)^T\colon H^{r-1}(C') \to H^{r-1}(C) \oplus H^{r-1}(C'))\]
is the diagonal, so that the images of elements of the form $(0,y') \in H^{r-1}(C) \oplus H^{r-1}(C')$ generate $H^r((f,1))$.

The map $H^r((f,1)) \to H_{n-r+1}(C')$ generated by $(0,\varphi \oplus - \varphi')$, on the chain level, is
\[\left( 0, \left( \begin{array}{cc} f & 1 \end{array} \right)\left( \begin{array}{cc} \varphi_0 & 0 \\ 0 & -\varphi'_0 \end{array} \right)\right) \colon (C')^r \oplus (C \oplus C')^{r-1} \to C'_{n-r+1}\]
which sends $y' \in H^{r-1}(C')$ to $-\varphi'_0(y')$.  We therefore have an isomorphism on homology since $(C',\varphi')$ is a symmetric Poincar\'{e} complex. We have a symmetric Poincar\'{e} pair
\[((f,1) \colon C \oplus C' \to C',(0,\varphi \oplus - \varphi')),\]
as claimed.
\end{proof}

\begin{proposition}\label{prop:symmstructureiplusminus}
Recall the chain maps $i_{\pm} \colon C \to D_{\pm}$:
\[\xymatrix @R+1cm @C-0.9cm{D_- = C_*(S^1 \times D^1_-;\Z[\pi_1(X)])\colon & \Z[\pi_1(X)] \ar[rrrrrrrrrrrrrrrrrrr]^{\left(\begin{array}{c}g_1-1\end{array}\right)} &&&&&&&&&&&&&&&&&&& \Z[\pi_1(X)] \\
C = C_*(S^1 \times S^0;\Z[\pi_1(X)])\colon \ar[u]_{i_-} \ar[d]^{i_+} & \bigoplus_2\,\Z[\pi_1(X)] \ar[rrrrrrrrrrrrrrrrrrr]_{\left(\begin{array}{cc}g_1-1 & 0 \\ 0 & g_q-1 \end{array}\right)} \ar[u]^{\left(\begin{array}{c} 1 \\ l_a^{-1} \end{array}\right)} \ar[d]_{\left(\begin{array}{c} l_b^{-1} \\ 1 \end{array}\right)} &&&&&&&&&&&&&&&&&&& \bigoplus_2\,\Z[\pi_1(X)] \ar[u]_{\left(\begin{array}{c} 1 \\ l_a^{-1} \end{array}\right)} \ar[d]^{\left(\begin{array}{c} l_b^{-1} \\ 1 \end{array}\right)} \\
D_+ = C_*(S^1 \times D^1_+;\Z[\pi_1(X)])\colon & \Z[\pi_1(X)] \ar[rrrrrrrrrrrrrrrrrrr]_{\left(\begin{array}{c}g_q-1\end{array}\right)} &&&&&&&&&&&&&&&&&&& \Z[\pi_1(X)].
}\]
These chain maps $i_{\pm}$ induce symmetric Poincar\'{e} pairs:
\[(i_{\pm} \colon C \to D_{\pm},(\delta\varphi_{\pm}=0,\pm(\varphi \oplus -\varphi)) \in Q^2(i_{\pm})).\]
\end{proposition}
\begin{proof}
There is no relative fundamental class in the chain complex for $D_{\pm}$, since $(D_{\pm})_2 \cong 0$, so we take $\delta\varphi_{\pm}=0$.  The reader can check that the homomorphisms:
\[i_{\pm}(\pm(\varphi_s \oplus -\varphi_s))i_{\pm}^* = 0 \colon (D_{\pm})^{r} \to (D_{\pm})_{1-r+s}\]
for $r=0,1, s=0,1$, so that the equations for a symmetric pair are satisfied.  Lemma \ref{lemma:productcobordism} applies here to show that the pairs are also Poincar\'{e}.

\end{proof}

We now use the union construction of Definition \ref{Defn:unionconstruction} to glue the symmetric pairs $(i_{\pm}\colon C\to D_{\pm},(0,\pm(\varphi\oplus-\varphi)))$ together along $(C,\varphi \oplus -\varphi)$.  Recall from Proposition \ref{Prop:torussplitchainequivalence} that the mapping cone
\[E := \mathscr{C}((i_-,i_+)^T \colon C \to D_- \oplus D_+)\]
with
\[E_r = (D_-)_r \oplus C_{r-1} \oplus (D_+)_r,\]
is given by:
\[\xymatrix{ E_2 \cong \bigoplus_2 \, \zpx \xrightarrow{\partial_2} E_1 \cong \bigoplus_4 \, \zpx  \xrightarrow{\partial_1} E_0 \cong \bigoplus_2 \, \zpx
}\]
where:
\[\partial_1 = \left(\begin{array}{cc} g_1-1 & 0  \\ 1 & l_b^{-1} \\ l_a^{-1} & 1 \\ 0 & g_q-1 \end{array} \right)\]
\[\partial_2 = \left(\begin{array}{cccc} -1 & g_1-1 & 0 & -l_b^{-1} \\ -l_a^{-1} & 0 & g_q-1 & -1 \end{array} \right).\]
Recall also that the chain complex \[E':= C_*(S^1 \times S^1;\zpx) \subset C_*(X;\zpx)\] from Proposition \ref{prop:toruschaincopmlex} is given by:
\[E'_2 \cong \Z[\pi_1(X)] \xrightarrow{\partial_2} E'_1 \cong \bigoplus_2 \, \Z[\pi_1(X)] \xrightarrow{\partial_1} E'_0 \cong \Z[\pi_1(X)],\]
where
\[\partial_1 = \left(\begin{array}{c} g_1-1 \\ l_al_b - 1 \end{array} \right)\]
\[\partial_2 = \left(\begin{array}{cc} l_al_b-1 & 1-g_1 \end{array} \right),\]
and that there is a chain equivalence $\eta \colon E \to E'$:
\[\xymatrix @R+2cm @C+1.3cm{\bigoplus_2 \,\zpx \ar[r]^{\partial_E} \ar[d]^{\eta = \left(\begin{array}{c} -l_b^{-1}l_a^{-1} \\ 0 \end{array} \right)} & \bigoplus_4 \,\zpx \ar[r]^{\partial_E} \ar[d]^{\eta = \left(\begin{array}{cc} 1 & 0 \\ 0 & l_b^{-1}l_a^{-1} \\ 0 & 0 \\ -l_a^{-1} & 0 \end{array} \right)} & \bigoplus_2 \,\zpx \ar[d]^{\eta = \left(\begin{array}{c} 1 \\ -l_a^{-1} \end{array} \right)} \\
\zpx \ar[r]^{\partial_{E'}} &  \bigoplus_2 \,\zpx \ar[r]^{\partial_{E'}} & \zpx.}\]

\begin{proposition}\label{prop:symmstructuretorus}
The symmetric structure on $E = C_*(S^1 \times S^1;\zpx)$ is:
\[\phi:= 0 \cup_{\varphi\oplus-\varphi} 0.\]
The symmetric structure map $\phi_0 \colon E^{2-*} \to E_*$:
\[\xymatrix @R+1cm @C+1cm{ E^0 \ar[r]^{\delta_1} \ar[d]^{\phi_0} & E^1 \ar[r]^{\delta_2} \ar[d]^{\phi_0} & E^2 \ar[d]^{\phi_0}\\
E_2 \ar[r]^{\partial_2} & E_1 \ar[r]^{\partial_1} & E_0
}\]
is given by:
\[\xymatrix @R+3cm @C+1.3cm{ \bigoplus_2\,\zpx \ar[r]^{\delta_1} \ar[d]^{\left(\begin{array}{cc} -1 & l_a \\ 0 & 0 \end{array} \right)} & \bigoplus_4\,\zpx \ar[r]^{\delta_2} \ar[d]^{\left(\begin{array}{cccc} 0 & g_1 & -l_ag_q & 0 \\ 0 & 0 & 0 & l_b^{-1} \\ 0 & 0 & 0 & -1 \\ 0 & 0 & 0 & 0 \end{array} \right)} & \bigoplus_2\,\zpx \ar[d]^{\left(\begin{array}{cc} 0 & g_1 l_b^{-1} \\ 0 & -g_q \end{array} \right)}\\
\bigoplus_2\,\zpx \ar[r]^{\partial_2} & \bigoplus_4\,\zpx \ar[r]^{\partial_1} & \bigoplus_2\,\zpx
}\]
Taking the image
\[\phi' := \eta^{\%}\phi = \eta \phi \eta^*\]
of the chain duality maps under the chain equivalence
\[\eta \colon E \to E',\]
which maps from the chain complex of the torus split into two cylinders, to the smallest possible chain complex of the torus, yields the symmetric structure map $\phi'_0 \colon E'^* \to E'_*$ as follows.
\[\xymatrix @R+1cm @C+1cm{ E'^0 \ar[r]^{\delta_1} \ar[d]^{\phi'_0} & E'^1 \ar[r]^{\delta_2} \ar[d]^{\phi'_0} & E'^2 \ar[d]^{\phi'_0}\\
E'_2 \ar[r]^{\partial_2} & E'_1 \ar[r]^{\partial_1} & E'_0
}\]
is given by:
\[\xymatrix @R+2cm @C+1.3cm{ \zpx \ar[rr]^-{\left(\begin{array}{cc} g_1^{-1}-1 & l_b^{-1}l_a^{-1}-1 \end{array}\right)} \ar[d]^{\left(\begin{array}{c} l_b^{-1}l_a^{-1} \end{array}\right)} && \bigoplus_2\,\zpx \ar[r]^-{\left(\begin{array}{c} l_b^{-1}l_a^{-1}-1 \\ 1-g_1^{-1} \end{array}\right)} \ar[d]_{\left(\begin{array}{cc} 0 & g_1l_b^{-1}l_a^{-1} \\ -1 & 0 \end{array}\right)} & \zpx \ar[d]_{\left(\begin{array}{c} g_1 \end{array}\right)}\\
\zpx \ar[rr]^{\left(\begin{array}{cc} l_al_b-1 & 1-g_1 \end{array}\right)} && \bigoplus_2\,\zpx \ar[r]^-{\left(\begin{array}{c} g_1-1 \\ l_al_b-1 \end{array}\right)} & \zpx.
}\]
\end{proposition}
\begin{proof}
This is just an application of Definition \ref{Defn:unionconstruction} and a calculation.  Note that we could also calculate the higher chain homotopies $\phi_i$ for $i=1,2$ using the union construction but we do not need them explicitly here so we do not do so.
\end{proof}

We are now in the position that we have symmetric Poincar\'{e} pairs $$i_{\pm} \colon C_*(S^1 \times S^0;\zpx) \to C_*(S^1 \times D^1_{\pm};\zpx)$$ along with chain maps $$f_{\pm} \colon C_*(S^1 \times D^1_{\pm};\zpx) \to C_*(X;\zpx)$$ and our chain homotopy $$g \colon C_*(S^1 \times S^0;\zpx) \to C_{*+1}(X;\zpx).$$  In order to finish the construction of our symmetric Poincar\'{e} triad we need a symmetric structure on $C_*(X;\zpx)$ which is compatible with the structure on $\partial X$.

Recall that the 3-handles of our decomposition of $X$ are $h^3_o$ and $h^3_{\partial}$ corresponding to the identities of the presentation of $\pi_1(X)$:
\[s_o = \prod_{k=1}^c \, w_{j_k}r_{j_k}w_{j_k}^{-1} = 1 \in F(g_1,..g_c,\mu,\lambda);\]
\begin{multline*} s_{\partial} = (r_{\partial}^{-1})(\lambda r_{\mu}^{-1} \lambda^{-1})(r_{\lambda}^{-1}) \left(\prod_{j=0}^{c-1}\,u_{k+j}r^{-1}_{k+j}u_{k+j}^{-1}\right) (r_{\mu})(\mu r_{\lambda}\mu^{-1})\\
 = 1 \in F(g_1,..,g_c,\mu,\lambda).\end{multline*}
As in Remark \ref{Rmk:relativefundclass}, the chain
\[[X,\partial X] := h^3_o + h^3_{\partial} \in C_3(X;\Z) = \Z \otimes_{\Z[\pi_1(X)]} C_3(X;\Z[\pi_1(X)])\]
satisfies \[\partial_3([X,\partial X]) = -h^2_{\partial} = (-1)^3f([\partial X]) \in C_2(X;\Z),\]
where $f$ is the inclusion $E' \hookrightarrow C_*(X;\zpx)$.  $[X,\partial X]$ is the \emph{relative fundamental class} for the knot exterior, which we shall use to derive the symmetric structure on the chain complex.

\begin{proposition}\label{prop:symmstructureonX}
We denote the symmetric structure on $C_*(X;\zpx)$ by $\Phi$.  In particular, we use Trotter's formulae for $\Delta_0$ from Theorem \ref{trotterdelta0} to explicitly define
\[\Phi_0 := \backslash(\Delta_0([X,\partial X]) - \chi )= \backslash(\Delta_0(h^3_o+\hp^3) -\chi) \]
where
\begin{multline*}\chi:= h^1_{\lambda} \otimes h^2_{\partial} + h^2_{\partial} \otimes h^1_{\mu} \in (E' \otimes_{\zpx} E')_3 \\ \subset (C_*(X;\zpx) \otimes_{\zpx} C_*(X;\zpx))_3.\end{multline*}
The higher chain homotopies $\Phi_i$ for $i=1,2,3$ also exist but we do not give explicit formulae.
The composition of $\eta$ with the split monomorphism $$f \colon E' \to C_*(X;\zpx)$$ then yields a symmetric Poincar\'{e} pair
\[(f\circ \eta \colon E = \mathscr{C}((i_-,i_+)^T) \to C_*(X;\zpx), (\Phi,\phi)).\]
\end{proposition}
\begin{proof}
We check that the equations for a symmetric pair are satisfied.  We have:
\begin{eqnarray*} & & d(\Delta_0([X,\partial X]) - \chi) \\ &=& \Delta_0(\partial([X,\partial X])) - d\chi = \Delta_0(-f([\partial X])) - d\chi = \Delta_0(f(-h^2_{\partial})) - d\chi \\ &=& \Delta_0(-h^2_{\partial}) - d\chi = -(\Delta_0(\hp^2) + d\chi) \\ &=&  -(\hp^2 \otimes \hp^0 + \hp^0 \otimes \hp^2 - \gamma(\lambda \mu \lambda^{-1} \mu^{-1}) + d\chi)\\  &=& -(\hp^2 \otimes \hp^0 + \hp^0 \otimes \hp^2 - (h^1_{\la}\otimes \la h^1_{\mu} + h^1_{\la} \otimes h^1_{\la} + h^1_{\mu} \otimes h^1_{\mu} + h^1_{\la} \otimes \mu^{-1}h^1_{\mu}\\ & & - h^1_{\la} \otimes \mu h^1_{\la} - h^1_{\la} \otimes h^1_{\mu} -h^1_{\mu} \otimes \la^{-1} \mu h^1_{\la} - h^1_{\mu} \otimes \la^{-1} h^1_{\mu}) + d\chi) \\ & = & -(\hp^2 \otimes \hp^0 + \hp^0 \otimes \hp^2 - h^1_{\la}\otimes \la h^1_{\mu} - h^1_{\la} \otimes h^1_{\la} - h^1_{\mu} \otimes h^1_{\mu} - h^1_{\la} \otimes \mu^{-1}h^1_{\mu}\\ & & + h^1_{\la} \otimes \mu h^1_{\la} +  h^1_{\la} \otimes h^1_{\mu} +h^1_{\mu} \otimes \la^{-1} \mu h^1_{\la} + h^1_{\mu} \otimes \la^{-1} h^1_{\mu} + d\chi) \\ & = & -(\hp^2 \otimes \hp^0 + \hp^0 \otimes \hp^2 - h^1_{\la}\otimes \la h^1_{\mu} - h^1_{\la} \otimes h^1_{\la} - h^1_{\mu} \otimes h^1_{\mu} - h^1_{\la} \otimes \mu^{-1}h^1_{\mu}\\ & & + h^1_{\la} \otimes \mu h^1_{\la} +  h^1_{\la} \otimes h^1_{\mu} +h^1_{\mu} \otimes \la^{-1} \mu h^1_{\la} + h^1_{\mu} \otimes \la^{-1} h^1_{\mu} \\ & & + h^1_{\la} \otimes (\la-1)h^1_{\mu} + h^1_{\la} \otimes (1-\mu)h^1_{\la} + \hp^0 \otimes (\la^{-1}-1)\hp^2 \\ & & + \hp^2 \otimes (\mu-1)\hp^0 + h^1_{\mu} \otimes (1-\la^{-1})h^1_{\mu} + h^1_{\la} \otimes (\mu^{-1}-1)h^1_{\mu}) \\ &=& -(\hp^2 \otimes \mu\hp^0 + \hp^0 \otimes \la^{-1}\hp^2 + h^1_{\mu} \otimes \la^{-1}\mu h^1_{\la} -h^1_{\la} \otimes h^1_{\mu}) \\ &=& -(\hp^2 \otimes \mu\hp^0 + \hp^0 \otimes \la^{-1}\hp^2 + h^1_{\mu} \otimes \mu\la^{-1} h^1_{\la} -h^1_{\la} \otimes h^1_{\mu}) \\ &=& -(\hp^2 \otimes g_1\hp^0 + \hp^0 \otimes l_b^{-1}l_a^{-1}\hp^2 + h^1_{\mu} \otimes g_1 l_b^{-1}l_a^{-1} h^1_{\la} -h^1_{\la} \otimes h^1_{\mu}).
\end{eqnarray*}
Then on the one hand \[\backslash d(\Delta_0([X,\partial X]) - \chi) = d_{\Hom} \backslash (\Delta_0([X,\partial X]) - \chi) = d_{\Hom}\Phi_0 = \partial_X\Phi_0 + (-1)^r\Phi_0\delta_X,\]
while on the other hand, by comparing the image under the slant map of the result of the calculation above with the symmetric structure maps $\phi'_0$ on $E'$ from Proposition \ref{prop:symmstructuretorus}, we see that $\backslash d(\Delta_0([X,\partial X]) - \chi) = (-1)^3 f\phi'_0 f^* = -\phi'_0$ (since $f$ is just the inclusion).  The equations for a symmetric pair are therefore satisfied.

To see that the pair $(f\circ \eta \colon E \to C_*(X;\zpx), (\Phi,\phi))$ is Poincar\'{e}, we need to check that the maps $(\Phi_0,\phi_0)$ induce isomorphisms $H^{3-r}(C_*(X;\zpx),E) \xrightarrow{\simeq} H_r(C_*(X;\zpx))$.  First, since $X$ is a $K(\pi,1)$, its universal cover is contractible so that \[H_0(C_*(X;\zpx)) \cong \Z\] and \[H_r(C_*(X;\zpx)) \cong 0\] for $r \neq 0$.  Since $C_*(X;\zpx)$ is the handle chain complex of a manifold, we know that Poincar\'{e}-Lefschetz duality holds, and that abstractly $H^{3-r}(X,\partial X) \cong H_r(X)$.  We could prove this for example using the isomorphism of singular homology to handle homology, or more elegantly by turning the handles upside down i.e. we can consider a $r$-handle $D^r \times D^{3-r}$ as a \emph{relative} $(3-r)$-handle $D^{3-r} \times D^r$ with the handle chain complex boundary maps becoming the relative cochain complex coboundary maps. The homology $H_0(C_*(X;\zpx)$ is generated by $ 1 \in C_0(X;\zpx) \cong \zpx$.  The mapping cone $\mathscr{C}(f\circ \eta \colon E \to C_*(X;\zpx))$ is given, with $\zpx$ coefficients used for the chain groups of $X$, by:
\[E_2\oplus C_3(X) \to E_1 \oplus C_2(X) \to E_0 \oplus C_1(X) \to C_0(X), \]
so that the cochain complex is given by:
\[C^0(X) \to E^0 \oplus C^1(X) \to E^1 \oplus C^2(X) \to E^2 \oplus C^3(X).\]
Since $f$ is the inclusion of sub-complex, the $E^2$ term lies in the image of $f^*$, so does not generate cohomology, and in fact $(h^3_{\partial})^* = (0,1) \in C^3(X;\zpx) \cong \bigoplus_2\,\zpx$ (or $(1,0)$) generates $H^3(C_*(X;\zpx),E) \cong \Z$.  The element $(0,1)$ is sent to $1 \in C_0(X;\zpx)$ by $\Phi_0$, since $h^3_{\partial} \otimes \hp^0$ lies in the image of Trotter's $\Delta_0$ map applied to $(h^3_o + \hp^3)$.  This completes the check that $(\Phi_0,\phi_0)$ induces isomorphisms $H^{3-r}(C_*(X;\zpx),E) \xrightarrow{\simeq} H_r(C_*(X;\zpx))$ for $r=0,1,2,3$.  The higher chain homotopies $\Phi_i$ are then guaranteed to exist by Theorem \ref{Thm:davisdiag}, so that we indeed have a symmetric Poincar\'{e} pair as claimed.
\end{proof}

This completes our description of the \emph{fundamental symmetric Poincar\'{e} triad of a knot}:
\[\xymatrix{\ar @{} [dr] |{\stackrel{g}{\sim}}
(C = C_*(S^1 \times S^0;\Z[\pi_1(X)]),\varphi\oplus -\varphi) \ar[r]^-{i_-} \ar[d]_{i_+} & (D_- = C_*(S^1 \times D^1_-;\Z[\pi_1(X)]),0) \ar[d]^{f_-}\\ (D_+ = C_*(S^1 \times D^1_+;\Z[\pi_1(X)]),0) \ar[r]^-{f_+} & (C_*(X;\Z[\pi_1(X)]),\Phi),
}\]
as has been the goal of Chapters \ref{Chapter:handledecomp}, \ref{chapter:chaincomplex} and \ref{Chapter:duality_symm_structures}.
\begin{remark}
If, as in Remark \ref{rmk:zerosurgery1}, we form the chain complex of the zero surgery $M_K$ by adding two handles to our original handle decomposition for $X$ in Theorem \ref{Thm:unicoverchaincomplex}, a 2-handle $h^2_s$ along the longitude using the word $l$ and a 3-handle $h^3_s$ to fill the rest of the solid torus in using the words $u_j$, then a fundamental class is given by
\[[M_K] = [h^3_o + h^3_s] \in C_3(M_K;\Z),\]
and this can be used directly with Trotter's $\Delta_0$ formulae to obtain the symmetric structure on a closed manifold.  If we just want to extract a sliceness obstruction and do not need to add knots together then this allows a significant simplification of the formulae.
\end{remark} 

\chapter{Adding Knots and Second Derived Covers}\label{Chapter:2ndderived}

In this chapter we describe in some detail how to add together oriented knots, and translate this into addition of the corresponding fundamental cobordisms, which we will later apply in the algebraic setting to define addition of fundamental symmetric Poincar\'{e} triads of knots.  We will primarily be interested in working at the level of the second derived cover.  Since in future work we hope to extend our results beyond this we have worked up until now at the level of the universal cover, however we now begin to specialise to the covering space defined by the second derived subgroup of the knot exterior.  We make a detailed study of the behaviour of the knot groups under connected sum and under the operation which factors out the second derived subgroup.

Experts may prefer to skip to Chapter \ref{Chapter:semigroup}.

\begin{definition}\label{Defn:connectedsum}
We form the \emph{connected sum} of two oriented knots $K$ and $K^{\dag}$ as follows.  Parametrise the regular neighbourhood $N(K)$ of a knot $K$ as $\left[0,1\right] \times D^2 \to S^1 \times D^2$ so that $(0,x) \sim (1,x)$.  We require that this defines a framing of the knot which is compatible with a Seifert surface $F \subset X$, by which we mean that $\partial F = S^1 \times \{(1,0)\} \subset S^1 \times D^2$.  The canonical orientation of $[0,1]$ must agree with the orientation of the knot.  We then extend the embedding $K \colon S^1 \to S^3$ to an embedding $K \colon (([0,1] \times D^2)/\sim) \hookrightarrow S^3$.  Do this similarly for $K^{\dag}$.  Then remove from each copy of $S^3$ a solid cylinder:
\[ \ba{rcl} K([0,1/2] \times D^2) &\subset& N(K) \subset S^3\\
K^{\dag}([0,1/2] \times D^2) &\subset & N(K^{\dag}) \subset (S^3)^{\dag}, \ea\]
and form the closures:
\[\ba{rcl}Z &:=&\cl(S^3 \setminus (K([0,1/2] \times D^2)))\\
Z^{\dag}&:=&\cl(S^3 \setminus (K^{\dag}([0,1/2] \times D^2))).\ea\]
Note that since $[0,1/2] \times D^2 \approx D^3$, $Z \approx Z^{\dag} \approx D^3$ and \[\partial Z \approx \partial([0,1/2]\times D^2) \approx \{0,1/2\} \times D^2 \cup_{\{0,1/2\} \times S^1} [0,1/2]\times S^1 \approx S^2.\]  We then identify $\partial Z \approx \partial Z^{\dag}$ in order to form the union
\[S^3 \approx Z \cup_{\partial Z \approx \partial Z^{\dag}} Z^{\dag},\]
using the identification which inverts the first coordinate, so that our addition produces another oriented knot, as follows:
\[K((0,x)) \xrightarrow{\sim} K^{\dag}((1/2,x)) \text{ for } x \in D^2; \text{ and}\]
\[K((t,y)) \xrightarrow{\sim} K^{\dag}((1/2 - t,y)); \;\; t \in [0,1/2],y \in S^1 = \partial D^2.\]
The knot $K \,\sharp\, K^{\dag}$ is then given by \[K([1/2,1] \times \{0\}) \cup_{K(\{1/2,1\} \times \{0\})} K^{\dag}([1/2,1]\times \{0\}) \subseteq Z \cup_{\partial Z \approx \partial Z^{\dag}} Z^{\dag} \approx S^3.\]
\qed \end{definition}

We study the effect of connected sum on the knot groups using the Seifert-Van-Kampen Theorem.

\begin{proposition}\label{Prop:decompknot-extconnectedsum}
Denote by $X^{\ddag} := \cl(S^3 \setminus N(K \, \sharp \, K^{\dag}))$ the knot exterior for the connected sum $K^{\ddag} = K \, \sharp \, K^{\dag}$ of two oriented knots.  Let $g_1,g_1^{\dag}$ be chosen generators in the fundamental groups $\pi_1(X;x_0)$ and $\pi_1(X^{\dag};x_0^{\dag})$ respectively, generating preferred subgroups $\langle g_1 \rangle \xrightarrow{\simeq} \Z \leq \pi_1(X;x_0)$ and $\langle g_1^{\dag} \rangle \xrightarrow{\simeq} \Z \leq \pi_1(X^{\dag};x_0^{\dag})$.  Recall that the basepoints $x_0,x_0^{\dag}$ are chosen to lie in the boundaries.  $X^{\ddag}$ admits a decomposition as:
\[X^{\ddag} = X \cup_{S^1 \times \mathring{D}^1} X^{\dag}\]
where $S^1 \times \mathring{D}^1 \subseteq \partial X,\partial X^{\dag}$, are cylinders in the boundaries of $X$ and $X^{\dag}$ chosen so that if we deform $g_1$ and $g_1^{\dag}$ so that they lie in $\partial X$ and $\partial X^{\dag}$ respectively, then the images of $g_1$ and $g_1^{\dag}$ coincide with the boundary component $S^1 \times \{1\}$ of the closure of the subsets $S^1 \times \mathring{D}^1$ in  $\partial X$ and $\partial X^{\dag}$, and $x_0,x_0^{\dag}$ is $\{1\} \times \{1\}$. In addition, they should be chosen so that the orientation on $K$ agrees, whereas that on $K^{\dag}$ disagrees, with the orientation of $\mathring{D}^1$ in $S^1 \times \mathring{D}^1$.  The subset $S^1 \times \mathring{D}^1 \subset \partial X^\dag$ has the orientation of $\mathring{D}^1$ reversed from the orientation on the $S^1 \times \mathring{D}^1$ which was used in the identification of Definition \ref{Defn:connectedsum}
\end{proposition}

\begin{proof}

This is an application of the definition of connected sum to the knot exteriors.  The effect of Definition \ref{Defn:connectedsum} is to glue the two knot exteriors together along a copy of $S^1 \times D^1$ in each of the boundaries.  For definiteness we choose these to include the basepoint and certain specific generators of the fundamental groups.  With the boundary of $X$ split into two we can visualise this by considering the knot exterior to be $D^3 \setminus (\mathring{D}^2 \times D^1)$, with boundary: \[S^2 \setminus (\mathring{D}^2 \times S^0) \cup_{S^1 \times S^0} S^1 \times D^1 \approx S^1 \times D^1 \cup_{S^1 \times S^0} S^1 \times D^1 \approx S^1 \times S^1.\]
We then identify half of the boundary of each of $X$ and $X^{\dag}$ as shown in Figure \ref{Fig:addingknots2}.

\begin{figure}
\begin{center}
{\psfrag{A}{$X$}
\psfrag{Y}{$X^{\dag}$}
\includegraphics[width=10cm]{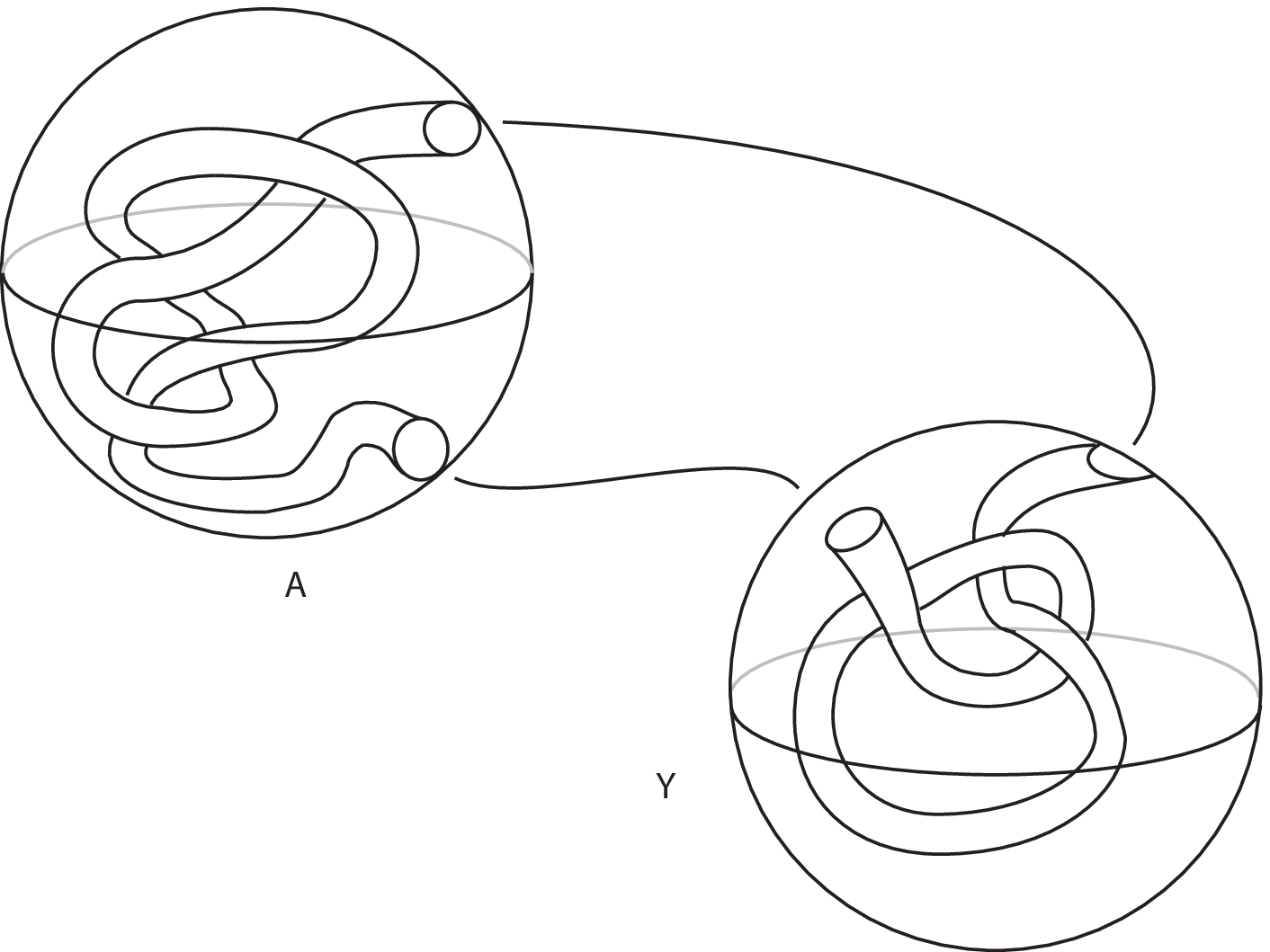}}
\caption{Gluing $X$ and $X^{\dag}$ together along part of their boundaries $S^2 \setminus (S^0 \times \mathring{D}^2)$ to create the exterior of the connected sum, $X^{\ddag}$.}
\label{Fig:addingknots2}
\end{center}
\end{figure}

\end{proof}

\begin{proposition}\label{prop:addingknotgroups}
The knot group for a connected sum $K\,\sharp\,K^{\dag}$ is given by the amalgamated free product of the knot groups of $K$ and $K^{\dag}$, with our chosen meridians identified: \[\pi_1(X^{\ddag}) = \pi_1(X) \ast_{\Z} \pi_1(X^{\dag}),\]
so that $g_1 = g_1^{\dag}$.
\end{proposition}

\begin{proof}
We make use of the decomposition of $X^{\ddag}$ from Proposition \ref{Prop:decompknot-extconnectedsum} above, and apply the Seifert-Van Kampen theorem to open subsets $U,V \subseteq X^{\ddag}$ where $U$ is an extension of $X$ to include a collar neighbourhood of $S^1 \times \mathring{D}^1$ in $X^{\dag}$: $X$ is a deformation retract of $U$. The advantage of $U$ is that it is an open set. Similarly we take $V$ to be an extension of $X^{\dag}$ into $X$.  We therefore have that
\[X^{\ddag} = U \cup_{S^1 \times \mathring{D}^1 \times \mathring{I}} V.\]
Now, $\pi_1(U) = \pi_1(X)$ and $\pi_1(V) = \pi_1(X^{\dag})$.  For the intersection we have $$U \cap V = S^1 \times \mathring{D}^1 \times \mathring{I} \simeq S^1,$$ so that $$\pi_1(U \cap V) \cong \pi_1(S^1) \cong \Z \cong \langle t \rangle.$$  Let $i_U,i_V \colon U \cap V \hookrightarrow U,V$ be the inclusion maps, with $$(i_U)_*,(i_V)_* \colon \Z \cong \pi_1(U \cap V) \rightarrow \pi_1(X),\pi_1(X^{\dag})$$ the induced maps on fundamental groups.  We have:
\[(i_U)_*(t) = g_1; \text{ and}\]
\[(i_V)_*(t) = g_1^{\dag}.\]
Note that we arranged the basepoints of $U, V$ and $U \cap V$ to coincide.  The Seifert Van-Kampen theorem then implies that
\[\pi_1(X^{\ddag}) = \frac{\pi_1(U) \ast \pi_1(V)}{((i_U)_*(t)(i_V)_*(t)^{-1})} = \pi_1(X) \ast_{\small{\Z}} \pi_1(X^{\dag})\]
as claimed.
\end{proof}

We are aiming to specialise to the case of the second derived cover of the knot exterior: we now investigate the structure of the relevant quotient of the fundamental group.

\begin{definition}\label{Defn:derived_subgroups}
The \emph{derived subgroups} $G^{(n)}$ of a group $G$, also denoted $G' = G^{(1)}$, $G'' = G^{(2)}$ et c{\ae}tera, are defined inductively via
\[G^{(0)} := G;\;\;\; G^{(n)} := [G^{(n-1)},G^{(n-1)}],\]
where as usual square brackets indicate the commutator subgroup.
\qed \end{definition}

\begin{lemma}\label{lemma:longitude}
For any knot $K$, the longitude $l$ satisfies $l \in \pi_1(X)^{(2)} = \pi_1(X)''$, the second derived subgroup of the knot group.
\end{lemma}

\begin{proof}
By definition of the zero framing, the longitude $l$ lies in a Seifert surface for the knot $K$, that is some choice of a compact, connected, orientable surface $F^2 \subseteq S^3$ such that $\partial F = K$. The longitude is isotopic to $K$ in $S^3$ via an isotopy which moves across an annulus in the regular neighbourhood of the knot.  A Seifert surface is homeomorphic to a surface of genus $g$, for some $g$, with a single $S^1$ boundary component.  Therefore $F$ is homotopy equivalent to a wedge of circles $\bigvee_{2g}\,S^1,$ so $\pi_1(F) \cong \underset{\small{2g}}{\ast}\,\Z$, generated by curves $a_1,\dots,a_g$ and their geometric duals $b_1,\dots,b_{2g}$.  The longitude, the element of $\pi_1(F)$ which represents the boundary of $F \cap X$, is the product of commutators
\[l = [a_1,b_1][a_2,b_2]\dots[a_g,b_g].\]
Now, \[H_1(X;\Z) \cong \frac{\pi_1(X)}{\pi_1(X)'} \cong \Z\]
generated by a meridian of the knot.  The isomorphism $H_1(X;\Z) \xrightarrow{\simeq} \Z$ is given by the linking number with $K$, which can be calculated by counting, with signs, transverse intersections of the cycle with a Seifert surface.  An element $x \in \pi_1(X)$ represents the zero cycle in $H_1(X;\Z)$ if and only if it has linking number zero with the knot, but also if and only if it lies in $\pi_1(X)'$, by the Hurewicz theorem.  The commutator subgroup therefore consists of those loops which do not link with the knot.  The image of the generators of $\pi_1(F)$ under the inclusion induced map $i_* \colon \pi_1(F) \to \pi_1(X)$ have linking number zero with the knot, since they can be pushed off the Seifert surface along the trivial normal bundle of $F \subseteq X$ so as not to intersect it at all.  This means that $l \in \pi_1(X)''$ as claimed.
\end{proof}

Complementary to this lemma, note that if the Alexander polynomial of $K$ is not equal to 1, then the longitude does not lie in $\pi_1(X)'''$, as shown in \cite[Proposition~12.5]{Cochran}.  Since we already know by work of Freedman \cite[Theorem~11.7B]{FQ}, that knots with Alexander polynomial one are slice, in all interesting cases we have $l \neq 1 \in \pi_1(X)^{(2)}/\pi_1(X)^{(3)}$.

\begin{proposition}\label{prop:2ndderivedsubgroup}
Let $\phi$ be the quotient map
\[\phi \colon \frac{\pi_1(X)}{\pi_1(X)^{(2)}} \to \frac{\pi_1(X)}{\pi_1(X)^{(1)}} \xrightarrow{\simeq} \Z.\]
Then for each choice of homomorphism
\[\psi \colon \Z \to \frac{\pi_1(X)}{\pi_1(X)^{(2)}}\]
such that $\phi \circ \psi = \Id$, there is an isomorphism:
\[ \theta \colon \frac{\pi_1(X)}{\pi_1(X)^{(2)}} \xrightarrow{\simeq} \Z \ltimes H,\]
where $H := H_1(X;\Z[\Z])$ is the Alexander module.  In the notation of Proposition \ref{Prop:decompknot-extconnectedsum}, and denoting $H^{\dag} := H_1(X^{\dag};\Z[\Z])$ and $H^{\ddag} := H_1(X^{\ddag};\Z[\Z])$, the behaviour of the second derived quotients under connected sum is given by:
\[\frac{\pi_1(X^{\ddag})}{\pi_1(X^{\ddag})^{(2)}} \cong \Z \ltimes H^{\ddag} \cong \Z \ltimes (H \oplus H^{\dag}).\]
That is, we can take the direct sum of the Alexander modules.
\end{proposition}

\begin{proof}
The first statement follows from the exact sequence of groups:
\[\frac{\pi_1(X)'}{\pi_1(X)''} \rightarrowtail \frac{\pi_1(X)}{\pi_1(X)''} \twoheadrightarrow \frac{\pi_1(X)}{\pi_1(X)'}.\]
Note that $\pi_1(X_{\infty}) = \pi_1(X)'$ since loops which link the knot lift to paths in $X_{\infty}$ with end points in different sheets of the covering.  Therefore $\pi_1(X)'/\pi_1(X)'' \cong \pi_1(X_{\infty})/\pi_1(X_{\infty})' \cong H_1(X_{\infty};\Z) \cong H_1(X;\Z[\Z]) = H$.  After making a choice of isomorphism $\pi_1(X)/\pi_1(X)' \xrightarrow{\simeq} \Z$, which we make using the orientation rule and our Conventions \ref{conventions}, so that the isomorphism is given by linking number, the above short exact sequence therefore yields
\[H_1(X;\Z[\Z]) \rightarrowtail \frac{\pi_1(X)}{\pi_1(X)^{''}} \twoheadrightarrow \Z.\]
Since $\Z$ is a free group, the sequence splits.  We call by $\phi \colon \pi_1(X)/\pi_1(X)'' \to \Z$ the surjection from this sequence, and let the splitting be given by a map $\psi \colon \Z \to \pi_1(X)/\pi_1(X)''$ such that \[t := \psi(1),\] $\phi \circ \psi = \text{Id}$.  We therefore have a map \[\theta \colon \frac{\pi_1(X)}{\pi_1(X)^{''}} \to \Z \ltimes H\] given by \[g \mapsto (\phi(g), gt^{-\phi(g)}).\]  We claim that $\theta$ is an isomorphism.  Note that $\phi(gt^{-\phi(g)}) = \phi(g) - \phi(g) = 0$, so by exactness we have an element of $\pi_1(X)'/\pi_1(X)'' = H$.  We therefore have that $\theta$ is an injection.  To see that it is a surjection note that \[(n,h) = (n,(ht^n)t^{-n}) = \theta(ht^n)\]
for $n \in \Z, h \in H \leq \pi_1(X)/\pi_1(X)''$.  We also check that $\theta$ is a group homomorphism. Suppose $\theta(g_1) = (n,g_1t^{-n})$ and $\theta(g_2) = (m,g_2t^{-m})$.
\[\theta(g_1g_2) = (\phi(g_1g_2),g_1g_2t^{-\phi(g_1g_2)}) = (n+m,g_1g_2t^{-n-m}).\]
The action of $n \in \pi_1(X)/\pi_1(X)' \cong \Z$ on $h \in \pi_1(X)'/\pi_1(X)''$ which occurs in the semi-direct product is by conjugation:
\[(n,h) \mapsto t^{n}ht^{-n}.\]
Therefore:
\[(n,g)\cdot(m,h) = (n+m,gt^{n}ht^{-n}) \in \Z \ltimes \frac{\pi_1(X)'}{\pi_1(X)''}.\]
which yields:
\begin{eqnarray*}\theta(g_1)\theta(g_2) = (n,g_1t^{-n})\cdot(m,g_2t^{-m}) = (n+m,g_1t^{-n}(t^n(g_2t^{-m})t^{-n})) &=& \\ (n+m,g_1g_2t^{-n-m}) &=& \theta(g_1g_2).\end{eqnarray*}
This action by conjugation corresponds to translating in the infinite cyclic cover, so if $\Z \xrightarrow{\simeq} \langle t\rangle$, then the action of $\Z$ on $H = H_1(X;\Z[t,t^{-1}])$ is by (left) multiplication by $t$.  The group element $t$ corresponds to our chosen meridian for a knot as in Proposition \ref{Prop:decompknot-extconnectedsum}.  Considering $\Z = \langle t \rangle$ and $H$ as a left $\Z[t,t^{-1}]$-module, we can write the multiplication of $\Z \ltimes H$ as:
\[(t^n,g)\cdot(t^m,h) = (t^{n+m},g + t^nh).\]

To see how the second derived quotients of the fundamental groups add under addition of knots we take the quotient of the conclusion of Proposition \ref{prop:addingknotgroups} by the second derived subgroups.  Note that $H$, $H^{\dag}$ and $H^{\ddag}$ are modules over the group ring $\Z[t,t^{-1}]$ for the same $t$, which comes from the preferred meridian of each of $X,X^{\dag}$ and $X^{\ddag}$ respectively; when the spaces are identified these meridians all coincide.
\begin{eqnarray*} \Z \ltimes H^{\ddag} & \cong & \frac{\pi_1(X^{\ddag})}{\pi_1(X^{\ddag})^{''}}\\
& \cong & \frac{\pi_1(X) \ast_{\Z} \pi_1(X^{\dag})}{(\pi_1(X) \ast_{\Z} \pi_1(X^{\dag}))''}\\
& \cong & \left(\frac{\pi_1(X)}{\pi_1(X)''} \ast_{\Z} \frac{\pi_1(X^{\dag})}{\pi_1(X^{\dag})''}\right)\text{\huge{/}}[\pi_1(X)',\pi_1(X^{\dag})']\\
& \cong & \frac{(\Z \ltimes H)\ast_{\Z}(\Z \ltimes H^{\dag})}{[\pi_1(X)',\pi_1(X^{\dag})']}.\\
\end{eqnarray*}
We now need to be sure that the two group elements which we identify, $g_1$ and $g_1^{\dag}$, map to $(1,0) \in \Z \ltimes H$ and $(1,0^{\dag}) \in \Z \ltimes H^{\dag}$ respectively under the compositions
\[\pi_1(X) \to \frac{\px}{\px^{(2)}} \to \Z \ltimes H\]
and
\[\pi_1(X^{\dag}) \to \frac{\pi_1(X^{\dag})}{\pi_1(X^{\dag})^{(2)}} \to \Z \ltimes H^{\dag}.\]
If we had chosen \[\psi(1) = g_1 \in \frac{\pi_1(X)}{\pi_1(X)''}\] and \[\psi^{\dag}(1) = g_1^{\dag} \in \frac{\pi_1(X^{\dag})}{\pi_1(X^{\dag})''}\] then this would be the case and we would have:
\begin{eqnarray*}
& & \frac{(\Z \ltimes H)\ast_{\Z}(\Z \ltimes H^{\dag})}{[\pi_1(X)',\pi_1(X^{\dag})']}\\
& \cong & \frac{\Z \ltimes (H \ast H^{\dag})}{[H,H^{\dag}]}\\
& \cong & \Z \ltimes (H \oplus H^{\dag}),
\end{eqnarray*}
and the proof would be complete.  The point is that we can always arrange that the image of $g_1$ is $(1,0)$ by applying an inner automorphism of $\Z \ltimes H$, and similarly for $g_1^\dag$ and $\Z \ltimes H^\dag$.  Recall (\cite[Proposition~1.2]{Levine2}) that $1-t$ acts as an automorphism of $H$.  We can therefore choose $h'_1 \in H$ such that $(1-t)h'_1 = h_1$.  Then we have that:
\begin{eqnarray*}
(0,h'_1)^{-1}(1,h_1)(0,h'_1) &=& (0,-h'_1)(1,h_1)(0,h'_1) \\
&=& (1,-h'_1 + h_1)(0,h'_1)\\
&=& (1,-h'_1 + h_1 + th'_1)\\
&=& (1,h_1 - (1-t)h'_1)\\
&=& (1,h_1-h_1) = (1,0).
\end{eqnarray*}
So, as claimed, we can compose the splittings $\psi$ and $\psi^\dag$ with suitable inner automorphisms and so achieve the desired conditions on the meridians which we identify.  Therefore the second derived quotients of the fundamental groups indeed add under connected sum as claimed.
\end{proof}

\begin{remark}
The fact that we can use an \emph{inner} automorphism in the proof of Proposition \ref{prop:2ndderivedsubgroup} will be useful in the next chapter when we wish to show that our fundamental symmetric Poincar\'{e} triad does not depend on choices, in particular the choice of group elements $g_1$ and $l_a$: all possible choices are related by a suitable conjugation.
\end{remark} 

\chapter{A Monoid of Chain Complexes}\label{Chapter:semigroup}

We are now ready to define a set of purely algebraic objects which capture the necessary information to produce concordance obstructions at the metabelian level.  We define a set comprising 3-dimensional symmetric Poincar\'{e} triads over the group ring $\zh$ for certain $\Z[\Z]$-modules $H$.  We have shown explicitly how to construct these objects, starting with a diagram of a knot.  In some sense, we are to forget that these chain complexes originally arose from geometry, and to perform operations on them purely with reference to the algebraic data which we store with each element.  The primary operation which we introduce in this chapter is a way to add these chain complexes, so that we obtain an abelian monoid.  On the other hand, we would not do well pedagogically to forget the geometry.  The great merit of the addition operation we put forwards here is that it closely mirrors geometric addition of knots by connected sum.

Our elements, representing the knot exteriors, are essentially algebraic $\Z$--homology cobordisms from the chain complex of the cylinder $S^1 \times D^1$ to itself, all over the group rings $\zh$ over the semi--direct products which arise as the quotients of knot groups by their second derived subgroups, with $H$ an Alexander module (Theorem \ref{Thm:Levinemodule}).  The crucial extra condition is a consistency condition, which relates $H$ to the actual homology of the chain complex.  Since the Alexander module changes under addition of knots and in a concordance, this extra control is vital in order to formulate a concordance obstruction theory.

\begin{remark}\label{rmk:keepnotationunicover}
Even though we work at the level of the second derived quotient, we maintain the notation for group ring elements of the universal cover.  While this introduces a certain amount of redundancy at this stage, we consider this a small price since it leaves the way open for generalisation further up the derived series in future work.
\end{remark}

We quote the following theorem of Levine, specialised here to the case of knots in dimension 3, and use it to define the notion of an abstract Alexander module.

\begin{theorem}[\cite{Levine2}]\label{Thm:Levinemodule}
Let $K$ be a knot, and let  $H:= H_1(M_K;\Z[\Z]) \cong H_1(X_{\infty};\Z)$  be its Alexander module.  Take $\Z[\Z] = \Z[t,t^{-1}]$.  Then $H$ satisfies the following properties:
\begin{description}
\item[(a)] The Alexander module $H$ is of type $K$: that is, $H$ is finitely generated over $\Z[\Z]$, and multiplication by $1-t$ is a module automorphism of $H$.  These two properties imply that $H$ is $\Z[\Z]$-torsion.
\item[(b)] The Alexander module $H$ is $\Z$-torsion free.  Equivalently, for $\Z[\Z]$-modules of type $K$, the homological dimension\footnote{This is defined as the minimal possible length of a projective resolution.} of $H$ is 1.
\item[(c)] The Alexander module $H$ satisfies Blanchfield Duality: $$\overline{H} \cong \Ext^1_{\Z[\Z]}(H,\Z[\Z]) \cong \Ext^0_{\Z[\Z]}(H,\Q(\Z)/\Z[\Z]) \cong \Hom_{\Z[\Z]}(H,\Q(\Z)/\Z[\Z])$$ where $\overline{H}$ is the conjugate module defined by using the involution defined by $t \mapsto t^{-1}$ to make $H$ into a right module.
\end{description}
Conversely, given a $\Z[\Z]$-module $H$ which satisfies properties (a), (b) and (c), there exists a knot $K$ such that $H_1(X;\Z[\Z]) \cong H$.

We say that a $\Z[\Z]$-module which satisfies (a),(b) and (c) is an Alexander module, and denote the class of Alexander modules by $\mathcal{A}$.
\end{theorem}
\begin{remark}
Note that if $\Z[\Z]$ were a principal ideal domain (PID), then (b) would be immediate since any module over a PID has homological dimension 1.  The ideal $(2,1+t) \lhd \Z[\Z]$ is not principal and does not contain 1, so $\Z[\Z]$ is not a PID.  Nevertheless Levine shows that classical Alexander modules are $\Z$-torsion free.  Levine \cite{Levine2} defines Blanchfield duality; see also \cite{Kearton2} or \cite{Friedl} for excellent accounts of the Blanchfield pairing; the original reference is \cite{Blanchfield}.  We will define Blanchfield pairings algebraically in Proposition \ref{Prop:chainlevelBlanchfield}.
\end{remark}
We now give the definition of our set of symmetric Poincar\'{e} triads.
\begin{definition}\label{Defn:algebraicsetofchaincomplexes}
We define the set $\mathcal{P}$ to be the set of equivalence classes of triples $(H,\Y,\xi)$ where: $H \in \mathcal{A}$ is an Alexander module; $\Y$ is a 3-dimensional symmetric Poincar\'{e} triad of finitely generated projective $\zh$-module chain complexes of the form:
\[\xymatrix @R+0.5cm @C+0.5cm {\ar @{} [dr] |{\stackrel{g}{\sim}}
(C,\varphi_C) \ar[r]^{i_-} \ar[d]_{i_+} & (D_-,\delta\varphi_-) \ar[d]^{f_-}\\ (D_+,\delta\varphi_+) \ar[r]^{f_+} & (Y,\Phi),
}\]
with chain maps $i_{\pm}$, chain maps $f_{\pm}$ which induce $\Z$-homology equivalences, and a chain homotopy $g \colon f_- \circ i_- \sim f_+ \circ i_+ \colon C_* \to Y_{*+1}$; and $$\xi \colon H \toiso H_1(\Z[\Z] \otimes_{\zh} Y)$$ is a $\Z[\Z]$-module isomorphism.  We give model chain complexes and symmetric structures for $C,D_{\pm}$ and for the chain maps $i_{\pm}$, which define the chain equivalence classes of these complexes.  To exhibit representatives for these chain equivalence classes, we denote by $g_1 = (1,h_1) \in \Z \ltimes H$ a specified element, and require two elements\footnote{We maintain the superfluous notation $l_b$ for $l_a^{-1}$ in order to keep the notation of the universal cover as promised in Remark \ref{rmk:keepnotationunicover}.} $l_a$ and $l_b$ of $\Z \ltimes H$ such that $l_al_b=1 \in \Z \ltimes H$.  We denote by
\[g_q := l_a^{-1}g_1l_a,\]
which implies that also:
\[g_q = l_bg_1l_b^{-1}.\]
 A model for $(C,\varphi_C = \varphi \oplus -\varphi)$:
\[\xymatrix @C+2cm @R+1cm{C^0 \ar[r]^{\delta_1} \ar[d]_{\varphi_0 \oplus -\varphi_0} & C^1 \ar[d]^{\varphi_0 \oplus -\varphi_0} \ar[dl]^{\varphi_1 \oplus -\varphi_1}\\
C_1 \ar[r]^{\partial_1} & C_0}\]
is given by:
\[\xymatrix @C+3cm @R+3cm{\bigoplus_2\,\zh \ar[r]^{\left(\begin{array}{cc} g_1^{-1}-1 & 0 \\ 0 & g_q^{-1}-1 \\ \end{array} \right)} \ar[d]_{\left(\begin{array}{cc} 1 & 0 \\ 0 & -1 \\ \end{array} \right)} & \bigoplus_2\,\zh \ar[d]^{\left(\begin{array}{cc} g_1 & 0 \\ 0 & -g_q \\ \end{array} \right)} \ar[dl]_{\left(\begin{array}{cc} 1 & 0 \\ 0 & -1 \\ \end{array} \right)}\\
\bigoplus_2\,\zh \ar[r]^{\left(\begin{array}{cc} g_1-1 & 0 \\ 0 & g_q-1 \\ \end{array} \right)} & \bigoplus_2\,\zh.}\]
The corresponding models for $(D_-,0)$ and $(D_+,0)$, with the chain maps $i_-$ and $i_+$, are given by:
\[\xymatrix @R+1cm @C-0.4cm{D_- & \zh \ar[rrrrr]^{\left(\begin{array}{c}g_1-1\end{array}\right)} &&&&& \zh \\
C \ar[u]^{i_-} \ar[d]_{i_+} & \bigoplus_2\,\zh \ar[rrrrr]_{\left(\begin{array}{cc}g_1-1 & 0 \\ 0 & g_q-1 \end{array}\right)} \ar[u]^{\left(\begin{array}{c} 1 \\ l_a^{-1} \end{array}\right)} \ar[d]_{\left(\begin{array}{c} l_b^{-1} \\ 1 \end{array}\right)} &&&&& \bigoplus_2\,\zh \ar[u]_{\left(\begin{array}{c} 1 \\ l_a^{-1} \end{array}\right)} \ar[d]^{\left(\begin{array}{c} l_b^{-1} \\ 1 \end{array}\right)} \\
D_+ & \zh \ar[rrrrr]_{\left(\begin{array}{c}g_q-1\end{array}\right)} &&&&& \zh,
}\]
The chain complexes $D_{\pm}$ arise by taking the tensor products:
\[\Z[\Z \ltimes H] \otimes_{\Z[\Z]} C_*(S^1;\Z[\Z]),\]
with homomorphisms $\Z[\Z] \to \zh$ given by:
\[t \mapsto g_1\]
for $D_-$ and
\[t \mapsto g_q\]
for $D_+$.  There is therefore a canonical chain isomorphism $\varpi \colon D_- \to D_+$,
\[\xymatrix @R+1cm @C+1cm{ (D_-)_1 \ar[r]^{\partial_1} \ar[d]_{\varpi} & (D_-)_0 \ar[d]^{\varpi} \\
(D_+)_1 \ar[r]^{\partial_1} & (D_+)_0,}\]
given by:
\[\xymatrix @R+1cm @C+1cm {\zh \ar[r]^{(g_1-1)}  \ar[d]^{(l_a)} & \zh \ar[d]^{(l_a)}  \\ \zh \ar[r]^{(g_q - 1)} & \zh.}\]
We require that the maps $\delta\varphi_{\pm}$ have the property that $\varpi \delta\varphi_- \varpi^* = - \delta\varphi_+$, and that there is a chain homotopy
\[\mu \colon f_+ \circ \varpi \simeq f_-.\]
This implies that objects of our set are independent of the choice of $f_-$ and $f_+$.

We take the symmetric structure on our models for $D_{\pm}$ to be zero; $\delta\varphi_{\pm} = 0$, so we do not show this in a diagram.  For the models we therefore have that $\varpi \delta\varphi_- \varpi^* = - \delta\varphi_+$.  We recapitulate the definition of a symmetric Poincar\'{e} triad (Definition \ref{Defn:symmPoincaretriad}).  It means that:
\[i_{\pm} \colon C \to D_{\pm}, (\delta\varphi_{\pm},\pm\varphi_C)\]
are symmetric Poincar\'{e} pairs, and that we have a symmetric Poincar\'{e} pair \[(\eta \colon E := \mathscr{C}((i_-,i_+)^T \colon C \to D_- \oplus D_+) \to Y,(\delta\varphi_- \cup_{\varphi_C} \delta\varphi_+,\Phi))\]  with $\eta$ defined by the chain map:
\[\xymatrix @R+1cm @C+0.5cm { &  C_1 \ar[rr]^-{(-i_-,\partial_C,-i_+)^T} \ar[d]^{-g} && (D_-)_1 \oplus C_0 \oplus (D_+)_1 \ar[r]^-{(\partial_{E})_1} \ar[d]^{(f_-,g,-f_+)} & (D_-)_0 \oplus (D_+)_0 \ar[d]_{(f_-,-f_+)} \\ Y_3 \ar[r]^{\partial_{Y}} & Y_2 \ar[rr]^{\partial_{Y}} && Y_1 \ar[r]^{\partial_{Y}} & Y_0,
}\]
with
\[(\partial_{E})_1 = \left(\begin{array}{ccc} \partial_{D_-} & i_- & 0 \\ 0 & i_+ & \partial_{D_+}\end{array} \right).\]
The maps $f_{\pm}$ must induce $\Z$-homology isomorphisms; note that $H_*(\Z \otimes_{\zh} D_{\pm}) \cong H_*(S^1;\Z)$:
\[(f_{\pm})_* \colon H_*(\Z \otimes_{\zh} D_{\pm}) \xrightarrow{\simeq} H_*(\Z \otimes_{\zh} Y).\]
We call the condition that the isomorphism \[\xi \colon H \xrightarrow{\simeq} H_1(\Z[\Z] \otimes_{\zh} Y)\]
exists, the \emph{consistency condition}, and we call $\xi$ the \emph{consistency isomorphism}.

We say that two triples $(H,\Y,\xi)$ and $(H^\%,\Y^\%,\xi^\%)$ are equivalent if there exists a $\Z[\Z]$-module isomorphism $\omega \colon H \toiso H^\%$, which induces a ring isomorphism $\zh \toiso \zhpc$, and if there exists a chain equivalence of triads $$j \colon \Z[\Z \ltimes H^\%] \otimes_{\zh} \Y \to \Y^\%,$$ such that the following diagram commutes:
\[\xymatrix @R+1cm @C+1cm{H \ar[r]_-{\xi}^-{\cong} \ar[d]_{\omega}^-{\cong} & H_1(\Z[\Z] \otimes_{\zh} Y) \ar[d]_-{j_*}^-{\cong} \\
H^\% \ar[r]_-{\xi^\%}^-{\cong} & H_1(\Z[\Z] \otimes_{\zhpc} Y^\%). }\]
The induced map $j_*$ on $\Z[\Z]$-homology makes sense, as there is an isomorphism
\[\Z[\Z] \cong \Z[\Z] \otimes_{\zhpc} \zhpc,\]
so that $$H_1(\Z[\Z] \otimes_{\zh} Y) \toiso H_1(\Z[\Z] \otimes_{\zhpc} \zhpc \otimes_{\zh} Y).$$
This is an equivalence relation: symmetry is seen using the inverses of the vertical arrows and transitivity is seen by vertically composing two such squares.
\qed\end{definition}

\begin{proposition}\label{prop: fundtriaddefinesanelement}
Given a knot $K$, with the quotient $\pi_1(X)^{(1)}/\pi_1(X)^{(2)} =:H$ considered as a $\Z[\Z]$-module via the action given by conjugation with a meridian, taking the fundamental symmetric Poincar\'{e} triad $\Y$ of $K$ as constructed in Chapters \ref{Chapter:handledecomp}, \ref{chapter:chaincomplex} and \ref{Chapter:duality_symm_structures}, and the geometrically defined canonical isomorphism $$\xi \colon H \toiso H_1(X;\Z[\Z]) \cong H_1(\Z[\Z] \otimes_{\Z[\Z \ltimes H]} Y),$$ we define an element $(H,\Y,\xi) \in \mathcal{P}$.
\end{proposition}
\begin{proof}
We can define \[Y := C(X;\zh) := \zh \otimes_{\zpx} C(X;\zpx),\] using Theorem \ref{Thm:mainchaincomplex}, with $g_1,l_a$ and $l_b$ the images in $\pi_1(X)/\pi_1(X)^{(2)}$ of their original incarnations in Chapter \ref{chapter:chaincomplex}; see Proposition \ref{prop:chainmapsiplusiminus} for the definitions of $l_a$ and $l_b$.  Take $(C,\varphi_C),(D_{\pm},\delta\varphi_{\pm})$ and $i_{\pm}$ to be the models as defined in Definition \ref{Defn:algebraicsetofchaincomplexes} or indeed (for $\pi_1(X)$ coefficients) in Propositions \ref{Prop:circlechaincomplex}, \ref{prop:chainmapsiplusiminus}, \ref{prop:symmstructurecircle} and \ref{prop:symmstructureiplusminus}.  Define the map $\eta$ and therefore the maps $f_{\pm}$ and $g$ as in Proposition \ref{Prop:torussplitchainequivalence} and the symmetric structure $\Phi$ on $Y_* = C_*(X;\zh)$ to be as given in Proposition \ref{prop:symmstructureonX}.
We have that \[f_- = (1,0) \colon (D_-)_i \to E'_i \subset Y_i\]
for $i=0,1$, and
\[f_+ = (l_a^{-1},0) \colon (D_+)_i \to E'_i \subset Y_i\]
for $i=0,1$.  Also, $\varpi = (l_a) \colon (D_-)_i \to (D_+)_i$ so $f_+ \circ \varpi = f_-$ and we can take $\mu = 0$.

It is important that our objects do not depend on choices, so that equivalent knots define equivalent triads.
Different choices of $l_a$ and $l_b$ depend on a choice of the letter $p$ in Proposition \ref{prop:chainmapsiplusiminus}; this affects these elements only up to a conjugation, or in other words an application of an inner automorphism, which means we can vary $C, D_{+}$ and $f_+$ by a chain isomorphism and obtain chain equivalent triads.  A different choice of element $g_1 = (1,h_1) \in \Z \ltimes H$ is related by a conjugation, or in other words an application of an inner automorphism, as in the proof of Proposition \ref{prop:2ndderivedsubgroup}, so that we can change $C, D_{\pm}$ and $Y$ by chain isomorphisms and obtain chain equivalent triads.  The point is that we need to make choices of $g_1$ and of $l_a$ in order to write down a representative of an equivalence class of symmetric Poincar\'{e} triads, but still different choices yield equivalent triads.  We investigate the effect of such changes on the consistency isomorphism $\xi$.  A change in $l_a$ does not affect the isomorphism $\xi$.  A change in $g_1$ affects $\xi$ as follows.  When we wish to change the boundary maps and chain maps in a triad by applying an inner automorphism, conjugating by an element $h \in \Z \ltimes H$ say, we define the chain equivalence of triads $\Y \to \Y^\%$ which maps basis elements of all chain groups as follows: $e_i \mapsto he_i$: $\Y^\%$ has the same chain groups as $\Y$ but with the relevant boundary maps and chain maps conjugated by $h$.  This induces an isomorphism which by a slight abuse of notation we denote $h_* \colon H_1(\Z[\Z] \otimes_{\zh} Y) \toiso H_1(\Z[\Z] \otimes_{\zhpc} Y^\%)$.  We take $\omega \colon H \to H^\% = H$ as the identity.  In order to obtain an equivalent triple, we therefore take $\xi^\% = h_* \circ \xi$.

An isotopy of knots induces a homeomorphism of the exteriors $X \xrightarrow{\approx} X^\%$ which itself induces an isomorphism $$\omega \colon \pi_1(X)^{(1)}/\pi_1(X)^{(2)} = H \toiso \pi_1(X^\%)^{(1)}/\pi_1(X^\%)^{(2)} = H^\%.$$
Likewise the isotopy induces an equivalence of triads $\zhpc \otimes_{\zh} \Y \to \Y^\%$.  The geometrically defined maps $\xi$ and $\xi^\%$ fit into the commutative square as required in Definition \ref{Defn:algebraicsetofchaincomplexes}.



We finally have to show that the conditions on homology for an element of $\P$ are satisfied.  First, $\Z \otimes_{\zh} D_{\pm}$ is given by
\[ \Z \xrightarrow{0} \Z,\] which has the homology of a circle.  Alexander duality or an easy Mayer-Vietoris argument using the decomposition of $S^3$ as $X \cup_{\partial X \approx S^1 \times S^1} S^1 \times D^2$ shows that $H_*(C_*(X;\Z)) \cong H_*(S^1;\Z)$, with the generator of $H_1(X;\Z)$ being any of the meridians.  So the chain maps $\Id \otimes_{\zh} f_{\pm} \colon \Z \otimes D_{\pm} \to C_*(X;\Z)$ induce isomorphisms on homology, given as they are on the chain level by inclusion maps of direct summands: see the map $\eta$ from the proof of Proposition \ref{Prop:torussplitchainequivalence}.

The consistency condition is satisfied, since we have the canonical Hurewicz isomorphism $H \toiso H_1(X;\Z[\Z])$ as claimed. Therefore, we indeed have defined an element of $\mathcal{P}$.
\end{proof}

We now define the notion of addition of two triples $(H,\mathcal{Y},\xi)$ and $(H^\dag,\mathcal{Y}^{\dag},\xi^\dag)$ in $\mathcal{P}$.  In the following, the notation should be transparent: everything associated to $\mathcal{Y}^{\dag}$ will be similarly decorated with a dagger.

\begin{definition}\label{Defn:connectsumalgebraic}
We define the sum of two triples
\[(H^\ddag,\mathcal{Y}^{\ddag},\xi^\ddag) = (H,\mathcal{Y},\xi) \, \sharp \, (H^\dag,\mathcal{Y}^{\dag},\xi^\dag),\]
as follows.  The first step is to make sure that the two triads are over the same group ring.  Define $H^{\ddag} := H \oplus H^{\dag}$.
We define the pull-back group:
\[(\Z \ltimes H) \times_{\Z} (\Z \ltimes H^{\dag}) := \{(g,g^{\dag}) \in (\Z \ltimes H) \times (\Z \ltimes H^{\dag})\,|\, \phi(g) = \phi^{\dag}(g^{\dag}) \in \Z \}.\]
For any choice of splitting maps $\psi \colon \Z \to \Z \ltimes H$ and $\psi^{\dag} \colon \Z \to \Z \ltimes H^{\dag}$ (see Proposition \ref{prop:2ndderivedsubgroup} for the notation) we can define an isomorphism
\[\Xi \colon (\Z \ltimes H) \times_{\Z} (\Z \ltimes H^{\dag}) \xrightarrow{\simeq} \Z \ltimes (H \oplus H^{\dag})\]
by
\[(g,g^{\dag}) \mapsto (\phi(g),(g(\psi(1))^{-\phi(g)},g^{\dag}(\psi^{\dag}(1))^{-\phi^{\dag}(g^{\dag})})).\]
There are obvious inclusions:
\[\Z \ltimes H \rightarrowtail (\Z \ltimes H) \times_{\Z} (\Z \ltimes H^{\dag});\]
\[g \mapsto (g,0)\]
and
\[\Z \ltimes H^{\dag} \rightarrowtail (\Z \ltimes H) \times_{\Z} (\Z \ltimes H^{\dag});\]
\[g^{\dag} \mapsto (0,g^{\dag})\]
which, when composed with the isomorphism $\Xi$, using
\[\psi(1) = g_1 = (1,0) \in \Z \ltimes H\]
and
\[\psi^{\dag}(1) = g_1^{\dag} = (1,0^\dag) \in \Z \ltimes H^{\dag}\]
in the definition of $\Xi$, enable us to form the tensor products
\[\zhddag \otimes_{\zh} \Y \]
and
\[\zhddag \otimes_{\zhdag} \Y^{\dag}\]
so that both symmetric Poincar\'{e} triads are over the same group ring as required.  This will be assumed for the rest of the definition without further comment.

The next step is to exhibit a chain equivalence, in fact a chain isomorphism:
\[\nu \colon C^{\dag} \xrightarrow{\sim} C.\]
We show this for the models for each chain complex, since any $C,C^{\dag}$ which can occur is itself chain equivalent to these models.  In fact, for the operation of connected sum which we define here, we describe how to add our two symmetric Poincar\'{e} triads $\Y$ and $\Y^{\dag}$ using the models given for $i_{\pm} \colon (C,\varphi_C) \to (D_{\pm},\delta\varphi_{\pm})$ and $i^{\dag}_{\pm} \colon (C^{\dag},\varphi_{C^{\dag}}) \to (D^{\dag}_{\pm},\delta\varphi^{\dag}_{\pm})$ in Definition \ref{Defn:algebraicsetofchaincomplexes}, since by definition there is always an equivalence of symmetric triads mapping to one in which $C,C^{\dag}$ and $D^{\dag}_{\pm}$ have this form.  To achieve this with $g_1 = (1,0) = g_1^\dag$ we may have to change the isomorphisms $\xi$ and $\xi^\dag$ as in the proof of Proposition \ref{prop: fundtriaddefinesanelement}.

The chain isomorphism $\nu \colon C^{\dag}_* \to C_*$:
\[\xymatrix @R+1cm @C+1cm{ C^{\dag}_1 \ar[r]^{\partial^{\dag}_1} \ar[d]_{\nu} & C^{\dag}_0 \ar[d]^{\nu} \\
C_1 \ar[r]^{\partial_1} & C_0
}\]
is given by:
\[\xymatrix @R+2cm @C+2cm { \bigoplus_2\,\zhddag \ar[r]^{\left(\begin{array}{cc}  g^{\dag}_1-1 & 0 \\ 0 & g^{\dag}_q-1 \end{array}\right)} \ar[d]_{\left(\begin{array}{cc} 1 & 0 \\ 0 & (l_a^{\dag})^{-1}l_a \end{array}\right)} & \bigoplus_2\,\zhddag \ar[d]^{\left(\begin{array}{cc} 1 & 0 \\ 0 & (l_a^{\dag})^{-1}l_a \end{array}\right)} \\
\bigoplus_2\,\zhddag \ar[r]^{\left(\begin{array}{cc} g_1-1 & 0  \\ 0 & g_q-1 \end{array}\right)} & \bigoplus_2\,\zhddag.
}\]
In order to see that these are chain maps we need the relation:
\[g_1^{\dag} = g_1 \in \Z \ltimes H^{\ddag}\]
which, since by definition
\[g_q = l_a^{-1}g_1l_a\]
and
\[g_q^{\dag} = (l_a^{\dag})^{-1}g_1^{\dag}l_a^{\dag}\]
implies that
\[g_q = l_a^{-1}l_a^{\dag}g_q^{\dag}(l_a^{\dag})^{-1}l_a.\]
We can also use this to calculate that $\nu (\varphi^{\dag}\oplus -\varphi^{\dag}) \nu^{*} = \varphi \oplus -\varphi$.

Recall that we also have a chain isomorphism $\varpi \colon D_-^{\dag} = D_- \to D_+:$
\[\xymatrix @R+1cm @C+1cm{ (D_-^{\dag})_1 \ar[r]^{\partial^{\dag}_1} \ar[d]_{\varpi} & (D_-^{\dag})_0 \ar[d]^{\varpi} \\
(D_+)_1 \ar[r]^{\partial_1} & (D_+)_0,
}\]
given by:
\[\xymatrix @R+1cm @C+1cm {\zhddag \ar[r]^{(g^{\dag}_1-1)}  \ar[d]^{(l_a)} & \zhddag \ar[d]^{(l_a)}  \\ \zhddag \ar[r]^{(g_q - 1)} & \zhddag.}\]

We now glue the two symmetric triads together.  The idea is that we are following the geometric addition of knots, as in Proposition \ref{Prop:decompknot-extconnectedsum}, where the neighbourhood of a chosen meridian of each knot gets identified.  We have the following diagram:
\[\xymatrix @R+1cm @C-0.41cm {
(D_-,0 = \delta\varphi_-) \ar[d]^{f_-} & (C,\varphi \oplus -\varphi=\varphi_C) \ar[l]_-{i_-} \ar[d]_{i_+} \ar[dl]_{\stackrel{g}{\sim}} & (C^{\dag},\varphi^{\dag} \oplus -\varphi^{\dag}=\varphi_{C^{\dag}}) \ar[l]_-{\stackrel{\nu}{\simeq}} \ar[d]^{i_-^{\dag}} \ar[r]^-{i_+^{\dag}} \ar[dr]^{\stackrel{g^{\dag}}{\sim}} & (D_+^{\dag},0=\delta\varphi_+^{\dag}) \ar[d]^{f_+^{\dag}} \\
(Y,\Phi) & (D_+,0=\delta\varphi_+) \ar[l]_-{f_+} & (D_-^{\dag},0=\delta\varphi_-^{\dag}) \ar[r]^-{f_-^{\dag}} \ar[l]_-{\stackrel{\varpi}{\simeq}} & (Y^{\dag},\Phi^{\dag})
}\]
where the central square commutes.  We then use the union construction from Definition \ref{Defn:unionconstruction} to define $\Y^{\ddag}$:
\[\xymatrix @R+0.5cm @C+0.5cm {\ar @{} [dr] |{\stackrel{g^{\ddag}}{\sim}}
(C^{\ddag},\varphi_{C^{\ddag}}) \ar[r]^{i^{\ddag}_-} \ar[d]_{i^{\ddag}_+} & (D^{\ddag}_-,\delta\varphi^{\ddag}_-) \ar[d]^{f^{\ddag}_-}\\ (D^{\ddag}_+,\delta\varphi^{\ddag}_+) \ar[r]^{f^{\ddag}_+} & (Y^{\ddag},\Phi^{\ddag}).
}\]
where:
\[(C^{\ddag},\varphi_{C^{\ddag}}):=(C^{\dag},\varphi_{C^{\dag}});\]
\[i_+^{\ddag} := i_+^{\dag};\]
\[i_-^{\ddag} := i_-\circ \nu;\]
\[(D_-^{\ddag},\delta\varphi_-^{\ddag}):= (D_-,\delta\varphi_-=0);\]
\[(D_+^{\ddag},\delta\varphi_+^{\ddag}):= (D_+^{\dag},\delta\varphi_+^{\dag}=0);\]
\[(Y^{\ddag},\Phi^{\ddag}) := (\mathscr{C}((-f_+ \circ \varpi,f_-^{\dag})^T \colon D_-^{\dag} \to Y \oplus Y^{\dag}),\Phi \cup_{\delta\varphi_-^{\dag}} \Phi^{\dag}),\]
so that
\[Y^{\ddag}_r:= Y_r \oplus (D_-^{\dag})_{r-1} \oplus Y_r^{\dag};\]
\[d_{Y^{\ddag}} := \left(\begin{array}{ccc} d_Y & (-1)^{r}f_{+}\circ \varpi & 0 \\ 0 & d_{D_-^{\dag}} & 0 \\ 0 & (-1)^{r-1}f^{\dag}_{-} & d_{Y^{\dag}}
\end{array}\right)\colon Y^{\ddag}_r \to Y^{\ddag}_{r-1};\]
\[f^{\ddag}_{-} := \left(
              \begin{array}{c}
                f_- \\
                0 \\
                0
              \end{array}
            \right) \colon (D_-^{\ddag})_r = (D_-)_r \to Y^{\ddag}_r=Y_r \oplus (D_-^{\dag})_{r-1} \oplus Y^{\dag}_r;
\]
\[f^{\ddag}_+ = \left(
              \begin{array}{c}
                0 \\
                0 \\
                f^{\dag}_{+}
              \end{array}
            \right) \colon (D_+^{\ddag})_r = (D^{\dag}_+)_r \to Y^{\ddag}_r=Y_r \oplus (D_-^{\dag})_{r-1} \oplus Y^{\dag}_r;\]
\[\Phi^{\ddag}_s := (\Phi \cup_{\delta\varphi_-^{\dag}} \Phi^{\dag})_s = \left(
                        \begin{array}{ccc}
                          \Phi_s & 0 & 0 \\
                          0 & 0 & 0 \\
                          0 & 0 & \Phi^{\dag}_s \\
                        \end{array}
                      \right)\colon\]
\begin{multline*}(Y^{\ddag})^{3-r+s} = Y^{3-r+s} \oplus (D_-^{\dag})^{2-r+s} \oplus (Y^{\dag})^{3-r+s} \to Y^{\ddag}_r = Y_r \oplus (D_-^{\dag})_{r-1} \oplus Y^{\dag}_r \\ (0 \leq s \leq 3);\end{multline*}
\[g^{\ddag}:= \left(
              \begin{array}{c}
                g \circ \nu \\
                (-1)^{r+1} i_-^{\dag} \\
                g^{\dag}
              \end{array}\right) \colon C^{\ddag}_r = C_r^{\dag} \to Y_{r+1}^{\ddag} = Y_{r+1} \oplus (D_-^{\dag})_r \oplus Y^{\dag}_{r+1}.\]

The mapping cone is of the chain map $(-f_+ \circ \varpi,f_-^{\dag})^T$, with a minus sign to reflect the fact that when one adds together oriented knots, one must identify the boundaries with opposite orientations coinciding, as described in Definition \ref{Defn:connectedsum}, so that the resulting knot is also oriented.

We therefore have the chain maps $i_{\pm}^{\ddag}$, given by:
\[\xymatrix @R+1cm @C-0.4cm{
D^{\ddag}_- = D_-  & \zhddag \ar[rrrrr]^{\left(\begin{array}{c}g_1-1\end{array}\right)}  &&&&& \zhddag \\
C^{\ddag}=C^{\dag} \ar[u]^{i_-^{\ddag}=i_- \circ \nu} \ar[d]_{i_+^{\ddag} = i_+^{\dag}} & \bigoplus_2\,\zhddag \ar[rrrrr]_{\left(\begin{array}{cc}g^{\dag}_1-1 & 0 \\ 0 & g^{\dag}_q-1 \end{array}\right)} \ar[u]^{\left(\begin{array}{c} 1 \\ (l_a^{\dag})^{-1} \end{array}\right)} \ar[d]_{\left(\begin{array}{c} (l_b^{\dag})^{-1} \\ 1 \end{array}\right)} &&&&& \bigoplus_2\,\zhddag \ar[u]_{\left(\begin{array}{c} 1 \\ (l_a^{\dag})^{-1} \end{array}\right)} \ar[d]^{\left(\begin{array}{c} (l_b^{\dag})^{-1} \\ 1 \end{array}\right)} \\
D_+^{\ddag}=D_+^{\dag} & \zhddag \ar[rrrrr]_{\left(\begin{array}{c}g^{\dag}_q-1\end{array}\right)} &&&&& \zhddag,
}\]
which means we can take:
\[g_1^{\ddag}:= g_1^{\dag}=g_1 \in \Z \ltimes H^{\ddag} = \Z \ltimes (H \oplus H^{\dag});\]
\[l_a^{\ddag}:= l_a^{\dag} \in \Z \ltimes H^{\ddag}; \text{ and}\]
\[l_b^{\ddag}:= l_b^{\dag} \in \Z \ltimes H^{\ddag},\]
so that
\[g_q^{\ddag} := g_q^{\dag} \in \Z \ltimes H^{\ddag}.\]
We have a chain isomorphism $\varpi^{\dag} \colon D_- = D^{\dag}_- \to D_+^{\dag}$.  To construct a chain homotopy \[\mu^{\ddag} \colon \left( \begin{array}{c} 0 \\ 0 \\ f_+^{\dag} \circ \varpi^{\dag} \end{array} \right) \simeq  \left( \begin{array}{c} f_- \\ 0 \\ 0 \end{array} \right)\] we first use:
\[\mu^{\dag} \colon \left( \begin{array}{c} 0 \\ 0 \\ f_+^{\dag}\circ\varpi^{\dag} \end{array} \right) \simeq \left( \begin{array}{c} 0 \\ 0 \\ f_-^{\dag}\end{array} \right).\]
We then have a chain homotopy given by:
\[\left(
    \begin{array}{c}
      0 \\
      \Id \\
      0 \\
    \end{array}
  \right) \colon (D_-^{\dag})_0 \to Y^{\ddag}_1 = Y_1 \oplus (D_-^{\dag})_0 \oplus Y_1^{\dag},
\]
and
\[\left(
    \begin{array}{c}
      0 \\
      -\Id \\
      0 \\
    \end{array}
  \right) \colon (D_-^{\dag})_1 \to Y^{\ddag}_2 = Y_2 \oplus (D_-^{\dag})_1 \oplus Y_2^{\dag},\]
which shows that
\[\left( \begin{array}{c} 0 \\ 0 \\ f_-^{\dag} \end{array} \right) \simeq  \left( \begin{array}{c} f_+ \circ \varpi \\ 0 \\ 0 \end{array} \right) \colon D_-^{\dag} \to \mathscr{C}((-f_+ \circ \varpi, f_-^{\dag})^T).\]
We finally have \[\mu \colon \left( \begin{array}{c} f_+ \circ \varpi \\ 0 \\ 0 \end{array} \right) \simeq \left( \begin{array}{c} f_- \\ 0 \\ 0 \end{array} \right).\]  Combining these three homotopies yields
\[\mu^{\ddag} \colon \left( \begin{array}{c} 0 \\ 0 \\ f_+^{\dag} \circ \varpi^{\dag} \end{array} \right) \simeq \left( \begin{array}{c} f_- \\ 0 \\ 0 \end{array} \right).\]
This completes our description of the symmetric Poincar\'{e} triad
\[\mathcal{Y}^{\ddag} := \mathcal{Y} \, \sharp \, \mathcal{Y}^{\dag}.\]
We now need to check that the two homological conditions are satisfied for this triad, so that we indeed still have an element of $\P$, and the sum operation is well-defined.  For the $\Z$-homology condition, we have a mapping cone operation to construct $Y^{\ddag}$ which has an associated short exact sequence of chain complexes:
\[0 \to Y \oplus Y^{\dag} \to \mathscr{C}((-f_+\circ\varpi,f_-^{\dag})^T) = Y^{\ddag} \to SD_-^{\dag} \to 0,\] so there is therefore a $\Z$-homology Mayer-Vietoris long exact sequence:
\[H_3(Y;\Z) \oplus H_3(Y^{\dag};\Z) \to H_3(Y^{\ddag};\Z) \to H_2(D_-^{\dag};\Z) \to H_2(Y;\Z) \oplus H_2(Y^{\dag};\Z) \to \] \[ H_2(Y^{\ddag};\Z) \to
H_1(D_-^{\dag};\Z) \xrightarrow{(-1,1)^T} H_1(Y;\Z) \oplus H_1(Y^{\dag};\Z) \xrightarrow{(i_{Y},i_{Y^{\dag}})} H_1(Y^{\ddag};\Z) \xrightarrow{0} \]
\[H_0(D_-^{\dag};\Z) \xrightarrow{(-1,1)^T} H_0(Y;\Z) \oplus H_0(Y^{\dag};\Z) \xrightarrow{(i_{Y},i_{Y^{\dag}})} H_0(Y^{\ddag};\Z).\]
Since $H_*(S^1;\Z) \cong H_*(D_-^{\dag};\Z) \xrightarrow{\simeq} H_*(Y;\Z) \cong H_*(Y^{\dag};\Z)$ we deduce from this sequence that also $H_*(Y^{\ddag};\Z) \cong H_*(S^1;\Z)$.  Since we also have isomorphisms \[(f_-)_*\colon H_*(D_-;\Z) \xrightarrow{\simeq} H_*(Y;\Z)\]
and
\[(f_+^{\dag})_*\colon H_*(D_+^{\dag};\Z) \xrightarrow{\simeq} H_*(Y^{\dag};\Z),\]
and since the homology of $Y^{\ddag}$ is generated by either of the generators of $H_*(Y;\Z)$ or $H_*(Y^{\dag};\Z)$, amalgamated as they are by the gluing operation, we indeed have induced $\Z$-homology isomorphisms:
\[(f_{\pm}^{\ddag})_* \colon H_*(\Z \otimes_{\zhddag} D_{\pm}^{\ddag}) \xrightarrow{\simeq} H_*(\Z \otimes_{\zhddag} Y^{\ddag}) \cong H_*(S^1;\Z)\]
as claimed.
To check the consistency condition we only need look at the following part of the sequence, with $\Z[\Z]$ coefficients:
\begin{multline*} H_1(D_-^{\dag};\Z[\Z]) \to H_1(Y;\Z[\Z]) \oplus H_1(Y^{\dag};\Z[\Z]) \xrightarrow{(\Id,0,\Id)} H_1(Y^{\ddag};\Z[\Z]) \xrightarrow{0} \\ H_0(D_-^{\dag};\Z[\Z])\end{multline*}
Since $\ker((t-1)\colon \Z[\Z] \to \Z[\Z]) \cong 0$ we have that $H_1(D_-^{\dag};\Z[\Z]) \cong 0$ so that there is an isomorphism:
\[H_1(Y;\Z[\Z]) \oplus H_1(Y^{\dag};\Z[\Z]) \xrightarrow{\simeq} H_1(Y^{\ddag};\Z[\Z])\]
Composing this map with the isomorphism:
\[\left(
    \begin{array}{cc}
      \xi & 0 \\
      0 & \xi^{\dag}
    \end{array}
  \right) \colon H \oplus H^{\dag} = H^{\ddag} \toiso H_1(Y;\Z[\Z]) \oplus H_1(Y^{\dag};\Z[\Z])
\]
yields an isomorphism
\[\xi^{\ddag} \colon  H^{\ddag} \toiso H_1(Y^{\ddag};\Z[\Z]) ,\]
which shows that the consistency condition is satisfied and defines the third element of the triple
\[(H^\ddag,\mathcal{Y}^{\ddag},\xi^\ddag) = (H,\mathcal{Y},\xi) \, \sharp \, (H^\dag,\mathcal{Y}^{\dag},\xi^\dag) \in \P.\]
This completes the definition of the addition of two elements of $\P$.
\qed \end{definition}

\begin{lemma}\label{Lemma:homotopicmapscones}
Chain homotopic chain maps $f \simeq g \colon C \to D$ have chain isomorphic mapping cones.
\end{lemma}
\begin{proof}
Let $k \colon f \simeq g \colon C \to D$ be the chain homotopy.  We have a chain isomorphism of the mapping cones:
\[\xymatrix @R+2cm @C+2cm{ D_r \oplus C_{r-1} \ar[r]^{\left(
                                                        \begin{array}{cc}
                                                          d_D & (-1)^{r-1}f \\
                                                          0 & d_C
                                                        \end{array}
                                                      \right)
} \ar[d]_{\left(
            \begin{array}{cc}
              \Id & (-1)^{r+1}k \\
              0 & \Id \\
            \end{array}
          \right)
} & D_{r-1} \oplus C_{r-2} \ar[d]^{\left(
            \begin{array}{cc}
              \Id & (-1)^{r}k \\
              0 & \Id \\
            \end{array}
          \right)} \\
D_r \oplus C_{r-1} \ar[r]^{\left(
                                                        \begin{array}{cc}
                                                          d_D & (-1)^{r-1}g \\
                                                          0 & d_C
                                                        \end{array}
                                                      \right)}
& D_{r-1} \oplus C_{r-2}
}\]
with inverse given by:
\[\left(\begin{array}{cc}
              \Id & (-1)^{r}k \\
              0 & \Id \\
            \end{array}
          \right) \colon \mathscr{C}(g)_r = D_{r} \oplus C_{r-1} \to \mathscr{C}(f)_r = D_r \oplus C_{r-1}. \]
\end{proof}

\begin{proposition}\label{Prop:abelianmonoid}
The sum operation $\sharp$ on $\P$ is abelian, associative and has an identity, namely the triple containing the fundamental symmetric Poincar\'{e} triad of the unknot.  Therefore, $(\P,\sharp)$ is an abelian monoid.  Let ``$Knots$'' denote the abelian monoid of isotopy classes of locally flat knots in $S^3$ under the operation of connected sum.  Then we have a monoid homomorphism $Knots \to \P$.
\end{proposition}
\begin{proof}
Associativity is straight--forward.  If we add three triples $$(H,\Y,\xi) \,\sharp \, (H',\Y',\xi') \,\sharp \,(H'',\Y'',\xi'')$$ together, then the Alexander module will be $H \oplus H' \oplus H''$, no matter the order of addition.  The mapping cones which add the complexes $Y,Y'$ and $Y''$ together are associative operations, and the extra  data will be $C'', D''_{\pm}, g''_1, l''_a$ and $l''_b$ no matter the order in which we choose to perform the addition.

The identity element is given by the fundamental symmetric Poincar\'{e} triad of the unknot, which we denote $\Y^U$, in the triple $(\{0\},\Y^U,\Id_{\{0\}})$.  That is, $H^U = \{0\}$, so all the chain complexes comprise $\Z[\Z]$-modules.  We can take $g_1^U = t$, and $l_a^U=l_b^U=1$.  The chain map and complexes $i_{\pm}^U \colon (C^U,\varphi_C^U) \to (D_{\pm}^U,\delta\varphi_{\pm}^U=0)$ are given by the models of Definition \ref{Defn:algebraicsetofchaincomplexes} with $g^U_1=g^U_q=t$.  The chain complex for $(Y^U,\Phi^U=0)$ is the same as that for $D_{\pm}^U$:
\[\Z[\Z] \xrightarrow{(t-1)} \Z[\Z],\]
with the maps:
\[f_{\pm}^U = \Id \colon D_{\pm}^U \to Y^U,\]
so that $g^U=0$.
When we form the sum $(H^\ddag,\Y^{\ddag},\xi^\ddag) := (\{0\},\Y^U,\Id_{\{0\}}) \,\sharp\,(H,\Y,\xi)$, we have that $H^{\ddag} = H$, and the tensor operation has the effect of identifying $t=g_1$.  This means performing the sum operation as described in Definition \ref{Defn:connectsumalgebraic} yields $(i_{\pm}^\ddag \colon C^{\ddag} \to D_{\pm}^{\ddag}) = (i_{\pm} \colon C \to D_{\pm}), l_a^{\ddag} = l_a$ and $l_b^{\ddag} = l_b$.  The map $\varpi \colon D_- \to D^U_+$ is the identity map so we do not include it in the notation here and take $D_- = D^U_+$.  The chain complex for $Y^{\ddag}$, defined as the mapping cone of $(-f_+^U,f_-)^T$, is given by:
\[\xymatrix @C+0.3cm { Y_3 \ar[r]^-{\partial_Y} &  (D_-)_1 \oplus Y_2 \ar[rr]^-{\left(
                                                                      \begin{array}{cc}
                                                                        f_+^U & 0  \\
                                                                        \partial_{D_-} & 0 \\
                                                                       -f_- & \partial_Y \\
                                                                      \end{array}
                                                                    \right)
} & & Y^U_1 \oplus (D_-)_0 \oplus Y_1 \ar[r]^-{(\partial_{Y^{\ddag}})_1
} & Y_0^U \oplus Y_0, }\]
where
\[(\partial_{Y^{\ddag}})_1 = \left(\begin{array}{ccc} \partial_{Y^U} & -f_+^U & 0 \\
                                                    0 & f_- & \partial_Y \end{array} \right)\]
with
\[f^{\ddag}_{-} := \left(
              \begin{array}{c}
                f_-^U \\
                0 \\
                0
              \end{array}
            \right) \colon (D_-^{\ddag})_r = (D_-^U)_r \to Y^{\ddag}_r=Y^U_r \oplus (D_-)_{r-1} \oplus Y_r\]
and
\[f^{\ddag}_{+} := \left(
              \begin{array}{c}
                0 \\
                0 \\
                f_+
              \end{array}
            \right) \colon (D_+^{\ddag})_r = (D_+)_r \to Y^{\ddag}_r=Y^U_r \oplus (D_-)_{r-1} \oplus Y_r.\]

Since $f_+^U = \Id$, this chain complex is chain equivalent to $Y$ via the chain map:
\[\xymatrix @ R+2cm{ Y_3 \ar[r]^-{\partial_Y} \ar[d]^>>>>>>>>>>{\left(\begin{array}{c}\Id \end{array}\right)} &  (D_-)_1 \oplus Y_2 \ar[rrr]^-{\left(
                                                                      \begin{array}{cc}
                                                                        f_+^U & 0  \\
                                                                        \partial_{D_-} & 0 \\
                                                                       -f_- & \partial_Y \\
                                                                      \end{array}
                                                                    \right)
}  \ar[d]_<<<<<<<<<<{\left( \begin{array}{cc} 0 \, , & \Id \end{array} \right)} & & & Y^U_1 \oplus (D_-)_0 \oplus Y_1 \ar[rr]^-{(\partial_{Y^{\ddag}})_1}  \ar[d]_{\left( \begin{array}{ccc} f_-\circ(f_+^U)^{-1} \, , & 0 \, , & \Id \end{array} \right)} & & Y_0^U \oplus Y_0  \ar[d]_{\left( \begin{array}{cc} f_-\circ(f_+^U)^{-1} & \, , \Id \end{array}\right)}\\
Y_3 \ar[r]^{\partial_3} & Y_2 \ar[rrr]^{\partial_2} & & & Y_1 \ar[rr]^{\partial_1} & & Y_0
}\]
with chain homotopy inverse:
\[\xymatrix @C+0.5cm @R+2cm{
Y_3 \ar[r]^{\partial_3} \ar[d]^{\left( \begin{array}{c}\Id \end{array}\right)} & Y_2 \ar[rr]^{\partial_2} \ar[d]_{\left( \begin{array}{c} 0 \\ \Id \end{array}\right)} & & Y_1 \ar[r]^{\partial_1} \ar[d]_{\left( \begin{array}{c} 0 \\ 0 \\ \Id \end{array}\right)} & Y_0 \ar[d]_{\left( \begin{array}{c} 0 \\ \Id \end{array}\right)}\\
Y_3 \ar[r]^-{\partial_Y} &  (D_-)_1 \oplus Y_2 \ar[rr]^-{\left(
                                                                      \begin{array}{cc}
                                                                        f_+^U & 0  \\
                                                                        \partial_{D_-} & 0 \\
                                                                       -f_- & \partial_Y \\
                                                                      \end{array}
                                                                    \right)
} & & Y^U_1 \oplus (D_-)_0 \oplus Y_1 \ar[r]^-{(\partial_{Y^{\ddag}})_1} & Y_0^U \oplus Y_0. }\]
The chain homotopy which shows the composition of these chain maps is homotopic to the identity chain map is given by:
\[\left(
    \begin{array}{cc}
      0 & 0 \\
     (f_+^U)^{-1} & 0 \\
      0 & 0
    \end{array}
  \right)
 \colon Y_0^U \oplus Y_0 \to Y_1^U \oplus (D^U_+)_0 \oplus Y_1;\]
and
\[\left(
    \begin{array}{ccc}
      -(f_+^U)^{-1} & 0 & 0 \\
      0 & 0 & 0 \\
    \end{array}
  \right) \colon Y_1^U \oplus (D^U_+)_0 \oplus Y_1 \to (D_+^U)_1 \oplus Y_2.
\]
Under this chain equivalence, the map $f_-^U \colon D_-^U \to Y^U$ becomes $$f_- \circ (f_+^U)^{-1} \circ f_-^U \colon D_-^U \to Y,$$ which makes sense since $D^U_- = D_+^U = D_-$ after the identification $t=g_1$, and means that we obtain once more the chain map $f_- \colon D_- \to Y$.  Therefore addition with the triple containing the fundamental symmetric Poincar\'{e} triad of the unknot acts as identity, so we have a monoid as claimed.

We now show that our monoid is abelian.  Suppose that we have two triples $(H,\Y,\xi)$ and $(H^\dag,\Y^{\dag},\xi^\dag)$ as before, and we form the sum of these in two ways.  Taking the sum $\Y \,\sharp \, \Y^{\dag}$ yields the triad:
\[\xymatrix @R+1.5cm @C+1.5cm { C^{\dag} \ar[r]^{i_- \circ \nu}  \ar[d]^{i_+^{\dag}} & D_- \ar[d]^{\left(\begin{array}{c} f_- \\ 0 \\ 0 \end{array} \right)}
\\ D^{\dag}_+ \ar[r]_-{\left(\begin{array}{c} 0 \\ 0 \\ f^{\dag}_+ \end{array} \right)}
& \mathscr{C}((-f_+ \circ \varpi , f^{\dag}_-)^T)}\]
where
\[\mathscr{C}((-f_+ \circ \varpi,f_-^{\dag})^T)_r = Y_r \oplus (D_-^{\dag})_{r-1} \oplus Y_r^{\dag}.\]
On the other hand, taking the sum $\Y^{\dag} \,\sharp \, \Y$ yields the triad:
\[\xymatrix @R+1.5cm @C+1.5cm { C \ar[r]^{i^{\dag}_- \circ \nu^{-1}}  \ar[d]^{i_+} & D^{\dag}_- \ar[d]^{\left(\begin{array}{c} f^{\dag}_- \\ 0 \\ 0 \end{array} \right)}
\\ D_+ \ar[r]_-{\left(\begin{array}{c} 0 \\ 0 \\ f_+ \end{array} \right)}
& \mathscr{C}((-f^{\dag}_+ \circ \varpi^{\dag} , f_-)^T)}\]
where
\[\mathscr{C}((-f^{\dag}_+ \circ \varpi^{\dag},f_-)^T)_r = Y^{\dag}_r \oplus (D_-)_{r-1} \oplus Y_r.\]
We exhibit chain equivalences which induce the desired morphisms on symmetric structures.  First, we have the chain isomorphism:
\[\nu \colon C^{\dag} \to C\]
which we saw in Definition \ref{Defn:connectsumalgebraic} induces $\nu^{\%}(\varphi^{\dag} \oplus -\varphi^{\dag}) = \varphi \oplus -\varphi$.  Since $g_1 =g_1^{\dag}$, we have that
\[D_- = D^{\dag}_-.\]
We then check that:
\[(i_-^{\dag} \circ \nu^{-1}) \circ \nu = i_- \circ \nu,\]
i.e. that:
\[i_-^{\dag} = i_- \circ \nu,\]
which translates to:
\[\left(\begin{array}{c} 1 \\ (l_a^{\dag})^{-1} \end{array} \right) = \left(\begin{array}{cc} 1 & 0 \\ 0 & (l_a^{\dag})^{-1}l_a \end{array} \right) \left(\begin{array}{c} 1 \\ l_a^{-1} \end{array} \right).\]
Next, we have a chain isomorphism:
\[\varpi \circ (\varpi^{\dag})^{-1} \colon D_+^{\dag} \to D_+.\]
We check that:
\[\varpi \circ (\varpi^{\dag})^{-1}\circ i_+^{\dag} = i_+ \circ \nu,\]
which translates to:
\[\left(\begin{array}{c} l_a^{\dag} \\ 1 \end{array} \right)\left(\begin{array}{c} (l_a^{\dag})^{-1} \end{array} \right)\left(\begin{array}{c} l_a \end{array} \right) = \left(\begin{array}{cc} 1 & 0 \\ 0 & (l_a^{\dag})^{-1}l_a \end{array} \right) \left(\begin{array}{c} l_a \\ 1 \end{array} \right).\]
Next, we use Lemma \ref{Lemma:homotopicmapscones} to obtain the first two chain equivalences of the following:
\begin{eqnarray*}
\mathscr{C}((-f_+ \circ \varpi,f_-^{\dag})^T) & \simeq & \mathscr{C}((-f_-,f_-^{\dag})^T) \\
& \simeq & \mathscr{C}((-f_-,f_+^{\dag} \circ \varpi^{\dag})^T) \\
& \simeq & \mathscr{C}((f_+^{\dag} \circ \varpi^{\dag}, -f_-)^T) \\
& \simeq & \mathscr{C}((-f_+^{\dag} \circ \varpi^{\dag}, f_-)^T).
\end{eqnarray*}
The last two chain equivalences are simply given by swapping orders and changing the signs.
Explicitly, the chain isomorphism from Lemma \ref{Lemma:homotopicmapscones} which gives us the first two equivalences is
\[\left(
    \begin{array}{ccc}
      \Id & (-1)^{r+1}\mu & 0 \\
      0 & \Id & 0 \\
      0 & (-1)^{r+1}\mu^{\dag} & \Id
    \end{array}
  \right)
 \colon Y_r \oplus (D_-^{\dag})_{r-1} \oplus Y_r^{\dag} \to Y_r \oplus (D_-^{\dag})_{r-1} \oplus Y_r^{\dag}.\]
The last two chain equivalences are given by the map:
\[\left(
    \begin{array}{ccc}
      0 & 0 & \Id \\
      0 & -\Id & 0 \\
      \Id & 0 & 0
    \end{array}
  \right)
 \colon Y_r \oplus (D_-^{\dag})_{r-1} \oplus Y_r^{\dag} \to Y_r^{\dag} \oplus (D_-)_{r-1} \oplus Y_r\]
so that when these are all combined we have a chain equivalence $$\mathscr{C}((-f_+ \circ \varpi,f_-^{\dag})^T) \simeq \mathscr{C}((-f_+^{\dag} \circ \varpi^{\dag}, f_-)^T),$$
given by:
\[\left(
    \begin{array}{ccc}
      0 & (-1)^{r+1}\mu^{\dag} & \Id \\
      0 & -\Id & 0 \\
      \Id & (-1)^{r+1}\mu & 0
    \end{array}
  \right) \colon Y_r \oplus (D_-^{\dag})_{r-1} \oplus Y_r^{\dag} \to Y_r^{\dag} \oplus (D_-)_{r-1} \oplus Y_r.
\]
Note that the induced map on $Q$-groups sends the symmetric structure:
\[\Phi \cup _{\delta\varphi_-^{\dag}} \Phi^{\dag} = \left(\begin{array}{ccc} \Phi_s & 0 & 0 \\
                          0 & 0 & 0 \\
                          0 & 0 & \Phi^{\dag}_s \\
                        \end{array}
                      \right)\colon\]
\[Y^{3-r+s} \oplus (D_-^{\dag})^{2-r+s} \oplus (Y^{\dag})^{3-r+s} \to Y_r \oplus (D_-^{\dag})_{r-1} \oplus Y^{\dag}_r \;\;(0 \leq s \leq 3),\]
to the symmetric structure:
\[\Phi^{\dag} \cup _{\delta\varphi_-} \Phi = \left(\begin{array}{ccc} \Phi_s^{\dag} & 0 & 0 \\
                          0 & 0 & 0 \\
                          0 & 0 & \Phi_s \\
                        \end{array}
                      \right)\colon\]
\[(Y^{\dag})^{3-r+s} \oplus (D_-)^{2-r+s} \oplus Y^{3-r+s} \to Y_r^{\dag} \oplus (D_-)_{r-1} \oplus Y_r \;\;(0 \leq s \leq 3),\]
as required.

We now check that this chain equivalence commutes, up to homotopy, with the maps of the triads.  First, we need to show that
\[\xymatrix @C+3cm @R+2cm{D_- \ar[r]^{=} \ar[d]_{\left(\begin{array}{c} f_- \\ 0 \\ 0 \end{array} \right)} & D_-^{\dag} \ar[d]^{\left(\begin{array}{c} f^{\dag}_- \\ 0 \\ 0 \end{array} \right)}\\
\mathscr{C}((-f_+\circ \varpi,f_-^{\dag})^T) \ar[r]_-{\left(
    \begin{array}{ccc}
      0 & (-1)^{r+1}\mu^{\dag} & \Id \\
      0 & -\Id & 0 \\
      \Id & (-1)^{r+1}\mu & 0
    \end{array}
  \right)} & \mathscr{C}((-f_+^{\dag} \circ \varpi^{\dag}, f_-)^T) }\]
commutes up to homotopy.  As we saw in Definition \ref{Defn:connectsumalgebraic}, the two maps which occur in the mapping cone are homotopic:
\[\left( \begin{array}{c} 0 \\ 0 \\ f_- \end{array} \right) \simeq  \left( \begin{array}{c} f_+^{\dag} \circ \varpi^{\dag} \\ 0 \\ 0 \end{array} \right) \colon D_- \to \mathscr{C}((-f_+^{\dag} \circ \varpi^{\dag}, f_-)^T).\]
We then use $\mu^{\dag}$ to see that:
\[\mu^{\dag} \colon \left( \begin{array}{c} f_+^{\dag} \circ \varpi^{\dag} \\ 0 \\ 0 \end{array} \right) \simeq \left( \begin{array}{c} f_-^{\dag} \\ 0 \\ 0 \end{array} \right),\]
so that the square above commutes up to homotopy as claimed.  Similarly we require that the following diagram also commutes up to homotopy:
\[\xymatrix @C+3cm @R+2cm{D_+^{\dag} \ar[r]^{\varpi \circ (\varpi^{\dag})^{-1}} \ar[d]_{\left(\begin{array}{c} 0 \\ 0 \\ f_+^{\dag} \end{array} \right)} & D_+ \ar[d]^{\left(\begin{array}{c} 0 \\ 0 \\ f_+ \end{array} \right)}\\
\mathscr{C}((-f_+\circ \varpi,f_-^{\dag})^T) \ar[r]_-{\left(
    \begin{array}{ccc}
      0 & (-1)^{r+1}\mu^{\dag} & \Id \\
      0 & -\Id & 0 \\
      \Id & (-1)^{r+1}\mu & 0
    \end{array}
  \right)} & \mathscr{C}((-f_+^{\dag} \circ \varpi^{\dag}, f_-)^T) }\]
First, we have:
\begin{multline*}\mu \colon \left(\begin{array}{c} 0 \\ 0 \\ f_+ \circ \varpi \circ (\varpi^{\dag})^{-1} \end{array} \right) \simeq \left(\begin{array}{c} 0 \\ 0 \\ f_- \circ (\varpi^{\dag})^{-1} \end{array} \right) = \left(\begin{array}{c} 0 \\ 0 \\ f_- \end{array} \right) \circ (\varpi^{\dag})^{-1} \colon \\ D^{\dag}_+ \to \mathscr{C}((-f_+^{\dag} \circ \varpi^{\dag}, f_-)^T).\end{multline*}
Then since, again as in Definition \ref{Defn:connectsumalgebraic}:
\[ \left(\begin{array}{c} 0 \\ 0 \\ f_- \end{array} \right) \simeq \left(\begin{array}{c} f_+^{\dag} \circ \varpi^{\dag} \\ 0 \\ 0 \end{array} \right) \colon D_-= D^{\dag}_- \to \mathscr{C}((-f_+^{\dag} \circ \varpi^{\dag}, f_-)^T), \]
we have that:
\[\left(\begin{array}{c} 0 \\ 0 \\ f_+ \circ \varpi \circ (\varpi^{\dag})^{-1} \end{array} \right) \simeq  \left(\begin{array}{c} 0 \\ 0 \\ f_- \end{array} \right) \circ (\varpi^{\dag})^{-1} \simeq \left(\begin{array}{c} f_+^{\dag} \circ \varpi^{\dag} \\ 0 \\ 0 \end{array} \right) \circ (\varpi^{\dag})^{-1} = \left(\begin{array}{c} f_+^{\dag} \\ 0 \\ 0 \end{array} \right)\]
as required.  This completes the description of the equivalence of symmetric Poincar\'{e} triads:
\[\Z[\Z \ltimes (H^\dag \oplus H)] \otimes_{\Z[\Z \ltimes (H \oplus H^\dag)]}(\Y \, \sharp \, \Y^{\dag}) \xrightarrow{\sim} \Y^{\dag} \, \sharp \, \Y.\]
To see that we have an equivalence of triples:
\[(H \oplus H^\dag,\Y \,\sharp\,\Y^\dag,\xi \oplus \xi^\dag) \sim (H^\dag \oplus H,\Y^\dag \,\sharp\,\Y,\xi^\dag \oplus \xi),\]
note that the following two diagrams commute:
\[\xymatrix @R+1cm @C+1cm {H \oplus H^{\dag} \ar[r]^-{\left(\ba{cc} \xi & 0 \\ 0 & \xi^\dag \ea\right)} \ar[d]_-{\left(\ba{cc} 0 & \Id \\ \Id & 0 \ea\right)} & H_1(Y;\Z[\Z]) \oplus H_1(Y^\dag;\Z[\Z]) \ar[d]_-{\left(\ba{cc} 0 & \Id \\ \Id & 0 \ea\right)}  \\
H^\dag \oplus H \ar[r]_-{\left(\ba{cc} \xi^\dag & 0 \\ 0 & \xi \ea\right)} & H_1(Y^\dag;\Z[\Z]) \oplus H_1(Y;\Z[\Z]),}\]
and
\[\xymatrix @R+1cm @C+1cm {
H_1(Y;\Z[\Z]) \oplus H_1(Y^\dag;\Z[\Z]) \ar[r]^-{\cong} \ar[d]_-{\left(\ba{cc} 0 & \Id \\ \Id & 0 \ea\right)} & H_1(\Z[\Z] \otimes_{\zhddag} \mathscr{C}((-f_+\circ \varpi,f_-^{\dag})^T)) \ar[d]_-{\left(\Id \otimes \left( \begin{array}{ccc}0 & (-1)^{r+1}\mu^{\dag} & \Id \\0 & -\Id & 0 \\\Id & (-1)^{r+1}\mu & 0\end{array}\right)\right)_*} \\
H_1(Y^\dag;\Z[\Z]) \oplus H_1(Y;\Z[\Z]) \ar[r]^-{\cong} & H_1(\Z[\Z] \otimes_{\zhddag} \mathscr{C}((-f_+^{\dag} \circ \varpi^{\dag} ,f_-)^T)).}\]
Combining the two gives us the required commutative square to show that we have an equivalence of triples.  We have therefore defined an \emph{abelian} monoid of symmetric Poincar\'{e} triads as claimed.

Our operation of connected sum of Definition \ref{Defn:connectsumalgebraic} performs a gluing construction which precisely mirrors the geometric gluing construction of Definition \ref{Defn:connectedsum} and Proposition \ref{Prop:decompknot-extconnectedsum}, in that we identify the neighbourhoods of a meridian of either knot in order to combine the fundamental cobordisms of two knots to form their sum.  The algebraic sum is well-defined, in that it does not depend on a choice of meridian, as shown by the chain homotopies $\mu$ and $\mu^{\dag}$ and Lemma \ref{Lemma:homotopicmapscones} which says that chain homotopic maps have isomorphic mapping cones.  Furthermore, an equivalence of knots, or a different choice of handle decomposition of our knot exterior, or of the chain level diagonal approximation $\Delta$ which we use in the symmetric construction, produces equivalent fundamental symmetric Poincar\'{e} triads.  As in the proof of Proposition \ref{prop: fundtriaddefinesanelement}, different choices of $g_1$ and $l_a$ also yield equivalent triads: these choices were only necessary to explicitly write down the model chain complexes.  We therefore have a well defined homomorphism of abelian monoids $Knots \to \P$ as claimed.
\end{proof}

\begin{remark}
We hope that our method of adding knots together algebraically is an improvement on previous methods.  The common geometric method (see e.g \cite[pages~313--4]{Gilmer}) is to add a zero framed unknot which links both knots to a diagram, and then show by Kirby moves that this is a surgery diagram for the zero framed surgery on the connected sum.  One can then calculate the effect of a single surgery on homology groups.  As well as operating at the level of chain complexes rather than on the level of homology, our method keeps a tight control on the peripheral structure.  It also has the advantage that it does not simply define algebraic addition by direct sum.  This would crucially destroy the property of being a $\Z$-homology circle; rather we combine the generators of the $\Z$-homology in our gluing operation to preserve this property.  See \cite[section~4]{Ranicki4} for the expression of the high-dimensional knot concordance groups as certain Witt groups with addition by direct sum.  In order to do this Ranicki formally inverts the element $1-t$ of the group ring $\Z[\Z] = \Z[t,t^{-1}]$, which has the effect of killing the meridian of the knot algebraically, so that it is then not necessary to identify the meridians in the addition operation.  In the high-dimensional setting we work over the group ring $\Z[\Z]$, since any high-dimensional knot is concordant to one whose knot group is just $\Z$.  As this is absolutely not the case with knots in dimension 3, this procedure does not analogously apply when we work further up the derived series than $\pi_1(X)/\pi_1(X)^{(1)} \cong \Z$.  Our remedy is this more sophisticated addition of symmetric triads.

It is a special feature of the monoid of homology cylinders from $S^1 \times D^1$ to itself that it is abelian.  See e.g. \cite{chafriedlkim} for a definition and study of homology cylinders.  The monoid of homology cylinders from a surface $F$ to itself, where $F$ is not homeomorphic to $S^1 \times D^1$, will not typically be abelian.  The feature here is that the product cobordism of the boundary $S^1 \times S^0$ is also two copies of $S^1 \times D^1$, so that the boundaries can be slid around by an isotopy to swap them: this is the geometric idea behind the homotopy $\mu$ which we used to represent the fact that the element of $\P$ is independent of the choice of $f_{-}$ and $f_+$.
\end{remark}

The next step is to impose a further concordance relation on our monoid of symmetric Poincar\'{e} triads, and so turn it into a group.  First, we motivate the algebraic concordance relation which we will introduce by recalling some knot concordance theory, in particular the work of \COT \cite{COT}, which was the principal motivation for this present project. 

\chapter[The Cochran-Orr-Teichner Filtration]{The Cochran-Orr-Teichner Filtration}\label{chapter:COTsurvey}

The work of \COT is the main background and motivation for this present work.  We aim to capture their obstruction theory using our symmetric Poincar\'{e} triads, so in this chapter we present a survey of their advances in knot concordance, as is principally contained in the main \COT paper \cite{COT}.  We also give, in \ref{defn:COTobstructionset_2}, the definition of the \COT obstruction set, which we denote $\mathcal{COT}_{(\C/1.5)}$.  The definition of the second level \COT obstructions depends on a choice.  The purpose of this set is to encapsulate the obstructions which result from all possible choices into a single algebraic object.

In order to make this chapter self-contained there is some overlap with the introduction.  This chapter owes a lot to lectures of Kent Orr which I attended in Heidelberg in December 2008 and to lecture notes of Peter Teichner from San Diego in 2001 which Julia Collins and I typed up \cite{sliceknots2}.

Experts may wish to skip to Chapter \ref{chapter:algconcordance}, and then return to Definition \ref{defn:COTobstructionset_2} in order to read Chapter \ref{Chapter:extractingCOTobstructions}.

\begin{definition}\cite{fm}
An oriented knot $K \colon S^1 \subset S^3$ is topologically \emph{slice} if there is an oriented embedded locally flat disk $D^2 \subseteq D^4$ whose boundary $\partial D^2 \subset \partial D^4 = S^3$ is the knot $K$.  Here locally flat means locally homeomorphic to a standardly embedded $\R^2 \subseteq \R^4$.

Two knots $K_1, K_2 \colon S^1 \subset S^3$ are \emph{concordant} if there is an embedded locally flat oriented annulus $S^1 \times I \subset S^3 \times I$ such that $\partial (S^1 \times I)$ is $K_1 \times \{0\} \subseteq S^3 \times \{0\}$ and $-K_2 \times \{1\} \subset S^3 \times \{1\}$.   Given a knot $K$, the knot $-K$ arises by reversing the orientation of the knot and of the ambient space $S^3$: on diagrams reversing the orientation of $S^3$ corresponds to switching under crossings to over crossings and vice versa.  The set of concordance classes of knots form a group $\C$ under the operation of connected sum with the identity element given by the class of slice knots, or knots concordant to the unknot.
\qed \end{definition}

\section{The geometric filtration of the knot concordance group}

\begin{definition}
We recall the definition of the zero--framed surgery along $K$ in $S^3$, $M_K$: attach a solid torus to the boundary of the knot exterior $X=\cl(S^3 \setminus (K(S^1) \times D^2))$ in such a way that the zero--framed longitude of the knot bounds in the solid torus.
\[M_K := X \cup_{S^1 \times S^1} D^2 \times S^1.\]
The homology groups of $M_K$ are given by:
\[H_i(M_K;\Z) \cong \Z \text{ for } i = 0,1,2,3;\]
and are $0$ otherwise.  $H_1(M_K;\Z)$ is generated by a meridian of the knot, and $H_2(M_K;\Z)$ is generated by a Seifert surface for $K$ capped off with a disc in $D^2 \times S^1$.  The fundamental group is given by:
\[\pi_1(M_K) \cong \frac{\pi_1(X)}{\langle l \rangle}\]
where as before $l \in \pi_1(X)$ represents the longitude of $K$.
\qed \end{definition}

\COT \cite{COT} defined a geometric filtration of the knot concordance group which revealed the depth of its structure.  The filtration is based on the following characterisation of slice knots: notice that the exterior of a slice disc for a knot $K$ is a 4-manifold whose boundary is $M_K$: where the extra $D^2 \times S^1$ which is glued onto the knot exterior $X$ is the boundary of a regular neighbourhood of a slice disc.

\begin{proposition}\label{basicfact}
A knot $K$ is topologically slice if and only if $M_K$ bounds a topological 4-manifold $W$ such that
\begin{description}
\item[(i)]$i_* \colon H_1(M_K;\Z) \xrightarrow{\simeq} H_1(W;\Z)$ where $i \colon M_K \hookrightarrow W$ is the inclusion map;
\item[(ii)] $H_2(W;\Z) \cong 0$; and
\item[(iii)] $\pi_1(W)$ is normally generated by the meridian of the knot.
\end{description}
\end{proposition}
\begin{proof}
The exterior of a slice disc $D$, $W:= \cl(D^4 \setminus (D \times D^2))$, satisfies all the conditions of the proposition, as can be verified using Mayer-Vietoris and Seifert-Van Kampen arguments on the decomposition of $D^4$ into $W$ and $D \times D^2$.  Conversely, suppose we have a manifold $W$ which satisfies all the conditions of the proposition.  Glue in $D^2 \times D^2$ to the $D^2  \times S^1$ part of $M_K$.  This gives us a 4-manifold $W'$ with $H_*(W';\Z) \cong H_*(D^4;\Z)$, $\pi_1(W') \cong 0$ and $\partial W' = S^3$, so $K$ is slice in $W'$.  We can then apply Freedman's topological $h$-cobordism theorem \cite{FQ} to show that $W' \approx D^4$ and so $K$ is in fact slice in $D^4$.
\end{proof}

To filter the condition of sliceness with a geometric obstruction theory, \COT look for 4-manifolds which could potentially be changed to make a slice disk exterior.  We start with a 4-manifold $W$ with $\partial W = M_K$ which satisfies conditions (i) and (iii) of Proposition \ref{basicfact} and aim to perform homology surgery with respect to a circle; that is we aim to perform surgery on embedded 2-spheres in $W$ in order to kill $H_2(W;\Z)$ and obtain a $\Z$-homology circle.  Typically classes in $H_2(W;\Z)$ will be represented by immersed spheres or embedded surfaces of non-zero genus rather than by embedded spheres. We can measure how close we are to being able to kill $H_2(W;\Z)$ by surgery by looking at the middle-dimensional equivariant intersection form:
\[\lambda \colon H_2(W;\Z[\pi_1(W)]) \times H_2(W;\Z[\pi_1(W)]) \to \Z[\pi_1(W)].\]
Using coefficients in $\pi_1(W)$ allows us to detect surfaces and their intersections.  The next problem comes from the fact that there can be a large variation in $\pi_1(W)$ for different choices of $W$.  In order to define an obstruction theory, \COT take representations which factor through quotients by elements of the derived series to fixed groups $\G_{n-1}$:
\[\rho_{n-1} \colon \pi_1(W) \to \frac{\pi_1(W)}{\pi_1(W)^{(n)}} \to \G_{n-1}.\]
If there is an embedded surface $N \subseteq W$ with $\pi_1(N) \unlhd \pi_1(W)^{(n)}$, called an \emph{$(n)$-surface}, then as far as the $n$th level intersection form
\begin{multline*} \lambda_n \colon H_2(W;\Z[\pi_1(W)/\pi_1(W)^{(n)}]) \times H_2(W;\Z[\pi_1(W)/\pi_1(W)^{(n)}]) \\ \to \Z[\pi_1(W)/\pi_1(W)^{(n)}]\end{multline*}
can see we have an embedded sphere.  Of course it may not actually be embedded, but in this way \COT obtain calculable obstructions.  For $n=1$, this is  essentially the Cappell-Shaneson technique for obstructing the concordance of high-dimensional knots $S^m \subseteq S^{m+2}$.  We now give the definition of the \COT filtration:

\begin{definition}\label{Defn:COTnsolvable} \cite[Definition~1.2]{COT}
A \emph{Lagrangian} of a symmetric form $\lambda \colon P \times P \to R$ on a free $R$-module $P$ is a submodule $L \subseteq P$ of half-rank on which $\lambda$ vanishes.  For $n \in \mathbb{N}_0 := \mathbb{N} \cup \{0\}$, let $\lambda_n$ be the intersection form, and $\mu_n$ the self-intersection form, on the middle dimensional homology $H_2(W^{(n)};\Z) \cong H_2(W;\Z[\pi_1(W)/\pi_1(W)^{(n)}])$  of the $n$th derived cover of a 4-manifold $W$, that is the regular covering space $W^{(n)}$ corresponding to the subgroup $\pi_1(W)^{(n)} \leq \pi_1(W)$:
\begin{multline*}\lambda_n \colon H_2(W;\Z[\pi_1(W)/\pi_1(W)^{(n)}]) \times H_2(W;\Z[\pi_1(W)/\pi_1(W)^{(n)}]) \\ \to \Z[\pi_1(W)/\pi_1(W)^{(n)}].\end{multline*}
An $(n)$-\emph{Lagrangian} is a submodule of $H_2(W;\Z[\pi_1(W)/\pi_1(W)^{(n)}])$, on which $\lambda_n$ and $\mu_n$ vanish, which maps via the covering map onto a Lagrangian of $\lambda_0$.

We say that a knot $K$ is \emph{$(n)$-solvable} if $M_K$ bounds a topological spin 4-manifold $W$ such that the inclusion induces an isomorphism on first homology and such that $W$ admits two dual \emph{$(n)$-Lagrangians}.  In this setting, dual means that $\lambda_n$ pairs the two Lagrangians together non-singularly and their images freely generate $H_2(W;\Z)$.

We say that $K$ is \emph{$(n.5)$-solvable} if in addition one of the $(n)$-Lagrangians is the image of an $(n+1)$-Lagrangian.
\qed \end{definition}

\begin{remark}\label{Rmk:nsolvablity}
This filtration of the knot concordance group relates strongly to geometric filtrations using gropes and Whitney towers (see \cite[Section~8]{COT} for more information), objects which feature prominently in the theory of the classification of 4-manifolds (see e.g. \cite{FQ}).  A slice knot is $(n)$-solvable for all $n\in \mathbb{N}_0$ by Proposition \ref{basicfact}, and it is hoped, but not known to be true, that if a knot is $(n)$-solvable for all $n$ then it is topologically slice.

A knot is $(0)$-solvable if and only if its Arf invariant vanishes, and $(0.5)$-solvable if and only if it is algebraically slice i.e. its Seifert form is null-concordant.

The size of an $(n)$-Lagrangian is controlled only by its image under the map induced by the covering map $W^{(n)} \to W$ in $H_2(W;\Z)$; the intersection forms of $W^{(n)}$ are typically singular due to the presence of the boundary $M_K$.  The requirement roughly speaking is that we have a Lagrangian of $\lambda_n$ on the non-singular part of $H_2(W^{(n)};\Z)$.  We can see from the long exact sequence of a pair that the intersection form is non-singular on the part of $H_2(W^{(n)};\Z)$ which neither lies in the image under inclusion of $H_2(\partial W^{(n)};\Z)$ nor is Poincar\'{e} dual to a relative class in $H_2(W^{(n)},\partial W^{(n)};\Z)$ which has non-zero boundary in $H_1(\partial W^{(n)};\Z)$.  The existence of a dual $(n)$-Lagrangian means that we have a non-singular part of sufficient size.  The dual $(n)$-Lagrangian maps to a dual Lagrangian of $\lambda_0$, implying that the form $\lambda_0$ is hyperbolic on $H_2(W;\Z)$ (see \cite[Remark~7.6]{COT} for the required basis change), which is a necessary condition if we wish to modify $W$ by surgery into a homology circle.

Note that $H_2(W;\Z)$ is a free module since if it had torsion this would appear in $H^3(W;\Z)$ by universal coefficients.  However $H^3(W;\Z)$ is isomorphic to $H_1(W,M_K;\Z)$ by Poincar\'{e} duality, which is zero by the long exact sequence of a pair since the inclusion of the boundary $M_K$ into $W$ induces an isomorphism on first homology.

The dual classes are very important.  When looking for a half basis of embedded spheres, or perhaps just of $(n)$-surfaces, as candidates for surgery, or when looking for embedded gropes, we can use the duals to remove unwanted intersections between surfaces by tubing between an intersection point and the intersection of one of the surfaces with its dual - see \cite[Section~8]{COT}.  If we achieve a half basis of framed embedded spheres, when doing surgery on a such a 2-sphere $S^2 \times D^2$, and replacing it with $D^3 \times S^1$, without the existence of dual classes we would create new classes in $H_1(W;\Z)$.  This would ruin the condition that $i_* \colon H_1(M_K;\Z) \xrightarrow{\simeq} H_1(W;\Z)$ is an isomorphism, which is necessary for a 4-manifold to be a slice disc complement.

We ask for the 4-manifold to be \emph{spin} so that the self-intersection forms $\mu_n$ can be well-defined on homology classes of $H_2(W^{(n)};\Z)$ - see \cite[Section~7]{COT}.  The self-intersection form is crucial in surgery theory for keeping track of the bundle data.  The \COT obstructions do not depend on it and we have not yet included it into our algebraic framework, but hope to achieve this in the future in order to capture the geometric $(n)$-solvability criteria as closely as possible.
\end{remark}

Our aim in chapter \ref{chapter:algconcordance} will be to introduce an algebraic concordance relation on the elements of $\P$ which closely captures the notion of $(1.5)$-solvability\footnote{We hope that this relation will generalise to capture the notion of $(n.5)$-solvability for any $n \in \mathbb{N}_0$ -- see Appendix \ref{Appendix:nth_order_group}.}.  Just as Gilmer \cite{Gilmer} defined a group\footnote{Although unfortunately \cite[page~43]{Friedl}, there is a gap in Gilmer's proofs.} which aimed to capture the algebraic concordance group of Levine \cite{Levine} and the Casson-Gordon invariants \cite{CassonGordon} in a single algebraic object, we aim to define a group which captures the \COT filtration levels of $(0)$, $(0.5)$, $(1)$ and $(1.5)$-solvability in a single stage, in the sense that the \COT obstructions vanish if a knot is algebraically $(1.5)$-solvable (again, a notion to be defined in chapter \ref{chapter:algconcordance}) which in turn holds if a knot is geometrically $(1.5)$-solvable.

\section{The Cochran-Orr-Teichner obstruction theory}\label{Chapter:COTobstructionthy}

We now describe the obstruction theory of \COT \cite{COT} which they use to detect that certain knots are not $(1.5)$-solvable - and indeed in \cite{CochranTeichner} that certain knots are $(n)$-solvable but not $(n.5)$-solvable for any $n \in \mathbb{N}_0$, but we focus on $(1.5)$-solvability for this exposition.  To define their obstructions, \COT have representations $\rho$ of the fundamental group $\pi_1(M_K)$ of $M_K$ which extend to representations of $\pi_1(W)$ for $(1)$-solutions $W$:
\[\xymatrix{\pi_1(M_K) \ar[rr]^{i_*} \ar[dr]_{\rho}  && \pi_1(W) \ar[dl]^{\wt{\rho}} \\ & \Gamma &
}\]
where $\partial W = M_K$ and \[\Gamma = \Gamma_1 := \Z \ltimes \frac{\Q(t)}{\Q[t,t^{-1}]},\]
their \emph{universally $(1)$-solvable group}, where, to define the semi-direct product, $n \in \Z$ acts by left multiplication by $t^n$.  The representation:
\[\rho \colon \pi_1(M_K) \to \pi_1(M_K)/\pi_1(M_K)^{(2)} \to \Z \ltimes H_1(M_K;\Q[t,t^{-1}]) \to \Z \ltimes \frac{\Q(t)}{\Q[t,t^{-1}]}\]
is given by:
\[g \mapsto (n := \phi(g),h := gt^{-\phi(g)}) \mapsto (n,\Bl(p,h)),\]
where $\phi \colon \pi_1(M_K) \to \Z$ is the abelianisation homomorphism and $t$ is a preferred meridian in $\pi_1(M_K)$.  The pairing $\Bl$ is the Blanchfield form (Definition \ref{Defn:Blanchfieldform} below), and $p$ is an element of $H_1(M_K;\Q[t,t^{-1}])$ chosen to lie in a metaboliser of the Blanchfield form so that the representation extends over the 4-manifold $W$ (see Theorem \ref{Lemma:COT4.4}).

\begin{definition}\label{Defn:Blanchfieldform}
The rational \emph{Blanchfield form} is the non-singular Hermitian pairing
\[\Bl \colon H_1(M_K;\Q[\Z]) \times H_1(M_K;\Q[\Z]) \to \Q(\Z)/\Q[\Z] = \Q(t)/\Q[t,t^{-1}]\]
which is defined by the sequence of isomorphisms:
\begin{multline*}\ol{H_1(M_K;\Q[\Z])} \xrightarrow{\simeq} H^2(M_K;\Q[\Z]) \xrightarrow{\simeq} H^1(M_K;\Q(\Z)/\Q[\Z]) \\ \xrightarrow{\simeq} \Hom_{\Q[\Z]}(H_1(M_K;\Q[\Z]),\frac{\Q(\Z)}{\Q[\Z]}).\end{multline*}
The first isomorphism is Poincar\'{e} duality: this involves the involution on the group ring to convert right modules to left modules.  The second isomorphism is the inverse of a Bockstein homomorphism: associated to the short exact sequence of coefficient groups
\[0 \to \Q[t,t^{-1}] \to \Q(t) \to \Q(t)/\Q[t,t^{-1}] \to 0,\]
is a long exact sequence in cohomology
\[H^1(M_K;\Q(t)) \to H^1(M_K;\Q(t)/\Q[t,t^{-1}]) \xrightarrow{\beta} H^2(M_K;\Q[t,t^{-1}]) \to H^2(M_K;\Q(t)).\]
The homology $H_1(M_K;\Q[t,t^{-1}])$ is a torsion $\Q[t,t^{-1}]$-module, with the Alexander polynomial annihilating the module.  As a $\Q[t,t^{-1}]$-module, $\Q(t)$ is flat, so
\[H_1(M_K;\Q(t)) \cong \Q(t) \otimes_{\Q[t,t^{-1}]} H_1(M_K;\Q[t,t^{-1}]) \cong 0.\]
Then on the one hand universal coefficients, since $\Q(t)$ is a field, and on the other hand Poincar\'{e} duality, shows that:
\[H^1(M_K;\Q(t)) \cong \Hom_{\Q(t)}(H_1(M_K;\Q(t)),\Q(t)) \cong 0\]
and
\[H^2(M_K;\Q(t)) \cong H_1(M_K;\Q(t)) \cong 0,\]
which together imply that Bockstein homomorphism $\beta$ is an isomorphism.  The final isomorphism:
\[H^1(M_K;\Q(\Z)/\Q[\Z]) \xrightarrow{\simeq} \Hom_{\Q[\Z]}(H_1(M_K;\Q[\Z]),\Q(\Z)/\Q[\Z])\]
is given by the universal coefficient theorem.  This applies since $\Q[\Z]$ is a principal ideal domain.  To see that the map is an isomorphism we need to see that:
\[\Ext_{\Q[\Z]}^1(H_0(M_K;\Q[\Z]),\Q(\Z)/\Q[\Z]) \cong 0.\]
Now, $\Q(\Z)/\Q[\Z]$ is a divisible $\Q[\Z]$-module, which implies that it is injective since $\Q[\Z]$ is a PID (see \cite[I.6.10]{Stenstrom}).  We can calculate the $\Ext$ groups using the injective resolution of $\Q(\Z)/\Q[\Z]$ of length 0 (see \cite[IV.8]{HS}), which implies that
\[\Ext^j_{\Q[\Z]}(A,\Q(\Z)/\Q[\Z]) \cong 0\] for $j \geq 1$ for any $\Q[\Z]$-module $A$; in particular this holds for $A=H_0(M_K;\Q[\Z])$.
So indeed we have an isomorphism from the universal coefficient theorem as claimed and the rational Blanchfield form is non-singular.  For the improvements necessary to see that the Blanchfield form is also non-singular with $\Z[\Z]$ coefficients see \cite{Levine2}.


We say that the Blanchfield pairing is \emph{metabolic} if it has a metaboliser.  A \emph{metaboliser} for the Blanchfield form is a submodule $P \subseteq H_1(M_K;\Q[\Z])$ such that:
\[P=P^{\bot}:= \{v \in H_1(M_K;\Q[\Z])\,|\, \Bl(v,w) = 0 \text{ for all } w \in P\}.\]
\qed \end{definition}

One of the key theorems of \COT is \cite[Theorem~4.4]{COT}.  They show, using duality, that the kernel of the inclusion induced map:
\[i_* \colon H_1(M_K;\Q[\Z]) \to H_1(W;\Q[\Z]),\]
for $(1)$-solutions $W$, is a metaboliser for the Blanchfield form.  This then implies that choices of $p \in H_1(M_K;\Q[\Z])$ control which representations of the form of $\rho$ extend over the $(1)$-solution, where $p$ is in the definition of $\rho$.  Choosing $p \in P$, where $P$ is a metaboliser for the Blanchfield form, is necessary for the representation to extend to $\pi_1(W)$.  This is very useful for applications, since the Blanchfield form can be calculated explicitly for a given knot; one method \cite{Kearton} calculates the form in terms of a Seifert matrix.  The philosophy is that linking information in the 3-manifold controls intersection information in the 4-manifold.  Note that we require that the Blanchfield form is metabolic, i.e. that the first order obstruction vanishes, in order for the representation $\wt{\rho}$ and thence the second order obstruction to be defined.  This is the weakness of homology pairings which we avoid by working at the chain level.  We give the proof of the following theorem in full since it is a crucial argument and since we will need to construct an analogous argument in due course from our chain complexes.
\begin{theorem}[\cite{COT} Theorem 4.4]\label{Lemma:COT4.4}
Suppose $M_K$ is $(1)$-solvable via $W$.  Then if we define:
\[P:= \ker(i_* \colon H_1(M_K;\Q[\Z]) \to H_1(W;\Q[\Z])),\]
then the rational Blanchfield form $\Bl$ of $M_K$ is metabolic and in fact $P=P^{\bot}$ with respect to $\Bl$.
\end{theorem}

\begin{proof}
Before proving the theorem, \COT state in their Lemma 4.5, whose proof we only sketch, that the sequence:
\[TH_2(W,M_K;\Q[\Z]) \xrightarrow{\partial} H_1(M_K;\Q[\Z]) \xrightarrow{i_*} H_1(W;\Q[\Z]) \]
is exact, where $TH_2$ denotes the torsion part of the second homology.  The idea is that the $(1)$-Lagrangian and its duals generate the free part of $H_2(W;\Q[\Z])$.  More precisely, the existence of the duals is used to show that the intersection form:
\[\lambda_1 \colon H_2(W;\Q[\Z]) \to \Hom_{\Q[\Z]}(H_2(W;\Q[\Z]),\Q[\Z])\]
is surjective.  We consider those classes in $H_2(W,M_K;\Q[\Z])$ which map to zero under the composition of Poincar\'{e} duality and universal coefficients:
\[\kappa \colon H_2(W,M_K;\Q[\Z]) \xrightarrow{\simeq} H^2(W;\Q[\Z]) \to \Hom_{\Q[\Z]}(H_2(W;\Q[\Z]),\Q[\Z]).\]
Since the first map is an isomorphism and the final group is free, the kernel of this composition is torsion.  An element $p$ of $\ker(i_*)$ lifts to a relative class $x \in H_2(W;M_K;\Q[\Z])$.  We can remove any part of this which comes from $H_2(W;\Q[\Z])$ without affecting the boundary of $x$.  Choose $$y \in \lambda_1^{-1} (\kappa (x)) \subseteq H_2(W;\Q[\Z]),$$
and let $j_*(y)$ be its image in $H_2(W,M_K;\Q[\Z])$.  Then $x-j_*(y)$ lies in the kernel of $\kappa$, so is therefore torsion.  The boundary of $x-j_*(y)$ is still $p$.  This completes our sketch proof of \cite[Lemma~4.5]{COT}.

Next, we construct a non-singular \emph{relative linking pairing} on the 4-manifold with boundary $W$.
\[\beta_{rel} \colon \ol{TH_2(W,M_K;\Q[\Z])} \times H_1(W;\Q[\Z]) \to \Q(\Z)/\Q[\Z].\]
This is defined in a similar manner to the Blanchfield pairing on $M_K$; we use the composition of isomorphisms:
\begin{multline*} \ol{TH_2(W,M_K;\Q[\Z])} \xrightarrow{\simeq} TH^2(W;\Q[\Z]) \xrightarrow{\simeq} H^1(W,\Q(\Z)/\Q[\Z]) \\ \xrightarrow{\simeq} \Hom_{\Q[\Z]}(H_1(W;\Q[\Z]),\frac{\Q(\Z)}{\Q[\Z]})\end{multline*}
where as before these are given by Poincar\'{e} duality, a Bockstein homomorphism, and the universal coefficient theorem.  To see that the second map is an isomorphism, as before we have a long exact sequence with connecting homomorphism given by the Bockstein:
\[H^1(W;\Q(\Z)) \to H^1(W;\Q(\Z)/\Q[\Z]) \xrightarrow{b} H^2(W;\Q[\Z]) \to H^2(W;\Q(\Z)).\]
\cite[Proposition~2.11]{COT} says here that $H^1(W;\Q(\Z))\cong 0$, while $H^2(W;\Q(\Z))$ is $\Q[\Z]$-torsion free, so it follows that we have an isomorphism:
\[b^{-1} \colon TH^2(W;\Q[\Z]) \xrightarrow{\simeq} H^1(W,\Q(\Z)/\Q[\Z]).\]
The universal coefficients argument for the final map runs parallel to the corresponding argument for $M_K$ in Definition \ref{Defn:Blanchfieldform}.

We now make use of our non-singular pairings $\Bl$ and $\beta_{rel}$ in the following commuting diagram: all coefficients are taken to be $\Q[\Z]$ and the functor $\bullet^{\wedge}$ is the Pontryagin dual:
\[\bullet^{\wedge} := \ol{\Hom_{\Q[\Z]}(\bullet, \Q(\Z)/\Q[\Z])}.\]
We have:
\[\xymatrix @R+1cm @C+1cm{TH_2(W,M_K) \ar[r]^-{\partial_*} \ar[d]^{\cong}_{\beta_{rel}} & H_1(M_K) \ar[d]^{\cong}_{\Bl} \ar[r]^-{i_*} & H_1(W) \ar[d]^{\cong}_{\beta_{rel}} \\
H_1(W)^{\wedge} \ar[r]^-{i^{\wedge}} & H_1(M_K)^{\wedge} \ar[r]^-{\partial^{\wedge}} & TH_2(W,M_K)^{\wedge}.
}\]
We have shown above that the rows are exact and that the vertical maps are isomorphisms.  We show that $P:= \ker(i_*) \subseteq P^{\bot}$.  Let $x,y \in P$.  Then
\[i_*(x) =0\]
so there is a $w \in TH_2(W,M_K;\Q[\Z])$ such that $\partial (w) = x$.  By commutativity of the diagram above have that:
\[i^{\wedge} \circ \beta_{rel} (w) = \Bl(\partial(w)) = \Bl(x).\]
But also:
\[i^{\wedge} \circ \beta_{rel}(w) = \beta_{rel}(w) \circ i_*\]
which implies that:
\[\Bl(x) = \beta_{rel}(w) \circ i_*,\]
so
\[\Bl(x,y) = \Bl(x)(y) = \beta_{rel}(w)(i_*(y)) = \beta_{rel}(w)(0) = 0\]
since also $y \in P$.  Therefore $x \in P^{\bot}$ and $P \subseteq P^{\bot}$.

We now show that $P^{\bot} \subseteq P$.  Since $P = \ker(i_*)$ we have an induced monomorphism:
\[i_* \colon H_1(M_K;\Q[\Z])/P \to H_1(W;\Q[\Z])\]
As in Definition \ref{Defn:Blanchfieldform}, $\Q(\Z)/\Q[\Z]$ is a divisible $\Q[\Z]$-module, so it is injective since $\Q[\Z]$ is a PID (see \cite[I.6.10]{Stenstrom}).  This means that taking duals, we have that:
\[i^{\wedge} \colon H_1(W;\Q[\Z])^{\wedge} \to (H_1(M_K;\Q[\Z])/P)^{\wedge}\]
is surjective.
Let $x \in P^{\bot}$, so by definition $\Bl(x,y) = 0$ for all $y \in P$.  Therefore we can lift $\Bl(x) \in H_1(M_K;\Q[\Z])^{\wedge}$ to an element of $(H_1(M_K;\Q[\Z])/P)^{\wedge}$ and therefore to an element of $H_1(W;\Q[\Z])^{\wedge}$ since $i^{\wedge}$ is surjective.  As $\Bl(x) \in \im(i^{\wedge})$, and since the vertical maps of our diagram above are isomorphisms, we see that $x \in \im(\partial)$.  This means that $x \in P$ by exactness of the top row, so $P^{\bot} \subseteq P$ as claimed.
\end{proof}

\cite[Theorem~3.6]{COT} then shows that the representations $\rho \colon \pi_1(M_K) \to \Gamma$ with $p \in P$ extend over $\pi_1(W)$.

\begin{remark}
Note that we deliberately work over the PID $\Q[\Z]$ in the above argument, and that this is vital for the deductions in several instances.  There is always the problem in knot concordance that we do not know that $i_* \colon \pi_1(M_K) \to \pi_1(W)$ is surjective.  For ribbon knot exteriors $W$, this is the case, but otherwise we cannot guarantee a surjection.

In the case that $i_*$ were surjective for $(1)$-solutions $W$, $\pi_1(W)$ would be simply a quotient of $\pi_1(M)$, and then there would be  no need to localise coefficients; we would have that:
\[P:= \ker (i_* \colon H_1(M_K;\Z[\Z]) \to H_1(W;\Z[\Z]))\]
is a metaboliser for the Blanchfield form on $H_1(M_K;\Z[\Z])$.  However, one main reason for considering coefficients in $\Q[\Z]$ is that there could conceivably be $\Z$-torsion in $H_1(W;\Z[\Z])$, in which case the best we could hope to show is that $P \subset P^{\bot}$.  In order to get a metaboliser we need to introduce:
\[Q:=\{x \in H_1(M_K;\Z[\Z]) \,|\, nx \in P \text{ for some }n \in \Z\}.\]
Then $Q= Q^{\bot}$ with respect to the Blanchfield form (\cite[Proposition~2.7]{Friedl}, see also \cite{Letsche}).   For the proof of Theorem \ref{Lemma:COT4.4} above however, this means we lose control on the \emph{size} of $P$; the zero submodule also satisfies $P \subset P^{\bot}$.  If $i_* \colon \pi_1(M) \to \pi_1(W)$ is onto, then as in \cite[Proposition~6.3]{Friedl}, there is no $\Z$-torsion in $H_1(W;\Z[\Z])$.  Since we only know this to be the case for ribbon knots, \COT localise coefficients in order to get a principal ideal domain $\Q[\Z]$.  Since it is intimately related to the ribbon-slice problem this problem of $\Z$-torsion is often also referred to as a ribbon-slice problem.
\end{remark}

Now suppose that there is $(1)$-solution $W$.  Then for each $p \in P = \ker(i_*)$ we have a representation \[\wt{\rho} \colon \pi_1(W) \to \Gamma = \Z \ltimes \Q(\Z)/\Q[\Z]\] which enables us to define the intersection form:
\[\lambda_1 \colon H_2(W;\Q\Gamma) \times H_2(W;\Q\Gamma) \to \Q\Gamma.\]
$W$ is a manifold with boundary, so in general this will be a singular intersection form.  To define a non-singular form we localise coefficients: \COT use the non-commutative \emph{Ore localisation} to formally invert all the non-zero elements in $\Q\G$ to obtain a skew-field $\K$.

\begin{definition}\label{Defn:OreLocalisation}
A ring $A$ satisfies the \emph{Ore condition}, which defines when a multiplicative subset $S$ of a non-commutative ring without zero-divisors can be formally inverted, if, given $s \in S$ and $a \in A$, there exists $t \in S$ and $b \in A$ such that $at=sb$.  Then the Ore localisation $S^{-1}A$ exists.  If $S = A-\{0\}$ then $S^{-1}A$ is a skew-field which we denote by $\mathcal{K}(A)$, or sometimes just $\K$ if $A$ is understood.
\qed \end{definition}

See \cite[Chapter~2]{Stenstrom} for more details on the Ore condition.  Ore localisation is flat so \[H_2(W;\K) \cong \K \otimes_{\Q\G} H_2(W;\Q\G).\]
The idea is that if the homology of the boundary $M_K$ vanishes with $\K$ coefficients, as is proved in \cite[Section~2]{COT}, then the intersection form on the middle homology of $W$ becomes non-singular, and we have defined an element in the Witt group of non-singular Hermitian forms over $\K$.  Moreover, control over the size of the $\Z$-homology translates into control over the size of the $\K$-homology of $W$.  To explain how this gives us a well--defined obstruction, which does not depend on the choice of 4-manifold, and how this obstruction lives in a group, we define $L$-groups and the localisation exact sequence in $L$-theory.

\begin{definition}[\cite{Ranicki3} I.3]\label{Defn:Lgroups}
Two $n$--dimensional $\eps$--symmetric Poincar\'{e} finitely generated projective $A$-module chain complexes $(C,\varphi)$ and $(C',\varphi')$ are \emph{cobordant} if there is an $(n+1)$-dimensional $\eps$-symmetric Poincar\'{e} pair:
\[(f,f') \colon C \oplus C' \to D, (\delta\varphi,\varphi \oplus -\varphi').\]
The union operation of Definition \ref{Defn:unionconstruction} shows that cobordism of chain complexes is a transitive relation.  The equivalence classes of symmetric Poincar\'{e} chain complexes under the cobordism relation form a group $L^n(A,\eps)$, with
\[(C,\varphi) + (C',\varphi') = (C \oplus C',\varphi \oplus \varphi'); \; -(C,\varphi) = (C,-\varphi).\]
As usual if we omit $\eps$ from the notation we assume that $\eps=1$.  In the case $n=0$, $L^0(A)$ coincides with the Witt group of non-singular Hermitian forms over $A$.
\qed \end{definition}

Note that an element of an $L$-group is in particular a symmetric \emph{Poincar\'{e}} chain complex.  This means that the intersection forms of our 4-manifolds $W$ typically give elements of $L^0(\K)$ but not of $L^0(\Q\Gamma)$.

\begin{definition}[\cite{Ranicki2} Chapter 3]\label{defn:localisationexactsequence}
The \emph{Localisation Exact Sequence in $L$-theory} is given, for a ring $A$ and a multiplicative subset $S$ which satisfies the Ore condition, as follows:
\[ \cdots \to L^n(A) \to L_S^n(S^{-1}A) \to L^n(A,S) \to L^{n-1}(A) \to \cdots.\]
The relative $L$-groups $L^n(A,S)$ are defined to be the cobordism classes of $(n-1)$-dimensional symmetric Poincar\'{e} chain complexes over $A$ which become contractible over $S^{-1}A$, where the cobordisms are also required to be contractible over $S^{-1}A$.  For $n=2$ this is equivalent to the Witt group of $S^{-1}A/A$-valued linking forms on $H^1$ of the chain complex.

The decoration $S$ on $L^n_S(S^{-1}A)$ refers to a restriction on the class of modules involved in the chain complex.  Recall that, for a ring $A$, $K_0(A)$ is the Grothendieck group of isomorphism classes of finitely generated projective modules over $A$.  A ring homomorphism $g \colon A \to B$ induces a morphism $g \colon K_0(A) \to K_0(B)$ via $[P] \mapsto [B \otimes_{A} P]$.  The reduced $K_0$-groups are given by $\wt{K}_0(A):= K_0(A)/\im(K_0(\Z))$.  We define the subset \[S := \im(g \colon \wt{K}_0(A) \to \wt{K}_0(S^{-1}A)) \subset \wt{K}_0(S^{-1}A).\] We require that a chain complex $(C,\varphi) \in L^n_S(S^{-1}A)$ satisfies that the image of
\[\sum_{i}\, (-1)^i[C_i]\]
in $\wt{K}_0(S^{-1}A)$ lies in $S \subset \wt{K}_0(S^{-1}A)$,
so that an element $(C,\varphi) \in L^n_S(S^{-1}A)$ is chain equivalent to $S^{-1}(D,\phi) := (S^{-1}A \otimes_{A} D,\Id \otimes \phi)$ for a chain complex $D$ over $A$: $D$ is symmetric over $A$ but may not be Poincar\'{e} over $A$, so may not lift to an element of $L^n(A)$, as we shall see below.

The first map in the localisation sequence is given by considering a chain complex over the ring $A$ as a chain complex over $S^{-1}A$, by tensoring up using the inclusion $A \to S^{-1}A$.  The effect of this is that some maps become invertible which previously were not; when $n=4$, which is our primary case of interest, the boundary 3-dimensional chain complex, if it has torsion homology modules, becomes contractible and the middle-dimensional intersection form of the 4-dimensional chain complex $C_*(W,M_K;\K)$ becomes non-singular, so that we have a $4$-dimensional symmetric Poincar\'{e} complex.  The symmetric chain complex $C_*(W,M_K;\Q\G)$ does not typically lie in the image of this map, since its intersection form only becomes non-singular after localisation.  We say that a symmetric chain complex is $\K$-Poincar\'{e} if it is Poincar\'{e} after tensoring with $\K$.

The second map is the boundary construction: by definition of the $S$ decoration of $L^n_S(S^{-1}A)$ there is a chain complex which is chain equivalent to $(C_*,\varphi)$, in which all the maps are given in terms of $A$.  We may therefore assume that we have a symmetric but typically not Poincar\'{e} complex $(C_*,\varphi)$ over $A$, and take the mapping cone $\mathscr{C}(\varphi_0 \colon C^{n-*} \to C_*)$.  This gives an $(n-1)$-dimensional symmetric Poincar\'{e} chain complex over $A$ which becomes contractible over $S^{-1}A$, since $\varphi_0$ is a chain equivalence over $S^{-1}A$, i.e. we have an element of $L^n(A,S)$.  In our case we consider again $n=4$, and take $C_* = C_*(W,M_K;\Q\G)$; the boundary construction then yields a complex which is chain equivalent to $C_*(M_K;\Q\G)$.

On the level of Witt groups, this map sends a Hermitian $S^{-1}A$-non-singular intersection form over $A$, $$(L ,\lambda \colon L \to L^*),$$  to the linking form on $\coker \lambda \colon L \to L^*$ given by:
\[(x,y) \mapsto \frac{z(x)}{s}\]
where $x,y \in L^*, z \in L, sy=\lambda(z)$ \cite[pages~242--3]{Ranicki2}.

The third map is the forgetful map on the equivalence relation; it forgets the requirement that the cobordisms be contractible over $S^{-1}A$, simply asking for algebraic cobordisms over $A$.
\qed \end{definition}

The obstruction theory of \COTN, for suitable representations $\pi_1(M_K) \to \G$, detects the class of $C_*(M_K;\Q\Gamma)$ in $L^4(\Q\Gamma,S)$, where $S := \Q\G - \{0\}$; we have an invariant of the 3-manifold $M_K$ which does not depend on the choice of 4-manifold.  The first question we ask, corresponding to $(1)$-solvability, is whether the chain complex of $M_K$ bounds over $\Q\G$.  Supposing that it does, i.e. supposing that we have a $(1)$-solvable knot, we have a chain complex in
\[\ker(L^4(\Q\G,S) \to L^3(\Q\G)),\]
so we can express the group detecting that there is no $\K$-contractible null-cobordism of $C_*(M_K;\Q\G)$ as \[L^4(\K)/\im(L^4(\Q\G));\]
as noted above the intersection information of a 4-manifold is intimately related to the linking information of its boundary 3-manifold.

The $(1)$-solution $W$ defines an element of $L^4_S(\K)$ by taking the symmetric $\K$-Poincar\'{e} chain complex \[C_*(W,M_K;\K) = \K \otimes_{\Q\G} C_*(W,M_K;\Q\G).\]  The image of $L^4(\Q\G)$ represents the change corresponding to a different choice of 4-manifold $W$: the obstruction defined must be independent of this choice.
Since 2 is invertible in the rings $\K$ and $\Q\G$, we can do surgery below the middle dimension \cite[Part~I,~3.3~and~4.3]{Ranicki3} to see that our invariant lives in
\[\frac{L^0_S(\K)}{\im(L^0(\Q\G))}.\]
Taking two choices of 4-manifold $W,W'$ with boundary $M_K$ and gluing to form
\[V:= W \cup_{M_K} -W',\]
we obtain a 4-manifold whose image in $L^4(\Q\G) \cong L^0(\Q\G)$ gives the difference between the Witt classes of the intersection forms of $W$ and $W'$, showing that the invariant in $L^0_S(\K)/\im(L^0(\Q\G))$ is well-defined.  If this invariant vanishes then we can hope that the knot is slice or perhaps just $(1.5)$-solvable; more importantly if our class in $L^0_S(\K)/\im(L^0(\Q\G))$ does not vanish then it obstructs $(1.5)$-solvability and therefore in particular obstructs the possibility of the knot being slice.

The argument in \cite[page~458]{COT}, to show that the invariant is independent of the choice of $W$, uses $\Z\Gamma$ here instead of $\Q\G$.  In this case it is unclear that surgery below the middle dimension is possible on the chain complexes in symmetric $L$-theory, so that the symmetric signature of $V$ need not yield an element of $L^0(\Z\G)$.  Nothing is lost, however, by replacing $\Z$ with $\Q$ in their argument.
The main obstruction theorem of \COTN, at the $(1.5)$ level, is the following:

\begin{theorem}\label{Thm:COTmaintheorem}\cite[Theorem~4.2]{COT}
Let $K$ be a $(1)$-solvable knot.  Then there exists a metaboliser $P =P^\bot \subseteq H_1(M_K;\Q[\Z])$ such that for all $p \in P$, we obtain an obstruction
\[B:= (C_*(M_K;\Q\G),\backslash\Delta([M_K])) \in \ker(L^4(\Q\G,\Q\G-\{0\}) \to L^3(\Q\G)).\]
In addition, if $K$ is $(1.5)$-solvable, then $B=0$.
\end{theorem}

\begin{proof}
We give a sketch proof.  The fact that a meridian of $K$ maps non--trivially under $\rho$ is sufficient, as in \cite[Section~2]{COT}, to see that $C_*(M_K;\K) \simeq 0$, so that indeed $B \in L^4(\Q\G,\Q\G-\{0\})$.  The $(1)$-solvable condition ensures that certain representations extend over $\pi_1(W)$, for $(1)$-solutions $W$, so that $B \mapsto 0 \in L^3(\Q\G)$.  If $W$ is also a $(1.5)$-solution, there is a metaboliser for the intersection form on $H_2(W;\K)$: as mentioned above the fact that we have control over the rank of the $\Z$-homology translates into control on the rank of the $\K$-homology.  We have a half-rank summand on which the intersection form vanishes: it is therefore trivial in the Witt group $L^0_S(\K)$.  Since $L^4_S(\K) \cong L^0_S(\K)$ by surgery below the middle dimension, we indeed have $B=0$.
\end{proof}

We now define a pointed set, which is algebraically defined, which we call the \emph{\COT obstruction set}, and denote $(\mathcal{COT}_{(\C/1.5)},U)$.  The above exposition then enables us to define a map of pointed sets $\C/\mathcal{F}_{(1.5)} \to \mathcal{COT}_{(\C/1.5)}$: the \COT obstructions do not necessarily add well, so we are only able to consider pointed sets, requiring that $(1.5)$-solvable knots map to $U$, the marked point of $\mathcal{COT}_{(\C/1.5)}$.  The reason for this definition is that the second order \COT obstructions depend for their definitions on certain choices of the way in which the first order obstructions vanish.  More precisely, for each element $p \in H_1(M_K;\Q[\Z])$ we obtain a different representation $\pi_1(M_K) \to \G$ and therefore, if it is defined, a potentially different obstruction $B$ from Theorem \ref{Thm:COTmaintheorem}.  The following definition gives an algebraic object, $\mathcal{COT}_{(\C/1.5)}$, which encapsulates the choices in a single set.  Our second order algebraic concordance group $\ac2$, defined in Chapter \ref{chapter:algconcordance} as a quotient of $\P$, gives a single stage obstruction group from which an element of $\mathcal{COT}_{(\C/1.5)}$ can be extracted; for this see Chapter \ref{Chapter:extractingCOTobstructions}.  I would like to thank Peter Teichner for pointing out that I ought to make such a definition.

\textbf{Recommendation.}  The reader would perhaps be best served to skip Definition \ref{defn:COTobstructionset_2} on the first reading.  Chapters \ref{chapter:algconcordance} and \ref{chapter:one_point_five_solvable_knots} serve to provide important context for this definition, and knowledge of it is not required until Chapter \ref{Chapter:extractingCOTobstructions}.


In the following definition, for intuition, $(N,\theta)$ should be thought of as corresponding to the symmetric Poincar\'{e} chain complex of the zero surgery $M_K$ on a knot in $S^3$, $\G := \Z \ltimes \Q(t)/\Q[t,t^{-1}]$, and $H$ should be thought of as corresponding to $H_1(M_K;\Q[\Z])$.  There is no requirement that $(N,\theta)$ actually is the chain complex associated to a knot: we are working more abstractly.

\begin{definition}\label{defn:COTobstructionset_2}
Let $H$ be a rational Alexander module, that is a $\Q[\Z]$-module such that $H = \Q \otimes_{\Z} H'$ for some $H' \in \mathcal{A}$.  We denote the class of such $H$ by $\Q \otimes_{\Z} \mathcal{A}$.  Let \[\Bl\colon H \times H \to \Q(t)/\Q[t,t^{-1}]\] be a non-singular Hermitian pairing, and let $p \in H$.  We define the set:
\[L^4_{H,\Bl,p}(\Q\G,\Q\G-\{0\})\]
to comprise pairs $((N,\theta \in Q^3(N)),\xi)$,
where $(N,\theta)$ is a 3-dimensional symmetric Poincar\'{e} complex over $\Q\G$ which is contractible when tensored with the Ore localisation $\K$ of $\Q\G$:
\[\K \otimes_{\Q\G} N \simeq 0,\]
which satisfies:
\[H_*(\Q \otimes_{\Q\G} N) \cong H_*(S^1 \times S^2 ; \Q);\]
and where $\xi$ is an isomorphism
\[\xi \colon H \xrightarrow{\simeq} H_1(\Q[\Z] \otimes_{\Q\G} N).\]
Using the 3-dimensional symmetric Poincar\'{e} chain complex $(\Q[\Z] \otimes_{\Q\G} N,\Id \otimes \theta)$, we can define the rational Blanchfield form (see Proposition \ref{Prop:chainlevelBlanchfield}):
\[\wt{\Bl} \colon H_1(\Q[\Z] \otimes_{\Q\G} N) \times H_1(\Q[\Z] \otimes_{\Q\G} N) \to \Q(t)/\Q[t,t^{-1}].\]
We require that:
\[\Bl(x,y) = \wt{\Bl}(\xi(x),\xi(y))\]
for all $x,y \in H$.  We have a further condition that:
\begin{equation}\label{Eqn:common_boundary}((N,\theta),\xi)_0 \cong ((\Q[\Z] \otimes_{\Q\G} N,\Id \otimes \theta),\xi) \end{equation}
for $((N,\theta),\xi)_0 \in L^4_{H,\Bl,0}(\Q\G,\Q\G-\{0\})$.
We consider the union, for a fixed $H \in \Q \otimes_{\Z} \mathcal{A}$ and a fixed $\Bl \colon H \times H \to \Q(t)/\Q[t,t^{-1}]$:
\[\mathcal{AF}_{(\C/1.5)}(H,\Bl) := \bigsqcup_{p \in H} \, L^4_{H,\Bl,p}(\Q\G,\Q\G-\{0\}),\]
over all $p \in H$.
Next, we define a partial ordering on the class of certain special subsets of $\mathcal{AF}_{(\C/1.5)}(H,\Bl)$:
\[\bigcup_{\stackrel{H \in \Q \otimes_{\Z} \mathcal{A}}{ \Bl \colon \ol{H} \xrightarrow{\simeq} \Ext_{\Q[\Z]}^1(H,\Q[\Z])} }\, \Big\{ \, \bigsqcup_{p \in H} \,((N,\theta),\xi)_{p} \subset \mathcal{AF}_{(\C/1.5)}(H,\Bl)\Big\},\] ranging over all possible $H$ and $\Bl$, so that we can make this class into a set by taking an inverse limit.  For each $\Q[\Z]$-module isomorphism $\a \colon H \xrightarrow{\simeq} H^{\%}$, we define a map
\[\a_* \colon L^4_{H,\Bl,p}(\Q\G,\Q\G-\{0\}) \to L^4_{H^{\%},\Bl^{\%},\a(p)}(\Q\G,\Q\G-\{0\}),\]
where $\Bl^{\%}(x,y) := \Bl(\a^{-1}(x),\a^{-1}(y))$ by
\[((N,\theta \in Q^3(N)),\xi) \mapsto ((N,\theta \in Q^3(N)), \xi \circ \a^{-1}).\]
This defines a map:
\[\a_* \colon \mathcal{AF}_{(\C/1.5)}(H,\Bl) \to \mathcal{AF}_{(\C/1.5)}(H^{\%},\Bl^{\%}),\]
which we use to map subsets to subsets.  We say that a subset:
\[\bigsqcup_{p \in H} \,((N,\theta),\xi)_{p} \subset \mathcal{AF}_{(\C/1.5)}(H,\Bl),\]
is less than or equal to
\[\bigsqcup_{q \in H^\%} \,((N,\theta),\xi^{\%})_{q} \subset \mathcal{AF}_{(\C/1.5)}(H^{\%},\Bl^{\%}),\]
if the latter is the image of the former under $\a_*$.  We then define:
\begin{eqnarray*}\mathcal{AF}_{(\C/1.5)} := \underleftarrow{\lim} \bigg\{\bigsqcup_{p \in H} \,((N,\theta),\xi)_{p} \subset \mathcal{AF}_{(\C/1.5)}(H,\Bl)\, | \, H \in \Q \otimes_{\Z} \mathcal{A},\\ \Bl \colon \ol{H} \xrightarrow{\simeq} \Ext_{\Q[\Z]}^1(H,\Q[\Z])\bigg\}.\end{eqnarray*}

Finally, we must say what it means for two elements of $\mathcal{AF}_{(\C/1.5)}$ to be equivalent, in such a way that isotopic and concordant knots map to equivalent elements of $\mathcal{AF}_{(\C/1.5)}$, and we must define the class of the zero object, so that we have a pointed set.

The distinguished point is the equivalence class of the 3-dimensional symmetric Poincar\'{e} chain complex: \begin{multline*} U := \Big(\big(\Q\G \otimes_{\Q[\Z]} C_*(S^1 \times S^2; \Q[\Z]),\backslash\Delta([S^1 \times S^2])\big), \xi = \Id \colon \{0\} \to \{0\} \Big) \\ \in \mathcal{AF}_{(\C/1.5)}(\{0\},\Bl_{\{0\}}).\end{multline*}
We declare two elements of $\mathcal{AF}_{(\C/1.5)}$ to be equivalent, denoted $\sim$, if we can choose a representative class for the inverse limit construction of each i.e. pick representatives:
\[\bigsqcup_{p \in H} \,((N,\theta),\xi)_{p} \subset \mathcal{AF}_{(\C/1.5)}(H,\Bl)\]
and
\[\bigsqcup_{q \in H^\dag} \,((N^{\dag},\theta^{\dag}),\xi^{\dag})_{q} \subset \mathcal{AF}_{(\C/1.5)}(H^{\dag},\Bl^{\dag})\]
for some $H,H^{\dag} \in \Q \otimes_{\Z} \mathcal{A}$, such that there is a metaboliser $P \subseteq H \oplus H^\dag$ of
\[\Bl \oplus - \Bl^\dag \colon H \oplus H^\dag \times H \oplus H^\dag \to \Q(\Z)/\Q[\Z]\]
for which all the elements of $L^4(\Q\G,\Q\G -\{0\})$ in the disjoint union:
\[\bigsqcup_{(p,q) \in P}\, ((N_p \oplus N_q^{\dag},\theta_p \oplus -\theta^{\dag}_q),\xi_p \oplus \xi_q^{\dag})\]
satisfy:
 \[((N_p \oplus N_q^{\dag},\theta_p \oplus -\theta^{\dag}_q) = 0 \in L^4(\Q\G,\Q\G-\{0\}),\]
with the corresponding reason that $(N_p \oplus N_q^{\dag},\theta_p \oplus -\theta^{\dag}_q) = 0$ being a 4-dimensional symmetric Poincar\'{e} pair
\[(j_p \oplus j_q^{\dag} \colon N_p \oplus N_q^{\dag} \to V_{(p,q)},(\delta \theta_{(p,q)}, \theta_p \oplus -\theta^{\dag}_q) \in Q^4(j_p \oplus j_q^{\dag}))\]
over $\Q\G$ such that
$$H_1(\Q \otimes_{\Q\G} N_p) \toiso H_1(\Q \otimes_{\Q\G} V_{(p,q)}) \xleftarrow{\simeq} H_1(\Q \otimes_{\Q\G} N_q^\dag),$$
that the isomorphism
\[\xi_p \oplus \xi_q^\dag \colon H \oplus H^\dag \toiso H_1(\Q[\Z] \otimes_{\Q\G} N_p) \oplus H_1(\Q[\Z] \otimes_{\Q\G} N_q^\dag)\]
restricts to an isomorphism
\[P \toiso \ker \big(H_1(\Q[\Z] \otimes_{\Q\G} N_p) \oplus H_1(\Q[\Z] \otimes_{\Q\G} N_q^\dag) \to H_1(\Q[\Z] \otimes_{\Q\G} V_{(p,q)})\big),\]
and that the algebraic Thom complex (Definition \ref{Defn:algThomcomplexandthickening}), taken over the Ore localisation, is algebraically null-cobordant in $L^4_S(\K) \cong L^0_S(\K)$:
\[[(\K \otimes_{\Q\G} \mathscr{C}((j_p \oplus j_q^\dag)),\Id \otimes \delta \theta_{(p,q)}/(\theta_{p}\oplus-\theta^\dag_q))] = [0] \in L^4_S(\K).\]
These conditions imply that we can do algebraic surgery (Definition \ref{Defn:algebraicsurgerymatrices}) on $\mathscr{C}((j_p \oplus j_q^\dag))$ to make it contractible over $\K$.  The relation $\sim$ is an equivalence relation: see Proposition \ref{prop:COTobstr_equiv_rln}.

Taking the quotient of $\mathcal{AF}_{(\C/1.5)}$ by this equivalence relation defines the second order \COT obstruction pointed set $(\mathcal{COT}_{(\C/1.5)},U)$: there is a well--defined map from concordance classes of knots modulo $(1.5)$-solvable knots to this set, which maps $(1.5)$-solvable knots to the equivalence class of $U$, as follows.

Define $H:= H_1(M_K;\Q[\Z])$.  For each $p\in H$, we use the corresponding representation $\rho \colon \pi_1(M_K) \to \G$ to form the complex:
\begin{multline*} ((N,\theta),\xi)_{p} := ((\Q\G \otimes_{\Z[\pi_1(M_K)]} C_*(M_K;\Z[\pi_1(M_K)]), \backslash \Delta ([M_K])),\xi) \\ \in L^4_{H,\Bl,p}(\Q\G,\Q\G - \{0\}).\end{multline*}
This gives a well--defined map: see Proposition \ref{Prop:COTobstr_well_defined}.  This completes our description of the \COT pointed set.

\qed\end{definition}

\begin{proposition}\label{prop:COTobstr_equiv_rln}
The relation $\sim$ of Definition \ref{defn:COTobstructionset_2} is indeed an equivalence relation.
\end{proposition}
\begin{proof}
To see reflexivity, note that the diagonal $H \subseteq H \oplus H$ is a metaboliser for $\Bl \oplus - \Bl$.  Then take $V_{(p,p)} := N_p$ and $\delta\theta_{(p,p)} := 0$.  It is straight--forward to see that $\sim$ is symmetric.  For transitivity, suppose that
\[\bigsqcup_{p \in H} \,((N,\theta),\xi)_{p} \sim \bigsqcup_{q \in H^\dag} \,((N^\dag,\theta^\dag),\xi^\dag)_{q}\]
with a metaboliser $P \subseteq H \oplus H^\dag$ and chain complexes $(V_{(p,q)},\delta\theta_{(p,q)})$, and that
\[\bigsqcup_{q \in H^\dag} \,((N^\dag,\theta^\dag),\xi^\dag)_{q} \sim \bigsqcup_{r \in H^\ddag} \,((N^\ddag,\theta^\ddag),\xi^\ddag)_{r}.\]
with a metaboliser $Q \subseteq H^\dag \oplus H^\ddag$ and chain complexes $(\ol{V}_{(q,r)},\ol{\delta\theta}_{(q,r)})$.

We define the metaboliser $R \subseteq H \oplus H^\ddag$ by
\[R:= \{(p,r) \in H \oplus H^\ddag \, | \, \exists \, q \in H^\dag \text{ with } (p,q) \in P \text{ and } (q,r) \in Q\}.\]
The proof of Lemma \ref{Lemma:Cancellation_Blanchfield} shows that this is a metaboliser.
For each $(p,r) \in R$ we can therefore choose a suitable $q$ and so glue the chain complexes:
\[(\ol{\ol{V}}_{(p,r)},\ol{\ol{\delta\theta}}_{(p,r)}) := (V_{(p,q)} \cup_{N_q^\dag} \ol{V}_{(q,r)}, \delta\theta_{(p,q)} \cup_{\theta_q^\dag} \ol{\delta\theta}_{(q,r)}),\]
to create an algebraic cobordism for each $(p,r) \in R$.  Easy Mayer-Vietoris arguments show that the inclusions $N_p \to \ol{\ol{V}}_{(p,r)}$ and $N_r^\ddag \to \ol{\ol{V}}_{(p,r)}$ induce isomorphisms on first $\Q$-homology, and that
\[\xi_p \oplus \xi_r^\ddag \colon H \oplus H^\ddag \toiso H_1(\Q[\Z] \otimes_{\Q\G} N_p) \oplus H_1(\Q[\Z] \otimes_{\Q\G} N_r^\ddag)\]
restricts to an isomorphism
\[R \toiso \ker \big(H_1(\Q[\Z] \otimes_{\Q\G} N_p) \oplus H_1(\Q[\Z] \otimes_{\Q\G} N_r^\ddag) \to H_1(\Q[\Z] \otimes_{\Q\G} \ol{\ol{V}}_{(p,r)})\big).\]  Since $\K \otimes_{\Q\G} N_q^\dag \simeq 0$, the elements of $L^4_S(\K)$ add and we still have the zero element of $L^4_S(\K)$ as required.
\end{proof}

\begin{proposition}\label{Prop:COTobstr_well_defined}
The map $\C/\mathcal{F}_{(1.5)} \to \mc{COT}_{(\C/1.5)}$ in Definition \ref{defn:COTobstructionset_2} is well--defined.
\end{proposition}
\begin{proof}
To see that the map is well--defined, we show that if $K \, \sharp \, -K^\dag$ is $(1.5)$-solvable, then the image of $K$ is equivalent to the image of ${K^\dag}$ in $\mc{COT}_{(\C/1.5)}$.  Let $W$ be a $(1.5)$-solution for $K \, \sharp \, -K^\dag$, and let $$P := \ker(H_1(M_K;\Q[\Z]) \oplus H_1(M_{K^\dag};\Q[\Z]) \to H_1(W;\Q[\Z])),$$
noting that $$H_1(M_K;\Q[\Z]) \oplus H_1(M_{K^\dag};\Q[\Z]) \toiso H_1(M_{K \, \sharp \, -K^\dag};\Q[\Z]).$$
We define, for all $(p,q) \in P$, $V_{(p,q)} := C_*(W,M_{K \, \sharp \, -K^\dag};\Q\G)$ to be the chain complex of $W$ relative to $M_{K \, \sharp \, -K^\dag}$.

Then $\K \otimes_{\Q\G} V_{(p,q)}$ represents an element of $L^4_S(\K)$ as in Definition \ref{defn:localisationexactsequence}.  Since $W$ is a $(1.5)$-solution, as in Theorem \ref{Thm:COTmaintheorem}, we have $B=0$.  That is, the intersection form of $V_{(p,q)}$ is hyperbolic as required.

Applying the algebraic Poincar\'{e} thickening (Definition \ref{Defn:algThomcomplexandthickening}) yields a symmetric Poincar\'{e} pair
\[C_*(M_{K \, \sharp \, -K^\dag};\Q\G)_{(p,q)} \to V^{4-*}_{(p,q)}.\]
Now note that $$C_*(M_{K \, \sharp \, -K^\dag};\Q\G)_{(p,q)} \simeq C_*(X_K \cup S^1 \times S^1 \times I \cup X_{K^\dag};\Q\G)_{(p,q)}.$$
By gluing the chain complex $C_*(S^1 \times D^2 \times I;\Q\G)$ to $V_{(p,q)}^{4-*}$ along $C_*(S^1 \times S^1 \times I;\Q\G)$, we obtain a symmetric Poincar\'{e} pair \[(C_*(M_K;\Q\G)_{p} \oplus C_*(M_{K^\dag};\Q\G)_q \to \widehat{V}_{(p,q)},(\widehat{\delta\theta}_{(p,q)},\theta_p \oplus -\theta_q^\dag)).\]
This gluing does not change the element of $L^4_S(\K)$ produced, since $$C_*(S^1 \times D^2 \times I;\K) \simeq 0.$$  We therefore indeed have that $K$ and $K^\dag$ map to equivalent elements in $\mc{COT}_{(\C/1.5)}$, as claimed.
\end{proof}

\section{$L^{(2)}$-signatures}\label{Chapter:L2signature}

There remains the not insignificant task of detecting non-zero elements in the Witt group of Hermitian forms in $L^0(\K)$.  \COT use an \emph{$L^{(2)}$-signature} (see \cite[Section~5]{COT} for more complete details) to define a homomorphism
\[\sigma^{(2)} \colon L^0(\K) \to \R\]
which detects the Witt class of the intersection form and therefore obstructs $(1.5)$-solvability and in particular sliceness.  The $L^{(2)}$-signature agrees with the ordinary signature of $\Q$-homology on the image of $L^0(\Q\Gamma)$ so that we have a well defined obstruction, the reduced $L^{(2)}$-signature:
\[\sigma^{(2)}(W) - \sigma(W),\]
where $\sigma(W) \in \Z$ is the ordinary signature, for a $(1)$-solution $W$.  We now give an outline of the beautiful theory of $L^{(2)}$-signatures.  Once we have our notion of algebraic concordance of symmetric Poincar\'{e} triads we will describe a way to obtain these signatures algebraically without having to make a choice of a geometric 4-manifold.

The $L^{(2)}$-signature can be thought of as a way of taking a signature when the coefficients are in the group ring of an infinite group $\G$.  We first make the inclusion:
\[\Q\G \hookrightarrow \mathbb{C}\G.\]
We then consider $\mathbb{C}\G$ as a subset of $\mathcal{B}(l^2\G)$, the Hilbert space of square-summable sequences indexed by the elements $g \in \G$.  We complete $\mathbb{C}\G$ inside $\mathcal{B}(l^2\G)$ using pointwise convergence and obtain the \emph{Von Neumann Algebra} $\NG$.  We shall later include $\NG$ into the space $\UG$ of unbounded operators affiliated to $\NG$, the equivalent of Ore localisation for Von Neumann algebras.  For a $(1)$-solution $W$ the intersection form
\[\lambda_1 \colon H_2(W;\K) \times H_2(W;\K) \to \K\]
yields a Hermitian operator on the Hilbert space $(\UG)^m$.

We define the signature for Hermitian operators $\lambda$ on $(\NG)^m$, which will then extend to $(\UG)^m$.  To define a signature we need notions of the dimensions of the positive and negative eigenspaces of $\lambda$.

The \emph{functional calculus} yields a correspondence between bounded measurable functions on the spectrum $\spec(\lambda)$ of an operator $\lambda$:
\[f \colon \spec(\lambda) \to \mathbb{C}\]
and bounded operators $f(\lambda)$ on $(\NG)^m$, represented as $m \times m$ matrices with elements in $\NG$.  Choosing the characteristic functions
\[p_+,p_- \colon \spec(\lambda) \subseteq \R \to \{0,1\} \subset \mathbb{C}\]
of $(0,\infty)$ and $(-\infty,0)$ ($\spec(\lambda) \subseteq \R$ since $\lambda$ is Hermitian) we obtain two Hermitian projection operators $p_+(\lambda), p_-(\lambda)$.  The completion to the Von Neumann algebra is necessary for the functional calculus to be well defined on such Heaviside-type functions as $p_+,p_-$; they are limits of polynomials so the fact that limits commute with the functional calculus correspondence in Von Neumann algebras is crucial.  For example, let $p_i$ be a sequence of polynomials such that $\lim(p_i) = p_+.$  We have:
\[p_+(\lambda) = (\lim p_i)(\lambda) = \lim (p_i(\lambda)) \in M_m(\NG).\]
where $p_i(\lambda)$ makes immediate sense as we can evaluate polynomials on operators which live in a $C^*$-algebra.

We can then use the \emph{Von Neumann $\G$-trace} to define the dimension of the $\pm$ eigenspaces of $\lambda$.  The $\G$-trace of an operator $a$ is defined to be:
\[\tr_{\G}(a) := \langle (e)a,e \rangle_{l^2\G} \in \mathbb{C},\]
using the $l^2\G$ inner product, where $e \in \G \subseteq l^2\G$ is the identity element.  This extends to $m\times m$ matrices by taking the matrix trace, that is by summing over the $\G$-traces of the diagonal entries.  Recall that for projection operators on finite dimensional vector spaces, their trace is equal to the dimension of their image; the $\G$-trace is a generalisation of this concept.

We can now define the $L^{(2)}$-signature of a Hermitian operator $\lambda$ to be:
\[\sigma^{(2)}(\lambda) := \tr_{\G}(p_+(\lambda)) - \tr_{\G}(p_-(\lambda)) \in \R.\]
Hermitian projection operators $a = a^2 = a^*$ have real traces since:
\[\langle (e)a , e\rangle_{l^2\G} = \langle (e)a^2 , e \rangle_{l^2\G} = \langle (e) a, (e) a^* \rangle_{l^2\G} = \langle (e)a , (e)a \rangle_{l^2\G} \in \R.\]
Furthermore we can include $\NG \subset \UG$, where $\UG$ is the space of unbounded operators affiliated to $\NG$.  See \cite[Lemma~5.6]{COT} and the preamble to it for more details.  The functional calculus can be extended to unbounded operators, and it is a theorem that $\NG$ satisfies the Ore condition, with $S$ as the set of all non-zero divisors, and that this Ore localisation yields $\UG$.

The introduction of $\UG$ enables the definition of the $L^{(2)}$-signature to be extended from Hermitian forms over $\Q\Gamma^n$ to those on $\K^n$; $H_2(W;\Q\G)$ may not be free but $H_2(W;\K)$ is a module over a skew-field so is a free module, so we can express
\[\lambda_1 \colon H_2(W;\K) \times H_2(W;\K) \to \K\] as a matrix, and use this to obtain the $L^{(2)}$-signature.  The reduced $L^{(2)}$-signature:
\[\wt{\sigma}^{(2)}(W) := \sigma^{(2)}(W) - \sigma(W) \in \R\]
gives a real number which is independent of the choice of $W$, so can detect the image of $\lambda_1$ in $L^0(\K)/\im(L^0(\Q\G))$ and therefore obstructs the existence of a $(1.5)$-solution, provided we check all the metabolisers $P$ of the Blanchfield form and, for each $P$, at least one of the representations which arise from a choice of $p \in P \setminus \{0\}$.  Since the obstruction depends only on the 3-manifold, and the choice of representation, it is often referred to as the \emph{Cheeger--Gromov--Von--Neumann $\rho$-invariant} of $M_K$.  \COT and Cochran--Harvey--Leidy (\cite[Section~6]{COT} and \cite{COT2}, \cite{CochranTeichner}, \cite{cohale}) are able to use this obstruction and various satellite constructions to find knots which are $(n)$-solvable but which are not $(n.5)$-solvable for all $n\in \mathbb{N}$.  The beauty is that the $L^{(2)}$-signature of these knots can be calculated in a simple way by integrating the classical Levine-Tristram $\omega$-signatures of the \emph{infection knot} of the satellite construction as $\omega$ varies around the circle (see \cite[Lemma~5.4]{COT}, \cite{COT2} for more on this).

\begin{theorem}[\cite{COT2} Theorem 5.2]
Suppose $K$ is a (1.5)-solvable knot whose Alexander polynomial is not 1 and which admits a Seifert surface $F$ of genus 1.  Then there is a homologically essential simple closed curve $J$ on $F$, which has self linking number zero, so corresponds to a metaboliser of the Blanchfield form, such that the integral of the Levine-Tristram signature function of $J$ vanishes, considering $J$ as a knot in $S^3$.
\end{theorem}

As a great example, we can use this to recreate the original Casson-Gordon result that the twist knots of Figure \ref{Fig:TwistKnots} are not slice.  The zero-linking curves on Seifert surfaces for the algebraically slice twist knots, which are those with $4k+1 = n^2$ for some $n \in \mathbb{N}$, are torus knots.  There exists a closed formula for the integral of the $\omega$-signatures of the torus knots, integrating over $\omega \in S^1$, written about by several people: see \cite{collins2010} for an excellent exposition and further references.  The relevant $L^{(2)}$-signatures of the torus knots are non-zero, proving once again that the only twist knots which are slice are for $k=0,2$.

\begin{figure}[h]
\begin{center}
{\psfrag{k}{$k$ full twists}
\includegraphics[width=7cm]{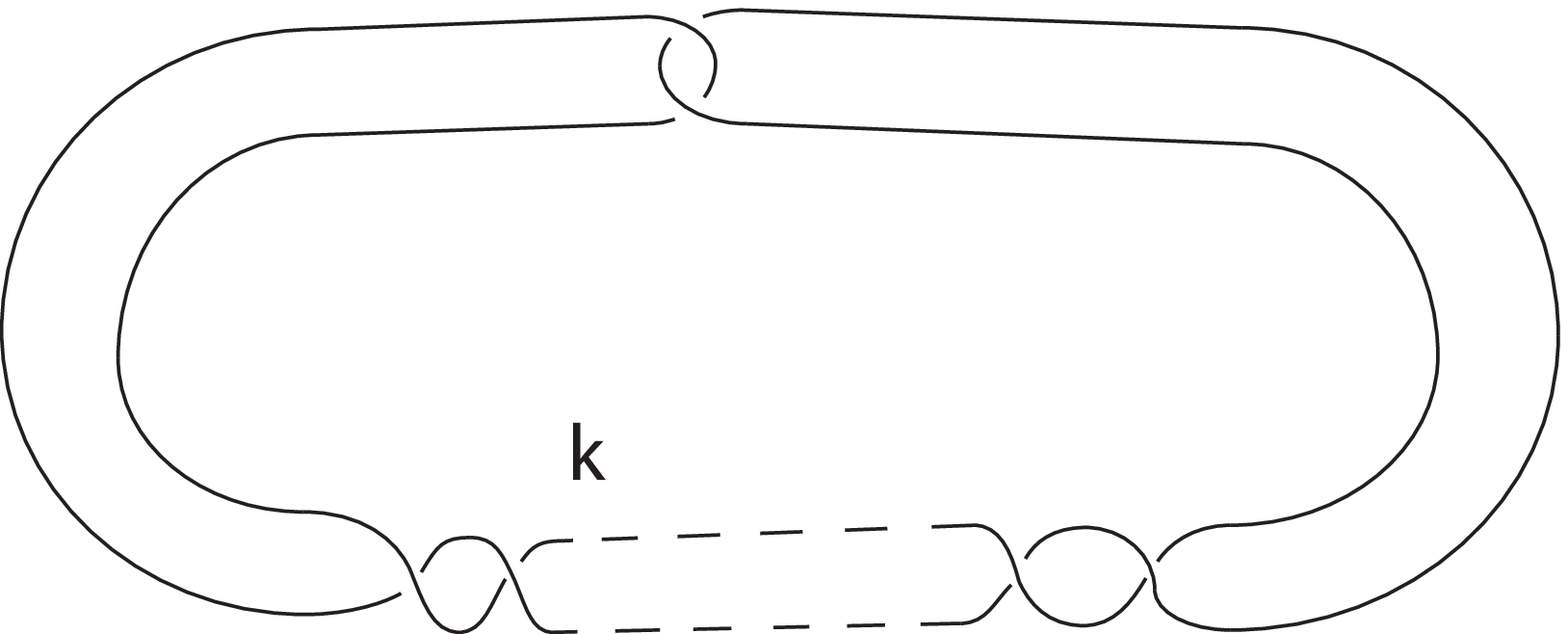}}
\caption{The $k$th Twist Knot} \label{Fig:TwistKnots}
\end{center}
\end{figure}

\chapter{Algebraic Concordance}\label{chapter:algconcordance}



The geometric obstruction theory of Chapter \ref{chapter:COTsurvey} motivates the definition of a purely algebraic obstruction theory, which we use to define a second order algebraic concordance group $\mathcal{AC}_2$.  We proceed as follows.

Given two triples $(H,\Y,\xi),(H^\dag,\Y^{\dag},\xi^\dag) \in \P$, we formulate an algebraic concordance equivalence relation, modelled on the concordance of knots and corresponding to $\Z$-homology cobordism of manifolds, with the extra control on the fundamental group which is evidently required, given the prominence of the Blanchfield form in the previous chapter when controlling representations.  We take the quotient of our monoid $\P$ by this relation, and obtain a group $\mathcal{AC}_2 := \P/\sim$.  Our main goal for this work is to complete the set up of the following commuting diagram, which has geometry in the left column and algebra in the right column:
\[\xymatrix @R+1cm @C+1cm{
Knots \ar[r] \ar @{->>} [d] & \P \ar @{->>} [d]  \\
\C \ar[r] \ar @{->>} [d] & \mathcal{AC}_2 \ar@{-->}[d]  \\
\C/\mathcal{F}_{(1.5)} \ar@{-->}[r]  \ar[ur] & \mathcal{COT}_{(\C/1.5)},}\]
where $Knots$ is the monoid of geometric knots under connected sum, $\C$ is the concordance group of knots and $\mathcal{F}_{(1.5)}$ denotes the subgroup of (1.5)-solvable knots. The top row consists of monoids, and arrows emanating from the top row should be monoid homomorphisms.  The rest of the maps should be homomorphisms of groups, apart from those with codomain $\mathcal{COT}_{(\C/1.5)}$ (the dotted arrows), which is the pointed set which contains the \COT obstructions.  Since the \COT obstructions are not guaranteed to behave well under connected sum (but see \cite{COT2}, where the obstructions do behave well under special circumstances), we require only that the maps with codomain $\mathcal{COT}_{(\C/1.5)}$ map zero to zero, so are morphisms of pointed sets: see Theorem \ref{Thm:betterstatementzeroAC2_to_zeroL4QG}.  When we are able to take $L^{(2)}$-signatures, as in Chapter \ref{chapter:COTsurvey}, to obstruct an element of $\mathcal{COT}_{(\C/1.5)}$ from being $U$, we are also obstructing triples in $\mathcal{AC}_2$ from being equivalent to the triple corresponding to the unknot (see Theorem \ref{Thm:extractingL2signatures}).

So far we have explained the diagram with $\mathcal{AC}_2$ and the maps with it as domain and codomain removed.  In this chapter we focus on completing the top square.  We define our concordance relation, and show that it is an equivalence relation.  We define an inverse $-(H,\Y,\xi)$ of a triple $(H,\Y,\xi)$, and show that $(H,\Y,\xi) \,\sharp\, -(H,\Y,\xi) \sim (\{0\},\Y^U,\Id_{\{0\}})$, where $(\{0\},\Y^U,\Id_{\{0\}})$ is the triple of the unknot, so that we obtain a group $\mathcal{AC}_2$.

\begin{proposition}\label{lemma:basicfactconcordance}
Two knots $K$ and $K^{\dag}$ are topologically concordant if and only if the 3-manifold
\[Z := X \cup_{\partial X = S^1 \times S^1} S^1 \times S^1 \times I \cup_{S^1 \times S^1 = \partial X^{\dag}} -X^{\dag}\]
is the boundary of a topological 4-manifold $W$ such that
\begin{description}
\item[(i)] the inclusion $i \colon Z \hookrightarrow W$ restricts to $\Z$-homology equivalences $$H_*(X;\Z) \xrightarrow{\simeq} H_*(W;\Z) \xleftarrow{\simeq} H_*(X^\dag;\Z); \text{ and}$$
\item[(ii)] the fundamental group $\pi_1(W)$ is normally generated by a meridian of (either of) the knots.
\end{description}
\end{proposition}
\begin{proof}
This is a generalisation of Proposition \ref{basicfact} which deals with the case that $K^{\dag}$ is the unknot.  Let $W$ be the exterior of the embedded annulus $S^1 \times I \subset S^3 \times I$ which gives a concordance:
\[W:= \cl((S^3 \times I)\setminus (S^1 \times D^2 \times I)).\]
Then a Mayer-Vietoris calculation and the Seifert-Van-Kampen theorem using the decomposition of $S^3 \times I$ as $W \cup_{S^1 \times S^1 \times I} S^1 \times D^2 \times I$ verify that $W$ satisfies the claimed properties.  Conversely, suppose that we have a $W$ which satisfies these properties.  Then we can glue in $S^1 \times D^2 \times I$ to the $S^1 \times S^1 \times I$ part of the boundary $\partial W = Z$.  This yields a simply-connected 4-manifold with the homology of $S^3$ and boundary $S^3 \times \{0,1\}$; $K$ and $K^{\dag}$ are concordant in $W$.  Gluing $D^4$ to both ends yields a homotopy 4-sphere, which is homeomorphic to $S^4$ by Freedman's topological $h$-cobordism theorem \cite[Theorem~7.1B]{FQ}.  Removing the images of our added 4-balls yields $S^3 \times I$ as claimed.
\end{proof}

We need to construct the algebraic version of $Z$ from two symmetric Poincar\'{e} triads $\Y$ and $\Y^{\dag}$ so that we can impose conditions on the algebraic 4-dimensional complexes which have it as their boundary.  As part of the definition of a symmetric Poincar\'{e} triad $\Y$ over $\zh$ (Definition \ref{Defn:symmPoincaretriad}),
\[\xymatrix @R+0.5cm @C+0.5cm {\ar @{} [dr] |{\stackrel{g}{\sim}}
(C,\varphi_C) \ar[r]^{i_-} \ar[d]_{i_+} & (D_-,\delta\varphi_-) \ar[d]^{f_-}\\ (D_+,\delta\varphi_+) \ar[r]^{f_+} & (Y,\Phi),
}\]
we can construct a symmetric Poincar\'{e} pair
\[(\eta \colon E := D_- \cup_{C} D_+ \to Y ,(\Phi, \delta\varphi_- \cup_{\varphi_C} \delta\varphi_+))\]
where
\[\eta = \left(\begin{array}{ccc} f_-\, , & (-1)^{r-1}g \, , & -f_+ \end{array} \right) \colon E_r = (D_-)_r \oplus C_{r-1} \oplus (D_+)_r \to Y_r.\]
In our case of interest, $E$, for the standard models of $C, D_{\pm}$, is given by:
\[\xymatrix{ E_2 \cong \bigoplus_2 \, \zh \xrightarrow{\partial_2} E_1 \cong \bigoplus_4 \, \zh  \xrightarrow{\partial_1} E_0 \cong \bigoplus_2 \, \zh,
}\]
where:
\[\partial_1 = \left(\begin{array}{cc} g_1-1 & 0  \\ 1 & l_a \\ l_a^{-1} & 1 \\ 0 & g_q-1 \end{array} \right) ; \text{ and}\]
\[\partial_2 = \left(\begin{array}{cccc} -1 & g_1-1 & 0 & -l_a \\ -l_a^{-1} & 0 & g_q-1 & -1 \end{array} \right),\]
with $\phi_0 \colon E^{2-r} \to E_r$:
\[\xymatrix @R+1cm @C+1cm{ E^0 \ar[r]^{\delta_1} \ar[d]^{\phi_0} & E^1 \ar[r]^{\delta_2} \ar[d]^{\phi_0} & E^2 \ar[d]^{\phi_0}\\
E_2 \ar[r]^{\partial_2} & E_1 \ar[r]^{\partial_1} & E_0
}\]
given by:
\[\xymatrix @R+3cm @C+1.5cm{ \bigoplus_2\,\zh \ar[r]^{\delta_1} \ar[d]^{\left(\begin{array}{cc} -1 & l_a \\ 0 & 0 \end{array} \right)} & \bigoplus_4\,\zh \ar[r]^{\delta_2} \ar[d]^{\left(\begin{array}{cccc} 0 & g_1 & -l_ag_q & 0 \\ 0 & 0 & 0 & l_a \\ 0 & 0 & 0 & -1 \\ 0 & 0 & 0 & 0 \end{array} \right)} & \bigoplus_2\,\zh \ar[d]^{\left(\begin{array}{cc} 0 & g_1 l_a \\ 0 & -g_q \end{array} \right)}\\
\bigoplus_2\,\zh \ar[r]^{\partial_2} & \bigoplus_4\,\zh \ar[r]^{\partial_1} & \bigoplus_2\,\zh.
}\]
We have replaced $l_b^{-1}$ with $l_a$ here.  Note that the boundary and symmetric structure maps still depend on the group element $l_a$.  The next lemma shows that, over the group ring $\Z[\Z \ltimes (H \oplus H^{\dag})] = \zhddag$, the chain complexes $E,E^{\dag}$ of the boundaries of two different triads $\Y,\Y^{\dag}$ are isomorphic.

\begin{lemma}\label{lemma:torusindependantofl_a}
There is a chain isomorphism:
\[\varpi_E \colon \zhddag \otimes_{\zh} E \to \zhddag \otimes_{\zhdag} E^{\dag}, \]
\[\xymatrix @R+2cm @C+2cm{
E_2 \ar[r]^{\partial_2} \ar[d]_{\varpi_E} & E_1 \ar[r]^{\partial_1} \ar[d]_{\varpi_E} & E_0 \ar[d]_{\varpi_E} \\
E^{\dag}_2 \ar[r]^{\partial_2^{\dag}} & E^{\dag}_1 \ar[r]^{\partial_1^{\dag}} & E^{\dag}_0 \\
}\]
omitting $\zhddag \otimes_{\zh}$ and $\zhddag \otimes_{\zhdag}$ from the notation of the diagram, given by:
\[\xymatrix @R+2cm @C+2cm{
\bigoplus_2\,\Ups^{\ddag} \ar[rr]^{\left(\begin{array}{cccc} -1 & g_1-1 & 0 & -l_a \\ -l_a^{-1} & 0 & g_q-1 & -1 \end{array} \right)} \ar[d]^{\left(\begin{array}{cc} 1 & 0 \\ 0 & l_a^{-1}l^{\dag}_a \end{array} \right)} && \bigoplus_4\,\Ups^{\ddag} \ar[r]^{\left(\begin{array}{cc} g_1-1 & 0  \\ 1 & l_a \\ l_a^{-1} & 1 \\ 0 & g_q-1 \end{array} \right)} \ar[d]_{\left(\begin{array}{cccc} 1 & 0  & 0 & 0\\ 0 & 1 & 0 & 0 \\ 0 & 0 & l_a^{-1}l^{\dag}_a  & 0 \\ 0 & 0 & 0 & l_a^{-1}l_a^{\dag}\end{array} \right)} & \bigoplus_2\,\Ups^{\ddag} \ar[d]_{\left(\begin{array}{cc} 1 & 0 \\ 0 & l_a^{-1}l^{\dag}_a \end{array} \right)} \\
\bigoplus_2\,\Ups^{\ddag} \ar[rr]_{\left(\begin{array}{cccc} -1 & g^{\dag}_1-1 & 0 & -l^{\dag}_a \\ -(l^{\dag}_a)^{-1} & 0 & g^{\dag}_q-1 & -1 \end{array} \right)} && \bigoplus_4\,\Ups^{\ddag} \ar[r]_{\left(\begin{array}{cc} g^{\dag}_1-1 & 0  \\ 1 & l^{\dag}_a \\ (l^{\dag}_a)^{-1} & 1 \\ 0 & g^{\dag}_q-1 \end{array} \right)} & \bigoplus_2\,\Ups^{\ddag} \\
}\]
where $\Ups^{\ddag} := \zhddag$.
\end{lemma}
\begin{proof}
To see that $\varpi_E$ is a chain map, as usual one needs the identities:
\[l_ag_ql_a^{-1} = g_1 = g_1^{\dag} = l_a^{\dag}g_q^{\dag}(l_a^{\dag})^{-1}.\]
The maps of $\varpi_E$ are isomorphisms, and the reader can calculate that $$\varpi_E \phi \varpi_E^* = \phi^{\dag}.$$  Note that this proof relies on the fact that $l_al_b=1$ and would require extra control over the longitude if we were not working modulo the second derived subgroup, but instead were only factoring out further up the derived series.
\end{proof}
\begin{definition}\label{defn:2ndorderconcordant}
We say that two triples $(H,\Y,\xi),(H^\dag,\Y^{\dag},\xi^\dag) \in \P$ are \emph{second order algebraically concordant} or \emph{algebraically $(1.5)$-equivalent}, written $\sim$, if there is a $\Z[\Z]$ module $H'$ of type $K$, that is $H'$ satisfies the properties of (a) of Theorem \ref{Thm:Levinemodule}, with a homomorphism \[(j_{\flat},j_{\flat}^{\dag}) \colon H \oplus H^{\dag} \to H'\] which induces a homomorphism
\[ \Z[\Z \ltimes (H \oplus H^{\dag})] \to \zhd\]
and therefore, by composition with the maps
\[\zh \to \Z[\Z \ltimes (H \oplus H^\dag)]\]
and
\[\zhdag \to \Z[\Z \ltimes (H \oplus H^\dag)]\]
from Definition \ref{Defn:connectsumalgebraic}, homomorphisms
\[\zh \to \zhd\]
and \[\zhdag \to \zhd,\]
along with a finitely generated projective $\zhd$-module chain complex with structure maps $(V,\Theta)$, the requisite chain maps $j,j^{\dag},\delta$, and chain homotopies $\g,\g^{\dag}$, such that there is a 4-dimensional symmetric Poincar\'{e} triad:
\[\xymatrix @R+1cm @C+1cm {\ar @{} [dr] |{\stackrel{(\gamma,\g^{\dag})}{\sim}}
(\zhd \otimes (E,\phi)) \oplus (\zhd \otimes (E^{\dag},-\phi^{\dag})) \ar[r]^-{(\Id, \Id \otimes \varpi_{E^{\dag}})} \ar[d]_{\left( \begin{array}{cc} \Id \otimes \eta & 0 \\ 0 & \Id \otimes \eta^{\dag}\end{array}\right)} & \zhd \otimes (E,0) \ar[d]^{\delta}
\\ (\zhd \otimes (Y,\Phi)) \oplus (\zhd \otimes (Y^{\dag},-\Phi^{\dag})) \ar[r]^-{(j,j^{\dag})} & (V,\Theta).
}\]
In what follows we frequently omit the tensor products when reproducing versions of the preceding diagram, taking as understood that all chain complexes are tensored up to be over $\zhd$ and all homomorphisms act with an identity on the $\zhd$ component of the tensor products.  The top row is a symmetric Poincar\'{e} pair by Lemma \ref{lemma:productcobordism}.  We require that the symmetric Poincar\'{e} triad satisfies two homological conditions.  The first is that:
\[j \colon H_*(\Z \otimes_{\zhd} (\zhd \otimes_{\zh} Y)) \xrightarrow{\simeq} H_*(\Z \otimes_{\zhd} V)\]
and
\[j^{\dag} \colon H_*(\Z \otimes_{\zhd} (\zhd \otimes_{\zhdag} Y^{\dag})) \xrightarrow{\simeq} H_*(\Z \otimes_{\zhd} V)\]
are isomorphism of $\Z$-homology, so that
\[H_*(\Z \otimes_{\zhd} V) \cong H_*(S^1;\Z).\]
The second homological condition is the consistency condition, that there is a consistency isomorphism:
\[\xi' \colon H' \xrightarrow{\simeq} H_1(\Z[\Z] \otimes_{\zhd} V),\]
such that the diagram below commutes:
\[\xymatrix @R+1cm @C+1cm{H \oplus H^{\dag} \ar[r]^{(j_{\flat},j_{\flat}^{\dag})} \ar[d]^{\left(\begin{array}{cc} \xi & 0 \\ 0 & \xi^{\dag} \end{array} \right)}_{\cong} & H' \ar[d]^{\xi'}_{\cong} \\
H_1(\Z[\Z] \otimes_{\zh} Y) \oplus H_1(\Z[\Z] \otimes_{\zhdag} Y^{\dag}) \ar[r]^-{\Id_{\Z[\Z]} \otimes (j_*,j^{\dag}_*)} & H_1(\Z[\Z] \otimes_{\zhd} V).
}\]
We say that two knots are second order algebraically concordant if their triples are, and we say that a knot is \emph{second order algebraically slice} or \emph{algebraically $(1.5)$-solvable} if it is second order algebraically concordant to the unknot.
\qed \end{definition}

\begin{definition}
The quotient of $\P$ by the relation $\sim$ of Definition \ref{defn:2ndorderconcordant} is the \emph{second order algebraic concordance group} $\ac2$.  See Proposition \ref{prop:equivrelation} for the proof that $\sim$ is an equivalence relation and Proposition \ref{prop:inverseswork} for the proof that $\ac2$ is a group.
\qed \end{definition}

\begin{remark}
A symmetric Poincar\'{e} triad is the natural way to algebraically encode a cobordism of cobordisms.  In particular, as with the Cappell-Shaneson method which underlies the \COT filtration, we are dealing with $\Z$-homology cobordism.  The \COT idea is to filter the condition of a knot exterior being $\Z$-homology cobordant to the exterior of the unknot by how far up the derived series their algebraic vanishing condition holds on the homology intersection pairing of a geometric 4-manifold.  We pass to algebra much sooner, and then filter the idea of the chain complex of the knot exterior being algebraically $\Z$-homology cobordant to the chain complex of the exterior of the unknot by how far up the derived series we can take our coefficients.

The consistency condition is crucial in order to have some control on the fundamental group.  Note the absence of Blanchfield linking pairings as well as intersection pairings.  As we will see, as long as the consistency condition holds, we can construct the Blanchfield pairing if desired and see that, due to the duality information stored in the symmetric structure, we still have the control it provided in Chapter \ref{chapter:COTsurvey} on the kernel of the induced map on homology of the inclusion of the 3-manifold into the 4-manifold.
\end{remark}

\begin{proposition}
Two concordant knots $K$ and $K^{\dag}$ are second order algebraically concordant.
\end{proposition}
\begin{proof}
Let $W$ be the exterior of the concordance as in Proposition \ref{lemma:basicfactconcordance}. Define:
\[H' := \pi_1(W)^{(1)}/\pi_1(W)^{(2)},\]
with the $\Z$ action given by conjugation with a choice of meridian.
We claim that $H'$ is of type $K$; that is we claim that $H'$ is finitely generated over $\Z[\Z]$ and that $1-t$ acts on $H'$ as an automorphism.  To see the claim, first note that $H' \cong H_1(W;\Z[\Z])$ by the Hurewicz isomorphism.  We modify \cite[Propositions~1.1~and~1.2]{Levine2}.  We see that $H'$ is finitely generated since $W$ is a compact topological 4-manifold, and so has the homotopy type of a finite simplicial complex, by \cite[Annex~B,~III,~page~301]{KS}. Therefore the infinite cyclic cover $\wt{W}$ has a chain complex whose chain groups are finitely generated free over $\Z[\Z]$, which implies in particular, since $\Z[\Z]$ is Noetherian, that the homology $H_1(W;\Z[\Z])$ is finitely generated over $\Z[\Z]$ \cite[Proposition~1.1]{Levine2}.  Inspection of the proof of \cite[Proposition~1.2]{Levine2} shows that the only hypothesis required is that $X$ is a $\Z$-homology circle.  Since $W$ is also a $\Z$-homology circle, the result also applies to $W$, and so $1-t$ acts on $H'$ as an automorphism.  This completes the proof of the claim.

We also define:
\[(V,\Theta') := (C_*(W;\zhd),\backslash \Delta([W,\partial W])).\]
Then $d_{\Hom}(\backslash \Delta([W,\partial W])) = \backslash\Delta([Z])$, where $Z$ is as in Proposition \ref{lemma:basicfactconcordance}.  Note that $Z \approx M_{K \, \sharp \, K^\dag}$: we then know by Theorem \ref{Thm:EMspaces} that $Z$ is an Eilenberg-Maclane space as long as we do not have $K=K^\dag=U$.  Therefore, by Theorem \ref{Thm:davisdiag}, any two choices of diagonal chain approximation are chain homotopic.  Therefore $\backslash \Delta([Z]) = \Phi \cup_{\phi} -\Phi^{\dag} \in Q^3(C_*(Z;\zhd))$ so there is a set of structure maps $\Theta$ which are equivalent to the maps $\Theta':= \backslash \Delta([W,\partial W])$, and which fit into the symmetric Poincar\'{e} triad required in Definition \ref{defn:2ndorderconcordant}.  To see this, note that there exists maps $\chi$, arising from the chain homotopy between the two diagonal approximations, such that:
\[\backslash\Delta([Z]) - d_{\Hom}(\backslash \Delta([W,\partial W]))=d_{\Hom}\chi.\]
Therefore defining $$\Theta := \Theta' + \chi = \backslash \Delta([W,\partial W]) + \chi,$$ we have that $$d_{\Hom}\Theta = d_{\Hom}\Theta' + d_{\Hom} \chi = \backslash\Delta([Z]),$$
as required.  The first homological condition is satisfied by (i) of Proposition \ref{lemma:basicfactconcordance}, and the consistency condition is satisfied by the Hurewicz isomorphism.
\end{proof}

\begin{proposition}\label{prop:equivrelation}
The relation $\sim$ of Definition \ref{defn:2ndorderconcordant} is an equivalence relation.
\end{proposition}
\begin{proof}
We begin by showing that $\sim$ is reflexive: that $$(H,\Y,\xi) \sim (H^\%,\Y^\%,\xi\%),$$ where $(H,\Y,\xi)$ and $(H^\%,\Y^\%,\xi\%)$ are equivalent in the sense of Definition \ref{Defn:algebraicsetofchaincomplexes}.  This is the algebraic equivalent of the geometric fact that isotopic knots are concordant.  Suppose that we have an isomorphism $\omega \colon H \to H^\%$, and a chain equivalence of triads
\[j \colon \Z[\Z \ltimes H^\%] \otimes_{\zh} \Y \to \Y^\%,\] such that the relevant square commutes, as in Definition \ref{Defn:algebraicsetofchaincomplexes} (see below). To show reflexivity, we take $H':= H^\%$, and take $(j_{\flat},j_{\flat}) = (\omega,\Id) \colon H \oplus H^\% \to H^\%$ and $(V,\Theta):= (Y^\%,0)$.  We tensor all chain complexes with $\zhpc$, which do not already consist of $\zhpc$-modules.  We have, induced by $j$, an equivalence of symmetric Poincar\'{e} pairs:
 \begin{multline*} (j_E,j_Y;\, k) \colon (\Id \otimes \eta \colon \zhpc \otimes_{\zh} E \to \zhpc \otimes_{\zh} Y) \\ \to (\eta^\% \colon E^\% \to Y^\%),\end{multline*}
where $k \colon \eta^\% j_E \sim j_Y \eta$ is a chain homotopy (see \cite[Part~I,~page~140]{Ranicki3}).
We therefore have the symmetric triad:
\[\xymatrix @R+1cm @C+1cm {\ar @{} [dr] |{\stackrel{(k,0)}{\sim}}
\zhpc \otimes_{\zh} (E,\phi) \oplus (E^\%,-\phi^\%) \ar[r]^-{(j_E,\Id)} \ar[d]_{\left( \begin{array}{cc} \Id \otimes \eta & 0 \\ 0 & \eta^\% \end{array} \right)} & (E^\%,0) \ar[d]^{\eta^\%} \\ (Y,\Phi) \oplus (Y^\%,-\Phi^\%) \ar[r]^-{(j_Y,\Id)} & (Y^\%,0).
}\]
The proof of Lemma \ref{lemma:productcobordism} shows that it is a symmetric Poincar\'{e} triad.  Applying the chain isomorphism $\varpi_{E^\%} \colon E^\% \toiso \zhpc \otimes_{\zh} E$ to the top right corner produces the triad:
\[\xymatrix @R+1cm @C+1cm {\ar @{} [dr] |{\stackrel{(k,0)}{\sim}}
\zhpc \otimes_{\zh} (E,\phi) \oplus (E^\%,-\phi^\%) \ar[r]^-{(\varpi_{E^\%}\circ j_E,\varpi_{E^\%})} \ar[d]_{\left( \begin{array}{cc} \Id \otimes \eta & 0 \\ 0 & \eta^\% \end{array} \right)} & (\zhpc \otimes_{\zh} E,0) \ar[d]^{\eta^\% \circ(\varpi_{E^\%})^{-1}} \\ (Y,\Phi) \oplus (Y^\%,-\Phi^\%) \ar[r]^-{(j_Y,\Id)} & (Y^\%,0),
}\]
as required.

The homological conditions are satisfied since the maps $j,j^{\dag}$ from Definition \ref{defn:2ndorderconcordant} are chain equivalences and the chain complex $V=Y^\%$.  The consistency condition is satisfied since the commutativity of the square \[\xymatrix @R+1cm @C+1cm{H \ar[r]_-{\xi}^-{\cong} \ar[d]_{\omega}^-{\cong} & H_1(\Z[\Z] \otimes_{\zh} Y) \ar[d]_-{j_*}^-{\cong} \\
H^\% \ar[r]_-{\xi^\%}^-{\cong} & H_1(\Z[\Z] \otimes_{\zhpc} Y^\%), }\]which shows that $(H,\Y,\xi)$ and $(H^\%,\Y^\%,\xi\%)$ are equivalent in the sense of Definition \ref{Defn:algebraicsetofchaincomplexes}, extends to show that the square
\[\xymatrix @R+1cm @C+1cm{H \oplus H^{\%} \ar[r]^{(\omega,\Id)} \ar[d]^{\left(\begin{array}{cc} \xi & 0 \\ 0 & \xi^{\%} \end{array} \right)} & H^\% \ar[d]^{\xi^\%} \\
H_1(\Z[\Z] \otimes_{\zh} Y) \oplus H_1(\Z[\Z] \otimes_{\zh} Y^{\%}) \ar[r]^-{(j_*,\Id_*)} & H_1(\Z[\Z] \otimes_{\zhd} Y^\%)
}\]
is also commutative.  Therefore Definition \ref{defn:2ndorderconcordant} is satisfied, so $\sim$ is indeed a reflexive relation.

Symmetry is straight--forward.  If $(H,\Y,\xi) \sim (H^\dag,\Y^{\dag},\xi^\dag)$, that is there is a diagram:

\[\xymatrix @R+1cm @C+1.5cm {\ar @{} [dr] |{\stackrel{(\gamma,\g^{\dag})}{\sim}}
(E,\phi) \oplus (E^{\dag},-\phi^{\dag}) \ar[r]^-{(\Id, \varpi_{E^{\dag}})} \ar[d]_{\left( \begin{array}{cc} \eta & 0 \\ 0 & \eta^{\dag}\end{array}\right)} & (E,0) \ar[d]^{\delta}
\\ (Y,\Phi) \oplus (Y^{\dag},-\Phi^{\dag}) \ar[r]^-{(j,j^{\dag})} & (V,\Theta),
}\]
with a commutative square
\[\xymatrix @R+1cm @C+1cm{H \oplus H^{\dag} \ar[r]^{(j_{\flat},j_{\flat}^{\dag})} \ar[d]^{\left(\begin{array}{cc} \xi & 0 \\ 0 & \xi^{\dag} \end{array} \right)} & H' \ar[d]^{\xi'} \\
H_1(\Z[\Z] \otimes_{\zh} Y) \oplus H_1(\Z[\Z] \otimes_{\zh} Y^{\dag}) \ar[r]^-{\Id_{\Z[\Z]} \otimes (j_*,j^{\dag}_*)} & H_1(\Z[\Z] \otimes_{\zhd} V),
}\]
then there is also a diagram:
\[\xymatrix @R+1cm @C+1.5cm {\ar @{} [dr] |{\stackrel{(\gamma^{\dag},\gamma)}{\sim}}
(E^{\dag},\phi^{\dag}) \oplus (E,-\phi) \ar[r]^-{(\Id, \varpi_{E})} \ar[d]_{\left( \begin{array}{cc} \eta^{\dag} & 0 \\ 0 & \eta \end{array}\right)} & (E^{\dag},0) \ar[d]^{\delta \circ \varpi_{E^{\dag}}}
\\ (Y^{\dag},\Phi^{\dag}) \oplus (Y,-\Phi) \ar[r]^-{(j^{\dag},j)} & (V,-\Theta)
}\]
with a commutative square
\[\xymatrix @R+1cm @C+1cm{H^{\dag} \oplus H \ar[r]^{(j_{\flat}^{\dag},j_{\flat})} \ar[d]^{\left(\begin{array}{cc} \xi^{\dag} & 0 \\ 0 & \xi \end{array} \right)} & H' \ar[d]^{\xi'} \\
H_1(\Z[\Z] \otimes_{\zh} Y^{\dag}) \oplus H_1(\Z[\Z] \otimes_{\zh} Y) \ar[r]^-{\Id_{\Z[\Z]} \otimes (j^{\dag}_*,j_*)} & H_1(\Z[\Z] \otimes_{\zhd} V),
}\]
which shows that $\sim$ is a symmetric relation.  Finally, to show transitivity, suppose that $(H,\Y,\xi) \sim (H^\dag,\Y^{\dag},\xi^\dag)$ using $H'$, and also that $(H^\dag,\Y^{\dag},\xi^\dag) \sim (H^\ddag,\Y^{\ddag},\xi^\ddag)$, using \[(\ol{j_{\flat}},\ol{j_{\flat}^{\ddag}}) \colon H^{\dag} \oplus H^{\ddag} \to \ol{H'},\] so that there is a diagram of $\Z[\Z \ltimes \ol{H'}]$-module chain complexes:
\[\xymatrix @R+1cm @C+1.5cm {\ar @{} [dr] |{\stackrel{(\ol{\gamma^{\dag}},\ol{\g^{\ddag}})}{\sim}}
(E^{\dag},\phi^{\dag}) \oplus (E^{\ddag},-\phi^{\ddag}) \ar[r]^-{(\Id, \wt{\varpi}_{E^{\ddag}})} \ar[d]_{\left( \begin{array}{cc} \eta^{\dag} & 0 \\ 0 & \eta^{\ddag}\end{array}\right)} & (E^{\dag},0) \ar[d]^{\ol{\delta}}
\\ (Y^{\dag},\Phi^{\dag}) \oplus (Y^{\ddag},-\Phi^{\ddag}) \ar[r]^-{(\ol{j^{\dag}},\ol{j^{\ddag}})} & (\ol{V},\ol{\Theta}).
}\]
In this proof the bar is a notational device and has nothing to do with involutions.  To show that $(H,\Y,\xi) \sim (H^\ddag,\Y^{\ddag},\xi^\ddag)$, first we must define a $\Z[\Z]$-module $\ol{\ol{H'}}$ so that we can tensor everything with $\Z[\Z \ltimes \ol{\ol{H'}}]$.  We will glue the symmetric Poincar\'{e} triads together to show transitivity; first we must glue together the $\Z[\Z]$-modules.  Define:
\[(j_{\flat},\ol{j_{\flat}^{\ddag}}) \colon H \oplus H^{\ddag} \to \ol{\ol{H'}} := \coker((j_{\flat}^{\dag},-\ol{j_{\flat}^{\dag}}) \colon H^{\dag} \to H' \oplus \ol{H}').\]
Now, use the inclusions followed by the quotient maps: $$H' \to H' \oplus \ol{H'} \to \ol{\ol{H'}}$$ and $$\ol{H'} \to H' \oplus \ol{H'} \to \ol{\ol{H'}}$$ to take the tensor product of both the 4-dimensional symmetric Poincar\'{e} triads which show that $(H,\Y,\xi) \sim (H^\dag,\Y^{\dag},\xi^\dag)$, and that $(H,\Y^{\dag},\xi^\dag) \sim (H^\ddag,\Y^{\ddag},\xi^\ddag)$, with $\Z[\Z \ltimes \ol{\ol{H'}}]$, so that both contain chain complexes of modules over the same ring  $\Z[\Z \ltimes \ol{\ol{H'}}]$.  Then algebraically gluing the triads together, as in \cite[pages~117--9]{Ranicki2}, we obtain the 4-dimensional symmetric Poincar\'{e} triad:
\[\xymatrix @R+2cm @C+2cm {\ar @{} [dr] |{\ol{\ol{\g}} = \left(\begin{array}{cc} \g & 0 \\  0 & 0 \\  0 & \ol{\g^{\ddag}} \\ \end{array} \right)}
(E,\phi) \oplus (E^{\ddag},-\phi^{\ddag}) \ar[r]^-{\left(\begin{array}{cc} \Id & 0 \\  0 & 0 \\  0 & \wt{\varpi}_{E^{\ddag}} \\ \end{array} \right)} \ar[d]_{\left( \begin{array}{cc} \eta & 0 \\ 0 & \eta^{\ddag}\end{array}\right)} & (\ol{\ol{E}},- 0 \cup_{\phi^{\dag}} 0) \ar[d]^{\ol{\ol{\delta}} = \left( \begin{array}{ccc} \delta & (-1)^{r-1}\g^{\dag} & 0 \\ 0 & \eta^{\dag} & 0 \\ 0 & (-1)^{r-1}\ol{\g^{\dag}} & \ol{\delta} \\ \end{array} \right)}
\\ (Y,\Phi) \oplus (Y^{\ddag},-\Phi^{\ddag}) \ar[r]_-{\left(\begin{array}{cc} j & 0 \\  0 & 0 \\  0 & \ol{j^{\ddag}} \\ \end{array} \right)} & (\ol{\ol{V}},\ol{\ol{\Theta}}).
}\]
where:
\[\ol{\ol{E}} := \mathscr{C}((\varpi_{E^{\dag}},\Id)^T \colon E^{\dag} \to E \oplus E^{\dag});\]
\[\ol{\ol{V}} := \mathscr{C}((j^{\dag},\ol{j^{\dag}})^T \colon Y^{\dag} \to V \oplus \ol{V}); \text{ and}\]
\[\ol{\ol{\Theta}} := \Theta \cup_{\Phi^{\dag}} \ol{\Theta}.\]
We need to show that this is equivalent to a triad where the top right term is $(E,0)$.  First, to see that $E \simeq \ol{\ol{E}}$, the chain complex of $\ol{\ol{E}}$ is given by:
\[E_2^{\dag} \xrightarrow{\partial^{\ol{\ol{E}}}_3} E_2 \oplus E_1^{\dag} \oplus E_2^{\dag} \xrightarrow{\partial^{\ol{\ol{E}}}_2} E_1 \oplus E_0^{\dag} \oplus E_1^{\dag} \xrightarrow{\partial^{\ol{\ol{E}}}_1} E_0 \oplus E_0^{\dag},\]
where:
\[\partial^{\ol{\ol{E}}}_3 = \left(
                               \begin{array}{c}
                                 \varpi_{E^{\dag}} \\
                                 \partial_{E^{\dag}} \\
                                 \Id \\
                               \end{array}
                             \right);
\]
\[\partial^{\ol{\ol{E}}}_2 = \left(
                               \begin{array}{ccc}
                                 \partial_{E} & -\varpi_{E^{\dag}} & 0 \\
                                 0 & \partial_{E^{\dag}} & 0 \\
                                 0 & -\Id & \partial_{E^{\dag}} \\
                               \end{array}
                             \right); \text{ and}\]
\[\partial^{\ol{\ol{E}}}_1 = \left(
                               \begin{array}{ccc}
                                 \partial_E & \varpi_{E^{\dag}} & 0 \\
                                 0 & \Id & \partial_{E^{\dag}} \\
                               \end{array}
                             \right),
\]
and we have the chain map:
\[\nu' := \left(
     \begin{array}{ccc}
       \Id\, , & 0\, , & -\varpi_{E^{\dag}} \\
     \end{array}
   \right) \colon E_r \oplus E_{r-1}^{\dag} \oplus E_{r}^{\dag} \to E_r,\]
with a chain homotopy inverse:
\[ \nu'^{-1} := \left(
    \begin{array}{c}
      \Id \\
      0 \\
      0 \\
    \end{array}
  \right)\colon E_r \to E_r \oplus E_{r-1}^{\dag} \oplus E_r^{\dag}.\]
Now,
\[\left(
     \begin{array}{ccc}
       \Id\, , & 0\, , & -\varpi_{E^{\dag}} \\
     \end{array}
   \right)
\left(
    \begin{array}{c}
      \Id \\
      0 \\
      0 \\
    \end{array}
  \right) - \left(
              \begin{array}{c}
                \Id \\
              \end{array}
            \right) = \left(
              \begin{array}{c}
                0 \\
              \end{array}
            \right),
  \]
whereas
\[\left(
    \begin{array}{c}
      \Id \\
      0 \\
      0 \\
    \end{array}
  \right)
  \left(
     \begin{array}{ccc}
       \Id\, , & 0\, , & -\varpi_{E^{\dag}} \\
     \end{array}
   \right) - \left(
               \begin{array}{ccc}
                 \Id & 0 & 0 \\
                 0 & \Id & 0 \\
                 0 & 0 & \Id \\
               \end{array}
             \right) = \left(
                         \begin{array}{ccc}
                           0 & 0 & -\varpi_{E^{\dag}} \\
                           0 & -\Id & 0 \\
                           0 & 0 & -\Id \\
                         \end{array}
                       \right)
\]
which is equal to $k'\partial + \partial k'$ where the chain homotopy $k'$ is given by:
\[k' = \left(
         \begin{array}{ccc}
           0 & 0 & 0 \\
           0 & 0 & (-1)^{r+1}\Id \\
           0 & 0 & 0 \\
         \end{array}
       \right)
 \colon E_r \oplus E_{r-1}^{\dag} \oplus E_r \to E_{r+1} \oplus E_r^{\dag} \oplus E_{r+1}^{\dag}.\]
We therefore have the diagram:
\[\xymatrix @R+2cm @C+2cm {
 & & (E,0) \ar@/^5pc/[ddl]^{\ol{\ol{\delta}}\circ \nu'^{-1}}\\
(E,\phi) \oplus (E^{\ddag},-\phi^{\ddag}) \ar@/^5pc/[rru]^>>>>>>>>>>>>>>>{\left(\Id,- \varpi_{E^{\dag}} \circ \wt{\varpi}_{E^{\ddag}} \right)
} \ar[r]^-{\left(\begin{array}{cc} \Id & 0 \\  0 & 0 \\  0 & \wt{\varpi}_{E^{\ddag}} \\ \end{array} \right)} \ar @{} [dr] |{\stackrel{\ol{\ol{\g}}}{\sim}} \ar[d] & (\ol{\ol{E}},- 0 \cup_{\phi^{\dag}} 0) \ar[ru]^{\simeq}_{\nu'} \ar[d]^{\ol{\ol{\delta}}}
\\ (Y,\Phi) \oplus (Y^{\ddag},-\Phi^{\ddag}) \ar[r] & (\ol{\ol{V}},\ol{\ol{\Theta}}).
}\]
The top triangle commutes, while the bottom triangle commutes up to the homotopy $k'$: $k'$ gets composed with $\ol{\ol{\gamma}}$ to make the new triad.  Furthermore, $$\nu' (- 0 \cup_{\phi^{\dag}} 0) \nu'^* = 0,$$ so that we indeed have an equivalent triad with the top right as $(E,0)$.  To complete the proof, we need to see that the consistency condition holds.  The following commutative diagram has exact columns, the right hand column being part of the Mayer-Vietoris sequence.  The horizontal maps are given by consistency isomorphisms.  Recall that $\ol{\ol{H'}} := \coker((j_{\flat}^{\dag},-\ol{j_{\flat}^{\dag}}) \colon H^{\dag} \to H' \oplus \ol{H}').$  All homology groups in this diagram are taken with $\Z[\Z]$-coefficients.
\[\xymatrix @R+0.5cm @C+0.6cm {H^\dag \ar[d] \ar[rrr]^-{\xi^\dag}_-{\cong} & & & H_1(Y^\dag) \ar[d] \\
 H' \oplus \ol{H'} \ar[dd] \ar[rrr]_-{\cong}^-{\left(\ba{cc}\xi' & 0 \\ 0 & \ol{\xi'} \ea\right)} & & & H_1(V) \oplus H_1(\ol{V}) \ar[dd]\\
 & H \oplus H^\ddag \ar[ul] \ar@{-->}[dl] \ar[r]_-{\cong}^-{\left(\ba{cc}\xi & 0 \\ 0 & \xi^\ddag \ea\right)} & H_1(Y) \oplus H_1(Y^\ddag) \ar[ur] \ar@{-->}[dr] & \\
\ol{\ol{H'}} \ar[d] \ar@{-->}[rrr]_-{\cong}^-{\ol{\ol{\xi'}}} & & & H_1(\ol{\ol{V}}) \ar[d]\\
   0 & & & 0}\]
The diagonal dotted arrows are induced by the diagram, so as to make it commute.  The horizontal dotted arrow $\ol{\ol{H'}} \to H_1(\Z[\Z] \otimes_{\Z[\Z \ltimes \ol{\ol{H'}}]}\ol{\ol{V}})$ is induced by a diagram chase: the quotient map $H' \oplus \ol{H'} \to \ol{\ol{H'}}$ is surjective.  We obtain a well--defined isomorphism
\[\ol{\ol{\xi'}} \colon \ol{\ol{H'}} \toiso H_1(\Z[\Z] \otimes_{\Z[\Z \ltimes \ol{\ol{H'}}]}\ol{\ol{V}}).\]
The commutativity of the diagram above implies the commutativity of the induced diagram:
\[\xymatrix @R+1cm @C+1cm{H \oplus H^{\ddag} \ar[r] \ar[d]^{\left(\begin{array}{cc} \xi & 0 \\ 0 & \xi^{\ddag} \end{array} \right)} & \ol{\ol{H'}} \ar[d]^-{\ol{\ol{\xi'}}} \\
H_1(\Z[\Z] \otimes_{\zh} Y) \oplus H_1(\Z[\Z] \otimes_{\zh} Y^{\ddag}) \ar[r] & H_1(\Z[\Z] \otimes_{\zhd} \ol{\ol{V}}).
}\]
This completes the proof that $\sim$ is transitive and therefore completes the proof that $\sim$ is an equivalence relation.
\end{proof}

\begin{definition}\label{defn:inverses}
Given an element $(H,\Y,\xi) \in \P$, choose a representative with the boundary given by the model chain complexes.
\[\xymatrix @R+0.5cm @C+0.5cm {\ar @{} [dr] |{\stackrel{g}{\sim}}
(C,\varphi \oplus -\varphi) \ar[r]^{i_-} \ar[d]_{i_+} & (D_-,0) \ar[d]^{f_-}\\ (D_+,0) \ar[r]^{f_+} & (Y,\Phi).
}\]
The following is also a symmetric Poincar\'{e} triad:
\[\xymatrix @R+0.5cm @C+0.5cm {\ar @{} [dr] |{\stackrel{g}{\sim}}
(C,-\varphi \oplus \varphi) \ar[r]^{i_-} \ar[d]_{i_+} & (D_-,0) \ar[d]^{f_-}\\ (D_+,0) \ar[r]^{f_+} & (Y,-\Phi).
}\]
which define as the element $-\Y$.  This is the algebraic equivalent of changing the orientation of the ambient space and of the knot simultaneously. The chain equivalence:
\[\varsigma = \left(
    \begin{array}{cc}
      0 & l_a \\
      l_a^{-1} & 0 \\
    \end{array}
  \right)
 \colon C_i \to C_i\]
for $i=0,1$ sends $\varphi \oplus -\varphi$ to $-\varphi \oplus \varphi$ and satisfies $i_{\pm} \circ \varsigma = i_{\pm}$.  We can therefore define the inverse $-(H,\Y,\xi) \in \P$ to be the triple $(H,-\Y,\xi)$, where $-\Y$ is the symmetric Poincar\'{e} triad:
\[\xymatrix @R+0.5cm @C+0.5cm {\ar @{} [dr] |{\stackrel{g\circ \varsigma}{\sim}}
(C,\varphi \oplus -\varphi) \ar[r]^{i_-} \ar[d]_{i_+} & (D_-,0) \ar[d]^{f_-}\\ (D_+,0) \ar[r]^{f_+} & (Y,-\Phi),
}\]
Summarising, to form an inverse we replace $g$ with $g \circ \varsigma$, and change the sign on the symmetric structures everywhere but on $C$ in the top left of the triad.
\qed \end{definition}

\begin{figure}[h]
    \begin{center}
 {\psfrag{A}{$(Y,\Phi)$}
 \psfrag{B}{$(D_+,0)$}
 \psfrag{C}{$(D_-,0)$}
 \psfrag{D}{$(V,\Theta)$}
 \psfrag{E}{$(D_-,0)$}
 \psfrag{F}{$(D_+,0)$}
 \psfrag{G}{$(D^{\dag}_-,0)$}
 \psfrag{H}{$(Y^{\dag},-\Phi^{\dag})$}
 \psfrag{J}{$(D_+^{\dag},0)$}
 \psfrag{K}{$(C,\varphi \oplus - \varphi)$}
 \psfrag{L}{$(C^{\dag},\varphi^{\dag} \oplus -\varphi^{\dag})$}
\includegraphics[width=8cm]{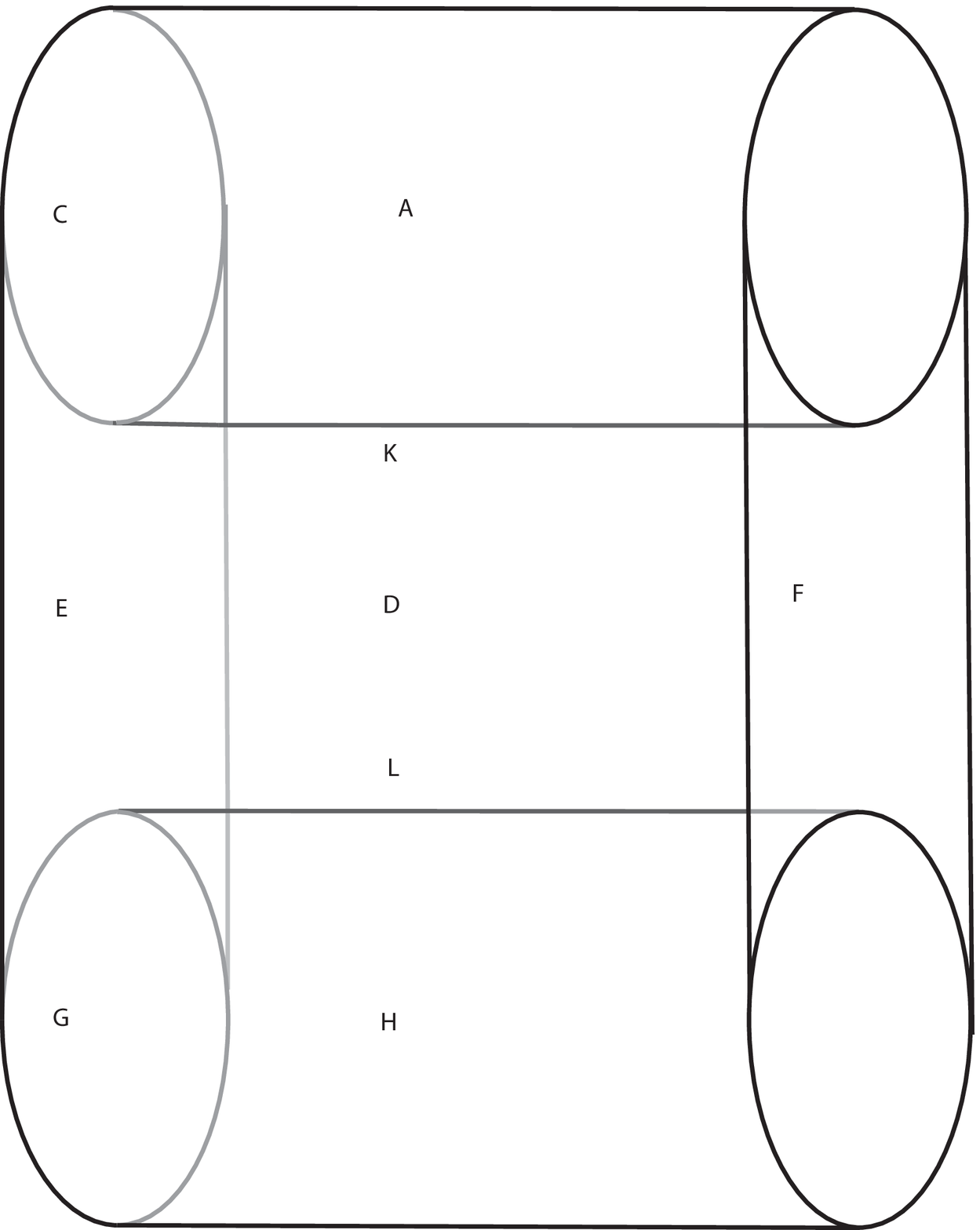}
 }
 \caption{The cobordism which shows that $\mathcal{Y} \sim \mathcal{Y}^{\dag}$.}
 \label{Fig:existenceofinverse2}
 \end{center}
\end{figure}

\begin{figure}[h]
    \begin{center}
 {\psfrag{A}{$(Y,\Phi)$}
 \psfrag{B}{$D_-^{\dag}$}
 \psfrag{C}{$D_-$}
 \psfrag{D}{$(V,\Theta)$}
 \psfrag{E}{$D_-$}
 \psfrag{F}{$D_-^{\dag}$}
 \psfrag{G}{$D_-$}
 \psfrag{H}{$D_- = Y^U$}
 \psfrag{J}{$D_-$}
 \psfrag{K}{$C$}
 \psfrag{L}{$C$}
 \psfrag{M}{$D_+$}
 \psfrag{N}{$D_+^{\dag}$}
 \psfrag{P}{$D_+^{\dag}$}
 \psfrag{Q}{$(Y^{\dag},-\Phi^{\dag})$}
 \psfrag{S}{$C^{\dag}$}
 \includegraphics[width=9cm]{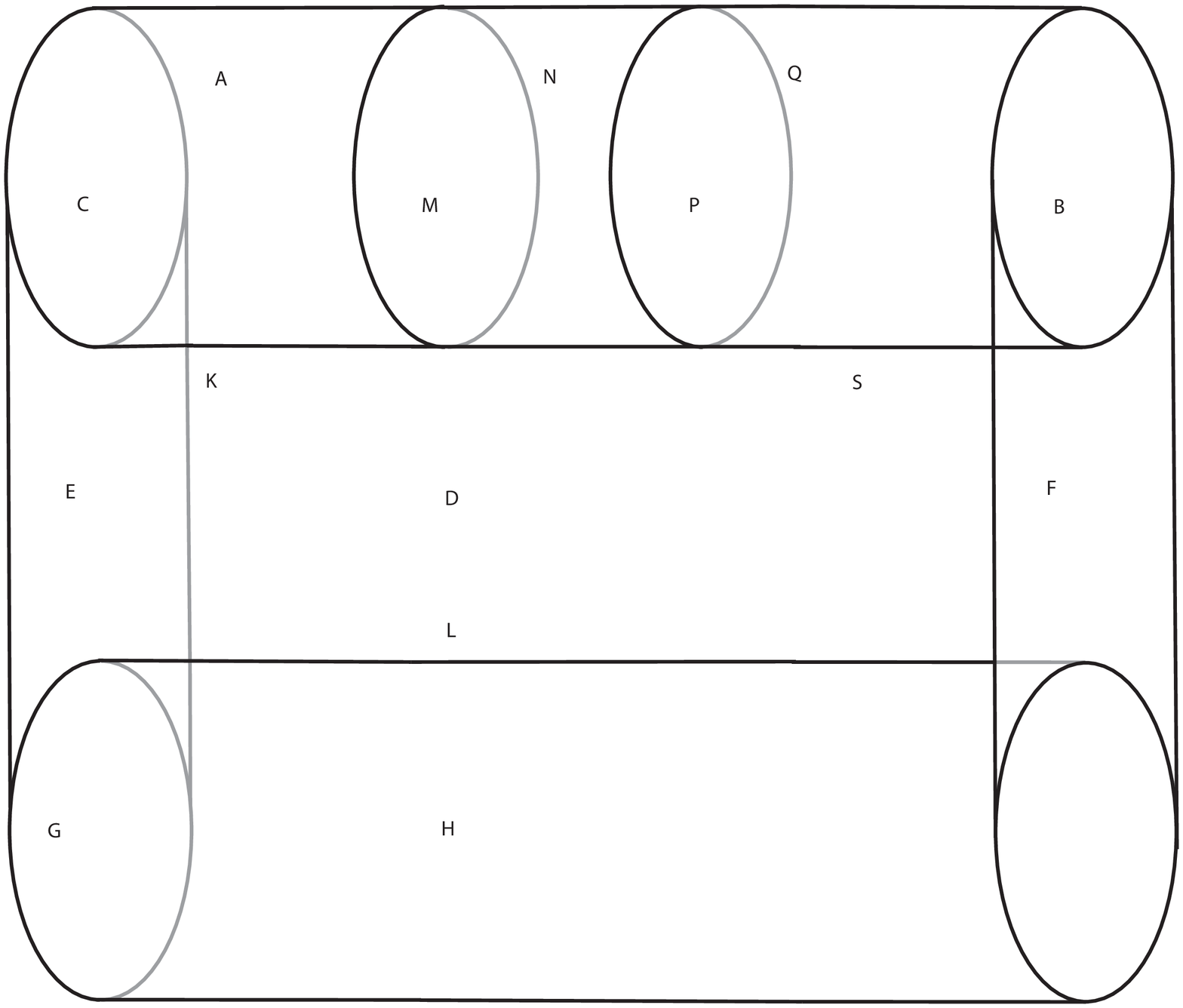}
 }
 \caption{The cobordism which shows that $\mathcal{Y} \, \sharp \, - \mathcal{Y}^{\dag} \sim \mathcal{Y}^U$.}
 \label{Fig:existenceofinverse3}
 \end{center}
\end{figure}

\begin{remark}
We now describe in detail why, for two knots $K$ and $K^{\dag}$, $K \, \sharp \, - K^{\dag}$ is slice if and only if $K$ is concordant to $K^{\dag}$.  The manifold $Z$ obtained by gluing two knot exteriors $X$ and $X^{\dag}$ together along their boundaries:
\[Z := X \cup_{\partial X = S^1 \times S^1} S^1 \times S^1 \times I \cup_{S^1 \times S^1 = \partial X^{\dag}} -X^{\dag},\]
as in Proposition \ref{lemma:basicfactconcordance}, can also be decomposed in a different way using the splitting of the boundary as $S^1 \times D^1 \cup_{S^1 \times S^0 \times I} S^1 \times D^1$.  First, using half the boundary we have the exterior of the connected sum:
\[X^{\ddag} = X \cup_{(S^1 \times D^1)_+} S^1 \times D^1 \times I \cup_{(S^1 \times D^1)^{\dag}_+} -X^{\dag},\]
so that
\begin{eqnarray*}
Z &\approx & X^{\ddag} \cup_{(S^1 \times D^1)_- \cup_{S^1 \times S^0} (S^1 \times D^1)^{\dag}_-} S^1 \times D^1 \times I \\
 & \approx & X^{\ddag} \cup_{S^1 \times S^1} S^1 \times S^1 \times I \cup_{S^1 \times S^1} X^{U},
\end{eqnarray*}
since \[S^1 \times D^1 \times I \approx S^1 \times D^2 \approx S^1 \times S^1 \times I \cup_{S^1 \times S^1} S^1 \times D^2  \approx S^1 \times S^1 \times I \cup_{S^1 \times S^1} X^U,\]
where $X^U \approx S^1 \times D^2$ is the exterior of the unknot.  The same 4-manifold therefore shows that $K$ is concordant to $K^{\dag}$ and that $K \,\sharp \,-K^{\dag}$ is concordant to the unknot.  For a schematic of the former cobordism see Figure \ref{Fig:existenceofinverse2} and for a schematic of the latter see Figure \ref{Fig:existenceofinverse3}.  We proceed in the next proposition to copy this motivating geometric argument in algebra.
\end{remark}

\begin{proposition}\label{prop:inverseswork}
Recall that $(\{0\},\Y^U,\Id_{\{0\}})$ is the triple of the unknot, and let $(H,\Y,\xi)$ and $(H^\dag,\Y^{\dag},\xi^\dag)$ be two triples in $\P$.  Then $$(H,\Y,\xi) \, \sharp \, -(H^\dag,\Y^{\dag},\xi^\dag) \sim (\{0\},\Y^U,\Id_{\{0\}})$$ if and only if $$(H,\Y,\xi) \sim (H^\dag,\Y^{\dag},\xi^\dag).$$
\end{proposition}

\begin{proof}
First, note that both boundaries use the Alexander module $H \oplus H^{\dag}$, so that the same homomorphism
\[(j_{\flat},j_{\flat}^{\dag}) \colon H \oplus H^{\dag} \to H'\]
can be used in both equivalences, fitting into the same commutative square.

The next step is to switch $D_-^{\dag}$ and $D_+^{\dag}$ in the symmetric Poincar\'{e} triad $-\Y^{\dag}$.  This is possible thanks to the chain homotopy $\mu^{\dag} \colon f_+^{\dag} \circ \varpi^{\dag} \simeq f_-^{\dag}$.  We have the following diagram for the equivalence of symmetric triads:
\[\xymatrix @R+1.5cm @C+1.3cm {
 & & & D_+^{\dag} \ar@/^5pc/[ldd]^{f_+^{\dag}} \\
 & \ar @{} [dr] |{\stackrel{g^{\dag}\circ \varsigma^{\dag}}{\sim}}
(C^{\dag},\varphi^{\dag} \oplus -\varphi^{\dag}) \ar[r]_{i^{\dag}_-} \ar[d]^{i^{\dag}_+} \ar@/^5pc/[rru]^{i_+^{\dag}} \ar@/_5pc/[ldd]_{i_-^{\dag}} & (D^{\dag}_-,0) \ar[d]_{f^{\dag}_-} \ar[ru]^{\varpi^{\dag}} \ar @{} [r] |{\stackrel{\mu^{\dag}}{\sim}} & \\  & (D^{\dag}_+,0) \ar[r]^{f^{\dag}_+} \ar[ld]_{(\varpi^{\dag})^{-1}} \ar @{} [d] |{\stackrel{\mu^{\dag} \circ (\varpi^{\dag})^{-1}}{\sim}} & (Y^{\dag},-\Phi^{\dag}) & \\ D_-^{\dag} \ar@/_5pc/[rru]_{f_-^{\dag}} & & &.
}\]
The outside square becomes the new triad $-\Y^{\dag}$, with all the chain homotopies shown combined to become a single homotopy.  We now follow the geometric argument above to construct something chain equivalent to the chain complex $E \cup_{E \oplus E^{\dag}} Y \oplus Y^{\dag}$, over $\Z[\Z \ltimes (H\oplus H^{\dag})]$, which must be the boundary of a 4-dimensional symmetric Poincar\'{e} pair in order for $(H,\Y,\xi)$ and $(H^\dag,\Y^{\dag},\xi^\dag)$ to be second order algebraically concordant.  The reader is advised to follow the rest of this proof while looking at Figures \ref{Fig:existenceofinverse2} and \ref{Fig:existenceofinverse3}.

To glue $Y$ and $Y^{\dag}$ together we use only the $D_+$ part of $E$ to begin with.  Note that:
\[Y^{\ddag}_r = Y_r \oplus (D_+^{\dag})_{r-1} \oplus Y^{\dag} \simeq Y_r \oplus (D_+)_{r-1} \oplus (D_+)_{r} \oplus (D_+^{\dag})_{r-1} \oplus Y_r^{\dag}.\]
Now, to form the manifold $Z$ we attached another $S^1 \times D^1 \times I$ to this, which corresponds to attaching the chain complex $D_-$.  However, we can first take the algebraic mapping cylinder of the map $E^{\ddag} \simeq E^U \to Y^U = D_-$ to see that:
\begin{multline*} Y^U_r = (D_-)_r \simeq \\ E^\ddag_r \oplus E^U_{r-1} \oplus Y_U \simeq (D_-)_r \oplus C^{\dag}_{r-1} \oplus (D_-^{\dag})_r \oplus (D_-)_{r-1} \oplus C_{r-2} \oplus (D_-)_{r-1} \oplus (D_-)_r.\end{multline*}
Therefore, gluing $Y^{\ddag}$ to $D_- = Y^U$ along $E^{\ddag}$ we make the chain complex $E^{\ddag} \cup_{E^{\ddag} \oplus E^U} Y^{\ddag} \oplus Y^U$:
\begin{eqnarray*}Y^\ddag_r \oplus E^\ddag_{r-1} \oplus Y^U_{r} \simeq \\ Y_r \oplus (D_+)_{r-1} \oplus (D_+)_{r} \oplus (D_+^{\dag})_{r-1} \oplus Y_r^{\dag} \oplus (D_-)_{r-1} \oplus C^{\dag}_{r-2} \oplus (D_-^{\dag})_{r-1} \\ \oplus (D_-)_r \oplus C^{\dag}_{r-1} \oplus (D_-^{\dag})_r \oplus (D_-)_{r-1} \oplus C_{r-2} \oplus (D_-)_{r-1} \oplus (D_-)_r\end{eqnarray*}
which is the chain complex over $\Z[\Z \ltimes (H\oplus H^{\dag})]$ which must bound a 4-dimensional symmetric Poincar\'{e} pair in order for $(H,\Y,\xi) \, \sharp \, -(H^\dag,\Y^{\dag},\xi^\dag)$ to be second order algebraically null-concordant.

Finally notice that $E^{\ddag} \cup_{E^{\ddag} \oplus E^U} Y^{\ddag} \oplus Y^U$ is, as claimed, chain equivalent to $E \cup_{E \oplus E^{\dag}} Y \oplus Y^{\dag}$, which is the complex which we were constructing to begin with.  To see this, glue $D_- = Y^U$ on to $Y^{\ddag}$ as above, again along $E^\ddag_r = (D_-)_r \oplus C^{\dag}_{r-1} \oplus (D_-^\dag)_r$ but without expanding $D_-$ first, to get:
\begin{eqnarray*}  Y_r \oplus (D_+)_{r-1} \oplus (D_+)_{r} \oplus (D_+^{\dag})_{r-1} \oplus Y_r^{\dag} \oplus (D_-)_{r-1} \oplus C^{\dag}_{r-2} \oplus (D_-^{\dag})_{r-1} \oplus (D_-)_r \simeq & \\  Y_r \oplus (D_+)_{r-1} \oplus (D_+)_{r} \oplus (D_+^{\dag})_{r-1} \oplus Y_r^{\dag} \oplus (D_-)_{r-1} & \\ \oplus C_{r-2} \oplus C_{r-1} \oplus C^{\dag}_{r-2} \oplus (D_-^{\dag})_{r-1} \oplus (D_-)_r  = E \cup_{E \oplus E^{\dag}} Y \oplus Y^{\dag}. &
\end{eqnarray*}
We expand $C$ in the chain equivalence here to get the algebraic equivalent of $S^1 \times S^0 \times I \times I$ inside the $S^1 \times S^1 \times I$, represented by $E$, which glues together $X$ and $-X^{\dag}$ to form $Z$ as in Proposition \ref{lemma:basicfactconcordance}.  Since the two chain complexes are chain equivalent, we see that if one chain complex $(V,\Theta)$ which fits into a 4-dimensional symmetric Poincar\'{e} triad exists, then so does the other, since we can compose the equivalences with the maps in the triad which we know exists, to show that the maps exist in the other triad.  This completes the proof.
\end{proof}

\begin{remark}
Proposition \ref{prop:inverseswork} shows that the putative inverse defined in Definition \ref{defn:inverses} does indeed give us an inverse, so that we have completed the task of showing the existence of the diagram below:
\[\xymatrix @R+1cm @C+1cm{
Knots \ar[r] \ar @{->>}[d] & \P \ar @{->>}[d]  \\
\C \ar[r] & \mathcal{AC}_2,
}\]
with $\C \to \ac2$ a group homomorphism, as we set out to achieve in this chapter.  We proceed in the next chapter to show the relationship of our constructions to the \COT filtration.
\end{remark}



\chapter{$(1.5)$-Solvable Knots are Algebraically $(1.5)$-Solvable}\label{chapter:one_point_five_solvable_knots}



This Chapter contains the proof of the following theorem.

\begin{theorem}\label{Thm:1.5solvable=>2nd_alg_slice}
A $(1.5)$-solvable knot is algebraically $(1.5)$-solvable.
\end{theorem}

We proceed as follows.  After recalling some definitions, we give a motivating discussion.  The main tool for the proof will be the chain complex operation of algebraic surgery, so before giving the proof of Theorem \ref{Thm:1.5solvable=>2nd_alg_slice} we introduce and explain this theory.

To aid the ensuing discussion we first recall once again the definition of $(1.5)$-solubility (from Definition \ref{Defn:COTnsolvable}) and the definition of geometric surgery.

\begin{definition}[\cite{COT} Definition 1.2]\label{Defn:1.5solvableagain}
A \emph{Lagrangian} of a symmetric form $\lambda \colon P \times P \to R$ on a free $R$-module $P$ is a submodule $L \subseteq P$ of half-rank on which $\lambda$ vanishes.  For $n \in \mathbb{N}_0 := \mathbb{N} \cup \{0\}$, let $\lambda_n$ be the intersection form, and $\mu_n$ the self-intersection form, on the middle dimensional homology $H_2(W^{(n)};\Z) \cong H_2(W;\Z[\pi_1(W)/\pi_1(W)^{(n)}])$  of the $n$th derived cover of a 4-manifold $W$, that is the regular covering space $W^{(n)}$ corresponding to the subgroup $\pi_1(W)^{(n)} \leq \pi_1(W)$:
\begin{multline*} \lambda_n \colon H_2(W;\Z[\pi_1(W)/\pi_1(W)^{(n)}]) \times H_2(W;\Z[\pi_1(W)/\pi_1(W)^{(n)}]) \\ \to \Z[\pi_1(W)/\pi_1(W)^{(n)}].\end{multline*}
An $(n)$-\emph{Lagrangian} is a submodule of $H_2(W;\Z[\pi_1(W)/\pi_1(W)^{(n)}])$, on which $\lambda_n$ and $\mu_n$ vanish, which maps via the covering map onto a Lagrangian of $\lambda_0$.

We say that a knot $K$ is $(1.5)$-solvable if the zero surgery $M_K$ bounds a topological spin 4-manifold $W$ such that the inclusion induces an isomorphism on first homology and such that $W$ admits a $(2)$-Lagrangians with a dual $(1)$-Lagrangian.  In this setting, dual means that $\lambda_1$ pairs the image of the $(2)$-Lagrangian non-singularly with the $(1)$-Lagrangian, and that their images freely generate $H_2(W;\Z)$.
\qed \end{definition}

\begin{remark}
The symmetric structure is not subtle enough to allow us to define the self-intersection forms $\mu_n$.  For this, one needs a quadratic enhancement of the symmetric structure.  Our obstruction theory is really obstructing knots from being \emph{rationally} $(1.5)$-solvable, as in \cite[Section~4]{COT}, and in the \CHL work (\cite{cohale}, \cite{cohale2}).  At the $\Z[\Z]$ level, however, there is no difference between the symmetric and quadratic theories \cite[Proposition~7.9.2~(ii)]{Ranicki2}.
\end{remark}

\begin{definition}\label{Defn:Geometricsurgery}
An \emph{elementary geometric $r$-surgery} on an $n$-dimensional manifold $M$ has as data an embedding $g \colon S^r \times D^{n-r} \hookrightarrow M$.  The effect of the surgery is the manifold
\[M' = \cl(M \setminus g(S^r \times D^{n-r})) \cup_{S^r \times S^{n-r-1}} D^{r+1} \times S^{n-r-1}\]
which is the result of cutting out our embedded thickened sphere and gluing in instead $D^{r+1} \times S^{n-r-1}$.
There is a cobordism, called the \emph{trace} of the surgery:
\[W = M \times I \cup_{g \colon S^r \times D^{n-r} \hookrightarrow M \times \{1\}} D^{r+1} \times D^{n-r}\]
between $M$ and $M'$.
Up to homotopy equivalence, $M'$ is the result of attaching an $(r+1)$-cell to $M$ along $g|_{S^r \times \{0\}}$, and then removing a dual cell of dimension $(n-r)$.
\qed \end{definition}

The idea of the proof of Theorem \ref{Thm:1.5solvable=>2nd_alg_slice} is as follows.  The Cappell-Shaneson technique looks for obstructions to being able to perform surgery on a 4-manifold $W$ whose boundary is the zero framed surgery $M_K$, in order to excise the second $\Z$-homology and create a homotopy slice disc exterior.  The main obstruction to being able to do this surgery is the middle-dimensional intersection form of $W$, as in the \COT definition of $(n)$-solubility.  However, even if the Witt class of the intersection form vanishes, with coefficients in $\Z[\pi_1(W)/\pi_1(W)^{(2)}]$ for testing $(1.5)$-solubility, this does not imply that we have a half basis of the second homology  $$H_2(W;\Z[\pi_1(W)/\pi_1(W)^{(2)}])$$  representable by disjointly embedded spheres, as our data for surgery: typically the homology classes will be represented as embedded surfaces of non-zero genus, whose fundamental group maps into $\pi_1(W)^{(2)}$.  We cannot do surgery on such surfaces.

However, the conditions on a $(1.5)$-solution are, as we shall see, precisely the conditions required for being able to perform \emph{algebraic surgery on the chain complex} of the $(1.5)$-solution.  The $(1.5)$-level algebra cannot see the differences between $(2)$-surfaces and spheres, so that we can obtain an \emph{algebraic $(1.5)$-solution} $V$.

In particular, the existence of the dual $(1)$-Lagrangian allows us to perform algebraic surgery \emph{without changing the first homology} at the $\Z[\Z]$ level, therefore maintaining the consistency condition.  When performing geometric surgery on a 4-manifold $W$ along a 2-sphere, we remove $S^2 \times D^2$ and glue in $D^3 \times S^1$.  As mentioned in Remark \ref{Rmk:nsolvablity}, removing the thickening $D^2$ potentially creates new elements of $H_1(W;\Z[\Z])$.  However, the existence of a dual surface to the $S^2$ which we remove guarantees that the boundary $S^1$ of the thickening $D^2$ bounds a surface on the other side, so that we do not create extra 1-homology.  This phenomenon will also be seen when performing algebraic surgery; as ever, the degree of verisimilitude provided by the chain level approach is as high as one could ever hope.

Next, we give the definition of the algebraic surgery operation, which is the chain complex version of the surgery operation on manifolds, followed by some motivation of the construction.

\begin{definition}
An $n$-dimensional symmetric complex $(C,\varphi \in Q^n(C,\eps))$ is \emph{connected} if
\[H_0(\varphi_0 \colon C^{n-*} \to C_*) = 0.\]
An $n$-dimensional symmetric pair
\[(f \colon C \to D,(\delta\varphi,\varphi) \in Q^n(f,\eps))\]
is \emph{connected} if
\[H_0( \left(\begin{array}{c} \delta\varphi_0 \\ \varphi_0 f^* \end{array} \right) \colon D^{n-*} \to \mathscr{C}(f)_*) = 0.\]
\qed \end{definition}

\begin{definition}\label{Defn:algebraicsurgerymatrices}
\cite[Part~I,~page~145]{Ranicki3} Algebraic surgery is a machine which takes as input a connected $n$-dimensional symmetric chain complex over a ring $A$, $(C,\varphi \in Q^n(C,\eps))$, and which takes as data a connected $(n+1)$-dimensional symmetric pair:
\[(f \colon C \to D,(\delta\varphi,\varphi) \in Q^{n+1}(f,\eps)).\]
The output, or effect, of algebraic surgery is the connected $n$-dimensional symmetric chain complex over $A$, $(C',\varphi' \in Q^n(C',\eps))$, given by:
\begin{eqnarray*}
d_{C'} &= &\left(
             \begin{array}{ccc}
               d_C & 0 & (-1)^{n+1}\varphi_0f^* \\
               (-1)^rf & d_D & (-1)^r\delta\varphi_0 \\
               0 & 0 & (-1)^r\delta_D \\
             \end{array}
           \right)
\colon \\& &C'_r = C_r \oplus D_{r+1} \oplus D^{n-r+1} \to C'_{r-1} = C_{r-1} \oplus D_r \oplus D^{n-r+2},\end{eqnarray*}

with the symmetric structure given by:
\begin{eqnarray*}
\varphi'_0 &=& \left(
                 \begin{array}{ccc}
                   \varphi_0 & 0 & 0 \\
                   (-1)^{n-r}fT_{\eps}\varphi_1 & (-1)^{n-r}T_{\eps}\delta\varphi_1 & (-1)^{r(n-r)}\eps \\
                   0 & 1 & 0 \\
                 \end{array}
               \right)
               \colon\\& & C'^{n-r} = C^{n-r} \oplus D^{n-r+1} \oplus D_{r+1} \to C'_r = C_r \oplus D_{r+1} \oplus D^{n-r+1};\text{ and}\\
\varphi'_s &=&  \left(
                 \begin{array}{ccc}
                   \varphi_s & 0 & 0 \\
                   (-1)^{n-r}fT_{\eps}\varphi_{s+1} & (-1)^{n-r}T_{\eps}\delta\varphi_{s+1} & 0 \\
                   0 & 0 & 0 \\
                 \end{array}
               \right)
               \colon\\& & C'^{n-r+s} = C^{n-r+s} \oplus D^{n-r+s+1} \oplus D_{r-s+1} \to C'_r = C_r \oplus D_{r+1} \oplus D^{n-r+1}
\end{eqnarray*}
for $s \geq 1$.
The reader can check that $d_{C'}^2 = 0$ and that $\{\varphi'_s\} \in Q^n(C',\eps)$.
Algebraic surgery on a chain complex which is symmetric but not Poincar\'{e} preserves the homotopy type of the boundary: see \cite[Part~I,~Proposition~4.1~(i)]{Ranicki3} for the proof.
\qed \end{definition}

\begin{definition}
The \emph{suspension morphism} $S$ on chain complexes raises the degree: $(SC)_r = C_{r-1};\; d_{SC} = d_C$.
\qed \end{definition}

\begin{remark}
We give some geometric motivation for the formulae of algebraic surgery.  When performing algebraic surgery, the complex $D$ corresponds to the geometric relative complex $C(W,M')$.  For an \emph{elementary algebraic $r$-surgery}, which should correspond to an elementary geometric surgery, by excision $D = C(W,M') \simeq C(D^{n-r},S^{n-r-1}) \simeq S^{n-r}A$. There is a chance for $\delta\varphi_0 \colon D^{n+1-r} \to D_r$ to be non-zero if $r=n+1-r$; $\delta\varphi_0$ is necessarily zero otherwise.  In general for an elementary algebraic $r$-surgery there will only be one chance for a non-zero $\delta\varphi_s \colon D^{n+1+s-j} \to D_j$, precisely when $s=2j-n-1$ and $j=n-r$, so $s=n-2r-1$.  The choice of $\delta\varphi$ represents the choice of the framing, that is a choice of trivialisation of the normal bundle of our embedded sphere $S^r$.

Throughout an algebraic surgery operation the ring $A$ remains unchanged; for the low-dimensional examples which we are interested in we have to take care of any changes in fundamental group and therefore in the group ring separately, as we have done throughout this work.

By Poincar\'{e}-Lefschetz duality, $D^r = C(W,M')^{r} \simeq C(W,M)_{n+1-r}$ via $\delta\varphi_0$.  Consider the cofibration sequence:
\[C(M) \to C(W) \to C(W,M) \to SC(M) \to SC(W).\]
By taking the algebraic mapping cone on the map $$\varphi_0 f^* \colon D^r \to C_{n-r} = SC_{n-r+1} = SC(M)_{n-r+1},$$ we can attach cells algebraically, and recover the complex $SC(W)$.  Note that for an elementary algebraic surgery the image of the map $f^*$ is the cohomology class which is dual to the homology class we are trying to kill.  In geometry, this is the homology class given under the Hurewicz map by the map \[g|_{S^r \times \{0\}} \colon S^r \to M \in \pi_r(M),\] where $g$ is the data for the corresponding geometric surgery.  We can therefore see that taking a mapping cone on $\varphi_0f^*$ attaches algebraically the required $(r+1)$-cell.

The key fact then is that we can always trivially desuspend algebraically; just lower indices.  Geometrically desuspending is often difficult and in general not possible.  We can therefore recover $C(W)$ from $SC(W)$.  Consider another cofibration sequence:
\[C(M') \to C(W) \to C(W,M') \to SC(M').\]
We can now take another algebraic mapping cone on the map $(f,\delta\varphi_0) \colon C(W) = C_r \oplus D^{n-r+1} \to C(W,M')_r = D_r$ to obtain $SC(M')$.
Recall that above we used the dual complex $C(W,M')^{n+1-r}$ to represent the complex $C(W,M)_{r}$ without using the duality map, $\delta\varphi_0$, whence its inclusion here.  This has the effect, for an elementary surgery, of removing the dual cell algebraically; in algebra it is not possible to remove cells, only to take mapping cones.  The appearance of $\delta\varphi_0$ here means that the choice from geometry of framing in $\pi_r(SO(n-r))$ for the thickening disk $D^{n-r}$ under the embedding $g \colon S^r \times D^{n-r} \hookrightarrow M$ is taken into account in the algebra.  For an elementary surgery, when $D = S^{n-r}A$, the map $f$ represents a cohomology class in $H^{n-r}(C)$, which is killed by the surgery.  Finally, we desuspend $SC(M')$ to get $C(M')_r = C_r \oplus D_{r+1} \oplus D^{n-r+1}$.  I would like to thank Tibor Macko for telling me about this explanation of algebraic surgery using cofibration sequences.
\end{remark}

\begin{proof}[Proof of Theorem \ref{Thm:1.5solvable=>2nd_alg_slice}]
We need to show that the triple $(H^K,\Y^K,\xi^K)$ of a $(1.5)$-solvable knot $K$, with a $(1.5)$-solution $W$, is equivalent to the identity element of $\mathcal{AC}_2$, which is represented by the triple $(\{0\},\Y^U,\Id_{\{0\}})$ corresponding to the unknot.

The chain complex \[N_K := E^K \cup_{E^K \oplus E^U} Y^K \oplus Y^U\] is chain equivalent to the chain complex $C_*(M_K;\Z[\Z \ltimes H_1(M_K;\Z[\Z])])$ of the second derived cover of the zero framed surgery on $K$.  Our first attempt for chain complex which fits into a 4-dimensional symmetric Poincar\'{e} triad as required in Definition \ref{defn:2ndorderconcordant} is the chain complex of the second derived cover of the $(1.5)$ solution $W$
\[(V',\Theta') := (C_*(W;\Z[\Z \ltimes H_1(W;\Z[\Z])]),\backslash\Delta([W,M_K])),\] so that \[H' := \pi_1(W)^{(1)}/\pi_1(W)^{(2)} \toiso H_1(W;\Z[\Z]),\] and we have the triad:

\[\xymatrix @R+1cm @C+1cm {\ar @{} [dr] |{\stackrel{(\gamma^K,\gamma^{U})}{\sim}}
(E^K,\phi^K) \oplus (E^U,-\phi^U) \ar[r]^-{(\Id, \Id \otimes \varpi_{E^{K}})} \ar[d]_{\left( \begin{array}{cc} \eta^K & 0 \\ 0 & \eta^{U}\end{array}\right)} & (E^K,0) \ar[d]^{\delta}
\\ (Y^K,\Phi^K) \oplus  (Y^{U},-\Phi^{U}) \ar[r]^-{(j^K,j^{U})} & (V',\Theta'),
}\]
with a geometrically defined consistency isomorphism $$H' \toiso H_1(W;\Z[\Z]) = H_1(\Z[\Z] \otimes_{\zhd} V).$$

The problem is that $H_2(W;\Z)$ is typically non-zero: if it were zero, we would have our topological concordance exterior and in particular $K$ would be second order algebraically slice.  We therefore need, as indicated above, to perform algebraic surgery on $V'$ to transform it into a $\Z$-homology circle.  We form the algebraic Thom complex (Definition \ref{Defn:algThomcomplexandthickening}):
\begin{multline*} C_*(W,M_K;\Z[\Z \ltimes H']) \simeq \ol{V} := \mathscr{C}((\delta,(-1)^{r-1}\gamma^K, (-1)^{r-1}\gamma^U, -j^K,-j^U) \colon \\ (N_K)_r = E^K_r \oplus E_{r-1}^K \oplus E_{r-1}^U \oplus Y^K_r \oplus Y^U_r \to V'_r),\end{multline*}
with symmetric structure $\ol{\Theta}:= \Theta' / (0 \cup_{\phi^K \oplus -\phi^U} \Phi^K \oplus -\Phi^U)$.  In this chapter the bar is again a notational device and has nothing to do with involutions.

This gives us the input for surgery, since the input for algebraic surgery must be a symmetric chain complex.  Next, we need the data for surgery.

As in the proof of \cite[Proposition~4.3]{COT}, any compact topological 4-manifold has the homotopy type of a finite simplicial complex.  \cite[Proposition~4.3]{COT} cites \cite[Theorem~4.1]{KS}, but it might be better to look at \cite[Annex~B~III,~page~301]{KS}.  In particular this means that $H_2(W;\Z)$ is finitely generated.  We therefore have homology classes $l'_1,\dots,l'_k \in H_2(W;\Z[\Z \ltimes H'])$ which generate the $(2)$-Lagrangian whose existence is guaranteed by definition of a $(1.5)$-solution $W$.  There are therefore dual cohomology classes $l_1,\dots,l_k \in H^2(W,M_K;\Z[\Z \ltimes H'])$, by Poincar\'{e}-Lefschetz duality.  Taking cochain representatives for these, we have maps $l_i \colon \ol{V}_2 \to \Z[\Z \ltimes H']$.  We then take as our data for algebraic surgery the symmetric pair:
\[(\ol{f} \colon \ol{V} \to B:= S^2(\bigoplus_k \, \Z[\Z \ltimes H']),(0,\ol{\Theta})).\]
where
\[\ol{f} = (l_1,\dots,l_k)^T \colon \ol{V}_2 \to B_2 = \bigoplus_k \, \Z[\Z \ltimes H'].\]

The fact that the $l_i$ are cohomology classes means that $l_i d_{\ol{V}} = 0 $, so that $\ol{f}$ is a chain map.  The requirement that the $l_i'$ generate a submodule of $$H_2(W;\Z[\Z \ltimes H']) = H_2(V')$$ on which the intersection form vanishes means that the duals $l_i$ generate a submodule of $H^2(\ol{V})$ on which the cup product vanishes.  The cup product of any two $l_i,l_j$ is given by:
\[\Delta_0^*(l_i \otimes l_j)([W,M_K]) = (l_i\otimes l_j)(\Delta_0([W,M_K])) = (l_i \otimes l_j)\ol{\Theta}_0,\] which under the slant isomorphism is  $l_i \ol{\Theta}_0 l_j^*,$ and so we see that each of these composites vanishes.

The only possibility for non-zero symmetric structure in the data for surgery would arise when $s = n-2r-1 = 4-2\cdot 2-1 = -1$, so no such non-zero structure maps exist.  Therefore the condition for our data for surgery to be a symmetric pair is that:
\[\ol{f} \,\ol{\Theta}_0 \ol{f}^* = 0;\]
which is the condition that the $k \times k$ matrix with $(i,j)$th entry $l_i \ol{\Theta}_0 l_j^*$, is zero.  This is satisfied as we saw above, since $l_i \ol{\Theta}_0 l_j^* \colon \Z[\Z \ltimes H'] \to \Z[\Z \ltimes H']$ is a module homomorphism given by multiplication by the same group ring element as the evaluation on the relative fundamental class $[W,M_K]$ of the cup product of two cohomology classes dual to the $(2)$-Lagrangian, and so equals the value of $\lambda_2(l'_i,l'_j)$.  This means that we can proceed with the operation of algebraic surgery to form the symmetric chain complex $(V,\Theta)$, which is the effect of algebraic surgery, shown below.  We may assume, since $W$ is a 4-manifold with boundary, that we have a chain complex $V'$ whose non-zero terms are $V'_0, V'_1, V'_2$ and $V'_3$.  The non-zero terms in $\ol{V}$ will therefore be of degree less than or equal to four.

The output of algebraic surgery, which we denote as $(V,\Theta)$ is then given, from Definition \ref{Defn:algebraicsurgerymatrices}, by:
\[\xymatrix @R+2cm @C-0.4cm{ \ol{V}^0 \ar[rr]^-{\left( \begin{array}{c} d^*_{\ol{V}} \\ 0 \end{array}\right)} \ar[d]^{\left(\begin{array}{c} \ol{\Theta}_0 \end{array} \right)} &&
\ol{V}^1 \oplus B^2 \ar[rrrr]^-{\left(\begin{array}{cc} d^*_{\ol{V}} & \ol{f}^* \end{array}\right)} \ar[d]^<<<<<<<<<{\left(\begin{array}{cc} \ol{\Theta}_0 & 0 \\ 0 & 1 \end{array} \right)} &&&&
\ol{V}^2 \ar[rrr]^-{\left( \begin{array}{c} d^*_{\ol{V}} \\ -\ol{f}\,\ol{\Theta}_0^*\end{array} \right)} \ar[d]_>>>>>>>{\left(\begin{array}{c} \ol{\Theta}_0 \end{array} \right)} &&&
\ol{V}^3 \oplus B_2 \ar[rrrr]^-{\left( \begin{array}{cc} d^*_{\ol{V}} & 0 \end{array}\right)} \ar[d]^>>>>>>>>{\left(\begin{array}{cc} \ol{\Theta}_0 & 0 \\ -f T\ol{\Theta}_1 & -1 \end{array} \right)} &&&&
\ol{V}^4 \ar[d]_<<<<<<<<<<{\left(\begin{array}{c} \ol{\Theta}_0 \end{array} \right)} \\
\ol{V}_4 \ar[rr]_-{\left( \begin{array}{c} d_{\ol{V}} \\ 0 \end{array}\right)} &&
\ol{V}_3 \oplus B^2 \ar[rrrr]_-{\left(\begin{array}{cc} d_{\ol{V}} & -\ol{\Theta}_0\, \ol{f}^* \end{array}\right)} &&&&
\ol{V}_2 \ar[rrr]_-{\left( \begin{array}{c} d_{\ol{V}} \\ \ol{f}\end{array} \right)} &&&
\ol{V}_1 \oplus B_2 \ar[rrrr]_-{\left( \begin{array}{cc} d_{\ol{V}} & 0 \end{array}\right)} &&&&
\ol{V}_0.
}\]
The higher symmetric structures $\Theta_s$ are just given by the maps $\ol{\Theta}_s$ for $s=1,2,3,4$ except for the map:
\[\Theta_1 = \left(\begin{array}{c} \ol{\Theta}_1 \\ -\ol{f} T \ol{\Theta}_2 \end{array} \right) \colon \ol{V}^4 \to \ol{V}_1 \oplus B_2.\]
Next, we take the algebraic Poincar\'{e} thickening (Definition \ref{Defn:algThomcomplexandthickening}) of $V$ to get:
\[i_V \colon \partial V \to V^{4-*},\]
where, as in Chapter \ref{Chapter:symmconstruction}, we define the complex $V^{4-*}$ by:
\[(V^{4-*})_r = \Hom_{\zhd}(V_{4-r},\zhd),\]
with boundary maps
\[\partial^* \colon (V^{4-*})_{r+1} \to (V^{4-*})_{r}\]
given by
\[\partial^* = (-1)^{r+1}d^*_V,\]
where $d^*_V$ is the coboundary map.
By \cite[Part~I,~Proposition~4.1~(i)]{Ranicki3}, the operation of algebraic surgery does not change the homotopy type of the boundary.  There is therefore a chain equivalence:
\[(N_K, 0 \cup_{\phi^K \oplus -\phi^U} \Phi^K \oplus -\Phi^U) \xrightarrow{\sim} (\partial V, \partial \Theta), \]
so that using the composition of the relevant maps in:
\[N_K = E^K \cup_{E^K \oplus E^U} Y^K \oplus Y^U \xrightarrow{\sim} \partial V \to V^{4-*}\]
we again have a 4-dimensional symmetric Poincar\'{e} triad:
\[\xymatrix @R+1cm @C+1cm {\ar @{} [dr] |{\sim}
(E^K,\phi^K) \oplus (E^U,-\phi^U) \ar[r] \ar[d] & (E^K,0) \ar[d]
\\ (Y^K,\Phi^K) \oplus  (Y^{U},-\Phi^{U}) \ar[r] & (V^{4-*},0).
}\]
To complete the proof we need to check the homology conditions of Definition \ref{defn:2ndorderconcordant}, namely that $V^{4-*}$ has the $\Z$-homology of a circle and the consistency condition that there is an isomorphism $\xi' \colon H' \toiso H_1(\Z[\Z] \otimes_{\zhd} V^{4-*})$.  We have:
\[H_4(\Z \otimes_{\zhd} V^{4-*}) \cong H^0(\Z \otimes_{\zhd} \ol{V}) \cong H^0(W,M_K;\Z) \cong H_4(W;\Z) \cong 0,\]
and
\[H_0(\Z \otimes_{\zhd} V^{4-*}) \cong H^4(\Z \otimes_{\zhd} \ol{V}) \cong H^4(W,M_K;\Z) \cong H_0(W;\Z) \cong \Z,\]
as required.
For each basis element $(0,\dots,0,1,0,\dots,0) \in B^2$, where the $1$ is in the $i$th entry, we have, for $v \in \ol{V}_2$, \[\ba{rcl} \ol{f}^*(0,\dots,0,1,0,\dots,0)(v) & = & (0,\dots,0,1,0,\dots,0)\ol{f}(v) \\ & = & (0,\dots,0,1,0,\dots,0)(l_1,\dots,l_k)^T(v) = l_i(v). \ea\]
This means, since no $l_i$ lies in the image of $d_{\ol{V}}^* \colon \ol{V}^1 \to \ol{V}^2$, that the kernel of $(d_{\ol{V}}^*, \ol{f}^*)$ is zero, so that:
\[H_3(\Z \otimes_{\zhd} V^{4-*}) \cong H^1(\Z \otimes_{\zhd} \ol{V}) \cong H^1(W,M_K;\Z) \cong 0.\]
Also, since the $l_i$ are in the image of $\ol{f}^*$, they are no longer cohomology classes of $V^{4-*}$ as they were of $\ol{V}$.
At this point we need the dual classes; recall that we have, from Definition \ref{Defn:1.5solvableagain}, classes $d'_1,\dots,d'_k \in H_2(W;\Z[\Z])$, whose images in $H_2(W;\Z)$ we also denote by $d'_1,\dots,d'_k$, which satisfy
\[\lambda_1(l'_i,d'_j) = \delta_{ij}.\]
We therefore have, by Poincar\'{e}--Lefschetz duality, classes: \[d_1,\dots,d_k \in H^2(W,M_K;\Z[\Z]),\]
with representative cochains which we also denote $d_1,\dots,d_k \in \ol{V}^2$.

Since, as above, the intersection form is defined in terms of the cup product, we have, over $\Z[\Z]$ and $\Z$, that:
\[l_i\ol{\Theta}_0^* d_j^* = \delta_{ij}.\]
We can use $\ol{\Theta}_0^* = T\ol{\Theta}_0$ instead of $\ol{\Theta}_0$ to calculate the cup products due to the existence of the higher symmetric structure chain homotopy $\ol{\Theta}_1$.  Then
\begin{eqnarray*} -\ol{f}\,\ol{\Theta}_0^*(d_j) &=& -\ol{f}\,\ol{\Theta}_0^*d_j^*(1) = -(l_1\ol{\Theta}_0^*d_j^*(1),\dots,l_k\ol{\Theta}_0^*d_j^*(1))^T \\ &=& -(0,\dots,0,1,0,\dots,0)^T = -e_j,\end{eqnarray*}
where the $1$ is in the $j$th position, and for $j = 1,\dots,k$ we denote the standard basis vectors by $$e_j := (0,\dots,0,1,0,\dots,0)^T \in B_2.$$  This means that the $d_j$ are not in the kernel of $-\ol{f}\ol{\Theta}_0^*$.  Then, since $d_{\ol{V}}^*(d_j) = 0$ as the $d_j$ are cocycles in $\ol{V}$, we know that the $d_j$ are no longer cohomology classes in $H_2(\Z \otimes_{\zhd} V^{4-*})$.  The group $H^2(\Z \otimes_{\zhd} \ol{V})$ was generated by the classes $l_1,\dots,l_k,d_1,\dots,d_k$, which means that we now have \[H_2(\Z \otimes_{\zhd} V^{4-*}) \cong 0.\]

Moreover, over both $\Z[\Z]$ and $\Z$, taking the element $D := \sum_{i=1}^k \, a_j d_j$, for any elements $a_1,\dots,a_k \in \Z[\Z]$, we have that:
\[-\ol{f}\,\ol{\Theta}_0^*(-D) = \sum_{j=1}^k \, a_j (\ol{f}\,\ol{\Theta}_0^*d_j^*(1)) = \sum_{j=1}^k \, a_j e_j \in B_2.\]
This means that $-\ol{f}\,\ol{\Theta}_0^*$ is onto $B_2$.  Therefore:
\[H_1(\Z \otimes_{\zhd} V^{4-*}) \cong H^3(\Z \otimes_{\zhd} \ol{V}) \cong H^3(W,M_K;\Z) \cong H_1(W;\Z) \cong \Z,\]
so the first homology remains unchanged at the $\Z$ level as required.  Similarly, with $\Z[\Z]$ coefficients, we have the isomorphisms:
\begin{eqnarray*} H' & \toiso & H_1(W;\Z[\Z]) \toiso H^3(W,M_K;\Z[\Z]) \\ & \toiso & H^3(\Z[\Z] \otimes_{\zhd} \ol{V}) \toiso H_1(\Z[\Z] \otimes_{\zhd} V^{4-*}), \end{eqnarray*}
which define the map
\[\xi' \colon H' \toiso H_1(\Z[\Z] \otimes_{\zhd} V^{4-*}),\]
so that the consistency condition is satisfied.  Since $H'$ is isomorphic to the $\Z[\Z]$-homology of a finitely generated projective module chain complex which is a $\Z$-homology circle, we can apply Levine's arguments \cite[Propositions~1.1~and~1.2]{Levine2}, to see that $H'$ is of type $K$.  This completes the proof that $(1.5)$-solvable knots are second order algebraically slice, or algebraically $(1.5)$-solvable.
\end{proof}

\begin{remark}
Theorem \ref{Thm:1.5solvable=>2nd_alg_slice} shows that we can extend our diagram to the following:
\[\xymatrix @R+1cm @C+1cm{
Knots \ar[r] \ar[d]^{/\sim} & \P \ar[d]^{/\sim}   \\
\C \ar[r] \ar[d] & \mathcal{AC}_2   \\
\C/\mathcal{F}_{(1.5)}, \ar[ur] &
}\]
so that the homomorphism from $\C$ to $\mathcal{AC}_2$ factors through $\mathcal{F}_{(1.5)}$ as claimed.  In the next chapter we show how to extract the \COT obstructions from an element of $\mathcal{AC}_2$.
\end{remark}



\chapter[Extracting the Cochran-Orr-Teichner obstructions]{Extracting the Cochran-Orr-Teichner Concordance Obstructions}\label{Chapter:extractingCOTobstructions}


In this chapter we aim to complete our diagram:
\[\xymatrix @R+1cm @C+1cm{
Knots \ar[r] \ar @{->>} [d] & \P \ar @{->>} [d]  \\
\C \ar[r] \ar @{->>} [d] & \mathcal{AC}_2 \ar@{-->}[d]  \\
\C/\mathcal{F}_{(1.5)} \ar@{-->}[r]  \ar[ur] & \mc{COT}_{(\C/1.5)},}\]
by explaining the map $\mathcal{AC}_2 \to \mc{COT}_{(\C/1.5)}$ and showing that it is a morphism of pointed sets.  Recall that $\G := \Z \ltimes \Q(t)/\Q[t,t^{-1}]$.  The map $\C/\mathcal{F}_{(1.5)} \to \mc{COT}_{(\C/1.5)}$ was defined in Section \ref{Chapter:COTobstructionthy}.  We will show that:

\begin{theorem}\label{Thm:zeroAC2_goes_to_zeroL4QG}
A triple in $\mathcal{AC}_2$ which is second order algebraically concordant to the triple of the unknot has zero \COT metabelian obstruction; i.e. it maps to $U$ in $\mc{COT}_{(\C/1.5)}$.  See Theorem \ref{Thm:betterstatementzeroAC2_to_zeroL4QG} for a more general and precise statement.
\end{theorem}

To define the map $\ac2 \to \mc{COT}_{(\C/1.5)}$, we begin by taking an element $(H,\Y,\xi) \in \ac2$, and forming the algebraic equivalent of the zero surgery $M_K$.  We construct the symmetric \emph{Poincar\'{e}} complex:
\[(N,\theta) := ((Y \oplus (\zh \otimes_{\Z[\Z]} Y^U)) \cup_{E \oplus (\zh \otimes_{\Z[\Z]} E^U)} E, (\Phi \oplus 0) \cup_{\phi \oplus -\phi^U} 0).\]
In the case that $\Y = \Y^K$ is the fundamental symmetric Poincar\'{e} triad of a knot $K$, we have that:
\[N_K \simeq C_*(M_K;\zh).\]
By defining representations $\Z \ltimes H \to \G$, we will obtain elements of $L^4(\Q\G,\Q\G-\{0\})$.  Recall that $L^4(\Q\G,\Q\G-\{0\})$ is the group of 3-dimensional symmetric Poincar\'{e} chain complexes over $\Q\G$ which become contractible when we tensor over the Ore localisation (Definition \ref{Defn:OreLocalisation}) $\K$ of $\Q\G$ with respect to $\Q\G - \{0\}$.  The group $L^4(\Q\G,\Q\G-\{0\})$ fits into the localisation exact sequence:
\[L^4(\Q\G) \to L^4_S(\K) \to L^4(\Q\G,\Q\G-\{0\}) \to L^3(\Q\G).\]
In geometry, a $(1)$-solvable knot $K$ has a zero-surgery $M_K$ which bounds over $\Q\G$ for a subset of the possible representations, by \cite[Theorems~3.6~and~4.4]{COT}, so that:
\[N_K \in \ker(L^4(\Q\G,\Q\G-\{0\}) \to L^3(\Q\G)).\]
In such circumstances, $N_K = 0 \in L^4(\Q\G,\Q\G-\{0\})$ if and only if a lift into $L^4_S(\K)$ is zero in
\[L^4_S(\K)/\im(L^4(\Q\G)) \cong L^0_S(\K)/\im(L^0(\Q\G)).\]
In turn the reduced $L^{(2)}$-signature (Section \ref{chapter:COTsurvey}.\ref{Chapter:L2signature}) obstructs the vanishing of an element of $L^0_S(\K)/\im(L^0(\Q\G))$.  We will describe how to define the signatures purely in terms of the algebraic objects in $\ac2$.  By making use of a result of Higson-Kasparov \cite{HigsonKasparov} which applies to PTFA groups, we do not need to appeal to geometric 4-manifolds in order to define Von Neumann $\rho$-invariants.

The first step is to define the representation
\[\rho \colon \Z \ltimes H \to \G,\]
which sends $\Z \ltimes H$, for varying $H$, to a fixed group, the so-called universally $(1)$-solvable group of \COTN:
\[\G := \Z \ltimes \frac{\Q(t)}{\Q[t,t^{-1}]}.\]
To define a representation, just as in Chapter \ref{chapter:COTsurvey}, choose a $p \in H$, and define:
\[\rho \colon (n,h) \mapsto (n,\Bl(p,h)) \in \G,\]
where $\Bl$ is the Blanchfield pairing,
\[\Bl \colon H \times H \to \frac{\Q(t)}{\Q[t,t^{-1}]},\]
which we will define below.  The key point is that the chain complex with symmetric structure:
\[(\Z[\Z] \otimes_{\zh} N, \Id \otimes \theta),\]
contains the information necessary to extract the Blanchfield pairing.  Note that
\[H_1(\Z[\Z] \otimes_{\zh} Y) \xrightarrow{\simeq} H_1(\Z[\Z] \otimes_{\zh} N).\]
This isomorphism, which is the algebraic equivalent of $H_1(X;\Z[\Z]) \cong H_1(M_K;\Z[\Z])$, arises in the Mayer-Vietoris sequence, since
\[H_1(\Z[\Z] \otimes_{\zh}(\zh \otimes_{\Z[\Z]} Y^U)) \cong H_1(Y^U) \cong 0.\]
We compose $\xi$ with the rationalisation map, to get:
\[\xi \colon H \xrightarrow{\simeq} H_1(\Z[\Z] \otimes_{\zh} N) \rightarrowtail H_1(\Q[\Z] \otimes_{\zh} N).\]
The second map is injective by Theorem \ref{Thm:Levinemodule} (b): $H$ is $\Z$-torsion free.  In this chapter we abuse notation and also refer to this composition of $\xi$ with the rationalisation map as $\xi$.
\begin{proposition}\label{Prop:chainlevelBlanchfield}
Given $[x],[y] \in H_1(\Q[\Z] \otimes N)$, the rational Blanchfield pairing of $[x]$ and $[y]$ is given by:
\[\Bl([x],[y]) = \frac{1}{s} \ol{z(x)}\]
where: \[x,y \in (\Q[\Z] \otimes_{\zh} N)_1,\] \[z \in (\Q[\Z] \otimes_{\zh} N)^1;\] \[\partial^*(z) = s\theta'_0 (y) \text{ for some } s \in \Q[\Z]- \{0\},\] and \[\theta'_0 \colon (\Q[\Z] \otimes_{\zh} N)_1 \to (\Q[\Z] \otimes_{\zh} N)^2\] is part of a chain homotopy inverse \[\theta_0' \colon (\Q[\Z] \otimes_{\zh} N)_r \to (\Q[\Z] \otimes_{\zh} N)^{3-r},\] so that \[\theta_0 \circ \theta'_0 \simeq \Id,\, \theta'_0 \circ \theta_0 \simeq \Id.\]  The Blanchfield pairing is non-singular, sesquilinear and Hermitian.
\end{proposition}
\begin{proof}
For this proof, write $(C,\theta) := (\Q[\Z] \otimes_{\zh} N, \Id \otimes \theta)$. We give some detail when checking the properties of the algebraically defined linking form in this proof, since as far as the author is aware these details, which are admittedly fairly straight--forward, do not appear in the literature.  The complex $(C,\theta)$ is a symmetric Poincar\'{e} complex, which implies that $\theta_0$ is a chain equivalence.  Therefore there exists a chain homotopy inverse $\theta'_0$.  Inspection of the sequence of isomorphisms which defined the Blanchfield form in Definition \ref{Defn:Blanchfieldform} shows that the formula given in Proposition \ref{Prop:chainlevelBlanchfield} is the corresponding chain level calculation.  The isomorphisms, given by Poincar\'{e} duality, a Bockstein, and universal coefficients, are defined algebraically: the only one which was not a chain complex construction was taking the Poincar\'{e} dual, and this became an algebraically defined chain complex map with the use of the symmetric structure.  Therefore the pairing is non-singular.  This also follows by algebraic surgery below the middle dimension, since $C$ is a Poincar\'{e} complex and $\Q[\Z]$ is a principal ideal domain and therefore a Dedekind domain: see \cite[Section~4.2]{Ranicki2}.  We can define a linking pairing \[\wt{\Bl} \colon TH^2(C) \times TH^2(C) \to \Q(t)/\Q[t,t^{-1}],\] as on \cite[page~185]{Ranicki2}, as follows.  For torsion $u,v \in H^2(C)$, we define:
\[\wt{\Bl}(u,v) := \frac{1}{r} \theta_0(u)(w),\]
where $w \in C^1$ is such that $\partial^*(w) = rv$ for some $r \in \Q[\Z] -\{0\}$.  This uses the identification of a module with its double dual as in Definition \ref{Defn:signsoontensor}:
\[C_* \xrightarrow{\simeq} C^{**};\; x \mapsto (f \mapsto \ol{f(x)}).\]
We can show that this definition corresponds to our definition of the linking form.  We can then define:
\[\Bl \colon TH_1(C) \times TH_1(C) \to \Q(t)/\Q[t,t^{-1}]\]
by
\[\Bl(x,y) := \wt{\Bl}(\theta'_0(x),\theta'_0(y)).\]
This means that:
\[\Bl(x,y) = \frac{1}{s} (\theta_0(\theta'_0(x)))(z),\]
where $z \in C^1, s \in \Q[\Z] -\{0\}$ are such that $\partial^*(z) = s \theta'_0(y)$.  Let $k: C_i \to C_{i+1}$ be the chain homotopy which shows that $\theta_0 \circ \theta'_0 \simeq \Id.$  Then:
\begin{eqnarray*}\Bl(x,y) &=& \frac{1}{s} (\Id +k\partial + \partial k)(x) (z)\\ &=& \frac{1}{s}(x + \partial k(x))(z) \\ &=& \frac{1}{s} \ol{z(x)} + \frac{1}{s} \ol{(\partial^*(z))(k(x))}\\ & = & \frac{1}{s}\ol{z(x)} + \frac{1}{s}\ol{s\theta'_0(y)(k(x))} \\ & = & \frac{1}{s}\ol{z(x)} + \frac{1}{s}\ol{\theta'_0(y)(k(x)) \ol{s}} \\ & = & \frac{1}{s}\ol{z(x)} + \frac{1}{s}s\ol{\theta'_0(y)(k(x))} \\ & = & \frac{1}{s}\ol{z(x)} + \ol{\theta'_0(y)(k(x))},\end{eqnarray*}
which means that:
\[\Bl(x,y)= \frac{1}{s}\ol{z(x)} \in \Q(t)/\Q[t,t^{-1}],\]
since $\theta'_0(y)(k(x)) \in \Q[t,t^{-1}].$  To show that this definition is independent of the choice of $s$ and $z$, suppose that also there is $s' \in \Q[\Z] - \{0\}, z' \in C^1$ such that:
\[\partial^*(z') = s'\theta_0'(y).\]
Since also $x$ is a torsion element of $H_1(C)$, there is a chain $w \in C_2$ and an $r \in \Q[t,t^{-1}]$ such that $\partial(w) = rx$.  Then:
\begin{eqnarray*} \frac{1}{s}\ol{z(x)} - \frac{1}{s'}\ol{z'(x)}
& = &  \left(\frac{1}{s} \ol{z(x)} - \frac{1}{s'} \ol{z'(x)}\right)\frac{\ol{r}}{\ol{r}} \\
& = & \left(\frac{1}{s} \ol{z(x)}\,\ol{t} - \frac{1}{s'} \ol{z'(x)}\,\ol{r}\right)\frac{1}{\ol{r}} \\
& = & \left(\frac{1}{s} \ol{r(z(x))} - \frac{1}{s'} \ol{r(z'(x))}\right)\frac{1}{\ol{r}} \\
& = & \left(\frac{1}{s} \ol{z(rx)} - \frac{1}{s'} \ol{z'(rx)}\right)\frac{1}{\ol{r}} \\
& = & \left(\frac{1}{s} \ol{z(\partial w)} - \frac{1}{s'} \ol{z'(\partial w)}\right)\frac{1}{\ol{r}} \\
& = & \left(\frac{1}{s} \ol{\partial^*(z)(w)} - \frac{1}{s'} \ol{\partial^*(z')(w)}\right)\frac{1}{\ol{r}} \\
& = & \left(\frac{1}{s} \ol{(s\theta'_0(y))(w)} - \frac{1}{s'} \ol{(s'\theta'_0(y))(w)}\right)\frac{1}{\ol{r}} \\
& = & \left(\frac{1}{s} \ol{(\theta'_0(y))(w)\ol{s}} - \frac{1}{s'} \ol{(\theta'_0(y))(w)\ol{s'}}\right)\frac{1}{\ol{r}} \\
& = & \left(\frac{1}{s} \ol{\ol{s}}\ol{(\theta'_0(y))(w)} - \frac{1}{s'} \ol{\ol{s'}}\ol{(\theta'_0(y))(w)}\right)\frac{1}{\ol{r}} \\
& = & \left(\ol{\theta'_0(y)(w)} - \ol{\theta'_0(y)(w)}\right) \frac{1}{\ol{r}} = 0.\end{eqnarray*}
Furthermore, for $p,q \in \Q[t,t^{-1}]$:
\[\Bl(px,qy) = \frac{1}{s}\ol{(qz)(px)} = \frac{1}{s} \ol{p z(x) \ol{q}} = \frac{1}{s} q \ol{z(x)} \ol{p},\]
so that $\Bl$ is sesquilinear.  To show that $\Bl$ is Hermitian, we will show that $\wt{\Bl}$ is Hermitian.  Recall that for $x,y \in C^2$, \[\wt{\Bl}(x,y) = \frac{1}{s}\ol{z(\theta_0(x))},\]  where $s \in \Q[t,t^{-1}]- \{0\},z \in C^1$ are such that $\partial^*z = sy$.

First, we claim that we can calculate $\wt{\Bl}$ using $T\theta_0 = \theta_0^*$ instead of $\theta_0$.  That is:
\begin{eqnarray*}\frac{1}{s} \ol{z (\theta_0(x))} - \frac{1}{s} \ol{z (\theta_0^*(x))}
& = & \frac{1}{s} \ol{z (\theta_0 - \theta_0^*)(x)} \\
& = & \frac{1}{s} \ol{z((\partial \theta_1 - \theta_1 \partial^*)(x))}\\
& = & \frac{1}{s} \ol{z(\partial \theta_1(x))}\\
& = & \frac{1}{s} \ol{(\partial^*(z))(\theta_1(x))}\\
& = & \frac{1}{s} \ol{(s y)(\theta_1(x))}\\
& = & \frac{1}{s} \ol{y(\theta_1(x))\ol{s}}\\
& = & \frac{1}{s} \ol{\ol{s}}\ol{y(\theta_1(x))}\\
& = & \ol{y(\theta_1(x))}.
\end{eqnarray*}
Since $\ol{y(\theta_1(x))} \in \Q[t,t^{-1}]$, this is zero in $\Q(t)/\Q[t,t^{-1}]$ as claimed.  Now, suppose we also have an $r \in \Q[t,t^{-1}] - \{0\}, w \in C^1$, such that $\partial w = rx$.  Then:
\begin{eqnarray*}\wt{\Bl}(x,y) & = & \frac{1}{s}\,\ol{z(\theta_0^*(x))} \\
& = & \frac{1}{s}\,\ol{z(\theta_0^*(x))}\,\ol{r}\, \frac{1}{\ol{r}} \\
& = & \frac{1}{s}\,\ol{r z(\theta_0^*(x))}\, \frac{1}{\ol{r}} \\
& = & \frac{1}{s}\,\ol{z(\theta_0^*(rx))}\, \frac{1}{\ol{r}} \\
& = & \frac{1}{s}\,\ol{z(\theta_0^*(\partial^* w))}\, \frac{1}{\ol{r}} \\
& = & \frac{1}{s}\,\ol{z(\partial\theta_0^*(w))}\, \frac{1}{\ol{r}} \\
& = & \frac{1}{s}\,\ol{\partial^*z(\theta_0^*(w))}\, \frac{1}{\ol{r}} \\
& = & \frac{1}{s}\,\ol{(s y)(\theta_0^*(w))}\, \frac{1}{\ol{r}} \\
& = & \frac{1}{s}\,\ol{(y)(\theta_0^*(w))\ol{s}}\, \frac{1}{\ol{r}} \\
& = & \frac{1}{s}\,\ol{\ol{s}}\,\ol{y(\theta_0^*(w))}\, \frac{1}{\ol{r}} \\
& = & \ol{y(\theta_0^*(w))}\, \frac{1}{\ol{r}} \\
& = & \theta_0^*(w)(y)\, \frac{1}{\ol{r}} \\
& = & w(\theta_0(y))\, \frac{1}{\ol{r}} \\
& = & \ol{\frac{1}{r}\,\ol{w(\theta_0(y))}} = \ol{\wt{\Bl}(y,x)},
\end{eqnarray*}
which shows that $\wt{Bl}$ and therefore $\Bl$ is Hermitian.  This completes the proof of Proposition \ref{Prop:chainlevelBlanchfield}.
\end{proof}

\begin{definition}
We define $\Bl \colon H \times H \to \Q(t)/\Q[t,t^{-1}]$ by:
\[\Bl(p,h) := \Bl(\xi(p),\xi(h)),\]
recalling that we also use $\xi$ to denote the map:
\[\xi \colon H \xrightarrow{\simeq} H_1(\Z[\Z] \otimes_{\zh} N) \rightarrowtail H_1(\Q[\Z] \otimes_{\zh} N).\]
\qed \end{definition}

\begin{definition}
A \emph{Poly--Torsion--Free--Abelian}, or PTFA, group $\G$ is a group which admits a finite sequence of normal subgroups $$\{1\} = \G_0 \lhd \G_1 \lhd ... \lhd \G_k = \G$$ such that the successive quotients $\G_{i+1}/\G_i$ are torsion-free abelian for each $i \geq 0$.
\qed\end{definition}
\begin{proposition}\label{Prop:NinL4QG}
The chain complex:
\[(\Q\G \otimes_{\zh} N, \Id \otimes \theta) \]
defines an element of $L^4(\Q\G,\Q\Gamma - \{0\})$.  That is, $\K \otimes_{\Q\G} \Q\G \otimes_{\zh} N$ is contractible.
\end{proposition}

\begin{proof}
First note that $\G$ is a PTFA group (see \cite[Sections~2~and~3]{COT}, since
\[[\G,\G] = \frac{\Q(t)}{\Q[t,t^{-1}]},\]
which is abelian and means that \[\frac{\G}{[\G,\G]} \cong \Z.\]

The fact that $\G$ is PTFA means that, by \cite[Proposition 2.5]{COT}, the Ore localisation of $\Q\G$ with respect to non-zero elements $\Q\G -\{0\}$ exists.  We will need the following proposition.
\begin{proposition}\label{Prop:COT2.10chainversion} \cite[Proposition~2.10]{COT}
If $C_*$ is a nonnegative chain complex over $\Q\G$ for a PTFA group $\G$ which is finitely generated projective in dimensions $0 \leq i \leq n$ and such that $H_i(\Q \otimes_{\Q\G} C_*) \cong 0$ for $0 \leq i \leq n$, then $H_i(\K \otimes_{\Q\G} C_*) \cong 0$.
\end{proposition}
Note that the hypothesis that the chain complex is finitely generated free for Proposition \ref{Prop:COT2.10chainversion} can be relaxed to $C$ being a finitely generated projective module chain complex, since this still allows the lifting of the partial chain homotopies.  The rest of the proof of Proposition \ref{Prop:NinL4QG} follows closely that of \cite[Proposition~2.11]{COT}, but in terms of chain complexes.  The chain complex of the circle $C_*(S^1;\Q[\Z])$ is given by:
\[\Q[\Z] \xrightarrow{t-1} \Q[\Z].\]
Tensor with $\Q\G$ over $\Q[\Z]$ using the homomorphism $\rho \circ (f_-)_*$, where we have to define $(f_-)_* \colon \Z \to \Z \ltimes H$.  Recall that $f_-$ is a chain map in our symmetric Poincar\'{e} triad $\Y$ (Definition \ref{Defn:algebraicsetofchaincomplexes}), and so we define $(f_-)_*$ to be the corresponding homomorphism of groups: there is, as ever, a symbiosis between the group elements and the 1-chains of the complex.  The homomorphism $(f_-)_* \colon \Z \to \Z \ltimes H$ sends $t \mapsto (1,h_1)$, where $h_1$ is, as in Definition \ref{Defn:algebraicsetofchaincomplexes}, the element of $H$ which makes $f_-$ a chain map.  Thus, passing from $C_*(S^1;\Q[\Z])$ to $C_*(S^1;\Q\G)$, we obtain:
\[\Q\G \otimes_{\Q[\Z]} \Q[\Z] \cong \Q\G \xrightarrow{(\rho\circ (f_-)_*(t) - 1)} \Q\G \otimes_{\Q[\Z]} \Q[\Z] \cong \Q\G ,\]
The chain map \[1 \otimes f_- \colon C_*(S^1;\Q\G) = \Q\G \otimes_{\zh} D_- \to \Q\G \otimes_{\zh} Y \to \Q\G \otimes_{\zh} N,\] is 1-connected on rational homology.  Therefore, by the long exact sequence of a pair,
\[H_k(\Q \otimes_{\Q\G} \mathscr{C}(1 \otimes f_- \colon C_*(S^1;\Q\G) \to \Q\G \otimes_{\zh} N)) \cong 0\]
for $k = 0, 1$.
We apply Proposition \ref{Prop:COT2.10chainversion}, with $n=1$ and $C_* = \mathscr{C}(1 \otimes f_-)$, to show that:
\[H_k(\K \otimes_{\Q\G} \mathscr{C}(1 \otimes f_- \colon C_*(S^1;\Q\G) \to \Q\G \otimes_{\zh} N)) \cong 0\]
for $k = 0,1$.  This implies, again by the long exact sequence of a pair, that there is an isomorphism:
\[H_0(S^1;\K) \cong H_0(\K \otimes_{\zh} N)\]
and a surjection:
\[H_1(S^1;\K) \twoheadrightarrow H_1(\K \otimes_{\zh} N).\]
As in the proof of \cite[Proposition~2.11]{COT}, $t$ maps to a non-trivial element \[\rho\circ (f_-)_*(t) = \rho(1,h_1) = (1,\Bl(p,h_1)) \in \Gamma.\]    Therefore $\rho\circ (f_-)_*(t) -1 \neq 0 \in \Q \G$ is invertible in $\K$, so $H_*(S^1;\K) \cong 0$.  This then implies that \[H_k(\K \otimes_{\zh} N) \cong 0\]
for $k=0,1$.


The proof that $\Q\G \otimes_{\zh} N$ is acyclic over $\K$ is then finished by applying Poincar\'{e} duality and universal coefficients.  The latter theorem is straight-forward since $\K$ is a skew-field, so we see that:
\[H_k(\K \otimes_{\Q\G} (\Q\G \otimes_{\zh} N)) \cong 0\]
for $k=2,3$ as a consequence of the corresponding isomorphisms for $k=0,1$.  A projective module chain complex is contractible if and only if its homology modules vanish \cite[Proposition~3.14~(iv)]{Ranicki}, which completes the proof.
\end{proof}

\begin{remark}
We can always define, for any representation which maps $g_1$ to a non-trivial element of $\G$, a map \[\ac2 \to L^4(\Q\G,\Q\G - \{0\}).\]  However, we will only show that it has the desired property: namely that it maps $0 \in \ac2$ to $0 \in L^4(\Q\G,\Q\G - \{0\})$, in the case that $\xi(p) \in P$, where $p$ is in the definition of the representation $\rho \colon \Z \ltimes H \to \Gamma$, for at least one of the submodules $P \subseteq H_1(\Q[\Z] \otimes_{\zh} N)$ such that $P=P^{\bot}$: that is, $P$ is a metaboliser of the rational Blanchfield form:
\[\Bl \colon H_1(\Q[t,t^{-1}] \otimes_{\zh} N) \times H_1(\Q[t,t^{-1}] \otimes_{\zh} N) \to \frac{\Q(t)}{\Q[t,t^{-1}]}.\]
This complicated vanishing for the \COT obstruction theory  is encoded in the definition of $\mc{COT}_{(\C/1.5)}$: see Definition \ref{defn:COTobstructionset_2}.  We have a two stage definition of the metabelian \COT obstruction set, since we need the Blanchfield form to define the elements and the notion of vanishing in $\mc{COT}_{(\C/1.5)}$; whereas an element of the group $\ac2$ is defined in a single stage from the geometry, via a handle decomposition.  Both stages of the \COT obstruction can be extracted from the single stage element of $\ac2$.  First, we explain the map $\ac2 \to \mathcal{AC}_1$.
\end{remark}

\begin{definition}\label{Defn:AC1}
We recall the definition of the algebraic concordance group, which we denote $\mathcal{AC}_1$.  We give three equivalent formulations; for proofs of their equivalence, see \cite{Ranicki4}.
A \emph{Seifert Form} is a finitely generated free $\Z$-module $S$ with a $\Z$-module homomorphism:
\[V \colon S \to S^* = \Hom_{\Z}(S,\Z),\]
such that $V-V^*$ is an isomorphism.
We define the Witt group of equivalence classes of Seifert forms, with addition by direct sum and the inverse of $(S,V)$ given by $(S,-V)$.  We call an element $(S,V)$ \emph{metabolic} if there is a basis of $S$ with respect to which $V$ has the matrix:
\[\left(
    \begin{array}{cc}
      0 & A \\
      B & C \\
    \end{array}
  \right),
\]
for block matrices $A,B,C$ such that $C=C^T$ and $A-B^T$ is invertible.  We say that $(S,V)$ is equivalent to $(S',V')$ if $(S \oplus S',V \oplus -V')$ is metabolic. Levine \cite[Lemma~1]{Levine} proves that this is an equivalence relation.

A Blanchfield form is an Alexander $\Z[\Z]$-module $H$ (Theorem \ref{Thm:Levinemodule}) with a $\Z[\Z]$-module isomorphism:
\[\Bl \colon H \toiso H^{\wedge}:= \ol{\Hom_{\Z[\Z]}(H, \Q(\Z)/\Z[\Z])},\]
which satisfies $\Bl = \Bl^{\wedge}$.
We define the Witt group of equivalence classes of Blanchfield forms, with addition by direct sum and the inverse of $(H,\Bl)$ given by $(H,-\Bl)$. We call an element $(H,\Bl)$ metabolic if there exists a metaboliser $P \subseteq H$ such that $P = P^{\bot}$ with respect to $\Bl$.  We say that $(H,\Bl)$ is equivalent to $(H',\Bl')$ is $(H \oplus H',\Bl \oplus -\Bl')$ is metabolic.
See Lemma \ref{Lemma:Cancellation_Blanchfield} for the rational version of the proof that this transitive and is therefore an equivalence relation.  The integral version is harder, but follows from the proof (see \cite[Theorems~3.10~and~4.2]{Ranicki4}) of the fact that the Witt group of Seifert forms and the Witt group of Blanchfield forms are isomorphic.

As in \cite{Ranicki4}, both of these Witt groups can be expressed in terms of symmetric $L$-theory by inverting the element $1-t \in \Z[t,t^{-1}]$, as:
\[L^4(\Z[t,t^{-1},(1-t)^{-1}],\Lambda),\]
where $\Lambda := \{p \in \Z[t,t^{-1}] \,| \, p(1) = \pm 1\}$.  This is a group under the addition of chain complexes by direct sum, the inverse of an element $(N,\theta)$ is given by $(N,-\theta)$, and an element is zero if it is the boundary of a 4-dimensional symmetric Poincar\'{e} pair $j \colon N \to U$ over $\Z[t,t^{-1},(1-t)^{-1}]$ such that $U$ is contractible over $\Lambda^{-1}\Z[t,t^{-1},(1-t)^{-1}]$.
\qed\end{definition}

We only prove the rational version of the following lemma, since this was all we needed in Proposition \ref{prop:COTobstr_equiv_rln} to see that the equivalence relation on $\mc{COT}_{(\C/1.5)}$ was transitive.  In particular, in the proof of Proposition \ref{prop:COTobstr_equiv_rln}, we needed an explicit description of the new metaboliser, as provided by Lemma \ref{Lemma:Cancellation_Blanchfield}.  In the case of integral Blanchfield forms, the result follows from the corresponding result for integral Seifert forms.  The proof of Lemma \ref{Lemma:Cancellation_Blanchfield} relies on the fact that $\Q[\Z]$ is a principal ideal domain.  The Witt group of integral Blanchfield forms injects into the Witt group of rational Blanchfield forms (see Proposition \ref{Thm:zeroAC2_goesto_zeroAC1}), so working with rational coefficients is not a large restriction.

\begin{lemma}\label{Lemma:Cancellation_Blanchfield}
Let $(H,\Bl)$ and $(H',\Bl')$ be rational Blanchfield forms.  Suppose that \\$(H \oplus H',\Bl \oplus \Bl')$ is metabolic with metaboliser $P = P^\bot \subseteq H \oplus H'$, and that $(H',\Bl')$ is metabolic with metaboliser $Q = Q^\bot \subseteq H'$.  Then $(H,\Bl)$ is also metabolic, and a metaboliser is given by
\[R := \{h \in H \, | \, \exists \, q \in Q \text{ with } (h,q) \in P\} \subseteq H.\]
\end{lemma}
\begin{proof}
A Blanchfield form is the same as a $0$-dimensional symmetric Poincar\'{e} complex in the category of finitely generated $\Q[t,t^{-1}]$-modules with $1-t$ acting as an automorphism.
By \cite[Propositions~3.2.2~and~3.4.5~(ii)]{Ranicki2}, a metaboliser $P$ for a Blanchfield form $(H,\Bl)$ is the same as a $1$-dimensional symmetric Poincar\'{e} pair
\[(f \colon C \to D, (0,\Bl^{\wedge})),\]
where $C = S^0 H^\wedge$ and $D = S^0P^\wedge$, in the category of finitely generated $\Q[t,t^{-1}]$-modules with $1-t$ acting as an automorphism.  This is an algebraic null--cobordism of $(H^\wedge,\Bl^\wedge)$.  Since $\Q[t,t^{-1}]$ is a PID, all modules are automatically of homological dimension 1.  We need to check that a submodule $P \subseteq H$ also has $1-t$ acting as an automorphism.  To see this, first note that since $P$ is a submodule, it is preserved by $1-t$, so $(1-t)(P) \subseteq P$.  Therefore $\ker(1-t\colon P \to P) \cong 0$, since there is no kernel of $1-t \colon H \to H$.  Note that submodules of $H$ are also finitely generated since $\Q[t,t^{-1}]$ is Noetherian.  Since we may also consider $1-t\colon P \to P$ as a linear transformation of a finite dimensional $\Q$-vector space, it must therefore be an automorphism as claimed, for dimension reasons.
Let
\[\left(\ba{c}g \\ g'\ea\right) \colon P \to H \oplus H'\]
and
\[h \colon Q \to H'\]
be the inclusions of the metabolisers.  We therefore have symmetric Poincar\'{e} pairs:
\[(\left(\ba{cc}g^\wedge & g'^\wedge\ea\right) \colon H^\wedge \oplus H'^\wedge \to P^\wedge = D_0, (0,\Bl^\wedge \oplus \Bl'^\wedge))\]
and
\[(h^\wedge \colon H'^\wedge \to Q^\wedge = D'_0,(0,-\Bl'^\wedge)).\]
We have introduced a minus sign in front of $\Bl'^\wedge$, so that we can glue the two algebraic cobordisms together along $H'^\wedge$ to yield another algebraic cobordism.
\[\xymatrix @R+1cm @C+1cm{ & H'^\wedge = D''_1 \ar[d]^-{\left(\ba{c} g'^\wedge \\ h^\wedge \ea \right)}\\
H^\wedge = C_0 \ar[r]^-{\left(\ba{c} g^\wedge \\ 0 \ea \right)} & P^\wedge \oplus Q^\wedge = D''_0.
}\]
Here $(D'',0 \cup_{\Bl'^{\wedge}} 0)$ is not 0-dimensional, so we cannot yet deduce that we have a metaboliser.  Since $1-t$ acts as an automorphism on submodules, it also acts as an isomorphism on $H_*(D'')$.  Also, again since $\Q[t,t^{-1}]$ is Noetherian, $H_*(D'')$ is finitely generated.  Therefore by \cite[Corollary~1.3]{Levine2}, $H_*(D'')$ is $\Q[t,t^{-1}]$-torsion.  Since $\Q[t,t^{-1}]$ is a PID, we have universal coefficient theorem isomorphisms.  We therefore have the following standard commutative diagram
\[\xymatrix @R+1cm @C+1cm{H^0(D'') \ar[r] \ar[d]^{\cong} & H^0(C) \ar[r] \ar[d]^{\cong} & H^1(D'',C) \ar[d]^{\cong} \\
H^1(D'',C)^\wedge \ar[r] & H^0(C)^\wedge \ar[r] & H^0(D'')^\wedge,
}\]
of $\Q[t,t^{-1}]$-torsion modules with exact rows and vertical isomorphisms, which by a standard argument, given in Theorems \ref{Lemma:COT4.4} and \ref{thm:COT4.4chainversion}, shows that
\[\ol{R}:= \im\Big(H^0(D'') \to H^0(C) \Big)\]
is a metaboliser for $\Bl^\wedge \colon H^0(C) = H^{\wedge\wedge} \times H^{\wedge\wedge} \to \Q(t)/\Q[t,t^{-1}]$, where the over--line indicates the use of the involution.  We make the identifications $$H^0(C) \cong H^{\wedge\wedge} \cong H,$$ $$(D''^0)^\wedge \cong (P^\wedge \oplus Q^\wedge)^\wedge \cong P \oplus Q$$ and $$(D''^1)^\wedge \cong H'^{\wedge\wedge} \cong H',$$ so that
\[H^0(D'') \cong \ker \big( \left(\ba{cc}g'  & h \ea\right)\colon P \oplus Q \to H'\big).\]
Since the identification $H^{\wedge\wedge} \cong H$ involves an involution, we have that
\[\ol{\ol{R}} = R = \im \Big( \left( \ba{cc}g  & 0 \ea\right) \colon \ker \big( \left(\ba{cc}g'  & h \ea\right)\colon P \oplus Q \to H'\big) \to H \Big),\]
is a metaboliser for $\Bl$.  Finally, this is indeed equal to
\[\{h \in H \, | \, \exists \, q \in Q \text{ with } (h,q) \in P\},\]
as required.
\end{proof}

\begin{proposition}\label{Thm:zeroAC2_goesto_zeroAC1}
There is a surjective homomorphism
\[\ac2 \to \mathcal{AC}_1,\]
where $\mathcal{AC}_1$ is the algebraic concordance group.
\end{proposition}
\begin{proof}
We use the formulation in terms of the Blanchfield form, since we have already explained how to extract the Blanchfield form from the chain complex, and since a very similar argument to that which would be used here in terms of $L$-theory will be given in the proof of Theorem \ref{Thm:zeroAC2_goes_to_zeroL4QG}.  Given an element $(H,\Y,\xi) \in \ac2$, we can find the Blanchfield form on the $\Z[\Z]$-module:
\[\Bl \colon H_1(\Z[\Z] \otimes_{\zh} Y) \times H_1(\Z[\Z] \otimes_{\zh} Y) \to \frac{\Q(\Z)}{\Z[\Z]},\]
just as in Proposition \ref{Prop:chainlevelBlanchfield}, but with $\Q$ replaced by $\Z$ in the coefficient ring.  The fact that $H$ is homological dimension 1 means that even though $\Z[\Z]$ is not a principal ideal domain, the universal coefficient spectral sequence still yields an isomorphism:
\[H^1(\Q(\Z)/\Z[\Z] \otimes_{\zh} Y) \xrightarrow{\simeq} \Hom_{\Z[\Z]}(H_1(\Z[\Z] \otimes_{\zh} Y),\Q(\Z)/\Z[\Z]),\] as proved in \cite{Levine2}. The integral Blanchfield form is therefore also non-singular.  To see that addition commutes with the map $\ac2 \to \mathcal{AC}_1$, note that the Alexander modules add as in Proposition \ref{prop:2ndderivedsubgroup}.  The symmetric structures also have no mixing between the chain complexes of $Y$ and $Y^{\dag}$ in the formulae in Definition \ref{Defn:connectsumalgebraic}, so that the Blanchfield form of a connected sum in $\ac2$ is the direct sum in the Witt group of Blanchfield forms.  Surjectivity follows from the fact (see \cite{Levine2}) that every Blanchfield form is realised as the Blanchfield form of a knot, and therefore as the Blanchfield form of the fundamental symmetric Poincar\'{e} triad of a knot.

We will show the following, which we state as a separate result, and prove after the rest of the proof of Proposition \ref{Thm:zeroAC2_goesto_zeroAC1}:
\begin{theorem}\label{thm:COT4.4chainversion}
For triple $(H,\Y,\xi) \in \ac2$ which is second order algebraically concordant to the unknot, via a 4-dimensional symmetric Poincar\'{e} pair:
\[(j \colon \zhd \otimes_{\zh} N \to V,\, (\Theta,\theta)),\]
if we define:
\[P:= \ker(j_* \colon H_1(\Q[\Z]\otimes_{\zhd}\zhd \otimes_{\zh} N) \to H_1(\Q[\Z] \otimes_{\zhd} V)),\]
then $P$ is a metaboliser for the rational Blanchfield form on $H_1(\Q[\Z] \otimes_{\zh} N)$.
\end{theorem}
Before proving Theorem \ref{thm:COT4.4chainversion}, we will first show how it implies Proposition \ref{Thm:zeroAC2_goesto_zeroAC1}.  Now recall that the Witt group of integral Blanchfield forms injects into the Witt group of rational Blanchfield forms. To see this, first note that:
\[H_1(\Z[\Z] \otimes_{\zh} N) \rightarrowtail H_1(\Q[\Z] \otimes_{\zh} N) \cong \Q \otimes_{\Z} H_1(\Z[\Z] \otimes_{\zh} N).\]
The first map is an injection since $H_1(\Z[\Z] \otimes_{\zh} N)$ is $\Z$-torsion free (Theorem \ref{Thm:Levinemodule}), while the second map is an isomorphism as $\Q$ is flat as a $\Z$-module.  Then suppose that we have a metaboliser $P_{\Q}$ for the rational Blanchfield form.  This restricts to a metaboliser \[P_{\Z} := P_{\Q} \cap (\Z \otimes_{\Z} H_1(\Z[\Z] \otimes_{\zh} N))\] for the integral Blanchfield form, since the calculation, restricted to the image of \\$H_1(\Z[\Z] \otimes_{\zh} N)$, is the same for the two forms.  The symmetric structure map in the rational case is just the integral map tensored up with the rationals; $(\theta'_0)_{\Q} = \Id_{\Q} \otimes_{\Z} (\theta'_0)_{\Z}$.

Therefore, the only place that the two calculations could differ is if one took \[s \in \Q[t,t^{-1}] \setminus \Z[t,t^{-1}]\] or \[z \in (\Q[\Z] \otimes_{\zh} N)^1 \setminus (\Z[\Z] \otimes_{\zh} N)^1.\]  In these cases we can clear denominators in the equation:
\[\partial^* (z) = s\theta'_0(y)\]
to get:
\[\partial^*(nz) = ns\theta'_0(y),\]
for some $n \in \Z$, so that now $ns \in \Z[t,t^{-1}]$ and $nz \in (\Z[\Z] \otimes_{\zh} N)^1$. Then:
\[\frac{1}{ns}\ol{(nz)(x)} = \frac{n}{ns}\ol{z(x)} = \frac{1}{s}\ol{z(x)},\]
which is the same outcome.
By Theorem \ref{thm:COT4.4chainversion}, second order algebraically slice triples map to metabolic rational Blanchfield forms, which we have now seen restrict to metabolic integral Blanchfield forms.  By applying Proposition \ref{prop:inverseswork}, we see that we have a well--defined homomorphism as claimed.  This completes the proof of Proposition \ref{Thm:zeroAC2_goesto_zeroAC1}.
\end{proof}

Modulo the proof of Theorem \ref{thm:COT4.4chainversion}, we have the following diagram of homomorphisms,
\[\xymatrix @R+1cm @C+1cm{
\C \ar[r] \ar @{->>} [d] & \mathcal{AC}_2 \ar @{->>} [d]  \\
\C/\mathcal{F}_{(0.5)} \ar[r]^{\simeq}  \ar[ur] & \mathcal{AC}_1,}\]
with geometry on the left and algebra on the right; the bottom map is an isomorphism: see \cite[Remark~1.3.2]{COT}.

Next, we will proof Theorem \ref{thm:COT4.4chainversion}.  This theorem is an algebraic reworking of \cite[Theorem~4.4]{COT} (our Theorem \ref{Lemma:COT4.4}): it is crucial for the control which the Blanchfield form provides on which 1-cycles of $\Q[\Z] \otimes_{\zh} N$ bound in some 4-dimensional pair, which in turn controls which representations extend over putative algebraic slice disc exteriors.

\begin{proof}[Proof of Theorem \ref{thm:COT4.4chainversion}]
A large part of this proof can be carried over almost verbatim from the proof of \cite[Theorem~4.4]{COT}, which was our Theorem \ref{Lemma:COT4.4}, subject to a manifold-chain complex dictionary, as follows.  The homology of $M_K$ with coefficients in a ring $R$ should be replaced with the homology of:
\[R \otimes_{\zh} N;\]
the (co)homology of $W$ with coefficients in $R$ should be replaced with the (co)homology of:
\[R \otimes_{\zhd} V; \text{ and}\]
the homology of the pair $(W,M_K)$ with coefficients in $R$ should be replaced with the homology of:
\[R \otimes_{\zhd} \mathscr{C}(j \colon \zhd \otimes_{\zh} N \to V).\]
To complete the proof we need to show that:
\begin{description}
         \item[(i)] The relative linking pairings $\beta_{rel}$ are non-singular.  This will follow from the argument in the proof of Theorem \ref{Lemma:COT4.4} once we show, for an algebraic $(1.5)$-solution $V$, that $$H_*(\Q(\Z) \otimes_{\zhd} V) \cong 0.$$  Note that this also implies by universal coefficients that $$H^*(\Q(\Z) \otimes_{\zhd} V) \cong 0,$$ and that $H_*(\Q[\Z] \otimes_{\zhd} V)$ is torsion, since $\Q(\Z)$ is flat over $\Q[\Z]$.
         \item[(ii)] The sequence \[TH_2(\Q[\Z] \otimes_{\zhd} \mathscr{C}(j)) \xrightarrow{\partial} H_1(\Q[\Z] \otimes_{\zh} N) \xrightarrow{j_*} H_1(\Q[\Z] \otimes_{\zhd} V)\] is exact.
       \end{description}
To prove (i) we once again apply Proposition \ref{Prop:COT2.10chainversion}, here to the chain complex \[\Q[\Z] \otimes_{\zhd} \mathscr{C}(j \circ f_- \colon \zhd \otimes_{\zh} D_- \to V).\]  Since $j \circ f_-$ induces isomorphisms on rational homology, the relative homology groups vanish:
\[H_*(\Q \otimes_{\Q[\Z]} \Q[\Z] \otimes_{\zhd} \mathscr{C}(j \circ f_-)) \cong 0.\]
Proposition \ref{Prop:COT2.10chainversion} then says that:
\[H_*(\Q(\Z) \otimes_{\Q[\Z]} \Q[\Z] \otimes_{\zhd} \mathscr{C}(j \circ f_-)) \cong 0,\]
which implies the second isomorphism of:
\begin{eqnarray*} H_*(\Q(\Z) \otimes_{\zhd} V) & \cong & H_*(\Q(\Z) \otimes_{\Q[\Z]} \Q[\Z] \otimes_{\zhd} V) \\ & \cong & H_*(\Q(\Z) \otimes_{\Q[\Z]} \Q[\Z] \otimes_{\zh} D_-) \cong 0.\end{eqnarray*}
So see the last isomorphism, note that as in the proof of Proposition \ref{Prop:NinL4QG} the homology of the circle with $(t-1)$ inverted vanishes, as long as $t$ maps non-trivially under the representation into $\Q[\Z]$, which in this case it certainly does.  This justifies the statement above that $$H_*(\Q(\Z) \otimes_{\zh} D_-) \cong 0.$$  The definitions of the relative linking pairings can be made purely algebraically using chain complexes, using the corresponding sequences of isomorphisms:
\begin{eqnarray*}\ol{TH_2(\Q[\Z] \otimes_{\zhd} \mathscr{C}(j))} &\xrightarrow{\simeq}& TH^2(\Q[\Z] \otimes_{\zhd} V) \xrightarrow{\simeq} \\ H^1(\Q(\Z)/\Q[\Z] \otimes_{\zhd} V) &\xrightarrow{\simeq}& \Hom_{\Q[\Z]}(H_1(\Q[\Z] \otimes_{\zhd} V),\Q(\Z)/\Q[\Z]); \end{eqnarray*} and
\begin{eqnarray*}
\ol{TH_1(\Q[\Z] \otimes_{\zhd} V)} &\toiso& TH^3(\Q[\Z] \otimes_{\zhd} V) \toiso \\ H^2(\Q(\Z)/\Q[\Z] \otimes_{\zhd} V)
 & \xrightarrow{\simeq}& \Hom_{\Q[\Z]}(H_2(\Q[\Z] \otimes_{\zhd} V),\Q(\Z)/\Q[\Z]).
\end{eqnarray*}
There are also an explicit chain level formulae for the pairings $\beta_{rel}$ in a similar vein to that for $\Bl$ in Proposition \ref{Prop:chainlevelBlanchfield}; for us, the important point is that the above maps are indeed isomorphisms.

To prove (ii), we show that in fact $H_2(\Q[\Z] \otimes_{\zhd} \mathscr{C}(j))$ is entirely torsion.
This follows from the long exact sequence of the pair \[\Id_{\Q(\Z)} \otimes j \colon \Q(\Z) \otimes_{\zh} N \to \Q(\Z) \otimes_{\zhd} V.\]
We have the following excerpt:
\[H_2(\Q(\Z) \otimes_{\zhd} V) \to H_2(\Q(\Z) \otimes_{\zhd} \mathscr{C}(j)) \to H_1(\Q(\Z) \otimes_{\zh} N).\]
We have already seen in (i) that $H_2(\Q(\Z) \otimes_{\zhd} V) \cong 0$.  We claim that $$H_1(\Q(\Z) \otimes_{\zh} N) \cong 0,$$ which then implies by exactness that the central module $H_2(\Q(\Z) \otimes_{\zhd} \mathscr{C}(j))$ is also zero.  Then note, since $\Q(\Z)$ is flat over $\Q[\Z]$, that
\[H_2(\Q(\Z) \otimes_{\zhd} \mathscr{C}(j)) \cong \Q(\Z) \otimes_{\Q[\Z]} H_2(\Q[\Z] \otimes_{\zhd} \mathscr{C}(j)).\]  That this last module vanishes means that $H_2(\Q[\Z] \otimes_{\zhd} \mathscr{C}(j))$ is $\Q[\Z]$-torsion.  To see the claim that $H_1(\Q(\Z) \otimes_{\zh} N) \cong 0$, recall that:
\[H_1(\Q[\Z] \otimes_{\zh} N) \cong H_1(\Q[\Z] \otimes_{\zh} Y) \cong \Q \otimes_{\Z} H_1(\Z[\Z] \otimes_{\zh} Y) \cong \Q \otimes_{\Z} H,\]
and that an Alexander module $H$ is $\Z[\Z]$-torsion, so that the $\Q[\Z]$-module $\Q \otimes_{\Z} H$ is $\Q[\Z]$-torsion.  This completes the proof of (ii); and therefore completes the proof of all the points that the chain complex argument for Theorem \ref{thm:COT4.4chainversion} is not directly analogous to the geometric argument in the proof of Theorem \ref{Lemma:COT4.4}.
\end{proof}


\begin{definition}\label{Defn:map_AC2_to_COT_(C/1.5)}
We define the map $\ac2 \to \mc{COT}_{(\C/1.5)}$ by mapping a triple $(H,\Y,\xi)$ to
\[\bigsqcup_{p \in \Q \otimes_{\Z} H} \, ((\Q\G \otimes_{\zh} N, \Id \otimes \theta)_p,\xi_p), \]
with each $(\Q\G \otimes_{\zh} N)_p$ defined using \[\ba{rcl} \rho \colon \Z \ltimes H &\to & \G \\ (n,h) &\mapsto & (n,\Bl(p,h))\ea\]
and $\xi_p$ given by the composition
\begin{eqnarray*} \xi_p \colon \Q \otimes_{\Z} H &\xrightarrow{\Id \otimes \xi}& \Q \otimes_{\Z} H_1(\Z[\Z] \otimes_{\zh} Y) \toiso H_1(\Q[\Z] \otimes_{\zh} Y) \\ &\toiso& H_1(\Q[\Z] \otimes_{\zh} N) \toiso H_1(\Q[\Z] \otimes_{\Q\G} (\Q\G \otimes_{\zh} N)_p). \end{eqnarray*}
The maps labelled as isomorphisms in this composition are given by the universal coefficient theorem, a Mayer-Vietoris sequence, and a simple chain level isomorphism for the final identification.
\qed\end{definition}

We give a more precise statement of Theorem \ref{Thm:zeroAC2_goes_to_zeroL4QG}, which shows that the map of pointed sets of Definition \ref{Defn:map_AC2_to_COT_(C/1.5)} is well--defined.

\begin{theorem}\label{Thm:betterstatementzeroAC2_to_zeroL4QG}
Let $(H,\Y,\xi) \sim (H^\dag,\Y^\dag,\xi^\dag) \in \ac2$  be equivalent triples.  Then there exists a metaboliser \[P = P^{\bot} \subseteq (\Q \otimes_{\Z} H) \oplus (\Q \otimes_{\Z} H^\dag)\] for the rational Blanchfield form $$\Bl \oplus -\Bl^\dag \colon (\Q \otimes_{\Z} H) \oplus (\Q \otimes_{\Z} H^\dag) \times (\Q \otimes_{\Z} H) \oplus (\Q \otimes_{\Z} H^\dag) \to \Q(t)/\Q[t,t^{-1}],$$ such that, for any $(p,q) \in (\Q \otimes_{\Z} H) \oplus (\Q \otimes_{\Z} H^\dag)$, the corresponding elements in $L^4(\Q\G,\Q\G-\{0\})$, which are obtained using the representation $\rho \colon \Z \ltimes H \to \G$ defined by $\xi, p$ and $\Bl$, satisfy:
\[((\Q\G \otimes_{\zh} N)_p,\theta_p) = ((\Q\G \otimes_{\zh} N^\dag)_q,\theta^\dag_q) \in L^4(\Q\G,\Q\G-\{0\}),\]
with the reason why this holds being a 4-dimensional symmetric Poincar\'{e} pair
\[(j_p \oplus j_q^{\dag} \colon (\Q\G \otimes_{\zh} N)_p \oplus (\Q\G \otimes_{\zh} N^\dag)_q \to V_{(p,q)},(\delta \theta_{(p,q)}, \theta_p \oplus -\theta^{\dag}_q))\]
over $\Q\G$ such that
$$H_1(\Q \otimes_{\Q\G} (\Q\G \otimes_{\zh} N)_p) \toiso H_1(\Q \otimes_{\Q\G} V_{(p,q)}) \xleftarrow{\simeq} H_1(\Q \otimes_{\Q\G} (\Q\G \otimes_{\zh} N^\dag)_q),$$
such that the isomorphism
\begin{multline*}\xi_p \oplus \xi_q^\dag \colon (\Q \otimes_{\Z} H) \oplus (\Q \otimes_{\Z} H^\dag) \toiso \\ H_1(\Q[\Z] \otimes_{\Q\G} (\Q\G \otimes_{\zh} N)_p) \oplus H_1(\Q[\Z] \otimes_{\Q\G} (\Q\G \otimes_{\zhdag} N^\dag)_q)\end{multline*}
restricts to an isomorphism
\[P \toiso \ker \big(H_1(\Q[\Z] \otimes_{\zh} N) \oplus H_1(\Q[\Z] \otimes_{\zhdag} N^\dag) \to H_1(\Q[\Z] \otimes_{\Q\G} V_{(p,q)})\big),\]
and such that the algebraic Thom complex (Definition \ref{Defn:algThomcomplexandthickening}), taken over the Ore localisation, is algebraically null-cobordant in $L^4_S(\K) \cong L^0_S(\K)$:
\[[(\K \otimes_{\Q\G} \mathscr{C}((j_p \oplus j_q^\dag)),\Id \otimes \delta \theta_{(p,q)}/(\theta_{p}\oplus-\theta^\dag_q))] = [0] \in L^4_S(\K).\]
That is,
\[\bigsqcup_{p \in H} \, ((\Q\G \otimes_{\zh} N)_p,\xi_p) \sim \bigsqcup_{q \in H^\dag} \, ((\Q\G \otimes_{\zhdag} N^\dag)_q,\xi^\dag_q)  \in \mc{COT}_{(\C/1.5)}.\]
\end{theorem}

\begin{proof}
By the hypothesis we have a symmetric Poincar\'{e} triad over $\zhd$:
\[\xymatrix @R+1cm @C+0.8cm {\ar @{} [dr] |{\stackrel{(\gamma,\g^{\dag})}{\sim}} (E,\phi) \oplus (E^{\dag},-\phi^{\dag}) \ar[r]^-{(\Id, \Id \otimes \varpi_{E^{\dag}})} \ar[d]_{\left( \begin{array}{cc} \eta & 0 \\ 0 & \eta^{\dag}\end{array}\right)} & (E,0) \ar[d]^{\delta}
\\ (Y,\Phi) \oplus (Y^{\dag},-\Phi^{\dag}) \ar[r]^-{(j,j^{\dag})} & (V,\Theta),
}\]
with isomorphisms
\[H_*(\Z \otimes_{\zh} Y) \toiso H_*(\Z \otimes_{\zhd} V) \xleftarrow{\simeq} H_*(\Z \otimes_{\zhdag} Y^\dag),\]
and a commutative square
\[\xymatrix @R+1cm @C+1cm{H \oplus H^{\dag} \ar[r]^{(j_{\flat},j_{\flat}^{\dag})} \ar[d]^{\left(\begin{array}{cc} \xi & 0 \\ 0 & \xi^{\dag} \end{array} \right)} & H' \ar[d]^{\xi'} \\
H_1(\Z[\Z] \otimes_{\zh} Y) \oplus H_1(\Z[\Z] \otimes_{\zh} Y^{\dag}) \ar[r]^-{\Id_{\Z[\Z]} \otimes (j_*,j^{\dag}_*)} & H_1(\Z[\Z] \otimes_{\zhd} V).
}\]
Corresponding to the manifold triad
\[\xymatrix @R+1cm @C+1cm{S^1 \times S^1 \sqcup S^1 \times S^1 \ar[r] \ar[d] & S^1 \times S^1 \times I \ar[d] \\ S^1 \times D^2 \sqcup S^1 \times D^2 \ar[r] & S^1 \times D^2 \times I,}\]
we have a symmetric Poincar\'{e} triad.
\[\xymatrix @R+1cm @C+1cm { (E^U,-\phi^U) \oplus (E^{U},\phi^U) \ar[r]^-{(\Id, \Id)} \ar[d]_{\left( \begin{array}{cc} \eta^U & 0 \\ 0 &  \eta^{U}\end{array}\right)} & (E^U,0) \ar[d]^{\delta^U}
\\ (Y^U,0) \oplus (Y^U,0) \ar[r]^-{(j^U,j^U)} & (Y^U,0).
}\]
With this triad tensored up over $\zhd$ sending $t \mapsto g_1$ as usual, we glue the two triads together as follows:
\[\xymatrix @R+1cm @C+1cm{(Y^U,0) \oplus (Y^U,0) \ar[r]^-{(j^U,j^U)} & (Y^U,0) \\
(E^U,-\phi^U) \oplus (E^U,\phi^U) \ar[u]^{\left( \begin{array}{cc} \eta^U & 0 \\ 0 &  \eta^{U}\end{array}\right)} \ar[r]^-{(\Id,\Id)} & (E^U,0) \ar[u]^{\delta^U} \\
(E,\phi) \oplus (E^\dag,-\phi^\dag) \ar @{} [dr] |{\stackrel{(\gamma,\g^{\dag})}{\sim}} \ar[u]_{\cong}^{\left( \begin{array}{cc} \varpi_E & 0 \\ 0 &  \varpi_{E^\dag}\end{array}\right)} \ar[r]^-{(\Id, \Id \otimes \varpi_{E^{\dag}})} \ar[d]_-{\left( \begin{array}{cc} \eta & 0 \\ 0 & \eta^{\dag}\end{array}\right)} & (E,0) \ar[u]^{\cong}_{\varpi_{E}} \ar[d]^{\delta} \\
(Y,\Phi) \oplus (Y^{\dag},-\Phi^{\dag}) \ar[r]^-{(j,j^{\dag})} & (V,\Theta),}\]
to obtain a symmetric Poincar\'{e} pair over $\zhd$:
\[((i,i^\dag)\colon N \oplus N^\dag \to \widehat{V} := V \cup_E Y^U,(\widehat{\Theta}:=\Theta \cup 0,\theta \oplus - \theta^{\dag})).\]
We can define $P$, by Theorem \ref{thm:COT4.4chainversion}, to be
\begin{multline*} P := \ker((\Q \otimes_{\Z} H) \oplus (\Q \otimes_{\Z} H^\dag) \to H_1(\Q[\Z] \otimes_{\zh} N) \oplus H_1(\Q[\Z] \otimes N^\dag) \\ \to  H_1(\Q[\Z] \otimes_{\zhd} \widehat{V})).\end{multline*}
Now, for all $(p,q) \in P$, the representation
\begin{multline*} (\Bl\oplus -\Bl^\dag)((\xi(p),\xi^\dag(q)),\bullet) \colon H_1(\Q[\Z] \otimes_{\zh} N) \oplus H_1(\Q[\Z] \otimes_{\zh} N^\dag) \\ \to \Q(t)/\Q[t,t^{-1}],\end{multline*}
extends, by \cite[Theorem~3.6]{COT}, to a representation \[H_1(\Q[\Z] \otimes_{\zhd} \widehat{V}) \to \Q(t)/\Q[t,t^{-1}].\]  This holds since the proof of \cite[Theorem~3.6]{COT} is entirely homological algebra, so carries over to the chain complex situation without the need for additional arguments.  We therefore have an extension:
\[\xymatrix @R+1cm @C+1cm{H \oplus H^\dag \ar[r]^-{(j_{\flat},j^\dag_{\flat})} \ar[d]^{\cong}_{\left(\begin{array}{cc} \xi & 0 \\ 0 & \xi^{\dag} \end{array} \right)} & H' \ar[d]^{\cong}_{\xi'} \\
H_1(\Z[\Z] \otimes_{\zh} N) \oplus H_1(\Z[\Z] \otimes_{\zhdag} N^\dag) \ar[r]^-{\Id_{\Z[\Z]} \otimes (i,i^\dag)} \ar @{>->} [d] & H_1(\Z[\Z] \otimes_{\zhd} \widehat{V}) \ar @{>->} [d] \\
H_1(\Q[\Z] \otimes_{\zh} N) \oplus H_1(\Q[\Z] \otimes_{\zhdag} N^\dag) \ar[r]^-{\Id_{\Q[\Z]} \otimes (i,i^\dag)} \ar[dr] |{(\Bl\oplus\-\Bl^\dag)((\xi(p),\xi^\dag(q)),\bullet)} & H_1(\Q[\Z] \otimes_{\zhd} \widehat{V}) \ar[d] \\
 & \frac{\Q(t)}{\Q[t,t^{-1}]}.
}\]
Noting that, from the Mayer-Vietoris sequence for $\widehat{V} = V \cup_E Y^U$, there is an isomorphism \[H_1(\Z[\Z] \otimes_{\zhd} V) \toiso H_1(\Z[\Z] \otimes_{\zhd} \widehat{V}),\] the top square commutes by the consistency condition.  We therefore have an extension of representations:
\[\xymatrix @R+1cm @C+1cm{\Z \ltimes (H\oplus H^\dag) \ar[r]^{(\Id_{\Z},(j_{\flat},j_{\flat}^\dag))} \ar[dr]_{\rho} & \Z \ltimes H' \ar[d]^{\wt{\rho}} \\
 & \G.
}\]
The element $$((\Q\G \otimes_{\zh} N)_p,\theta_p) \oplus ((\Q\G \otimes_{\zhdag} N^\dag)_p,-\theta_p^\dag)  \in L^4(\Q\G,\Q\G -\{0\})$$ therefore lies, by virtue of the existence of $\Q\G \otimes_{\zhd} \widehat{V}_{(p,q)}$, in
\[\ker (L^4(\Q\G,\Q\G -\{0\}) \to L^3(\Q\G)).\]
As in the $L$-theory localisation sequence (Definition \ref{defn:localisationexactsequence}), we therefore have the element:
\[(\ol{V}_{(p,q)},\ol{\Theta}_{(p,q)}) := ((\K \otimes_{\zhd} \mathscr{C}((i,i^\dag)))_{(p,q)},\Theta_{(p,q)}/(\theta_p \oplus - \theta_q^\dag)) \in L^4_S(\K),\]
whose boundary is $$((\Q\G \otimes_{\zh} N)_p,\theta_p) \oplus ((\Q\G \otimes_{\zhdag} N^\dag)_p,-\theta_p^\dag)  \in L^4(\Q\G,\Q\G -\{0\}).$$  Since $2$ is invertible in $\K$, we can do algebraic surgery below the middle dimension \cite[Part~I,~Proposition~4.4]{Ranicki3}, on $\ol{V}_{(p,q)}$, to obtain a non-singular Hermitian form:
\[(\lambda \colon H^2(\ol{V}_{(p,q)}) \times H^2(\ol{V}_{(p,q)}) \to \K) \in L^0_S(\K) \cong L^4_S(\K),\]
whose image in:
\[L^0_S(\K)/L^0(\Q\G)\]
detects the class of $\Q\G \otimes_{\zh} N \in L^4(\Q\G,\Q\G -\{0\})$.
Once again, we apply Proposition \ref{Prop:COT2.10chainversion}, again noting that it in fact applies just as well to finitely generated projective module chain complexes as to finitely generated free module complexes.  Since $j$ and $j^\dag$ induce isomorphisms on $\Z$-homology, and therefore on $\Q$-homology, we have that the chain map
\[\Id \otimes i \colon \Q \otimes_{\Q\G} (\Q\G \otimes_{\zh} N)_p \to \Q \otimes_{\Q\G} (\Q\G \otimes_{\zhd} \widehat{V}_{(p,q)})\]
induces isomorphisms
\[ i_* \colon H_k(\Q \otimes_{\zh} N) \toiso H_k(\Q \otimes_{\zhd} \widehat{V})\]
for all $k$, by a straight--forward Mayer-Vietoris argument.  Therefore $$H_k(\Q \otimes_{\zhd} \mathscr{C}(i)) \cong 0$$ for all $k$ by the long exact sequence of a pair.  By Proposition \ref{Prop:COT2.10chainversion}, we therefore have that
\[H_k(\K \otimes_{\zhd} \mathscr{C}(i)_{(p,q)}) \cong 0\]
for all $k$.  The long exact sequence in $\K$-homology associated to the short exact sequence
\[0\to \mathscr{C}(i)_{(p,q)} \to \mathscr{C}((i,i^\dag))_{(p,q)} \to S(\zhd \otimes_{\zhdag} N^\dag_q) \to 0\]
implies, noting that $H_*(\K \otimes_{\zhdag} N^\dag_q) \cong 0$, that
\[H_k(\K \otimes_{\zhd} \mathscr{C}((i,i^\dag))_{(p,q)}) = H_k(\ol{V}_{(p,q)}) \cong 0\]
for all $k$.  In particular, since \[H_2(\ol{V}_{(p,q)}) \cong H^2(\ol{V}_{(p,q)}) \cong 0,\] we see that the image of $\ol{V}_{(p,q)}$ in $L^0_S(\K)$, which is the intersection form $\lambda$, is trivially hyperbolic and represents the zero class of $L^0_S(\K)$.  This completes the proof that
\begin{multline*} \bigsqcup_{p \in H} \, ((\Q\G \otimes_{\zh} N,\Id \otimes \theta)_p,\xi_p) \sim \bigsqcup_{q \in H^\dag} \, ((\Q\G \otimes_{\zhdag} N^\dag,\Id \otimes \theta^\dag)_q,\xi^\dag_q) \\ \in \mc{COT}_{(\C/1.5)}.\end{multline*}

\end{proof}

Finally, we have a non-triviality result, which shows that we can extract the $L^{(2)}$-signatures from $\ac2$.  In order to obstruct the equivalence of triples $(H,\Y,\xi) \sim (H^\dag,\Y^\dag,\xi^\dag) \in \ac2$, we just need, by Proposition \ref{prop:inverseswork}, to be able to obstruct an equivalence $(H,\Y,\xi) \sim (\{0\},\Y^U,\Id_{\{0\}})$.  To achieve this, as in Definition \ref{defn:COTobstructionset_2} we need to obstruct the existence of a 4-dimensional symmetric Poincar\'{e} pair over $\Q\G$
\[(j \colon (\Q\G \otimes_{\zh} N)_p \to V_p,(\Theta_p,\theta_p)),\]
for at least one $p \neq 0$, with $\xi(p) \in P$, for each metaboliser $$P = P^\bot \subseteq H_1(\Q[\Z] \otimes_{\zh} N)$$ of the Blanchfield form, where $V_p$ satisfies that
\[\xi(p) \in \ker(j_* \colon H_1(\Q[\Z] \otimes_{\zh} N_p) \to H_1(\Q[\Z] \otimes_{\Q\G} V_p)),\]
that
\[j_*\colon H_1(\Q \otimes_{\zh} N) \xrightarrow{\simeq} H_1(\Q \otimes_{\Q\G} V_p)\]
is an isomorphism, and that
\[[\K \otimes_{\Q\G} \mathscr{C}(j)] = [0] \in L^4_S(\K).\]
We do this by taking $L^{(2)}$-signatures of the middle dimensional pairings on putative such $V_p$, to obstruct the Witt class in $L^4_S(\K) \cong L^0_S(\K)$ from vanishing.  First, we have a notion of algebraic $(1)$-solvability.

\begin{definition}\label{Defn:alg(1)solvable}
We say that an element $(H,\Y,\xi) \in \ac2$ with image $0 \in \mathcal{AC}_1$ is \emph{algebraically $(1)$-solvable} if the following holds. There exists a metaboliser $P = P^{\bot} \subseteq H_1(\Q[\Z] \otimes_{\zh} N)$ for the rational Blanchfield form such that for any $p \in H$ such that $\xi(p) \in P$, we obtain an element:
\[\Q\G \otimes_{\zh} N_p \in \ker(L^4(\Q\G,\Q\G -\{0\}) \to L^3(\Q\G)),\]
via a symmetric Poincar\'{e} pair over $\Q\G$: \[(j \colon \Q\G \otimes_{\zh} N_p \to V_p, (\Theta_p,\theta_p)),\] with
\[P = \ker(j_* \colon H_1(\Q[\Z] \otimes_{\zh} N) \to H_1(\Q[\Z] \otimes_{\Q\G} V_p)),\]
and such that:
\[j_*\colon H_1(\Q \otimes_{\zh} N) \xrightarrow{\simeq} H_1(\Q \otimes_{\Q\G} V_p)\]
is an isomorphism.  We call each such $(j \colon \Q\G \otimes_{\zh} N_p \to V_p, (\Theta_p,\theta_p))$ an \emph{algebraic $(1)$-solution}.
\qed\end{definition}

\begin{theorem}\label{Thm:extractingL2signatures}
Suppose that $(H,\Y,\xi) \in \ac2$ is algebraically $(1)$-solvable with algebraic $(1)$-solution $(V_p,\Theta_p)$ and $\xi(p) \in P$.  Then since:
\[\ker(L^4(\Q\G,\Q\G -\{0\}) \to L^3(\Q\G)) \cong \frac{L^4(\K)}{L^4(\Q\G)} \cong \frac{L^0(\K)}{L^0(\Q\G)},\]
we can apply the $L^{(2)}$-signature homomorphism:
\[\sigma^{(2)} \colon L^0(\K) \to \R,\]
to the intersection form:
\[\lambda_{\K} \colon H_2(\K \otimes_{\Q\G} V_p) \times H_2(\K \otimes_{\Q\G} V_p) \to \K.\]
We can also calculate the signature $\sigma(\lambda_{\Q})$ of the ordinary intersection form:
\[\lambda_{\Q} \colon H_2(\Q \otimes_{\Q\G} V_p) \times H_2(\Q \otimes_{\Q\G} V_p) \to \Q,\]
and so calculate the reduced $L^{(2)}$-signature \[\wt{\sigma}^{(2)}(V_p) = \sigma^{(2)}(\lambda_{\K}) - \sigma(\lambda_{\Q}).\]  This is independent, for fixed $p$, of changes in the choice of chain complex $V_p$.    Provided we check that the reduced $L^{(2)}$-signature does not vanish, for each metaboliser $P$ of the rational Blanchfield form with respect to which $(H,\Y,\xi)$ is algebraically $(1)$-solvable, and for each $P$, for at least one $p \in P \setminus \{0\}$, then we have a \emph{chain--complex--Von--Neumann $\rho$--invariant} obstruction.  This obstructs the image of the element $(H,\Y,\xi)$ in $\mathcal{COT}_{(\C/1.5)}$ from being $U$, and therefore obstructs $(H,\Y,\xi)$ from being second order algebraically slice.


\end{theorem}
\begin{remark}
We do not require any references to 4-manifolds, other than for pedagogic reasons, to extract the \COT $L^{(2)}$-signature metabelian concordance obstructions from the triple of a $(1)$-solvable knot, or indeed for any algebraically $(1)$-solvable triple in $\ac2$.  This result relies strongly on the reason for the invariance of the reduced $L^{(2)}$-signatures which is least emphasised in the paper of \COT \cite{COT}.  This is the result of Higson-Kasparov \cite{HigsonKasparov} that the analytic assembly map is onto for PTFA groups - see \cite[Proposition~5.12]{COT}, where it is shown that the surjectivity of the assembly map implies that the $L^{(2)}$-signature and the ordinary signature coincide on the image of $L^0(\Q\G)$.  The key point is that this result does not depend on manifolds; it is a purely algebraic result.

The Higson-Kasparov result does not hold for groups with torsion, a fact made use of in e.g. \cite{chaorr}.  Homology cobordism invariants which use representations to torsion groups appear to be using deeper manifold structure than is captured by symmetric Poincar\'{e} complexes alone.
\end{remark}

\begin{proof}[Proof of Theorem \ref{Thm:extractingL2signatures}]
For this proof we omit the $p$ subscripts from the notation; it is to be understood that tensor products with $\Q\G$ depend on a choice of representation.  Given a pair
\[(j \colon \Q\G \otimes_{\zh} N \to V, (\Theta,\theta)),\]
which exhibits $(H,\Y,\xi)$ as being algebraically $(1)$-solvable, we again take the element:
\[(\K \otimes_{\Q\G} \mathscr{C}(j),\Theta/\theta) \in L^4(\K),\]
and look at its image
\[\lambda_{\K} \in L^0(\K).\]
We can calculate an intersection form $\lambda_{\K}$ on $H^2(\K \otimes_{\Q\G} \mathscr{C}(j))$, as in \cite[page~19]{Ranicki2}, by taking \[x,y \in (\K\ \otimes_{\Q\G} \mathscr{C}(j))^2 \cong \Hom_{\K}((\K \otimes_{\Q\G} \mathscr{C}(j))_2,\K),\] and calculating:
\[y' = (\Theta/\theta)_0(y) \in (\K\ \otimes_{\Q\G} \mathscr{C}(j))_2.\]
Then  \[\lambda_{\K} (x,y) := y'(x) = \ol{x(y')} \in \K.\]
This uses, as in the definition of $\Bl$ in Proposition \ref{Prop:chainlevelBlanchfield}, the identification of $(\K\ \otimes_{\Q\G} \mathscr{C}(j))_2$ with its double dual.
By taking the chain complex $\Q \otimes_{\Q\G} \mathscr{C}(j)$ we can also calculate the intersection form $\lambda_{\Q} \in L^0(\Q)$, with an analogous method.  To see that the intersection form on $H^2(\Q \otimes_{\Q\G} \mathscr{C}(j))$ is non-singular, consider the following long exact sequence of the pair; we claim that the maps labelled as $j^*$ and $\kappa$ are isomorphisms.
\[\xymatrix @C-0.3cm {H^1(\Q \otimes_{\Q\G} V) \ar[r]^-{\cong}_-{j^*} & H^1(\Q \otimes_{\zh} N) \ar[r]^{0} & H^2(\Q \otimes_{\Q\G} \mathscr{C}(j)) \ar[r]^-{\cong}_-{\kappa} & H^2(\Q \otimes_{\Q\G} V). }\]
The intersection form is given by the composition:
\begin{multline*}\lambda_{\Q} \colon H^2(\Q \otimes_{\Q\G} \mathscr{C}(j)) \xrightarrow{\kappa} H^2(\Q \otimes_{\Q\G} V) \toiso H_2(\Q \otimes_{\Q\G} \mathscr{C}(j))\\ \toiso \Hom_{\Q}(H^2(\Q \otimes_{\Q\G} \mathscr{C}(j)),\Q),\end{multline*}
given by the map $\kappa$ from the long exact sequence of a pair, followed by a Poincar\'{e} duality isomorphism induced by the symmetric structure, and a universal coefficient theorem isomorphism.  To show that $\lambda_\Q$ is non-singular we therefore need to show that $\kappa$ is an isomorpism.  The assumption that there is an isomorphism
\[j_* \colon H_1(\Q \otimes_{\zh} N) \toiso H_1(\Q \otimes_{\Q\G} V)\]
on rational first homology implies that, as claimed, there is also an isomorphism $$j^* \colon H^1(\Q \otimes_{\Q\G} V) \toiso H^1(\Q \otimes_{\zh} N)$$ on rational cohomology, by the universal coefficient theorem (the relevant $\Ext$ groups vanish with rational coefficients).  Therefore, by exactness, the map: \[\kappa \colon H^2(\Q \otimes_{\Q\G} \mathscr{C}(j)) \to H^2(\Q \otimes_{\Q\G} V)\]
is injective. Over $\Q$, for dimension reasons, it must therefore, as marked on the diagram, be an isomorphism; the dimensions must be equal since the second and third maps in the composition which gives $\lambda_\Q$ show that $$H^2(\Q \otimes_{\Q\G} V) \cong \Hom_{\Q}(H^2(\Q \otimes_{\Q\G} \mathscr{C}(j)),\Q),$$ and the dimensions over $\Q$ of $\Hom_{\Q}(H^2(\Q \otimes_{\Q\G} \mathscr{C}(j)),\Q)$ and of $H^2(\Q \otimes_{\Q\G} \mathscr{C}(j))$ coincide.  Therefore the intersection form $\lambda_\Q$ is non-singular as claimed.

The reduced $L^{(2)}$-signature
\[\wt{\sigma}^{(2)}(V) = \sigma^{(2)}(\lambda_{\K}) - \sigma(\lambda_{\Q})\]
detects the group $L^0_S(\K)/L^0(\Q\G)$.  This will follow from \cite[Proposition~5.12]{COT}, which uses a result of Higson-Kasparov \cite{HigsonKasparov} on the analytic assembly map for PTFA groups such as $\G$, and says that the $L^{(2)}$-signature agrees with the ordinary signature on the image of $L^0(\Q\G)$.  We claim that a non-zero reduced $L^{(2)}$-signature, for all possible metabolisers $P=P^{\bot}$ of the rational Blanchfield form, implies that $(H,\Y,\xi)$ is not second order algebraically slice.  To see this, we need to show that, for a fixed representation $\rho$, the reduced $L^{(2)}$-signature does not depend on the choice of chain complex $V$.

We first note, by the proof of Theorem \ref{Thm:betterstatementzeroAC2_to_zeroL4QG}, that a change in $(H,\Y,\xi)$ to an equivalent element in $\ac2$ produces an algebraic concordance which we can glue onto $V$ as in Proposition \ref{prop:equivrelation}, which neither changes the second homology of $V$ with $\K$ nor with $\Q$ coefficients, so does not change the corresponding signatures.

To show that the reduced $L^{(2)}$-signature does not depend on the choice of $V$, suppose that we have two algebraic $(1)$-solutions, that is two 4-dimensional symmetric Poincar\'{e} pairs over $\Q\G$:
\[(j \colon \Q\G \otimes_{\zh} N \to V, (\Theta,\theta))\]
and
\[(j^{\diamondsuit} \colon \Q\G \otimes_{\zh} N \to V^{\diamondsuit}, (\Theta^{\diamondsuit},\theta)),\]
such that $p = p^{\diamondsuit} \in H$.  Use the union construction to form the symmetric Poincar\'{e} complex:
\[(V \cup_{\Q\G \otimes N} V^{\diamondsuit},\Theta \cup_{\theta} -\Theta^{\diamondsuit}) \in L^4(\Q\G).\]
Over $\K$, $\Q\G \otimes_{\zh} N$ is contractible, so that:
\[(V \cup_{\Q\G \otimes N} V^{\diamondsuit},\Theta \cup_{\theta} -\Theta^{\diamondsuit}) \simeq (V \oplus V^{\diamondsuit},\Theta \oplus -\Theta^{\diamondsuit}) = (V,\Theta) - (V^{\diamondsuit},\Theta^{\diamondsuit}) \in L^4_S(\K).\]
Therefore
\[(V,\Theta) - (V^{\diamondsuit},\Theta^{\diamondsuit}) = 0 \in L^4(\K)/L^4(\Q\G),\]
which means that the images in $L^0_S(\K)$ satisfy:
\[\lambda_{\K} - \lambda_{\K}^{\diamondsuit} = 0 \in L^0_S(\K)/L^0(\Q\G)\]
If $\lambda_{\K} - \lambda_{\K}^{\diamondsuit} \in L^0(\Q\G)$, then by \cite[Proposition~5.12]{COT}:
\[\sigma^{(2)}(\lambda_{\K} - \lambda_{\K}^{\diamondsuit}) = \sigma(\Q \otimes_{\Q\G} V \cup_{\Q\G \otimes N} V^{\diamondsuit},\Id_{\Q} \otimes (\Theta \cup_{\theta} -\Theta^{\diamondsuit})) = \sigma(\lambda_{\Q}) - \sigma(\lambda_{\Q}^{\diamondsuit}),\]
where the last equality is by Novikov Additivity.  Novikov Additivity also holds for $\sigma^{(2)}$: see \cite[Lemma~5.9.3]{COT}, so that:
\[\sigma^{(2)}(\lambda_{\K}) - \sigma^{(2)}(\lambda_{\K}^{\diamondsuit}) = \sigma(\lambda_{\Q}) - \sigma(\lambda_{\Q}^{\diamondsuit})\]
and therefore:
\[\wt{\sigma}^{(2)}(V) = \wt{\sigma}^{(2)}(V^{\diamondsuit}),\]
as claimed.
\end{proof}

\begin{remark}
The results of \COT \cite{COT2} and Cochran--Harvey--Leidy (\cite{cohale}, \cite{cohale2},\cite{cohale3_primary}), which use Von Neumann $\rho$--invariants to show the existence of infinitely many linearly independent injections of $\Z$ and of $\Z_2$ into $\mathcal{F}_{(1)}/\mathcal{F}_{(1.5)}$, can therefore be applied, so that we can use the chain-complex-Von-Neumann $\rho$-invariant of Theorem \ref{Thm:extractingL2signatures} to show the existence of infinitely many injections of $\Z$ and $\Z_2$ into $\ker(\ac2 \to \mathcal{AC}_1)$.
\end{remark}



\appendix
\chapter[An $N$th order algebraic concordance group]{An $n$th Order Algebraic Concordance Group}\label{Appendix:nth_order_group}


One obvious extension to this project, which the author intends to complete in future work, is to define an algebraic concordance group which captures all of the \COT $(n)$-solvable filtration.  We will give an outline of how we conjecture that this should proceed.  The material in this appendix is presented without proof.

\begin{theorem}\label{Thm:knot_groups2}
Let $\pi = \pi_1(X)$ be the fundamental group of a knot exterior.  Then $\pi$ satisfies the following:
\begin{description}
\item[(a)] the group $\pi$ is finitely presented, where all of the generators are conjugates of one generator;
\item[(b)] the homology groups are $H_1(\pi;\Z) \cong \Z$ and $H_k(\pi;\Z) \cong 0$ for $k \geq 2$; and
\item[(c)] the deficiency of $\pi$, defined to be the maximum over all possible presentations of $g-r$, where $g$ is the number of generators and $r$ is the number of relations, is one.
\end{description}
\end{theorem}

\begin{definition}
First, define the set $\P$ to be given by the set of equivalence classes of triples $(\pi,\Y,\xi)$, where $\pi$ is a knot group, by which we mean it satisfies the conditions of Theorem \ref{Thm:knot_groups2}; $\Y$ is a 3-dimensional symmetric Poincar\'{e} triad of finitely generated projective $\zp$-module chain complexes:
\[\xymatrix @R+0.5cm @C+0.5cm {\ar @{} [dr] |{\stackrel{g}{\sim}}
(C,\varphi_C) \ar[r]^{i_-} \ar[d]_{i_+} & (D_-,\delta\varphi_-) \ar[d]^{f_-}\\ (D_+,\delta\varphi_+) \ar[r]^{f_+} & (Y,\Phi),
}\]
as in Definition \ref{Defn:algebraicsetofchaincomplexes}, such that the induced maps
\[f_{\pm} \colon H_*(\Z \otimes_{\zp} D_{\pm}) \toiso H_*(\Z \otimes_{\zp} Y)\]
are isomorphisms; and $\xi$ is a sequence of isomorphisms $\xi = (\xi_0,\xi_1,\xi_2,\dots)$ which fit into the tower:
\[\xymatrix @C+1cm{\pi/\pi^{(1)} \ar[r]^-{\xi_0}_-{\cong} & H_1(\Z[\pi/\pi^{(0)}] \otimes_{\zp} Y) \\
\pi^{(1)}/\pi^{(2)} \ar[r]^-{\xi_1}_-{\cong} & H_1(\Z[\pi/\pi^{(1)}] \otimes_{\zp} Y) \ar[u]_{0} \\
\pi^{(2)}/\pi^{(3)} \ar[r]^-{\xi_2}_-{\cong} & H_1(\Z[\pi/\pi^{(2)}] \otimes_{\zp} Y) \ar[u]_{0} \\
\vdots & \vdots \ar[u]_{0} \\
\pi^{(k)}/\pi^{(k+1)} \ar[r]^-{\xi_k}_-{\cong} & H_1(\Z[\pi/\pi^{(k)}] \otimes_{\zp} Y). \ar[u]_{0} \\
\vdots & \vdots \ar[u]_{0}
}\]
Note the extra condition that the induced vertical maps are the zero maps, using the homomorphisms \[\pi/\pi^{(k)} \to (\pi/\pi^{(k)})/(\pi^{(k-1)}/\pi^{(k)}) \toiso \pi/\pi^{(k-1)}\]
to define the maps and to consider both $$H_1(\Z[\pi/\pi^{(k)}] \otimes_{\zp} Y)$$ and $$H_1(\Z[\pi/\pi^{(k-1)}] \otimes_{\zp} Y)$$ as $\Z[\pi/\pi^{(k)}]$-modules.  This was automatic in the second order case: since we only used the first two levels of the tower, this zero was guaranteed from the $\Z$-homology isomorphism induced by $f_{\pm}$.  It was pointed out to me by Peter Teichner that this was implicit and would need to be considered in the higher order case.

We say that two elements $(\pi,\Y,\xi)$ and $(\pi^\%,\Y^\%,\xi^\%)$ are equivalent if there is an isomorphism $\omega \colon \pi \toiso \pi^\%$, which induces isomorphisms $$\omega_{(k)} \colon \pi/\pi^{(k)} \toiso \pi^\%/(\pi^\%)^{(k)}$$ and $$\omega_{(k/k+1)} \colon \pi^{(k)}/\pi^{(k+1)} \toiso (\pi^\%)^{(k)}/(\pi^\%)^{(k+1)},$$ and an equivalence of triads
\[j \colon \Z[\pi^\%] \otimes_{\zp} \Y \xrightarrow{\sim} \Y^\%\]
such that, for all $k$, the following diagram commutes:
\[\xymatrix @R+1cm @C+1cm{\pi^{(k)}/\pi^{(k+1)} \ar[r]^-{\xi_k}_-{\cong} \ar[dd]_{\omega_{(k/k+1)}}^{\cong} & H_1(\Z[\pi/\pi^{(k)}] \otimes_{\zp} Y) \ar[d]_{\cong}\\
 & H_1(\Z[(\pi^\%)/(\pi^\%)^{(k)}] \otimes_{\Z[\pi^\%]} \Z[\pi^\%] \otimes_{\zp} Y) \ar[d]_{\cong} \\
(\pi^\%)^{(k)}/(\pi^\%)^{(k+1)} \ar[r]^-{\xi^\%_k}_-{\cong} & H_1(\Z[(\pi^\%)/(\pi^\%)^{(k)}] \otimes_{\Z[\pi^\%]} Y^\%).}\]
The upper right vertical map comes from the chain isomorphism
\begin{align*} \Z[\pi/\pi^{(k)}] \otimes_{\zp} Y &\xrightarrow{\omega_{(k)} \otimes \Id} \Z[(\pi^\%)/(\pi^\%)^{(k)}] \otimes_{\zp} Y \\  &\toiso \Z[(\pi^\%)/(\pi^\%)^{(k)}] \otimes_{\Z[\pi^\%]} \Z[\pi^\%] \otimes_{\zp} Y.\end{align*}
The isomorphisms in the above square are a priori isomorphisms of groups, but we require that they are also isomorphisms of $\Z[\pi/\pi^{(k)}]$-modules, using the isomorphism $\omega_{(k)} \colon \pi/\pi^{(k)} \toiso \pi^\%/(\pi^\%)^{(k)}$ to define the module structure on those modules which are ostensibly $\Z[\pi^\%/(\pi^\%)^{(k)}]$-modules.  Note that $\pi^{(k)}/\pi^{(k+1)}$ is abelian, and that $\pi/\pi^{(k)}$ acts on $\pi^{(k)}/\pi^{(k+1)}$ by conjugation, so we can consider $\pi^{(k)}/\pi^{(k+1)}$ as a $\Z[\pi/\pi^{(k)}]$-module.  Similar remarks apply to $\pi^\%$, $\pi^\dag$ and $\pi'$ throughout this appendix.

We conjecture that this defines an equivalence relation.
\qed\end{definition}

We also conjecture that such triples can be combined to give $\P$ the structure of an abelian monoid with a well--defined monoid homomorphism $Knots \to \P$, similarly to Definition \ref{Defn:connectsumalgebraic} and Propositions \ref{prop: fundtriaddefinesanelement},  and \ref{Prop:abelianmonoid}.  We add knot groups $\pi$ and $\pi^\dag$ using the free product to obtain $\pi^\ddag = \pi \ast_{\Z} \pi^\dag$.

Note that for knots whose groups have perfect commutator subgroups (the Alexander polynomial one knots), there is no data in the tower beyond the $k = 0$ level. Each stage in the tower corresponds to the information in a higher--order Alexander module, in the sense of \cite[Section~3]{COT}.

\begin{definition}
We define two triples $(\pi,\Y,\xi), (\pi^\dag,\Y^\dag,\xi^\dag) \in \P$ to be $(n+1)$th order algebraically concordant, or $(n.5)$-solvable equivalent, if there is a finitely presented group $\pi'$ with group homomorphisms
\[j_{\flat} \colon \pi \to \pi'\] and \[j_{\flat}^\dag \colon \pi^\dag \to \pi',\]
which induce homomorphisms
\[(j_{\flat})_{(k/k+1)} \colon \pi^{(k)}/\pi^{(k+1)}  \to (\pi')^{(k)}/(\pi')^{(k+1)},\]
\[(j_{\flat}^\dag)_{(k/k+1)} \colon (\pi^\dag)^{(k)}/(\pi^\dag)^{(k+1)} \to (\pi')^{(k)}/(\pi')^{(k+1)},\]
\[(j_{\flat})_{(k)} \colon \pi/\pi^{(k)} \to \pi'/(\pi')^{(k)}\]
and
\[(j_{\flat}^\dag)_{(k)} \colon \pi^\dag/(\pi^\dag)^{(k)} \to \pi'/(\pi')^{(k)},\]
if there is a finitely generated projective $\zpd$-module chain complex with structure maps $(V,\Theta)$, the requisite chain maps $j,j^{\dag},\delta$, and chain homotopies $\g,\g^{\dag}$ such that there is a 4-dimensional symmetric Poincar\'{e} triad:
\[\xymatrix @R+1cm @C+1.1cm {\ar @{} [dr] |{\stackrel{(\gamma,\g^{\dag})}{\sim}}
\zpd \otimes_{\zp} (E,\phi) \oplus \zpd \otimes_{\zpdag} (E^{\dag},-\phi^{\dag}) \ar[r]^-{(\Id,\Id \otimes \varpi_{E^{\dag}})} \ar[d]_{\Id \otimes \left( \begin{array}{cc} \eta & 0 \\ 0 & \eta^{\dag}\end{array}\right)} & \zpd \otimes_{\Z[\pi]} (E,0) \ar[d]^{\delta}
\\  \zpd \otimes_{\Z[\pi]}(Y,\Phi) \oplus  \zpd \otimes_{\zpdag}(Y^{\dag},-\Phi^{\dag}) \ar[r]^-{(j,j^{\dag})} & (V,\Theta),
}\]
with
\[H_*(\Z \otimes_{\zp} Y) \toiso H_*(\Z \otimes_{\zpd} V) \xleftarrow{\simeq} H_*(\Z \otimes_{\zpdag} Y^\dag),\]
and if there is a sequence of isomorphisms $\xi' = (\xi'_0,\dots,\xi'_n)$, such that there is a tower:
\[\xymatrix @C+1cm{\pi'/\pi'^{(1)} \ar[r]^-{\xi'_0}_-{\cong} & H_1(\Z[\pi'/\pi'^{(0)}] \otimes_{\zpd} V) \\
\pi'^{(1)}/\pi'^{(2)} \ar[r]^-{\xi'_1}_-{\cong} & H_1(\Z[\pi'/\pi'^{(1)}] \otimes_{\zpd} V) \ar[u]_{0} \\
\vdots & \vdots \ar[u]_{0} \\
\pi'^{(n)}/\pi'^{(n+1)} \ar[r]^-{\xi'_n}_-{\cong} & H_1(\Z[\pi'/\pi'^{(n)}] \otimes_{\zpd} V) \ar[u]_{0}
}\]
and such that, for $k = 1,\dots,n$, we have a commutative diagram:
\[\xymatrix @R+1cm @C+0.8cm{\frac{\pi^{(k)}}{\pi^{(k+1)}} \ar[dd]^>>>>>>>>>>>>>>>>{(j_{\flat})_{(k/k+1)}}  \ar[r]^-{\xi_k} & H_1(\Z\left[\frac{\pi}{\pi^{(k)}}\right] \otimes_{\zp} Y) \ar[d] \\
 & H_1(\Z\left[\frac{\pi'}{\pi'^{(k)}}\right] \otimes_{\zp} Y) \ar[d]^{j} \\
\pi'^{(k)}/(\pi')^{(k+1)} \ar[r]^-{\xi'_k} & H_1(\Z\left[\frac{\pi'}{\pi'^{(k)}}\right] \otimes_{\zpd} V),
}\]
and a corresponding commutative diagram with daggers on each occurrence of the letters $\pi, Y, \xi$ and $j$ not in the bottom row.
The upper right vertical map is defined using $(j_{\flat})_{(k)}$ for the map and to define the tensor product in the codomain.  All maps in the above diagram are considered as $\Z[\pi/\pi^{(k)}]$-module homomorphisms, using, when required, the map
\[(j_{\flat})_{(k)} \colon \pi/\pi^{(k)} \to \pi'/(\pi')^{(k)},\]
to define the $\Z[\pi/\pi^{(k)}]$-module structures.
\qed\end{definition}

We conjecture that this defines an equivalence relation, and that this enables us to define an $(n+1)$th order algebraic concordance group $\mc{AC}_{n+1}$.  It seems likely that it will be necessary to have more control on the longitude $l$ of a knot, perhaps including it as part of the data.  We know (Lemma \ref{lemma:longitude}) that $l \in \pi_1(X)^{(2)}$, but that typically $l \notin \pi_1(X)^{(3)}$.  Therefore it will play more of a r\^{o}le in higher order obstruction groups: we can also no longer take $l_b = l_a^{-1}$.

We conjecture that the group $\mc{AC}_{n+1}$ fits into a diagram
\[\xymatrix @R+1cm @C+1cm{
\C \ar[r] \ar @{->>} [d] & \mathcal{AC}_{n+1} \ar@{-->}[d]  \\
\C/\mathcal{F}_{(n.5)} \ar@{-->}[r]  \ar[ur] & \mathcal{COT}_{(\C/n.5)},
}\]
with solid arrows as group homomorphisms and dotted arrows as morphisms of pointed sets, analogously to the results of Chapters \ref{chapter:one_point_five_solvable_knots} and \ref{Chapter:extractingCOTobstructions}.  The pointed set $\mc{COT}_{(\C/n.5)}$ is defined, analogously to $\mc{COT}_{\C/1.5}$, to be the equivalence classes of disjoint unions taken over all the possible choices of representations:
\begin{multline*}\bigsqcup_{(p_1,\dots, p_n) \in \bigoplus_{i=1}^n \frac{\pi^{(i)}}{\pi^{(i+1)}}}\,((N,\theta),\xi)_{(p_1,\dots,p_n)} \\ \subset \bigsqcup_{(p_1,\dots, p_n) \in \bigoplus_{i=1}^n \frac{\pi^{(i)}}{\pi^{(i+1)}}}\,L^4_{\pi,(\Bl_1,\dots,\Bl_n),(p_1,\dots,p_n)}(\Q\G_n,\Q\G_n - \{0\}).\end{multline*}
The higher order Blanchfield forms $\Bl_k$ and universally $(k)$-solvable groups $\G_k$ are defined in \cite[Sections~2~and~3]{COT}.  I have only checked the details for the material in this appendix for $n=1$, but, as mentioned above, hope to prove the general result in future work.



\backmatter
\bibliographystyle{amsalpha}
\bibliography{markbib1}
\include{index}

\end{document}